\def\DateTime{16/March/2013, 18:30 (Kyoto)}
\def\Version{Version 2.0}
\def\yes{\if00}
\def\no{\if01}
\def\ifmonog{\no}
\def\ifpaper{\yes}
\ifmonog
\def\iftenpt{\no}
\def\ifelevenpt{\no}
\def\iftwelvept{\no}
\def\ifmathscr{\yes}
\def\ifutopia{\no}
\fi
\ifpaper
\def\iftenpt{\no}
\def\ifelevenpt{\no}
\def\iftwelvept{\yes}
\def\ifmathscr{\no}
\def\ifutopia{\yes}
\fi

\def\ifquery{\yes}

\ifmonog
\documentclass[leqno]{memo-l}
\fi

\iftenpt
\documentclass[leqno]{amsart}
\fi
\ifelevenpt
\documentclass[leqno,11pt]{amsart}
\fi
\iftwelvept
\documentclass[leqno,12pt]{amsart}
\fi

\usepackage{amssymb}
\usepackage{amscd}
\usepackage{verbatim}
\usepackage[all]{xy}

\usepackage{amsmidx}

\ifmathscr
\usepackage{mathrsfs}
\fi

\ifutopia
\usepackage[adobe-utopia]{mathdesign}
\usepackage[T1]{fontenc}
\fi

\ifelevenpt
\fi

\iftwelvept
\fi

\ifmonog
\theoremstyle{plain}
\newtheorem{Theorem}{Theorem}[section]
\newtheorem{Proposition}[Theorem]{Proposition}
\newtheorem{Lemma}[Theorem]{Lemma}
\newtheorem{PropDef}[Theorem]{Proposition-Definition}
\newtheorem{Corollary}[Theorem]{Corollary}
\newtheorem{Sublemma}[Theorem]{Sublemma}
\newtheorem{Claim}{Claim}[Theorem]

\theoremstyle{definition}
\newtheorem{Definition}[Theorem]{Definition}
\newtheorem{Remark}[Theorem]{Remark}
\newtheorem{Example}[Theorem]{Example}
\newtheorem{Conjecture}[Theorem]{Conjecture}
\newtheorem{Assertion}[Theorem]{Assertion}
\newtheorem{Condition}[Theorem]{Condition}
\newtheorem{Problem}[Theorem]{Problem}
\newtheorem{Question}[Theorem]{Question}
\newtheorem{FQuestion}[Theorem]{Fundamental question}
\newtheorem{Observation}[Theorem]{Observation}
\newtheorem{Assumption}[Theorem]{Assumption}
\newtheorem{Conj}[Claim]{Conjecture}
\newtheorem{Convention}{}[section]
\fi

\ifpaper
\theoremstyle{plain}
\newtheorem{Theorem}{Theorem}[subsection]
\newtheorem{Proposition}[Theorem]{Proposition}
\newtheorem{Lemma}[Theorem]{Lemma}
\newtheorem{PropDef}[Theorem]{Proposition-Definition}
\newtheorem{Corollary}[Theorem]{Corollary}

\newtheorem{Claim}{Claim}[Theorem]

\theoremstyle{definition}
\newtheorem{Definition}[Theorem]{Definition}
\newtheorem{Remark}[Theorem]{Remark}

\fi

\ifmonog
\renewcommand{\thesection}{\arabic{chapter}.\arabic{section}}
\renewcommand{\thesubsection}{\arabic{chapter}.\arabic{section}.\arabic{subsection}}
\renewcommand{\theTheorem}{\arabic{chapter}.\arabic{section}.\arabic{Theorem}}
\renewcommand{\theClaim}{\arabic{chapter}.\arabic{section}.\arabic{Theorem}.\arabic{Claim}}
\renewcommand{\theequation}{\arabic{chapter}.\arabic{section}.\arabic{Theorem}.\arabic{Claim}}
\fi

\ifpaper
\renewcommand{\thesubsection}{\arabic{section}.\arabic{subsection}}

\renewcommand{\theTheorem}{\arabic{section}.\arabic{subsection}.\arabic{Theorem}}
\renewcommand{\theClaim}{\arabic{section}.\arabic{subsection}.\arabic{Theorem}.\arabic{Claim}}
\renewcommand{\theequation}{\arabic{section}.\arabic{subsection}.\arabic{Theorem}.\arabic{Claim}}
\fi

\def\rom{\textup}
\newcommand{\ZZ}{{\mathbb{Z}}}
\newcommand{\QQ}{{\mathbb{Q}}}
\newcommand{\RR}{{\mathbb{R}}}
\newcommand{\KK}{{\mathbb{K}}}
\newcommand{\CC}{{\mathbb{C}}}
\newcommand{\PP}{{\mathbb{P}}}

\newcommand{\MM}{{\mathbb{M}}}
\ifmathscr
\newcommand{\AAA}{{\mathscr{A}}}
\newcommand{\BBB}{{\mathscr{B}}}

\newcommand{\DDD}{{\mathscr{D}}}
\newcommand{\EEE}{{\mathscr{E}}}
\newcommand{\FFF}{{\mathscr{F}}}
\newcommand{\HHH}{{\mathscr{H}}}
\newcommand{\III}{{\mathscr{I}}}
\newcommand{\JJJ}{{\mathscr{J}}}
\newcommand{\LLL}{{\mathscr{L}}}
\newcommand{\MMM}{{\mathscr{M}}}
\newcommand{\NNN}{{\mathscr{N}}}
\newcommand{\OOO}{{\mathscr{O}}}
\newcommand{\PPP}{{\mathscr{P}}}
\newcommand{\QQQ}{{\mathscr{Q}}}
\newcommand{\VVV}{{\mathscr{V}}}
\newcommand{\XXX}{{\mathscr{X}}}
\newcommand{\YYY}{{\mathscr{Y}}}
\newcommand{\ZZZ}{{\mathscr{Z}}}
\else
\newcommand{\AAA}{{\mathcal{A}}}
\newcommand{\BBB}{{\mathcal{B}}}

\newcommand{\DDD}{{\mathcal{D}}}
\newcommand{\EEE}{{\mathcal{E}}}
\newcommand{\FFF}{{\mathcal{F}}}
\newcommand{\HHH}{{\mathcal{H}}}
\newcommand{\III}{{\mathcal{I}}}
\newcommand{\JJJ}{{\mathcal{J}}}
\newcommand{\LLL}{{\mathcal{L}}}
\newcommand{\MMM}{{\mathcal{M}}}
\newcommand{\NNN}{{\mathcal{N}}}
\newcommand{\OOO}{{\mathcal{O}}}
\newcommand{\PPP}{{\mathcal{P}}}
\newcommand{\QQQ}{{\mathcal{Q}}}
\newcommand{\VVV}{{\mathcal{V}}}
\newcommand{\XXX}{{\mathcal{X}}}
\newcommand{\YYY}{{\mathcal{Y}}}
\newcommand{\ZZZ}{{\mathcal{Z}}}
\fi
\newcommand{\Proj}{\operatorname{Proj}}

\newcommand{\Pic}{\operatorname{Pic}}
\newcommand{\Rat}{\operatorname{Rat}}

\newcommand{\Hom}{\operatorname{Hom}}

\newcommand{\Coker}{\operatorname{Coker}}
\newcommand{\Spec}{\operatorname{Spec}}

\newcommand{\Supp}{\operatorname{Supp}}

\newcommand{\codim}{\operatorname{codim}}
\newcommand{\mult}{\operatorname{mult}}
\newcommand{\lentgh}{\operatorname{length}}

\newcommand{\trdeg}{\operatorname{trdeg}}
\newcommand{\rank}{\operatorname{rk}}
\newcommand{\ord}{\operatorname{ord}}

\newcommand{\adeg}{\widehat{\operatorname{deg}}}
\newcommand{\zeros}{\operatorname{div}}

\newcommand{\Bs}{\operatorname{Bs}}

\newcommand{\Div}{\operatorname{Div}}
\newcommand{\aDiv}{\widehat{\operatorname{Div}}}
\newcommand{\WDiv}{\operatorname{WDiv}}

\newcommand{\vol}{\operatorname{vol}}
\newcommand{\avol}{\widehat{\operatorname{vol}}}
\newcommand{\acvol}{\widehat{\operatorname{vol}}_{\chi}}
\newcommand{\aH}{\hat{H}^0}
\newcommand{\ah}{\hat{h}^0}
\newcommand{\achi}{\hat{\chi}}

\newcommand{\Tpsh}{\operatorname{PSH}}

\newcommand{\DVal}{\operatorname{DVal}}

\newcommand{\an}{\operatorname{an}}
\newcommand{\ad}{\operatorname{a}}

\newcommand{\integrable}{\operatorname{int}}

\newcommand{\rest}[2]{\left.{#1}\right\vert_{{#2}}}
\newcommand{\srest}[2]{{#1}\vert_{{#2}}}
\ifquery
\def\query#1{\setlength\marginparwidth{65pt} 
\marginpar{\raggedright\fontsize{7.81}{9} 
\selectfont\upshape\hrule\smallskip 
#1\par\smallskip\hrule}} 

\else
\def\query#1{}
\fi
\newcommand{\frontmatterforspececialeqn}{%
\ifmonog\renewcommand{\theequation}{\arabic{chapter}.\arabic{section}.\arabic{Theorem}}\fi%
\ifpaper\renewcommand{\theequation}{\arabic{section}.\arabic{subsection}.\arabic{Theorem}}\fi%
\addtocounter{Theorem}{1}}
\newcommand{\backmatterforspececialeqn}{%
\ifmonog\renewcommand{\theequation}{\arabic{chapter}.\arabic{section}.\arabic{Theorem}.\arabic{Claim}}\fi%
\ifpaper\renewcommand{\theequation}{\arabic{section}.\arabic{subsection}.\arabic{Theorem}.\arabic{Claim}}\fi}
\newcommand{\frontmatterforspececialclaim}{%
\ifmonog\renewcommand{\theClaim}{\arabic{chapter}.\arabic{section}.\arabic{Theorem}}\fi%
\ifpaper\renewcommand{\theClaim}{\arabic{section}.\arabic{subsection}.\arabic{Theorem}}\fi%
\addtocounter{Theorem}{1}}
\newcommand{\backmatterforspececialclaim}{%
\ifmonog\renewcommand{\theClaim}{\arabic{chapter}.\arabic{section}.\arabic{Theorem}.\arabic{Claim}}\fi%
\ifpaper\renewcommand{\theClaim}{\arabic{section}.\arabic{subsection}.\arabic{Theorem}.\arabic{Claim}}\fi}

\ifmonog
\makeindex{AdelDiv-Subject}
\makeindex{AdelDiv-Symbol}
\def\AdelDivSubject{AdelDiv-Subject}
\def\AdelDivSymbol{AdelDiv-Symbol}
\fi

\ifpaper
\makeindex{AdelDiv}
\def\AdelDivSubject{AdelDiv}
\def\AdelDivSymbol{AdelDiv}
\fi

\begin{document}

\ifmonog
\frontmatter
\fi

\title[Adelic divisors on arithmetic varieties]%
{Adelic divisors on arithmetic varieties%
\ifmonog\footnote{Date: \DateTime, (\Version)}\fi}
\author{Atsushi Moriwaki}
\address{Department of Mathematics, Faculty of Science,
Kyoto University, Kyoto, 606-8502, Japan}
\email{moriwaki@math.kyoto-u.ac.jp}
\ifpaper
\date{\DateTime, (\Version)}
\fi
\begin{abstract}
In this article, we generalize several fundamental results for arithmetic divisors,
such as the continuity of the volume function, the generalized Hodge index theorem,
Fujita's approximation theorem for arithmetic divisors and
Zariski decompositions for arithmetic divisors on arithmetic surfaces,
to the case of the adelic arithmetic divisors.
\end{abstract}


\maketitle

\ifpaper
\setcounter{tocdepth}{1}
\fi
\tableofcontents

\ifmonog
\mainmatter
\fi

\ifmonog\chapter*{Introduction}\fi
\ifpaper\section*{Introduction}\fi

The theory of birational Arakelov geometry has advanced 
tremendously over the last decade, 
such as
the continuity of the volume function, the generalized Hodge index theorem,
Fujita's approximation theorem for arithmetic divisors,
Zariski decompositions for arithmetic divisors on arithmetic surfaces and so on.
Besides them, non-Archimedean Arakelov geometry is also well-developed using Berkovich analytic spaces.
In this article, we would like to generalize the fundamental results for arithmetic divisors to
the case of adelic arithmetic divisors. 

\ifmonog\section{Birational Arakelov geometry}\fi
\ifpaper\subsection{Birational Arakelov geometry}\fi

Let $\XXX$ be a $(d+1)$-dimensional, generically smooth, projective, normal arithmetic variety,
that is, $\XXX$ is a projective and flat normal integral scheme over $\ZZ$ such that
$\XXX$ is smooth over $\QQ$ and the Krull dimension of $\XXX$ is $d+1$.
A pair $\overline{\DDD} = (\DDD, g_{\infty})$ is called an {\em arithmetic $\RR$-Cartier divisor of
$C^0$-type on $\XXX$}
\index{\AdelDivSubject}{arithmetic R-Cartier divisor of C^0-type@arithmetic $\RR$-Cartier divisor of $C^0$-type}%
if the following conditions are satisfied:
\begin{enumerate}
\renewcommand{\labelenumi}{(\roman{enumi})}
\item
The first $\DDD$ is an $\RR$-Cartier divisor on $\XXX$, that is,
$\DDD = a_1 \DDD_1 + \cdots + a_r \DDD_r$ for some
Cartier divisors $\DDD_1, \ldots, \DDD_r$ on $\XXX$ and $a_1, \ldots, a_r \in \RR$.

\item
The second $g_{\infty}$ is a real valued continuous function on $(\XXX \setminus \Supp(\DDD))(\CC)$ such that,
for each $x \in \XXX(\CC)$,
$g_{\infty} + \sum_{i=1}^r a_i \log \vert f_i \vert^2$ extends to a continuous function
around $x$, where $f_1, \ldots, f_r$ are local equations of $\DDD_1, \ldots, \DDD_r$ at $x$, respectively.
In addition, $g_{\infty}$ is invariant under the complex conjugation map.
\end{enumerate}
Let $\Rat(\XXX)$ be the rational function field of $\XXX$.
We define $H^0(\XXX, \DDD)$ to be
\[
H^0(\XXX, \DDD) := \left\{ \phi \in \Rat(\XXX)^{\times} \mid
\DDD + (\phi) \geq 0 \right\} \cup \{ 0 \}.
\]%
\index{\AdelDivSymbol}{0H:H^0(XXX,DDD)@$H^0(\XXX, \DDD)$}%
Note that $H^0(\XXX, \DDD)$ is a finitely generated $\ZZ$-module.
For $\phi \in H^0(\XXX, \DDD)$, we can see that $\vert \phi \vert \exp(-g_{\infty}/2)$ extends to
a continuous function $\vartheta$ on $\XXX(\CC)$, so that
\[
\sup \left\{ \vartheta(x) \mid x \in \XXX(\CC) \right\}
\]
is denoted by
$\Vert \phi \Vert_{g_{\infty}}$.
\index{\AdelDivSymbol}{0n:Vert phi Vert_{g_{infty}}@$\Vert \phi \Vert_{g_{\infty}}$}%
The volume $\avol(\overline{\DDD})$ of $\overline{\DDD}$, by definition,
is given by
\[
\avol(\overline{\DDD}) :=
\limsup_{n \to \infty} 
\frac{\log \# \left\{ \phi \in H^0(\XXX, n\DDD) \mid \Vert \phi \Vert_{n g_{\infty}} \leq 1 \right\}}%
{n^{d+1}/(d+1)!}.
\]%
\index{\AdelDivSymbol}{0v:avol(overline{DDD})@$\avol(\overline{\DDD})$}%
It is known that the volume function $\avol$ 
has the following fundamental properties (for details, see \cite{MoArZariski}):
\begin{enumerate}
\renewcommand{\labelenumi}{(\arabic{enumi})}
\item (Finiteness)
$\avol(\overline{\DDD}) < \infty$ (\cite{MoCont}, \cite{MoContExt}).

\item
(Limit theorem) 
${\displaystyle \avol(\overline{\DDD}) =
\lim_{n \to \infty} 
\frac{\log \# \left\{ \phi \in H^0(\XXX, n \DDD) \mid \Vert \phi \Vert_{n g_{\infty}} \leq 1 \right\}}%
{n^{d+1}/(d+1)!}
}$ (\cite{HChen}, \cite{MoContExt}).

\item (Positive homogeneity)
$\avol(a \overline{\DDD}) = a^{d+1} \avol(\overline{\DDD})$ for $a \in \RR_{\geq 0}$
(\cite{MoCont}, \cite{MoContExt}).

\item
(Continuity) The volume function
$\avol$ is continuous in the following sense:
Let $\overline{\DDD}_1,\ldots,\overline{\DDD}_r$, $\overline{\AAA}_1,\ldots, \overline{\AAA}_s$ be arithmetic $\RR$-Cartier divisors of $C^0$-type on $\XXX$.
For a compact subset $B$ in $\RR^r$ and a positive number $\epsilon$, 
there are positive numbers $\delta$ and $\delta'$ such that
\[
\left\vert \avol\left( \sum_{i=1}^r a_i \overline{\DDD}_i + \sum_{j=1}^s \delta_j \overline{\AAA}_j + (0, \phi) \right) - \avol\left( \sum_{i=1}^r a_i \overline{\DDD}_i  \right) \right| \leq \epsilon
\]
for all $a_1, \ldots, a_r, \delta_1, \ldots, \delta_s \in \RR$ and $\phi \in C^0(X)$ with
$(a_1, \ldots, a_r) \in B$, $\vert \delta_1 \vert + \cdots + \vert \delta_s \vert \leq \delta$ and $\Vert \phi \Vert_{\sup} \leq \delta'$ (\cite{MoCont}, \cite{MoContExt}).
\end{enumerate}
Here we would like to introduce several kinds of the positivity of an arithmetic $\RR$-Cartier divisor
$\overline{\DDD}$ of $C^0$-type on $\XXX$.
\begin{enumerate}
\item[$\bullet$] Big:
$\avol(\overline{\DDD}) > 0$.

\item[$\bullet$] Relatively nef: the first Chern current $c_1(\overline{\DDD})$ is positive and
$\DDD$ is relatively nef with respect to $\XXX \to \Spec(\ZZ)$, that is,
$\deg(\rest{\DDD}{C}) \geq 0$ 
for all vertical $1$-dimensional closed integral subschemes $C$ of $\XXX$.

\item[$\bullet$] Nef: $\overline{\DDD}$ is relatively nef and $\adeg(\srest{\overline{\DDD}}{C}) \geq 0$
for all horizontal $1$-dimensional closed integral subschemes $C$ of $\XXX$.
\end{enumerate}
In addition, $\overline{\DDD}$ is said to be {\em integrable} 
\index{\AdelDivSubject}{integrable arithmetic R-Cartier divisor@integrable arithmetic $\RR$-Cartier divisor}%
if $\overline{\DDD} = \overline{\DDD}' - \overline{\DDD}''$ for some
relatively nef arithmetic $\RR$-Cartier divisors $\overline{\DDD}'$ and
$\overline{\DDD}''$ of $C^0$-type.
For integrable arithmetic $\RR$-Cartier divisors
$\overline{\DDD}_1, \ldots, \overline{\DDD}_{d+1}$ of $C^0$-type,
the arithmetic intersection number
$\adeg \left( \overline{\DDD}_1 \cdots \overline{\DDD}_{d+1} \right)$ is well-defined
(cf. \cite[Subsection~6.4]{MoArZariski}, \cite[Subsection~2.1]{MoD}).
\index{\AdelDivSymbol}{0d:adeg(overline{DDD}_1 cdots overline{DDD}_{d+1})@$\adeg(\overline{\DDD}_1 \cdots \overline{\DDD}_{d+1})$}%
The following fundamental results were obtained by several authors such as
Faltings, Gillet-Soul\'{e}, S. Zhang, Moriwaki, H. Chen, X. Yuan and so on:

\begin{enumerate}
\renewcommand{\labelenumi}{(\arabic{enumi})}
\setcounter{enumi}{4}
\item
\rom{(Generalized Hodge index theorem)}
If $\overline{\DDD}$ is relatively nef, then
\[
\adeg(\overline{\DDD}^{d+1}) \leq
\avol(\overline{\DDD}).
\]
Moreover, if $\overline{\DDD}$ is nef, then
$\adeg(\overline{\DDD}^{d+1}) =
\avol(\overline{\DDD})$ (\cite{FaCAS}, \cite{GSRR}, \cite{ZhPos}, \cite{MoCont}, \cite{MoArZariski}).

\item
\rom{(Fujita's approximation theorem for arithmetic divisors)}
If $\overline{\DDD}$ is big, then,
for any positive number $\epsilon$, there are a birational morphism $\mu : \YYY \to \XXX$
of generically smooth, normal and projective arithmetic varieties and
a nef arithmetic $\RR$-Cartier divisor $\overline{\QQQ}$ of $C^0$-type on $\YYY$ such that
\[
\overline{\QQQ} \leq \mu^*(\overline{\DDD})
\quad\text{and}\quad 
\avol(\overline{\DDD}) - \epsilon \leq
\avol(\overline{\QQQ}) \leq \avol(\overline{\DDD})
\]
(\cite{HChen}, \cite{YuanVol}, \cite{MoArLin}).

\item
\rom{(Zariski decompositions for arithmetic divisors on arithmetic surfaces)}
We assume that $d=1$ and $\XXX$ is regular.
Let $\Upsilon(\overline{\DDD})$ be the set of all
nef arithmetic $\RR$-Cartier divisors $\overline{\LLL}$ of
$C^0$-type on $\XXX$ with $\overline{\LLL} \leq \overline{\DDD}$.
\index{\AdelDivSymbol}{0U:Upsilon(overline{DDD})@$\Upsilon(\overline{\DDD})$}%
If $\Upsilon(\overline{\DDD}) \not= \emptyset$,
then there is the greatest element $\overline{\QQQ}$ of
$\Upsilon(\overline{\DDD})$, that is,
$\overline{\QQQ} \in \Upsilon(\overline{\DDD})$ and
$\overline{\LLL} \leq \overline{\QQQ}$ for all $\overline{\LLL} \in \Upsilon(\overline{\DDD})$
(\cite{MoArZariski}, \cite{MoCharNef}).
\end{enumerate}
The purpose of this article is to generalize the above results to adelic
arithmetic divisors.

\ifmonog\section{Green functions on analytic spaces over a compete discrete valuation field}\fi
\ifpaper\subsection{Green functions on analytic spaces over a compete discrete valuation field}\fi

Let $k$ be a field and $v$ a non-trivial complete discrete valuation of $k$.
Let $X$ be a projective and geometrically integral variety over $k$.
Let $X^{\an}$ be the analytification of $X$ in the sense of
Berkovich \cite{Be}.
\index{\AdelDivSymbol}{0X:X^{an}@$X^{\an}$}%
Note that $X^{\an}$ is a compact Hausdorff space.
Let $\Rat(X)$ be the rational function field of $X$.
Let $D$ be an $\RR$-Cartier divisor on $X$, that is, 
$D = a_1 D_1 + \cdots + a_r D_r$ for some Cartier divisors
$D_1, \ldots, D_r$ on $X$ and $a_1, \ldots, a_r \in \RR$.
Let $X = \bigcup_{i=1}^N U_i$ be an affine open covering of $X$ such that
each $D_j$ is given by $f_{ji} \in \Rat(X)^{\times}$ on $U_i$ for $j=1, \ldots, r$.
We say a continuous function 
\[
g : X^{\an} \setminus \bigcup\nolimits_{j=1}^r \Supp(D_j)^{\an} \to \RR
\]
is a {\em $D$-Green function of $C^0$-type on $X^{\an}$}
\index{\AdelDivSubject}{Green function of C^0-type@Green function of $C^0$-type}%
if $g + \sum_{j=1}^r a_{j} \log \vert f_{ji} \vert^2$ 
extends to a continuous function on $U_i^{\an}$ for each $i = 1, \ldots, N$.

Let $\XXX$ be a model of $X$ over $\Spec(k^{\circ})$, that is,
$\XXX$ is a projective and flat integral scheme over $\Spec(k^{\circ})$ such that
the generic fiber of $\XXX \to \Spec(k^{\circ})$ is $X$,
where 
\[
k^{\circ} := \{ f \in k \mid v(f) \leq 1 \}.
\]
We assume that there are Cartier divisors $\DDD_1, \ldots, \DDD_r$ on $\XXX$
such that $\DDD_j \cap X = D_j$ for $j=1, \ldots, r$.
We set $\DDD := a_1 \DDD_1 + \cdots + a_r \DDD_r$.
The pair $(\XXX, \DDD)$ is called a {\em model of $(X, D)$}.
\index{\AdelDivSubject}{model of (X, D)@model of $(X, D)$}%
For $x \in X^{\an} \setminus \bigcup_{j=1}^r \Supp(D_j)^{\an}$, 
let $f_1, \ldots, f_r$ be local equations of $\DDD_1, \ldots, \DDD_r$ at 
$\xi = r_{\XXX}(x)$, respectively,
where $r_{\XXX} : X^{\an} \to \XXX_{\circ}$ is the reduction map and
$\XXX_{\circ}$ is the central fiber of $\XXX \to \Spec(k^{\circ})$.
We define $g_{(\XXX,\, \DDD)}(x)$ to be
\[
g_{(\XXX,\,\DDD)}(x) := -\sum_{j=1}^r a_j \log \vert f_j (x) \vert^2.
\]%
\index{\AdelDivSymbol}{0g:g_{(XXX,DDD)}@$g_{(\XXX,\,\DDD)}$}%
It is easy to see that $g_{(\XXX,\,\DDD)}$ is a $D$-Green function of $C^0$-type on $X^{\an}$.
We call it the {\em Green function induced by the model $(\XXX, \DDD)$}.
\index{\AdelDivSubject}{Green function induced by model@Green function induced by model}%

We say a $D$-Green function $g$ is {\em of $(C^0 \cap \Tpsh)$-type} 
\index{\AdelDivSubject}{of (C^0 \cap \Tpsh)-type@of $(C^0 \cap \Tpsh)$-type}%
if $D$ is nef and
there is a sequence $\{ (\XXX_n, \DDD_n) \}_{n=1}^{\infty}$ of
models of $(X, D)$ with the following properties:
\begin{enumerate}
\renewcommand{\labelenumi}{(\roman{enumi})}
\item
For each $n \geq 1$, $\DDD_n$ is relatively nef with respect to $\XXX_n \to \Spec(k^{\circ})$.

\item
If we set $\phi_n = g_{(\XXX_n,\, \DDD_n)} - g$, then $\lim_{n\to\infty} \Vert \phi_n \Vert_{\sup} = 0$.
\end{enumerate}

\ifmonog\section{Adelic arithmetic divisors}\fi
\ifpaper\subsection{Adelic arithmetic divisors}\fi

Let $K$ be a number field and $O_K$ the ring of integers in $K$.
We denote the set of all maximal ideals of $O_K$ by $M_K$.
For $P \in M_K$, 
the valuation $v_P$ of $K$ at $P$ is given by
\[
v_P(f) = \#(O_K/P)^{-\ord_P(f)}.
\]%
\index{\AdelDivSymbol}{0v:v_P@$v_P$}%
Let $K_P$ be the completion of $K$ with respect to $v_P$.
\index{\AdelDivSymbol}{0K:K_P@$K_P$}%
Let $X$ be a $d$-dimensional, projective, smooth and geometrically integral
variety over $K$ and let
$X_P := X \times_{\Spec(K)} \Spec(K_P)$, which is also
a projective, smooth and geometrically integral
variety over $K_P$.
\index{\AdelDivSymbol}{0X:X_P@$X_P$}%
Let $X(\CC)$ be the set of all $\CC$-valued points of $X$, that is,
\[
X(\CC) := \left\{ x : \Spec(\CC) \to X \mid \text{$x$ is a morphism as schemes} \right\}.
\]%
\index{\AdelDivSymbol}{0X:X(CC)@$X(\CC)$}%
Let $F_{\infty} : X(\CC) \to X(\CC)$ be the complex conjugation map, that is,
for $x \in X(\CC)$, $F_{\infty}(x)$ is given by the composition of
morphisms $\Spec(\CC) \overset{-^{a}}{\longrightarrow} \Spec(\CC)$ and
$\Spec(\CC) \overset{x}{\to} X$,
where $\Spec(\CC) \overset{-^{a}}{\to} \Spec(\CC)$ is the morphism induced by
the complex conjugation. 
\index{\AdelDivSymbol}{0F:F_{infty}@$F_{\infty}$}%
The space of $F_{\infty}$-invariant real valued continuous functions on $X(\CC)$
is denoted by $C^0_{F_{\infty}}(X(\CC))$, that is,
\[
C^0_{F_{\infty}}(X(\CC)) := \left\{ \varphi \in C^0(X(\CC)) \mid \varphi \circ F_{\infty} = \varphi \right\}.
\]%
\index{\AdelDivSymbol}{0C:C^0_{F_{infty}}(X(CC))@$C^0_{F_{\infty}}(X(\CC))$}%
A pair $\overline{D} = (D, g)$ of an $\RR$-Cartier divisor $D$ on $X$ and a collection
of Green functions 
\[
g = \left\{ g_P \right\}_{P \in M_K} 
\cup \left\{ g_{\infty} \right\}
\]
is called an {\em adelic arithmetic $\RR$-Cartier divisor of $C^0$-type on $X$} 
\index{\AdelDivSubject}{adelic arithmetic R-Cartier divisor of C^0-type@adelic arithmetic $\RR$-Cartier divisor of $C^0$-type}%
if the following conditions are satisfied:
\begin{enumerate}
\renewcommand{\labelenumi}{(\arabic{enumi})}
\item
For each $P \in M_K$,
$g_P$ is a $D$-Green function of $C^0$-type on $X^{\an}_P$.
In addition, there are a non-empty open set $U$ of $\Spec(O_K)$,
a normal model $\XXX_U$ of $X$ over $U$ and an $\RR$-Cartier divisor $\DDD_U$ on $\XXX_U$
such that $\DDD_U \cap X = D$ and $g_P$ is a $D$-Green function induced by the model
$(\XXX_U, \DDD_U)$ for all $P \in U \cap M_K$. 

\item
The function $g_{\infty}$ is an $F_{\infty}$-invariant $D$-Green function of $C^0$-type on $X(\CC)$.
\end{enumerate}
For simplicity, a collection
of Green functions $g = \{ g_P \}_{P \in M_K} \cup \{ g_{\infty} \}$
is often expressed by the following symbol:
\[
g = \sum_{P \in M_K} g_P [P] + g_{\infty} [\infty].
\]
We denote the space of all adelic arithmetic $\RR$-Cartier divisors of $C^0$-type on $X$ by
$\aDiv_{C^0}^{\ad}(X)_{\RR}$.
\index{\AdelDivSymbol}{0Div:aDiv_{C^0}^{ad}(X)_{RR}@$\aDiv_{C^0}^{\ad}(X)_{\RR}$}%

Let $\overline{D} = (D, g)$ be an adelic arithmetic $\RR$-Cartier divisor of $C^0$-type on $X$.
We define $H^0(X, D)$ to be
\[
H^0(X, D) := \left\{ \phi \in \Rat(X)^{\times} \mid D + (\phi) \geq 0 \right\} \cup \{ 0 \}.
\]%
\index{\AdelDivSymbol}{0H:H^0(X,D)@$H^0(X, D)$}%
For $\phi \in H^0(X, D)$ and $\wp \in M_K \cup \{ \infty \}$,
$\vert \phi \vert \exp(-g_{\wp}/2)$ extends to a continuous function, so that its supremum
is denoted by $\Vert \phi \Vert_{g_{\wp}}$. 
\index{\AdelDivSymbol}{0n:Vert phi Vert_{g_{wp}}@$\Vert \phi \Vert_{g_{\wp}}$}%
The set $\aH(X, \overline{D})$ of small sections 
of $\overline{D}$ and
the volume $\avol(\overline{D})$ of $\overline{D}$ are defined by
\[
\begin{cases}
\aH(X, \overline{D}) := \left\{ \phi \in H(X, D) \mid 
\text{$\Vert \phi \Vert_{g_{\wp}} \leq 1$ for all $\wp \in M_K \cup \{ \infty \}$} \right\},\\[2.0ex]
{\displaystyle \avol(\overline{D}) := \limsup_{n\to\infty}\frac{\log \# \aH(X, n \overline{D})}%
{n^{d+1}/(d+1)!}},
\end{cases}
\]
respectively.
\index{\AdelDivSymbol}{0H:aH(X,overline{D})@$\aH(X, \overline{D})$}%
\index{\AdelDivSymbol}{0v:avol(overline{D})@$\avol(\overline{D})$}%
Similarly as given for arithmetic divisors,
we can also introduce several kinds of the positivity of $\overline{D}$ as follows:
\begin{enumerate}
\item[$\bullet$] Big:
$\avol(\overline{D}) > 0$.

\item[$\bullet$] Relatively nef: 
$g_{P}$ is of $(C^0 \cap \Tpsh)$-type for
all $P \in M_K$ and
the first Chern current $c_1(D, g_{\infty})$
is positive.

\item[$\bullet$] Nef: $\overline{D}$ is relatively nef and 
the height function arising from $\overline{D}$ is non-negative.
\end{enumerate}
Further, $\overline{D}$ is said to be {\em integrable} 
\index{\AdelDivSubject}{integrable adelic arithmetic R-Cartier divisor@integrable adelic arithmetic $\RR$-Cartier divisor}%
if $\overline{D} = \overline{D}' - \overline{D}''$ for some
relatively nef adelic arithmetic $\RR$-Cartier divisors $\overline{D}'$ and
$\overline{D}''$ of $C^0$-type on $X$.
For integrable adelic arithmetic $\RR$-Cartier divisors
$\overline{D}_1, \ldots, \overline{D}_{d+1}$ of $C^0$-type on $X$, 
the arithmetic intersection number
\[
\adeg \left( \overline{D}_1 \cdots \overline{D}_{d+1} \right)
\]
can be defined 
\ifmonog(cf. Section~\ref{subsec:global:int:number}). \fi
\ifpaper(cf. Subsection~\ref{subsec:global:int:number}). \fi
\index{\AdelDivSymbol}{0d:adeg(overline{D}_1 cdots overline{D}_{d+1})@$\adeg(\overline{D}_1 \cdots \overline{D}_{d+1})$}%

\ifmonog\section{Main results}\fi
\ifpaper\subsection{Main results}\fi

Let $X$ be a $d$-dimensional, projective, smooth and geometrically integral
variety over a number field $K$.
The following theorems are the main results of this article.
Theorem~\ref{thm:cont:volume:intro},
Theorem~\ref{thm:G:H:I:T:adelic:arith:intro},
Theorem~\ref{thm:Fujita:approx:adel:intro} and
Theorem~\ref{thm:Zariski:decomp:adelic:arithmetic:divisor:intro}
are generalizations of (4), (5), (6) and (7), respectively.
The properties (1), (2) and (3) also hold for adelic arithmetic divisors
(cf. Theorem~\ref{thm:avol:lim}).
Several similar results on arithmetic toric varieties are known. For
details, see \cite{BPS} and \cite{BMPS}.
The adelic version of Fujita's approximation theorem has been already established by
Boucksom and Chen \cite{BC}.
In this article, we give another proof of it and generalize it to $\RR$-divisors.
Further, Theorem~\ref{thm:char:nef:general:intro} is a generalization of the result proved in
\cite{MoCharNef}.

\begin{Theorem}[Continuity of the volume function for adelic arithmetic divisors]
\label{thm:cont:volume:intro}
The volume function $\avol : \aDiv_{C^0}^{\ad}(X)_{\RR} \to \RR$ is continuous in the following sense:
Let $\overline{D}_1, \ldots, \overline{D}_r$, $\overline{A}_1, \ldots, \overline{A}_{r'}$ be 
adelic arithmetic $\RR$-Cartier divisors of $C^0$-type on $X$.
Let $\{ P_1, \ldots, P_s \}$ be a finite subset of $M_K$.
For a compact subset $B$ in $\RR^r$ and a positive number $\epsilon$, 
there are positive numbers $\delta$ and $\delta'$ such that
\[
\left\vert \avol\left(\sum_{i=1}^r a_i \overline{D}_i + \sum_{j=1}^{r'} \delta_j \overline{A}_j + 
\left(0, \sum_{l=1}^{s} \varphi_{P_l} [P_l] + \varphi_{\infty} [\infty] \right) \right) -
\avol\left(\sum_{i=1}^r a_i \overline{D}_i \right) \right\vert
\leq \epsilon
\]
holds for all $a_1, \ldots, a_r, \delta_1, \ldots, \delta_{r'} \in \RR$, 
$\varphi_{P_1} \in C^0(X_{P_1}^{\an}), \ldots, \varphi_{P_s} \in C^0(X_{P_s}^{\an})$
and $\varphi_{\infty} \in C^0_{F_{\infty}}(X(\CC))$
with $(a_1, \ldots, a_r) \in B$, $\sum_{j=1}^{r'} \vert \delta_j \vert \leq \delta$
and $\sum_{l=1}^s \Vert \varphi_{P_l} \Vert_{\sup} + \Vert \varphi_{\infty} \Vert_{\sup}
\leq \delta'$.
\end{Theorem}

\begin{Theorem}[Generalized Hodge index theorem for adelic arithmetic divisors]
\label{thm:G:H:I:T:adelic:arith:intro}
Let $\overline{D}$ be a relatively nef adelic arithmetic $\RR$-Cartier divisor of $C^0$-type on $X$.
Then
\[
\adeg(\overline{D}^{d+1}) \leq \avol(\overline{D}).
\]
Moreover, if $\overline{D}$ is nef, then
$\adeg(\overline{D}^{d+1}) = \avol(\overline{D})$.
\end{Theorem}

\begin{Theorem}[Fujita's approximation theorem for adelic arithmetic divisors]
\label{thm:Fujita:approx:adel:intro}
Let $\overline{D}$ be a big adelic arithmetic $\RR$-Cartier divisor of $C^0$-type on $X$.
Then, for any positive number $\epsilon$, there are a birational morphism $\mu : Y \to X$
of smooth, projective and geometrically integral varieties over $K$ and
a nef adelic arithmetic $\RR$-Cartier divisor $\overline{Q}$ of $C^0$-type on $Y$ such that
$\overline{Q} \leq \mu^*(\overline{D})$ and 
$\avol(\overline{D}) - \epsilon \leq \avol(\overline{Q}) \leq \avol(\overline{D})$.
\end{Theorem}

\begin{Theorem}[Zariski decompositions for adelic arithmetic divisors on curves]
\label{thm:Zariski:decomp:adelic:arithmetic:divisor:intro}
We assume $d=1$.
Let $\overline{D}$ be an adelic arithmetic $\RR$-Cartier divisor of $C^0$-type on $X$.
Let $\Upsilon(\overline{D})$ be the set of all
nef adelic arithmetic $\RR$-Cartier divisors $\overline{L}$ of
$C^0$-type on $X$ with $\overline{L} \leq \overline{D}$.
If $\Upsilon(\overline{D}) \not= \emptyset$,
then there is the greatest element $\overline{Q}$ of
$\Upsilon(\overline{D})$, that is,
$\overline{Q} \in \Upsilon(\overline{D})$ and
$\overline{L} \leq \overline{Q}$ for all $\overline{L} \in \Upsilon(\overline{D})$.
\index{\AdelDivSymbol}{0U:Upsilon(overline{D})@$\Upsilon(\overline{D})$}%
Moreover, 
the natural map 
$\aH(X, a \overline{Q}) \to \aH(X, a \overline{D})$
is bijective
for $a \in \RR_{>0}$.
In particular, $\avol(\overline{Q}) = \avol(\overline{D})$.
\end{Theorem}

\begin{Theorem}[Numerical characterization of nef adelic arithmetic divisors on curves]
\label{thm:char:nef:general:intro}
We assume $d=1$.
Let $\overline{D}$ be an integrable adelic arithmetic $\RR$-Cartier divisor on $X$.
Then $\overline{D}$ is nef if and only if $\overline{D}$ is pseudo-effective and
$\adeg(\overline{D}^2) = \avol(\overline{D})$.
\end{Theorem}

\ifmonog\section{Conventions and terminology}\fi
\ifpaper\subsection{Conventions and terminology}\fi

\ifmonog\subsection{}\fi
\ifpaper\subsubsection{}\fi
\label{CT:top:space:cont}
For a topological space $M$, the set of all real valued continuous functions on $M$ is denote by $C^0(M)$.
\index{\AdelDivSymbol}{0C:C^0(M)@$C^0(M)$}%
Note that $C^0(M)$ forms an $\RR$-algebra.

\ifmonog\subsection{}\fi
\ifpaper\subsubsection{}\fi
\label{CT:nonArchimedean:value}
Let $k$ be a field and $v$ a non-Archimedean valuation of $k$.
We define $k^{\circ}$ and $k^{\circ\circ}$  to be
\[
k^{\circ} := \{ x \in k \mid v(x) \leq 1 \}\quad\text{and}\quad
k^{\circ\circ} := \{ x \in k \mid v(x) < 1 \}.
\]%
\index{\AdelDivSymbol}{0kc:k^{circ}@$k^{\circ}$}%
\index{\AdelDivSymbol}{0kcc:k^{circ circ}@$k^{\circ\circ}$}%
Note that $k^{\circ}$ is a valuation ring and $k^{\circ\circ}$ is its maximal ideal.
If $v$ is discrete and compete, then $k^{\circ}$ is excellent.

\ifmonog\subsection{}\fi
\ifpaper\subsubsection{}\fi
\label{CT:norm:module}
Let $M$ be a finitely generated $\ZZ$-module and let $\Vert\cdot\Vert$ be a norm of $M_{\RR} := M \otimes_{\ZZ} \RR$.
We define $\ah(M, \Vert\cdot\Vert)$ and $\achi(M, \Vert\cdot\Vert)$ to be
\[
\begin{cases}
\ah(M, \Vert\cdot\Vert) := \log \# \{ x \in M \mid \Vert x \Vert \leq 1 \}, \\[2ex]
{\displaystyle \achi(M, \Vert\cdot\Vert) := \log \left( \frac{\vol(B(M, \Vert\cdot\Vert))}{\vol(M_{\RR}/(M/M_{tor}))} \right) 
+ \log \#(M_{tor})},
\end{cases}
\]
where $B(M, \Vert\cdot\Vert)$ is the unit ball with respect to $\Vert\cdot\Vert$
(i.e. $B(M, \Vert\cdot\Vert) := \{ x \in M_{\RR} \mid \Vert x \Vert \leq 1 \}$), 
$M_{tor}$ is the torsion subgroup of $M$ and $\vol(M_{\RR}/(M/M_{tor}))$ is the volume of the fundamental domain of 
$M_{\RR}/(M/M_{tor})$.
\index{\AdelDivSymbol}{0h:ah(M,Vert cdot Vert)@$\ah(M, \Vert\cdot\Vert)$}%
\index{\AdelDivSymbol}{0c:achi(M,Vert cdot Vert)@$\achi(M, \Vert\cdot\Vert)$}%

\ifmonog\subsection{}\fi
\ifpaper\subsubsection{}\fi
\label{CT:S:variety}
Let $S$ be a noetherian  integral scheme.
An integral scheme $X$ over $S$ is called a {\em variety over $S$}
\index{\AdelDivSubject}{variety over a noetherian  integral scheme@variety over a noetherian  integral scheme}%
if $X$ is flat, separated and of finite type over $S$.
If $S$ is given by $\Spec(O_K)$ (i.e. $K$ is a number field and $O_K$ is the ring of integers in $K$),
then a variety over $S$ is often called an {\em arithmetic variety}. 
\index{\AdelDivSubject}{arithmetic variety@arithmetic variety}%

\ifmonog\subsection{}\fi
\ifpaper\subsubsection{}\fi
\label{CT:model}
Let $S$ be a noetherian  integral scheme and $k$ the rational function field of $S$.
Let $X$ be a projective variety over $k$.
A projective variety $\XXX$ over $S$ is called a {\em model of $X$ over $S$} 
\index{\AdelDivSubject}{model@model}%
if the generic fiber of $\XXX \to S$ is $X$.
Moreover, if $\XXX$ are normal (resp. regular), 
then $\XXX$ is called a {\em normal model of $X$ over $S$} (resp {\em regular model of $X$ over $S$}).
\index{\AdelDivSubject}{normal model@normal model}%
\index{\AdelDivSubject}{regular model@regular model}%
Note that if $\XXX$ is normal (resp. regular), 
then $X$ is also normal (resp. regular).
We assume that $S$ is an excellent Dedekind scheme, $\dim X = 1$ and $X$ is smooth over $k$.
By \cite{Lip}, 
for any model $\XXX$ of $X$ over $S$, there is a regular model $\XXX'$ of $X$ over $S$
together with a birational morphism $\XXX' \to \XXX$.

\ifmonog\subsection{}\fi
\ifpaper\subsubsection{}\fi
\label{CT:rel:nef}
Let $f : \XXX \to S$ be a proper morphism of noetherian schemes.
Let $C$ be a curve on $\XXX$, that is,
a $1$-dimensional reduced and irreducible closed subscheme on $\XXX$.
The curve $C$ is said to be {\em vertical} with respect to $f : \XXX \to S$
\index{\AdelDivSubject}{vertical curve@vertical curve}%
if $f(C)$ is a closed point of $S$.
For $\LLL \in \Pic(\XXX) \otimes \RR$,
we say $\LLL$ is {\em relatively nef} 
\index{\AdelDivSubject}{relatively nef@relatively nef}%
with respect to $f : \XXX \to S$ if
$\deg(\rest{\LLL}{C}) \geq 0$ for all vertical curves 
$C$ on $\XXX$.
Let $\DDD$ be an $\RR$-Cartier divisor on $\XXX$, that is,
$\DDD = a_1 \DDD_1 + \cdots + a_r \DDD_r$ for some
Cartier divisors $\DDD_1, \ldots, \DDD_r$ on $\XXX$ and
$a_1, \ldots, a_r \in \RR$.
The $\RR$-Cartier divisor $\DDD$ is said to be {\em relatively nef} with respect to $f : \XXX \to S$ if 
$\OOO_{\XXX}(\DDD_1)^{\otimes a_1} \otimes \cdots \otimes \OOO_{\XXX}(\DDD_r)^{\otimes a_r} \in \Pic(\XXX) \otimes \RR$
is relatively nef with respect to $f : \XXX \to S$.

\ifmonog\subsection{}\fi
\ifpaper\subsubsection{}\fi
\label{CT:ord:function}
Let $(A, m)$ be a $1$-dimensional noetherian local domain.
For $x \in A \setminus \{ 0 \}$, we define $\ord_A(x)$ to be $\ord_A(x) := \lentgh_{A}(A/xA)$.
\index{\AdelDivSymbol}{0o:ord_A@$\ord_A$}%
It is easy to see that 
\[
\ord_A(xy) = \ord_A(x) + \ord_A(y)
\]
for $x, y \in A \setminus \{ 0 \}$, so that
it extends to $F^{\times}$ as a homomorphism,
where $F$ is the quotient field of $A$.
Further, if we set $F^{\times}_{\RR} := F^{\times} \otimes_{\ZZ} \RR$,
then $\ord_A$ also extends to $F^{\times}_{\RR}$.
Let $X$ be a noetherian integral scheme and $\gamma$ a point of $X$ such that
$\dim \OOO_{X, \gamma} = 1$.
Then $\ord_{\OOO_{X, \gamma}}$ is often denoted by $\ord_{\gamma}$ or $\ord_{\Gamma}$,
where $\Gamma$ is the closure of $\{ \gamma \}$.
\index{\AdelDivSymbol}{0o:ord_{gamma}@$\ord_{\gamma}$}%
\index{\AdelDivSymbol}{0o:ord_{Gamma}@$\ord_{\Gamma}$}%

\ifmonog\subsection{}\fi
\ifpaper\subsubsection{}\fi
\label{CT:max:min:divisor}
Let $X$ be a regular scheme and let $\Div(X)$ be the group of Cartier divisors on $X$.
\index{\AdelDivSymbol}{0Div:Div(X)@$\Div(X)$}%
We set $\Div(X)_{\RR} := \Div(X) \otimes_{\ZZ} \RR$, whose element is called an {\em $\RR$-Cartier divisor}.
\index{\AdelDivSubject}{R-Cartier divisor@$\RR$-Cartier divisor}%
\index{\AdelDivSymbol}{0Div:Div(X)_{RR}@$\Div(X)_{\RR}$}%
As $X$ is regular,
an $\RR$-Cartier divisor $D$ has a unique expression 
\[
D = \sum_{\Gamma} a_{\Gamma} \Gamma,
\]
where $a_{\Gamma} \in \RR$ and $\Gamma$ runs over all prime divisors on $X$.
For $D_1, \ldots, D_r \in \Div(X)_{\RR}$,
we set 
\[
D_1 = \sum_{\Gamma} a_{1, \Gamma} \Gamma, \ \ldots,\  D_r = \sum_{\Gamma} a_{r, \Gamma} \Gamma.
\]
We define $\max \{ D_1, \ldots, D_r \}$ and $\min \{ D_1, \ldots, D_r \}$
to be
\[
\begin{cases}
{\displaystyle \max \{ D_1, \ldots, D_r \} := \sum_{\Gamma} \max \{ a_{1, \Gamma}, \ldots, a_{r, \Gamma} \} \Gamma}, \\[3ex]
{\displaystyle  \min \{ D_1, \ldots, D_r \} := \sum_{\Gamma} \min \{ a_{1, \Gamma}, \ldots, a_{r, \Gamma} \} \Gamma}.
\end{cases}
\]%
\index{\AdelDivSymbol}{0m:max \{ D_1, ldots, D_r \}@$\max \{ D_1, \ldots, D_r \}$}%
\index{\AdelDivSymbol}{0m:min \{ D_1, ldots, D_r \}@$\min \{ D_1, \ldots, D_r \}$}%


\ifmonog\chapter{Preliminaries}\fi
\ifpaper\section{Preliminaries}\fi
The goal of this 
\ifmonog chapter \fi
\ifpaper section \fi
is to prepare several kinds of materials for the later sections.
In
\ifmonog Section~\ref{subsec:closedness:support:R:Cartier:divisor}, \fi
\ifpaper Subsection~\ref{subsec:closedness:support:R:Cartier:divisor}, \fi
we consider the support of an $\RR$-Cartier divisor.
In 
\ifmonog Section~\ref{subsec:analytification:algebraic:schemes}, \fi
\ifpaper Subsection~\ref{subsec:analytification:algebraic:schemes}, \fi
we quickly review the analytification of an algebraic scheme in the sense of Berkovich \cite{Be}.
\ifmonog Section~\ref{subsec:Misc:lemmas} \fi
\ifpaper Subsection~\ref{subsec:Misc:lemmas} \fi
is devoted to the proof of several lemmas.

\ifmonog\section{$\RR$-Cartier divisors on a noetherian integral scheme}\fi
\ifpaper\subsection{$\RR$-Cartier divisors on a noetherian integral scheme}\fi
\label{subsec:closedness:support:R:Cartier:divisor}

Let $A$ be a noetherian integral domain and $F$ the quotient field of $A$.
Let $\KK$ be either $\ZZ$ or $\QQ$ or $\RR$.
We set 
\[
F^{\times}_{\KK} := (F^{\times}, \times) \otimes_{\ZZ} \KK
\quad\text{and}\quad
(A_p^{\times})_{\KK} :=
(A_p^{\times}, \times) \otimes_{\ZZ} \KK
\]
for $p \in \Spec(A)$.
As $\KK$ is flat over $\ZZ$, we have
$(A_p^{\times})_{\KK} \subseteq F^{\times}_{\KK}$.
For $f \in F^{\times}_{\KK}$, we define $V_{\KK}(f)$ to be
\[
V_{\KK}(f) := \left\{ p \in \Spec(A) \mid f \not\in (A_p^{\times})_{\KK} \right\}.
\]
Let us begin with the following proposition:

\begin{Proposition}
\label{prop:closedness:support}
\begin{enumerate}
\renewcommand{\labelenumi}{(\arabic{enumi})}
\item
$V_{\RR}(f) = V_{\QQ}(f)$ for $f \in F^{\times}_{\QQ}$.

\item
Let $f \in F^{\times}$. Then $V_{\QQ}(f) = \bigcap_{n=1}^{\infty} V_{\ZZ}(f^n)$.
Moreover, if $A$ is normal,
then $V_{\QQ}(f) = V_{\ZZ}(f)$.

\item 
For $f \in F^{\times}_{\KK}$, the set
$V_{\KK}(f)$
is closed in $\Spec(A)$.
\end{enumerate}
\end{Proposition}

\begin{proof}
(1) By Lemma~\ref{lem:linear:comb:R} in
\ifmonog Section~\ref{subsec:Misc:lemmas}, \fi 
\ifpaper Subsection~\ref{subsec:Misc:lemmas}, \fi 
$(A_p^{\times})_{\QQ} = F^{\times}_{\QQ} \cap
(A_p^{\times})_{\RR}$, and hence (1) follows.

(2) As $F^{\times}_{\QQ}/(A_p^{\times})_{\QQ} = (F^{\times}/A_p^{\times}) \otimes_{\ZZ} \QQ$,
$f = 1$ in $F^{\times}_{\QQ}/(A_p^{\times})_{\QQ}$ if and only if
$f^n = 1$ in $F^{\times}/A_p^{\times}$ for some $n \in \ZZ_{>0}$.
Thus the first assertion follows.
The second assertion is obvious because $V_{\ZZ}(f) = V_{\ZZ}(f^n)$ if $A$ is normal.

\medskip
(3) First we prove that $V_{\ZZ}(f)$ is closed for $f \in F^{\times}$.
We set 
\[
I = \{ a \in A \mid a f \in A \}\quad\text{and}\quad J = I f.
\]
Clearly $I$ and $J$ are ideals of $A$.
Note that $I_p = \{ a \in A_p \mid a f \in A_p \}$ by
\cite[Corollary~3.15]{AM}). Thus,
\[
f \in A_p^{\times} \quad\Longleftrightarrow\quad
\text{$I_p = A_p$ and $J_p = A_p$},
\]
so that $V_{\ZZ}(f) = \Supp(\Spec(A/I)) \cup \Supp(\Spec(A/J))$,
which is closed.

Next let us see that $V(f)_{\QQ}$ is closed for $f \in F^{\times}_{\QQ}$.
Clearly we may assume that $f \in F^{\times}$ because, for $n \in \ZZ_{>0}$,
$f \in (A_p^{\times})_{\QQ}$ if and only if $f^n \in (A_p^{\times})_{\QQ}$.
Thus, by (2), $V_{\QQ}(f)$ is closed.

Finally we consider the case $\KK = \RR$.
We can find $f_1, \ldots, f_r \in F^{\times}$ and
$a_1, \ldots, a_r \in \RR$ such that
$f = f_1^{a_1} \cdots f_r^{a_r}$ and $a_1, \ldots, a_r$ are linearly independent over $\QQ$.
Then, by Lemma~\ref{lem:linear:comb:R},
\[
f \in (A_p^{\times})_{\RR}\quad\Longleftrightarrow\quad
f_1, \ldots, f_r \in (A_p^{\times})_{\QQ},
\]
and hence
$V_{\RR}(f)  = \bigcup_{i=1}^r V_{\QQ}(f_i)$, which is closed by the previous observation.
\end{proof}

\begin{Definition}
\label{def:Supp:RR:Cartier:div}
Let $X$ be a noetherian integral scheme and let $\Rat(X)$ be the rational function field of $X$.
\index{\AdelDivSymbol}{0R:Rat(X)@$\Rat(X)$}%
Let $\Div(X)$ be the group of Cartier divisors on $X$, that is,
\[
\Div(X) := H^0\left(X, \Rat(X)^{\times}/\OOO_X^{\times}\right).
\]%
\index{\AdelDivSymbol}{0Div:Div(X)@$\Div(X)$}%
Let $\KK$ be either $\ZZ$ or $\QQ$ or $\RR$.
We set 
\[
\Div(X)_{\KK} := \Div(X) \otimes_{\ZZ} \KK
\quad\text{and}\quad
\Rat(X)^{\times}_{\KK} := \Rat(X)^{\times} \otimes_{\ZZ} \KK.
\]%
\index{\AdelDivSymbol}{0Div:Div(X)_{KK}@$\Div(X)_{\KK}$}%
\index{\AdelDivSymbol}{0R:Rat(X)^{times}_{KK}@$\Rat(X)^{\times}_{\KK}$}%
An element of $\Div(X)_{\KK}$ (reps. $\Rat(X)^{\times}_{\KK}$)
is called a {\em $\KK$-Cartier divisor on $X$} (reps. {\em $\KK$-rational function on $X$}).
\index{\AdelDivSubject}{K-Cartier divisor@$\KK$-Cartier divisor}%
\index{\AdelDivSubject}{K-rational function@$\KK$-rational function}%
A $\KK$-rational function $f \in \Rat(X)^{\times}_{\KK}$ naturally gives rise to a $\KK$-Cartier divisor,
which is called the {\em $\KK$-principal divisor of $f$} 
\index{\AdelDivSubject}{K-principal divisor@$\KK$-principal divisor}%
and is denoted by
$(f)_{\KK}$.
Occasionally, $(f)_{\KK}$ is denoted by $(f)$ for simplicity.
For $D \in \Div(X)_{\KK}$ (i.e., $D= a_1 D_1 + \cdots + a_r D_r$ for
some $D_1, \ldots, D_r \in \Div(X)$ and $a_1, \ldots, a_r \in \KK$),
there is an affine open covering $X = \bigcup_{i=1}^N \Spec(A_i)$ of $X$ such that
$D$ is given by some $f_i \in \Rat(X)^{\times}_{\KK}$ 
and $f_i/f_j \in (\OOO_{X, p}^{\times})_{\KK} (:= \OOO_{X, p}^{\times} \otimes_{\ZZ} \KK)$
for all $p \in U_i \cap U_j$, so that
$V_{\KK}(f_i) = V_{\KK}(f_j)$ on $U_i \cap U_j$, where $U_i = \Spec(A_i)$ for $i=1, \ldots, N$.
Therefore, we have a closed set $Z$ on $X$ such that
$\rest{Z}{U_i} = V_{\KK}(f_i)$ for all $i=1,\ldots,N$.
It is called the {\em $\KK$-support of $D$} 
\index{\AdelDivSubject}{K-support@$\KK$-support}%
\index{\AdelDivSymbol}{0S:Supp_K(D)@$\Supp_{\KK}(D)$}%
and is denoted by $\Supp_{\KK}(D)$.
By Proposition~\ref{prop:closedness:support},
$\Supp_{\RR}(D) = \Supp_{\QQ}(D)$ for $D \in \Div(X)_{\QQ}$ and
$\Supp_{\QQ}(D) = \bigcap_{n=1}^{\infty} \Supp_{\ZZ}(nD)$ for $D \in \Div(X)$.

From now on, 
we assume that $X$ is normal. Then $\Supp_{\QQ}(D) = \Supp_{\ZZ}(D)$ for $D \in \Div(X)$.
For a $\KK$-Cartier divisor $D$ on $X$,
the {\em associated $\KK$-Weil divisor $D_W$ of $D$} 
\index{\AdelDivSubject}{associated K-Weil divisor@associated $\KK$-Weil divisor}%
is defined by
\[
D_W := \sum_{\text{$\Gamma$ : prime divisor}} \ord_{\Gamma}(f_{\Gamma}) \Gamma,
\]
where $f_{\Gamma}$ is a local equation of $D$ at $\Gamma$.
\index{\AdelDivSymbol}{0D:D_W@$D_W$}%
The {\em support of $D$ as a Weil-divisor} is denoted by $\Supp_W(D)$, that is,
\[
\Supp_W(D) := \bigcup_{\ord_{\Gamma}(f_{\Gamma}) \not= 0} \Gamma.
\] 
\index{\AdelDivSubject}{support as a Weil-divisor@support as a Weil-divisor}%
\index{\AdelDivSymbol}{0S:Supp_W(D)@$\Supp_W(D)$}%
\end{Definition}

\begin{Proposition}
\label{prop:supp:Cartier:Weil}
We assume that $X$ is normal.
Let $D$ be a $\KK$-Cartier divisor on $X$.
Then $\Supp_W(D) \subseteq \Supp_{\KK}(D)$.
Further, if $X$ is regular, then $\Supp_W(D) = \Supp_{\KK}(D)$.
\end{Proposition}

\begin{proof}
We use the same notation as in Definition~\ref{def:Supp:RR:Cartier:div}.
Let $p \in U_i \setminus \Supp_{\KK}(D)$.
Then $f_i \in (\OOO_{X, p}^{\times})_{\KK}$. In particular,
$\ord_{\Gamma}(f_i) = 0$ for all prime divisors $\Gamma$ with $p \in \Gamma$, and hence
$p \not\in \Supp_W(D)$, as desired.

We assume that $X$ is regular. Let $p \in U_i \cap \Supp_{\KK}(D)$.
As $\OOO_{X, p}$ is a UFD, there are distinct prime elements $h_1, \ldots, h_r \in \OOO_{X,p}$ modulo $\OOO_{X, p}^{\times}$, 
$u \in (\OOO_{X, p}^{\times})_{\KK}$
and $a_1, \ldots, a_r \in \RR$ such that
$f_i = u h_1^{a_1} \cdots h_r^{a_r}$.
If $a_1 = \cdots = a_r = 0$, then $f_i \in (\OOO_{X, p}^{\times})_{\KK}$, which contracts to $p \in U_i \cap \Supp_{\KK}(D)$,
so that we may assume that $a_1, \ldots, a_r \in \RR_{\not= 0}$.
Since $h_1, \ldots, h_r$ are distinct modulo $\OOO_{X, p}^{\times}$,
$\Gamma_1 = \Spec(\OOO_{X, p}/h_1), \ldots, \Gamma_r = \Spec(\OOO_{X, p}/h_r)$ give rise to distinct prime divisors.
In addition, $\ord_{\Gamma_j}(f_i) = a_j$ for $j=1, \ldots, r$.
Therefore, $p \in \Supp_W(D)$.
\end{proof}

Finally let us consider Hartogs' lemma for $\RR$-rational functions.

\begin{Lemma}[Hartogs' lemma for $\RR$-rational functions]
\label{lem:Hartogs:R:rat:fun}
Let $A$ be a normal and noetherian domain and
$F$ the quotient field of $A$.
For $x \in F_{\RR}^{\times}$, if $\ord_{\Gamma}(x) \geq 0$ for all prime divisors $\Gamma$ of $A$,
then there are $x_1,\ldots, x_r \in A \setminus \{ 0 \}$ and $a_1, \ldots, a_r \in \RR_{>0}$ with
$x = x_1^{a_1} \cdots x_r^{a_r}$.
\end{Lemma}

\begin{proof}
Let $\Sigma$ be the set of all prime divisors of $A$.
Clearly we can find $y_1, \ldots, y_r \in F^{\times}$ and $c_1, \ldots, c_r \in \RR_{>0}$
such that $x = y_1^{c_1} \cdots y_r^{c_r}$ and $c_1, \ldots, c_r$ are linearly independent over $\QQ$.
We set $\Sigma' = \{ \Gamma \in \Sigma \mid \ord_{\Gamma}(x) > 0 \}$.
Let us see the following:

\begin{Claim}
$\Sigma'$ is a finite set and $\ord_{\Gamma}(y_i) = 0$ for all $\Gamma \in \Sigma \setminus \Sigma'$ and $i=1, \ldots, r$.
\end{Claim}

\begin{proof}
If $\Gamma \not\subseteq \Supp_{\ZZ}((x)_{\RR})$, then $\ord_{\Gamma}(x) = 0$, so that
$\Sigma'$ is a finite set.
Moreover, if $\Gamma \in \Sigma \setminus \Sigma'$,
then 
\[
0 = \ord_{\Gamma}(x) = c_1 \ord_{\Gamma}(y_1) + \cdots + c_r \ord_{\Gamma}(y_r),
\]
and hence $\ord_{\Gamma}(y_1) = \cdots = \ord_{\Gamma}(y_r) = 0$ by
the linear independency of $c_1, \ldots, c_r$ over $\QQ$.
\end{proof}

By virtue of Lemma~\ref{lem:approx:positive} in 
\ifmonog Section~\ref{subsec:Misc:lemmas}, \fi
\ifpaper Subsection~\ref{subsec:Misc:lemmas}, \fi
there are $e_{ij} \in \QQ_{>0}$ ($i, j = 1, \ldots, r$) and $a_1, \ldots, a_r \in \RR_{>0}$
such that, if we set $x_i = y_1^{e_{i1}} \cdots y_r^{e_{ir}}$ for
$i=1, \ldots, r$, then $x = x_1^{a_1} \cdots x_r^{a_r}$ and
$\ord_{\Gamma}(x_i) > 0$ for all $\Gamma \in \Sigma'$ and $i=1, \ldots, r$.
Replacing $e_{ij}$ by $e e_{ij}$ and $a_i$ by $a_i/e$ for some $e \in \ZZ_{>0}$,
we may assume that $e_{ij} \in \ZZ$ for all $i, j$. In particular, $x_i \in F^{\times}$ for $i=1, \ldots, r$.
Note that, for $\Gamma \in \Sigma \setminus \Sigma'$,
\[
\ord_{\Gamma}(x_i) = e_{i1}\ord_{\Gamma}(y_1) + \cdots + e_{ir} \ord_{\Gamma}(y_r) = 0,
\]
and hence $\ord_{\Gamma}(x_i) \geq 0$ for all $\Gamma \in \Sigma$ and $i=1, \ldots, r$.
Therefore, by algebraic Hartogs' lemma, $x_i \in A \setminus \{ 0 \}$ for $i=1, \ldots, r$,
as required.
\end{proof}

\ifmonog \section[Analytification of algebraic schemes]{Analytification of algebraic schemes over a complete valuation field}\fi
\ifpaper \subsection{Analytification of algebraic schemes over a complete valuation field}\fi
\label{subsec:analytification:algebraic:schemes}

Throughout this 
\ifmonog section, \fi
\ifpaper subsection, \fi
$k$ is a field and $v$ is a complete valuation of $k$.
Here we quickly review the analytification of algebraic schemes over $k$
in the sense of Berkovich \cite{Be}.

Let $A$ be a $k$-algebra.
We say a map $\vert\cdot\vert : A \to \RR_{\geq 0}$ is a {\em multiplicative semi-norm over $k$} 
\index{\AdelDivSubject}{multiplicative semi-norm@multiplicative semi-norm}%
if the following conditions are satisfied:
\begin{enumerate}
\renewcommand{\labelenumi}{(\arabic{enumi})}
\item
$\vert a + b \vert \leq \vert a \vert + \vert b \vert$ for all $a, b \in A$.

\item
$\vert a b \vert = \vert a \vert \vert b \vert$ for all $a, b \in A$.

\item
$\vert a \vert = v(a)$ for all $a \in k$.
\end{enumerate}
Let $x = \vert\cdot\vert_x$ be a multiplicative semi-norm over $k$.
We set $p_{x} := \left\{ a \in A \mid \vert a \vert_x = 0 \right\}$, which
is a prime ideal of $A$.
We call $p_{x}$ the {\em associated prime of $x$}.
\index{\AdelDivSubject}{associated prime of the multiplicative semi-norm@associated prime of the multiplicative semi-norm}%
The residue field at $p_{x}$ is denoted by $k(x)$.
Clearly $x$ descents to a valuation $v_x$ of $k(x)$
such that $v_x(a) = v(a)$ for all $a \in k$.
The field $k(x)$ and the valuation $v_x$ are called
the {\em residue field of $x$} and the {\em associated valuation of $x$}, respectively.
\index{\AdelDivSubject}{residue field@residue field}%
\index{\AdelDivSubject}{associated valuation@associated valuation}%
Conversely, let $v'$ be a valuation of the residue field $k(p)$ at $p \in \Spec(A)$
such that $v'(a) = v(a)$ for all $a \in k$. If we set
$\vert a \vert := v'(a \mod p)$ for $a \in A$,
then $\vert\cdot\vert$ yields a multiplicative semi-norm over $k$ whose residue field and 
associated valuation are 
$k(p)$ and $v'$, respectively.
In particular, this observation shows that if $v$ is non-Archimedean,
then $\vert\cdot\vert$ is also non-Archimedean, that is,
$\vert a + b \vert \leq \max \{ \vert a \vert, \vert b \vert \}$ for all $a, b \in A$.

We denote the set of all multiplicative semi-norms over $k$ by $\Spec^{\an}_k(A)$.
\index{\AdelDivSymbol}{0S:Spec^{an}_k(A)@$\Spec^{\an}_k(A)$}%
For $x = \vert\cdot\vert_x \in \Spec^{\an}_k(A)$, $\vert a \vert_x$ is 
often denoted by $\vert a(x) \vert$.
We equip the weakest topology to $\Spec^{\an}_k(A)$ such that
the map $\Spec^{\an}_k(A) \to \RR_{\geq 0}$ given by $x \mapsto \vert a(x) \vert$ is continuous for every $a \in A$, that is,
the collection 
\[
  \Big\{ \{ x \in \Spec^{\an}_k(A) \mid \vert a(x) \vert \in U \} 
  \Big\}_{a, U}\quad\text{(where $a \in A$ and $U$ is an open set in $\RR_{\geq 0}$)}
\]
forms a subbasis of the topology.
A map $\Spec^{\an}_k(A) \to \Spec(A)$ given by $x \mapsto p_x$ is denoted by $p$.
It is easy to see that $p : \Spec^{\an}_k(A) \to \Spec(A)$ is continuous.

Let $f : A \to B$ be a homomorphism of $k$-algebras.
We define a map 
\[
f^{\an} : \Spec^{\an}_k(B) \to \Spec^{\an}_k(A)
\]
to be $\vert a \vert_{f^{\an}(y)} = \vert f(a) \vert_{y}$
for $y = \vert\cdot\vert_y \in \Spec^{\an}_k(B)$ and $a \in A$.
We can easily check that $f^{\an} : \Spec^{\an}_k(B) \to \Spec^{\an}_k(A)$ is continuous.
Let $s$ be a non-nilpotent element of $A$.
Let $\iota : A \to A_s$ be the canonical homomorphism. Then we can see that
$\iota^{\an}$ yields a homeomorphism
\frontmatterforspececialeqn
\begin{equation}
\label{eqn:homeo:open}
\Spec^{\an}_k(A_s)  \underset{\text{homeo}}{\overset{\approx}{\longrightarrow}} \left\{ x \in \Spec^{\an}_k(A) \mid \vert s(x) \vert \not= 0 \right\}.
\end{equation}
\backmatterforspececialeqn

\medskip
Let $X$ be an algebraic scheme over $k$, that is, a scheme separated and of finite type over $k$.
If $X = \Spec(A)$ is an affine scheme over $k$,
then $X^{\an} := \Spec^{\an}_k(A)$.
In general, if $X = \bigcup_{i=1}^N U_i$ is an affine open covering of $X$,
then $X^{\an}$ is defined by gluing together $U_i^{\an}$ as a topological space (cf. \eqref{eqn:homeo:open}).
\index{\AdelDivSymbol}{0X:X^{an}@$X^{\an}$}%
For each $i$, we can define $p : U_i^{\an} \to U_i$, which
can be extended to a continuous map
$p : X^{\an} \to X$.
Let $f : X \to Y$ be a morphism of algebraic schemes over $k$.
We can see that
$f$ induces a natural continuous map $f^{\an} : X^{\an} \to Y^{\an}$.
\index{\AdelDivSymbol}{0f:f^{an}@$f^{\an}$}%

\medskip
From now on, we assume that $v$ is non-Archimedean and $X$ is proper over $k$.
Let $\XXX$ be a proper and flat scheme over $\Spec(k^{\circ})$ such that
the generic fiber of $\XXX \to \Spec(k^{\circ})$ is $X$.
Let $\XXX_{\circ}$ be the central fiber of $\XXX \to \Spec(k^{\circ})$, that is,
$\XXX_{\circ} = \XXX \times_{\Spec(k^{\circ})} \Spec(k^{\circ}/k^{\circ\circ})$
(for the definitions of $k^{\circ}$ and $k^{\circ\circ}$, see Conventions and terminology~\ref{CT:nonArchimedean:value}). 
Let 
\[
r_{\XXX} : X^{\an} \to \XXX_{\circ}
\]
be the {\em reduction map} induced by $\XXX \to \Spec(k^{\circ})$,
which can be defined in the following way:
\index{\AdelDivSubject}{reduction map@reduction map}%
\index{\AdelDivSymbol}{0r:r_{XXX}@$r_{\XXX}$}%
For $x \in X^{\an}$, let $k(x)$ be the residue field of $x$. 
Then, by using the valuation criterion of properness,
there is a morphism $t : \Spec(k(x)^{\circ}) \to \XXX$ such that the following diagram is commutative:
\[
\xymatrix{
\Spec(k(x)) \ar[r] \ar[d] &  \XXX \ar[d] \\
\Spec(k(x)^{\circ}) \ar[ru]^t \ar[r] & \Spec(k^{\circ}) \\
}
\]
Then $r_{\XXX}(x)$ is given by $t(k(x)^{\circ\circ})$.
The morphism $t : \Spec(k(x)^{\circ}) \to \XXX$ yields a 
homomorphism $\OOO_{\XXX, r_{\XXX}(x)} \to k(x)^{\circ}$.
In particular, 
\frontmatterforspececialeqn
\begin{equation}
\label{eqn:norm:leq:1}
\text{$\vert f \vert_x \leq 1$ for all $f \in \OOO_{\XXX, r_{\XXX}(x)}$.}
\end{equation}
\backmatterforspececialeqn
It is well-known that $r_{\XXX} : X^{\an} \to \XXX_{\circ}$
is anti-continuous, that is,
for any open set $U$ of $\XXX_{\circ}$, $r_{\XXX}^{-1}(U)$ is closed (cf. \cite[Section~2.4]{Be}).
Let $Y$ be another proper algebraic scheme over $k$ and $\mu : Y \to X$ a morphism over $k$.
Let $\mathcal{Y} \to \Spec(k^{\circ})$ be 
a proper and flat scheme over $\Spec(k^{\circ})$ such that
the generic fiber of $\YYY \to \Spec(k^{\circ})$ is $Y$ and
there is a morphism $\tilde{\mu} : \mathcal{Y} \to \XXX$
over $\Spec(k^{\circ})$ as an extension of $\mu$.
It is easy to see that the following diagram is commutative:
\frontmatterforspececialeqn
\begin{equation}
\label{eqn:reduction:comm:1}
\xymatrix{
\ar @{} [dr] |{\Box}
Y^{\an} \ar[r]^{r_{\mathcal{Y}}} \ar[d]_{\mu^{\an}} & \YYY_{\circ} \ar[d]^{\tilde{\mu}} \\
X^{\an} \ar[r]^{r_{\XXX}} &  \XXX_{\circ}
}
\end{equation}
\backmatterforspececialeqn

\ifmonog\section{Miscellaneous lemmas}\fi
\ifpaper\subsection{Miscellaneous lemmas}\fi
\label{subsec:Misc:lemmas}

In this 
\ifmonog section, \fi
\ifpaper subsection, \fi
we prove seven lemmas, which are non-trivial and
indispensable for other 
\ifmonog sections. \fi
\ifpaper subsections. \fi

\begin{Lemma}
\label{lem:linear:comb:R}
Let $V$ be a vector space over $\QQ$ and $W$ a subspace of $V$ over $\QQ$.
Let $x_1, \ldots, x_r \in V$ and $a_1, \ldots, a_r \in \RR$ such that
$a_1, \ldots, a_r$ are linearly independent over $\QQ$.
If $a_1 x_1 + \cdots + a_r x_r \in W \otimes_{\QQ} \RR$,
then $x_1, \ldots, x_r \in W$.
\end{Lemma}

\begin{proof}
We set $H := \left\{ \phi \in \Hom_{\QQ}(V, \QQ) \mid \rest{\phi}{W} = 0 \right\} = \Hom_{\QQ}(V/W, \QQ)$.
For $\phi \in H$, the natural extension to $\Hom_{\RR}(V \otimes_{\QQ} \RR, \RR)$ 
is denoted by $\phi_{\RR}$.
As $a_1 x_1 + \cdots + a_r x_r \in W \otimes_{\QQ} \RR$, we have 
\[
0 = \phi_{\RR}(a_1 x_1 + \cdots + a_r x_r) = 
a_1 \phi(x_1) + \cdots + a_r \phi(x_r)
\] 
for all $\phi \in H$.
Thus $\phi(x_1) = \cdots = \phi(x_r) = 0$ for all $\phi \in H$ because
$a_1, \ldots, a_r$ are linearly independent over $\QQ$ and $\phi(x_1), \ldots, \phi(x_r) \in \QQ$.
Therefore, $x_1, \ldots,  x_r \in W$.
\end{proof}

\begin{Lemma}
\label{lem:approx:positive}
Let $V$ be a vector space over $\RR$.
Let $x_1, \ldots, x_r \in V$, $a_1, \ldots, a_r \in \RR_{>0}$,
$x := a_1 x_1 + \cdots + a_r x_r$ and
$\phi_1, \ldots, \phi_m \in \Hom_{\RR}(V, \RR)$.
If $\phi_l(x) > 0$ for all $l=1, \ldots, m$,
then there are $x'_1, \ldots, x'_r \in \QQ_{>0} x_1 + \cdots + \QQ_{>0} x_r$  and
$a'_1, \ldots, a'_r \in \RR_{>0}$ such that
$x = a'_1 x'_1 + \cdots + a'_r x'_r$ and
$\phi_l(x'_i) > 0$ for all $i=1, \ldots, r$ and $l=1, \ldots, m$.
\end{Lemma}

\begin{proof}
Let us begin with the following claim:

\begin{Claim}
\label{claim:lem:approx:positive:01}
Let $y_1, \ldots, y_s \in V$, $b_1, \ldots, b_s \in \RR_{>0}$ and $\psi \in \Hom_{\RR}(V, \RR)$.
If $\psi(y_i) > 0$ for $i = 1, \ldots, s-1$ and $\psi(b_1 y_1 + \cdots + b_s y_s) > 0$,
then there are $c_1, \ldots, c_{s-1} \in \QQ_{> 0}$
such that $b_i - c_i b_s  > 0$ for $i=1, \ldots, s-1$ and
\[
\psi(c_1 y_1 + \cdots + c_{s-1} y_{s-1} + y_s) > 0.
\]
\end{Claim}

\begin{proof}
As $b_1 \psi(y_1) + \cdots + b_{s-1}\psi(y_{s-1}) + b_s \psi(y_s) > 0$,
we have
\[
-\psi(y_s) < (b_1/b_s) \psi(y_1) + \cdots + (b_{s-1}/b_s) \psi(y_{s-1}).
\]
Therefore, we can find $c_1, \ldots, c_{s-1} \in \QQ_{> 0}$ such that
\[
-\psi(y_s) < c_1 \psi(y_1) + \cdots + c_{s-1} \psi(y_{s-1})
\quad\text{and}\quad
c_i < b_i/b_s
\]
for $i=1, \ldots, s-1$, as required.
\end{proof}

First we consider the case where $m=1$.
We set $I = \left\{ i \mid \phi_1(x_i) \leq 0 \right\}$ and $J = \left\{ i \mid \phi_1(x_i) > 0 \right\}$.
We prove the lemma for $m=1$ by induction on $\#(I)$.
As 
\[
\phi_1(x) = a_1 \phi_1(x_1) + \cdots + a_r \phi_1(x_r) > 0,
\]
we have $J \not= \emptyset$.
Clearly we may assume that $I \not= \emptyset$. 
Choose $s \in I$.
Note that 
\[
\phi_1\left(\sum_{j \in J} a_j x_j + a_s x_s \right) > 0,
\]
so that,
applying Claim~\ref{claim:lem:approx:positive:01} to $\{ x_j \}_{j\in J} \cup \{ x_s \}$, 
$\{ a_j \}_{j\in J} \cup \{ a_s \}$ and $\phi_1$,
we can find $c_j \in \QQ_{> 0}$ ($j \in J$) 
such that $a_j - c_j a_s  > 0$ for $j \in J$ and
\[
\phi_1\left(\sum_{j \in J} c_j x_j + x_s\right) > 0.
\]
Note that
\[
 x = \sum_{j \in J} (a_j - c_j a_s) x_j + a_s \left(\sum_{j \in J} c_j x_j + x_s\right)
 + \sum_{i \in I \setminus \{ s \}} a_i x_i.
\]
Thus, by the induction hypothesis,
the assertion of the lemma for $m=1$ follows.

\medskip
In general, we prove the lemma induction on $m$.
The previous observation shows that it holds for $m=1$.
By hypothesis of induction, there are $x'_1, \ldots, x'_r \in \QQ_{>0} x_1 + \cdots + \QQ_{>0} x_r$  and
$a'_1, \ldots, a'_r \in \RR_{>0}$ such that
$x = a'_1 x'_1 + \cdots + a'_r x'_r$ and
$\phi_l(x'_i) > 0$ for all $i=1, \ldots, r$ and $l=1, \ldots, m-1$.
Moreover, by the case where $m=1$, we can find
$x''_1, \ldots, x''_r \in \QQ_{>0} x'_1 + \cdots + \QQ_{>0} x'_r$  and
$a''_1, \ldots, a''_r \in \RR_{>0}$ such that
$x = a''_1 x''_1 + \cdots + a''_r x''_r$ and
$\phi_m(x''_i) > 0$ for all $i=1, \ldots, r$.
Note that $\phi_l$ ($l=1, \ldots, m-1$) is positive on $\QQ_{>0} x'_1 + \cdots + \QQ_{>0} x'_r$
and 
\[
\QQ_{>0} x'_1 + \cdots + \QQ_{>0} x'_r \subseteq \QQ_{>0} x_1 + \cdots + \QQ_{>0} x_r.
\]
Thus the assertion follows.
\end{proof}

\begin{Lemma}
\label{lem:h:0:chi}
Let $M$ be a finitely generated $\ZZ$-module and let $\Vert\cdot\Vert$ and $\Vert\cdot\Vert'$
be norms of $M_{\RR} := M \otimes_{\ZZ} \RR$.
Let $M'$ be a submodule of $M$ such that $M/M'$ is a torsion group, so that
$M_{\RR} = M'_{\RR} (:= M' \otimes_{\ZZ} \RR)$.
If $\Vert\cdot\Vert \leq \Vert\cdot\Vert'$,
then 
\begin{align}
\addtocounter{Claim}{1}
\label{eqn:lem:h:0:chi:01}
& \qquad \achi(M', \Vert\cdot\Vert') \leq \achi(M, \Vert\cdot\Vert). \\
\intertext{Moreover, for $\lambda \in \RR_{\geq 0}$, the following formulae hold:}
\addtocounter{Claim}{1}
\label{eqn:lem:h:0:chi:02}
& \qquad \ah(M, \exp(-\lambda)\Vert\cdot\Vert) \leq \ah(M, \Vert\cdot\Vert) + \lambda \rank M + \log(3) \rank M, \\
\addtocounter{Claim}{1}
\label{eqn:lem:h:0:chi:03}
& \qquad \achi(M, \exp(-\lambda)\Vert\cdot\Vert) = \achi(M, \Vert\cdot\Vert) + \lambda \rank M, \\
\addtocounter{Claim}{1}
\label{eqn:lem:h:0:chi:04}
& \qquad \ah(M, \Vert\cdot\Vert) \leq \ah(M', \Vert\cdot\Vert) + \log \#(M/M') + \log(6) \rank M, \\
\addtocounter{Claim}{1}
\label{eqn:lem:h:0:chi:05}
& \qquad \achi(M, \Vert\cdot\Vert) = \achi(M', \Vert\cdot\Vert) + \log \#(M/M').
\end{align}
\end{Lemma}

\begin{proof}
\eqref{eqn:lem:h:0:chi:01} is obvious.
\eqref{eqn:lem:h:0:chi:02} follows from \cite[Lemma~1.2.2]{MoArLin}.
\eqref{eqn:lem:h:0:chi:03} is also obvious because 
$B(M, \exp(-\lambda) \Vert\cdot\Vert) = \exp(\lambda) B(M, \Vert\cdot\Vert)$.

Let us consider \eqref{eqn:lem:h:0:chi:04}.
Let $\pi : M \to M/M'$ be the canonical homomorphism.
Let us choose $x_1, \ldots, x_N \in M$ with the following properties:
\begin{enumerate}
\renewcommand{\labelenumi}{(\roman{enumi})}
\item
$\Vert x_i \Vert \leq 1$ for all $i=1, \ldots, N$.

\item
$\pi(x_i) \not= \pi(x_j)$ for $i \not= j$.

\item
For any $x \in M$ with $\Vert x \Vert \leq 1$, there is $x_i$ with $\pi(x) = \pi(x_i)$.
\end{enumerate}
Then we can see that
\[
\left\{ x \in M \mid \Vert x \Vert \leq 1 \right\} \subseteq \left\{ x_i + x' \mid \text{$x' \in M'$ and $\Vert x' \Vert \leq 2$} \right\},
\]
and hence,
\[
\ah(M, \Vert\cdot\Vert) \leq \ah(M', (1/2)\Vert\cdot\Vert) + \log \#(M/M').
\]
Therefore, \eqref{eqn:lem:h:0:chi:04} follows from \eqref{eqn:lem:h:0:chi:02}.

For \eqref{eqn:lem:h:0:chi:05}, let us
consider the following commutative diagram:
\[
\begin{CD}
0 @>>> M' @>>> M @>>> M/M' @>>> 0 \\
@. @VVV @VVV @VVV @. \\
0 @>>> M'/M'_{tor} @>>> M/M_{tor} @>>> (M/M_{tor})/(M'/M'_{tor}) @>>> 0,
\end{CD}
\]
which show that
we may assume that $M$ is torsion free.
Let $\omega_1, \ldots, \omega_r$ be a free basis of $M$ such that
$a_1 \omega_1, \ldots, a_r \omega_r$ form a free basis of $M'$ for some $a_1, \ldots, a_r \in \ZZ_{>0}$.
Then 
\[
\vol(M_{\RR}/M') = a_1 \cdots a_r \vol(M_{\RR}/M).
\]
Thus \eqref{eqn:lem:h:0:chi:05} follows because
$\#(M'/M) = a_1 \cdots a_r$.
\end{proof}

\begin{Lemma}
\label{lem:mult:linear:diff:formula}
Let $A$ and $M$ be $\ZZ$-modules and let
$f : M^n \to A$ be a multi-linear map, that is,
\[
f(x_1, \ldots, x_i - x'_i, \ldots, x_n) = f(x_1, \ldots, x_i, \ldots, x_n) - f(x_1, \ldots, x'_i, \ldots, x_n)
\]
for all $i=1, \ldots, n$ and $x_1, \ldots, x_i, x'_i, \ldots, x_n \in M$.
Then, for $x_1, \ldots, x_n, x'_1, \ldots, x'_n \in M$,
\[
f(x'_1, \ldots, x'_n) = f(x_1, \ldots, x_n) +
\sum_{i=1}^n f(x'_1, \ldots, x'_{i-1}, \delta_i, x_{i+1}, \ldots, x_n),
\]
where $\delta_i = x'_i - x_i$.
\end{Lemma}

\begin{proof}
We prove it by induction on $n$. In the case where $n=1$, the above means that
$f(x'_1) = f(x_1) + f(x'_1 - x_1)$, which is obvious. In general, using the induction hypothesis,
we have
\[
f(x'_1, \ldots, x'_n) = f(x'_1, \ldots, x'_{n-1}, x_n) + f(x'_1, \ldots, x'_{n-1}, \delta_{n-1})
\]
and
\begin{multline*}
\qquad\qquad f(x'_1, \ldots, x'_{n-1}, x_n) = f(x_1, \ldots, x_{n-1}, x_n) \\
 + \sum_{i=1}^{n-1} f(x'_1, \ldots, x'_{i-1}, \delta_i, x_{i+1}, \ldots, x_{n-1}, x_n).
 \quad
\end{multline*}
Thus we have the assertion.
\end{proof}

\begin{Lemma}
\label{lem:multi:symmetric:partition}
Let $V$ be a vector space over $\RR$ and $W$ a subspace of $V$ over $\RR$.
Let 
\[
\langle\ \rangle : W \times V^{l-1} \to \RR
\]
be a multi-linear map over $\RR$.
We assume that $\langle\ \rangle$ is symmetric, that is,
\[
\langle x_1, x_2, \ldots, x_l \rangle = \langle x_2, x_1, \ldots, x_l \rangle
\]
for all $x_1, x_2 \in W$ and $x_3, \ldots, x_l \in V$ and
\[
\langle x_1, \ldots, x_i, \ldots, x_j, \ldots, x_l \rangle = \langle x_1, \ldots, x_j, \ldots, x_i, \ldots, x_l \rangle
\]
for all $x_1 \in W$, $x_2, \ldots, x_l \in V$ and $2 \leq i < j \leq l$.
For simplicity, $\langle x_1,\ldots, x_l \rangle$ is denoted by
$\langle x_1 \cdots x_l \rangle$ or $\left\langle \prod_i x_i \right\rangle$. Then 
\begin{multline*}
\sum_{\emptyset \not= I \subseteq \{ 1, \ldots, l \}} \left\langle \prod_{i \in I}(a_i + b_i ) \cdot \prod_{j \in \{ 1, \ldots, l \} \setminus I} x_j \right\rangle \\
= \sum_{\emptyset \not= I \subseteq \{ 1, \ldots, l \}} \left\langle \prod_{i \in I}b_i  \cdot \prod_{j \in \{ 1, \ldots, l \} \setminus I} x_j \right\rangle+
\sum_{\emptyset \not= I \subseteq \{ 1, \ldots, l \}} \left\langle \prod_{i \in I}a_i  \cdot \prod_{j \in \{ 1, \ldots, l \} \setminus I} (x_j + b_j) \right\rangle \\
\end{multline*}
holds for $a_1, \ldots, a_l, b_1, \ldots, b_l \in W$ and $x_1, \ldots, x_l \in V$.
\end{Lemma}

\begin{proof}
If we set 
\begin{multline*}
\qquad\quad
A := \sum_{\emptyset \not= I \subseteq \{ 1, \ldots, l \}} 
\left\langle \prod_{i \in I}(a_i + b_i ) \cdot \prod_{j \in \{ 1, \ldots, l \} \setminus I} x_j \right\rangle \\
- \sum_{\emptyset \not= I \subseteq \{ 1, \ldots, l \}} \left\langle \prod_{i \in I}b_i  \cdot \prod_{j \in \{ 1, \ldots, l \} \setminus I} x_j \right\rangle,
\qquad\quad
\end{multline*}
then
\begin{align*}
A & = \sum_{\emptyset \not= I \subseteq \{ 1, \ldots, l \}} \sum_{L \subseteq I}
\left\langle \prod_{i \in L}a_i \cdot \prod_{i' \in I \setminus L} b_{i'} \cdot \prod_{j \in \{ 1, \ldots, l \} \setminus I} x_j \right\rangle \\
& \qquad\qquad\qquad\qquad\qquad - \sum_{\emptyset \not= I \subseteq \{ 1, \ldots, l \}} \left\langle \prod_{i \in I}b_i  \cdot \prod_{j \in \{ 1, \ldots, l \} \setminus I} x_j \right\rangle\\
& = \sum_{\emptyset \not= I \subseteq \{ 1, \ldots, l \}} \sum_{\emptyset \not = L \subseteq I}
\left\langle \prod_{i \in L}a_i \cdot \prod_{i' \in I \setminus L} b_{i'} \cdot \prod_{j \in \{ 1, \ldots, l \} \setminus I} x_j \right\rangle \\
& = \sum_{\emptyset \not = L \subseteq  \{ 1, \ldots, l \}} \sum_{M \amalg M' = \{ 1, \ldots, l \} \setminus L} 
\left\langle \prod_{i \in L}a_i \cdot \prod_{m \in M} b_{m} \cdot \prod_{m' \in M'} x_{m'} \right\rangle \\
& = \sum_{\emptyset \not= L \subseteq \{ 1, \ldots, l \}} 
\left\langle \prod_{i \in L}a_i \cdot \prod_{j \in \{ 1, \ldots, l \} \setminus L } (x_j + b_j)  \right\rangle,
\end{align*}
as desired.
\end{proof}

\begin{Lemma}
\label{lem:dekind:extension:R:Cartier}
Let $S$ be a connected Dedekind scheme and $k$ the rational function field of $S$.
Let $X$ be a projective variety over $k$. Then we have the following:
\begin{enumerate}
\renewcommand{\labelenumi}{(\arabic{enumi})}
\item
There exists a model of $X$ over $S$ \rom{(}cf. Conventions and terminology~\rom{\ref{CT:model}}\rom{)}.

\item
Let $J$ be an invertible fractional ideal sheaf on $X$.
Then there are a model $\XXX$ of $X$ over $S$ and an invertible 
fractional ideal sheaf $\JJJ$ on $\XXX$ such that $\JJJ \cap X = J$.

\item
Let $D$ be an $\RR$-Cartier divisor on $X$.
Then there are a model $\XXX$ of $X$ over $S$ and an $\RR$-Cartier divisor $\DDD$
on $\XXX$ such that $\DDD \cap X = D$.
\end{enumerate}
\end{Lemma}

\begin{proof}
(1) As $X$ is projective over $k$, there is a closed embedding $\iota : X \hookrightarrow \PP_k^N$.
Let $\XXX$ be the closure of $\iota(X)$ in $\PP^N_S$.
Then $\XXX$ is integral, projective and flat over $S$ because $S$ is a 
connected Dedekind scheme.

(2) Let $\XXX'$ be a model of $X$ over $S$.
Then we can find a non-empty open set $U$ of $S$ and an invertible fractional ideal sheaf $\JJJ'_U$ on $\XXX'_U$ such that
$\JJJ'_U \cap X = J$. Therefore, 
as $\XXX$ is noetherian, 
by using the extension theorem of coherent sheaves (cf. \cite[Chapter~II, Exercise~5.15]{Hartshorne}),
we have a fractional ideal sheaf $\JJJ'$ on $\XXX'$ such that $\JJJ' \cap \XXX'_U = \JJJ'_U$.
Let $\pi : \XXX = \Proj\left( \bigoplus_{m=0}^{\infty} {\JJJ'}^m \right) \to \XXX'$
be the blowing-up by the fractional ideal sheaf $\JJJ'$.
Then, as $\JJJ := \JJJ' \OOO_{\XXX}$ is invertible, the assertion of (2) follows.

(3) is a consequence of (2).
\end{proof}

\begin{Lemma}
\label{lem:base:change:completion}
Let $k$ be a field and $v$ a valuation of $k$.
Let $k_v$ be the completion of $k$ with respect to $v$.
\index{\AdelDivSymbol}{0kv:k_v@$k_v$}%
By abuse of notation, the unique extension of $v$ to $k_v$ is also denoted by $v$.
Then we have the following:
\begin{enumerate}
\renewcommand{\labelenumi}{(\arabic{enumi})}
\item
We assume that $v$ is discrete.
Let $X$ be a projective and geometrically integral variety over $k$ and let
$\XXX$ be a model of $X$ over $\Spec(k^{\circ})$.
We set 
\[
X_v := X \times_{\Spec(k)} \Spec(k_v)
\quad\text{and}\quad
\XXX_v := \XXX \times_{\Spec(k^{\circ})} \Spec(k_v^{\circ}).
\] 
Let $\pi : \XXX_v \to \XXX$ be the projection, and let
$(\XXX_v)_{\circ}$ and $\XXX_{\circ}$ be the central fibers of
\[
\XXX_v \to \Spec(k_v^{\circ})
\quad\text{and}\quad
\XXX \to \Spec(k^{\circ}),
\]
respectively
\rom{(}cf. Conventions and terminology~\rom{\ref{CT:nonArchimedean:value}}\rom{)}.
\index{\AdelDivSymbol}{0X:XXX_{circ}@$\XXX_{\circ}$}%
If we choose $\xi_v \in (\XXX_v)_{\circ}$ and $\xi \in \XXX_{\circ}$ with $\pi(\xi_v) = \xi$,
then we have the following:
\begin{enumerate}
\renewcommand{\labelenumii}{(\arabic{enumi}.\arabic{enumii})}
\item
$\XXX_v$ is a model of $X_v$ over $\Spec(k_v^{\circ})$.

\item
$\OOO_{\XXX,\, \xi}$ is regular if and only if $\OOO_{\XXX_v,\, \xi_v}$ is regular.

\item
We assume that $k^{\circ}$ is excellent.
Then $\OOO_{\XXX,\, \xi}$ is normal if and only if $\OOO_{\XXX_v,\, \xi_v}$ is normal.
\end{enumerate}

\item
Let $A$ be a $k$-algebra and $A_v := A \otimes_k k_v$.
Let $x = \vert\cdot\vert_x$ and $x' = \vert\cdot\vert_{x'}$ be seminorms of $A_v$.
If $\vert a \otimes 1 \vert_x = \vert a \otimes 1 \vert_{x'}$
for all $a \in A$, then $x = x'$.
\end{enumerate}
\end{Lemma}

\begin{proof}
(1) We need to see that $\XXX_v$ is integral.
As $\XXX_v \to \Spec(k_v^{\circ})$ is flat and the generic fiber of
$\XXX_v \to \Spec(k_v^{\circ})$ is integral,
(1.1) follows from \cite[Lemma~4.2]{MoDC}.

\smallskip
Before staring the proofs of (1.2) and (1.3),
let us see
$m_{\xi_v} = m_{\xi}\OOO_{\XXX_v,\, \xi_v}$.
Let $\kappa(\xi)$ be the residue field at $\xi$.
Here we consider an exact sequence
\[
\begin{CD}
\kappa(\xi)\otimes_{k^{\circ}} k_v^{\circ\circ}  @>{\alpha}>> 
\kappa(\xi)\otimes_{k^{\circ}}
k_v^{\circ} @>>> \kappa(\xi) \otimes_{k^{\circ}} (k_v^{\circ}/k_v^{\circ\circ}) @>>> 0
\end{CD}
\]
induced by $0 \to k_v^{\circ\circ} \to 
k_v^{\circ} \to k_v^{\circ}/k_v^{\circ\circ} \to 0$.
Note that $\alpha = 0$
because $\alpha(a \otimes \varpi b)=  a \otimes \varpi b = 
\varpi a \otimes b  = 0$ for $a \in \kappa(\xi)$ and $b \in k_v^{\circ}$.
Therefore, 
\[
\kappa(\xi) \otimes_{k^{\circ}} k_v^{\circ} \simeq 
\kappa(\xi) \otimes_{k^{\circ}} (k_v^{\circ}/k_v^{\circ\circ}) \simeq 
\kappa(\xi) \otimes_{k^{\circ}} (k^{\circ}/k^{\circ\circ}) \simeq
\kappa(\xi).
\]
On the other hand, performing $\otimes_{\OOO_{\XXX,\,\xi}} \OOO_{\XXX_v,\, \xi_v}$ to
the exact sequence $0 \to m_{\xi} \to \OOO_{\XXX,\,\xi} \to \kappa(\xi) \to 0$,
we obtain
\[
0 \to m_{\xi} \OOO_{\XXX_v,\, \xi_v} \to \OOO_{\XXX_v,\, \xi_v} \to \kappa(\xi) 
\otimes_{\OOO_{\XXX,\,\xi}} \OOO_{\XXX_v,\, \xi_v} \to 0.
\]
As $\pi$ induces the isomorphism $(\XXX_v)_{\circ} \overset{\sim}{\longrightarrow} \XXX_{\circ}$,
$\pi^{-1}(\xi) = \{ \xi_v \}$, so that
$\OOO_{\XXX_v,\, \xi_v} = \OOO_{\XXX,\, \xi} \otimes_{k^{\circ}} k_v^{\circ}$.
Therefore,
\[
\kappa(\xi) \otimes_{\OOO_{\XXX,\,\xi}} \OOO_{\XXX_v,\, \xi_v} =
\kappa(\xi) \otimes_{\OOO_{\XXX,\,\xi}} (\OOO_{\XXX,\, \xi} \otimes_{k^{\circ}} k_v^{\circ}) \simeq
\kappa(\xi) \otimes_{k^{\circ}} k_v^{\circ} \simeq \kappa(\xi),
\]
which shows that $m_{\xi} \OOO_{\XXX_v,\, \xi_v}$ is the maximal ideal of
$\OOO_{\XXX_v,\, \xi_v}$, and hence $m_{\xi} \OOO_{\XXX_v,\, \xi_v} = m_{\xi_v}$.

\smallskip
Since $\OOO_{\XXX, \xi} \subseteq \OOO_{\XXX_v,\xi_v} \subseteq \widehat{\OOO}_{\XXX, \xi}$ and
$m_{\xi} \OOO_{\XXX_v,\, \xi_v} = m_{\xi_v}$, by \cite[Chapter~1, Theorem~3.16]{Liu},
we have $\widehat{\OOO}_{\XXX_v,\xi_v} \simeq \widehat{\OOO}_{\XXX, \xi}$.
Thus (1.2) follows from \cite[Proposition~11.24]{AM}.
Further, (1.3) follows from \cite[VI, 7.8.3, (v)]{EGA}.

\medskip
(2)
For $\alpha \in A_v$, we set $\alpha = a_1 \otimes \lambda_1 + \cdots + a_r \otimes \lambda_r$,
where $a_1, \ldots, a_r \in A$ and $\lambda_1, \ldots, \lambda_r \in k_v$.
Then we can find sequences $\{ \lambda_{1, n} \}_{n=1}^{\infty}, \ldots, \{ \lambda_{r, n} \}_{n=1}^{\infty}$
in $k$
such that 
$\lambda_i = \lim_{n\to\infty} \lambda_{i, n}$ for
$i=1, \ldots, r$.
Here we set 
\[
\alpha_n = a_1 \otimes \lambda_{1,n} + \cdots + a_r \otimes \lambda_{r, n} = 
(\lambda_{1,n} a_1 + \cdots + \lambda_{r, n} a_r) \otimes 1.
\]
Then,
\begin{align*}
\left| \vert \alpha_n \vert_{x} - \vert \alpha \vert_{x} \right| & \leq
\vert \alpha_n - \alpha \vert_x =
\left\vert a_1 \otimes (\lambda_{1, n} - \lambda_1) + \cdots + a_r \otimes (\lambda_{r, n} - \lambda_r)\right\vert_x \\
& \leq \left\vert a_1 \otimes (\lambda_{1, n} - \lambda_1) \right\vert_x
+ \cdots + \left\vert a_1 \otimes (\lambda_{r, n} - \lambda_r)\right\vert_x \\
& = \left\vert (a_1 \otimes 1) \cdot( 1 \otimes (\lambda_{1, n} - \lambda_1)) \right\vert_x
+ \cdots + \left\vert (a_r \otimes 1) \cdot (1 \otimes  (\lambda_{r, n} - \lambda_r))\right\vert_x \\
& = \vert a_1 \otimes 1 \vert_x v( \lambda_{1, n} - \lambda_1)
+ \cdots + \vert a_1 \otimes 1 \vert_x v(\lambda_{r, n} - \lambda_r),
\end{align*}
and hence $\lim_{n\to\infty} \vert \alpha_n \vert_{x} = \vert \alpha \vert_{x}$.
In the same way, $\lim_{n\to\infty} \vert \alpha_n \vert_{x'} = \vert \alpha \vert_{x'}$.
On the other hand, by our assumption,
\[
\vert \alpha_n \vert_x = \vert (\lambda_{1,n} a_1 + \cdots + \lambda_{r, n} a_r) \otimes 1 \vert_x =
\vert (\lambda_{1,n} a_1 + \cdots + \lambda_{r, n} a_r) \otimes 1 \vert_{x'} = \vert \alpha_n \vert_{x'}
\]
for all $n \geq 1$.
Therefore, $\vert \alpha \vert_x = \vert \alpha \vert_{x'}$, and hence $x = x'$.
\end{proof}

\ifmonog\chapter[Adelic $\RR$-Cartier divisors]{Adelic $\RR$-Cartier divisors over a discrete valuation field}\fi
\ifpaper\section{Adelic $\RR$-Cartier divisors over a discrete valuation field}\fi
In this 
\ifmonog chapter, \fi
\ifpaper section, \fi
we introduce an adelic $\RR$-Cartier divisor on a projective variety 
over a discrete valuation field and study 
their basic properties.
Roughly speaking, an adelic $\RR$-Cartier divisor is a pair of an $\RR$-Cartier divisor and 
a Green function on the analytification of the given variety, which is an analogue of
Arakelov divisors (i.e. arithmetic divisors) on an arithmetic variety.

Throughout this 
\ifmonog chapter, \fi
\ifpaper section, \fi
let $k$ be a field and $v$ a discrete valuation.
We set
\[
k^{\circ} := \{ a \in k \mid v(a) \leq 1 \}
\quad\text{and}\quad
k^{\circ\circ} := \{ a \in k \mid v(a) < 1 \}.
\]
Let $\varpi$ be a uniformizing parameter of $v$, that is, $k^{\circ\circ} = \varpi k^{\circ}$.
Note that $v$ might be trivial, so that we do not exclude the case where $\varpi = 0$.
Let us begin with Green functions on analytic spaces over a complete discrete valuation field.

\ifmonog\section[Green functions on analytic spaces]{Green functions on analytic spaces over a discrete valuation field}\fi
\ifpaper\subsection{Green functions on analytic spaces over a discrete valuation field}\fi
\label{subsec:green:function:analytic:space}

We assume that $v$ is complete.
Let $X$ be a projective and geometrically integral variety over $k$.
Let $\Rat(X)$ be the rational function field of $X$.
Let $U = \Spec(A)$ be an affine open set of $X$.
Let $p \in U$ and $x = \vert\cdot\vert_x \in U^{\an}$ such that $p_x \subseteq p$,
where $p_x$ is the associated prime of $x$ (cf. 
\ifmonog Section~\ref{subsec:analytification:algebraic:schemes}). \fi
\ifpaper Subsection~\ref{subsec:analytification:algebraic:schemes}). \fi
Then we have a natural extension $\vert\cdot\vert_x : A_p \to \RR_{\geq 0}$ of $\vert\cdot\vert_x$ on $A$
given by $\vert a/s \vert_x = \vert a \vert_x/\vert s \vert_x$ for $a \in A$ and $s \in A \setminus p$,
which yields the group homomorphism $\vert\cdot\vert_x : A_p^{\times} \to \RR_{>0}$.
Thus we obtain a canonical extension $(A_p^{\times})_{\RR} \to \RR_{>0}$,
which is also denoted by $\vert\cdot\vert_x$ by abuse of notation.
Let $f \in \Rat(X)^{\times}_{\RR}$ and $x \in U^{\an} \setminus V_{\RR}(f)^{\an}$
(see 
\ifmonog Section~\ref{subsec:closedness:support:R:Cartier:divisor} \fi 
\ifpaper Subsection~\ref{subsec:closedness:support:R:Cartier:divisor} \fi 
for the definitions of $ \Rat(X)^{\times}_{\RR}$ and
$V_{\RR}(f)$).
As $p_x \not\in V_{\RR}(f)$, we get
$f \in (A_{p_x}^{\times})_{\RR}$, and hence $\vert f(x) \vert \in \RR_{>0}$.
Therefore, we have a map
$U^{\an} \setminus V_{\RR}(f)^{\an} \to \RR$ given by $x \mapsto \log \vert f(x) \vert^2$. 
We denote it by $\log \vert f \vert^2$.
Clearly $\log \vert f \vert^2$ is continuous on $U^{\an} \setminus V_{\RR}(f)^{\an}$.

\begin{Definition}
Let $D$ be an $\RR$-Cartier divisor on $X$, that is, 
$D \in \Div(X)_{\RR} (:= \Div(X) \otimes_{\ZZ} \RR)$.
Let $X = \bigcup_{i=1}^N U_i$ be an affine open covering of $X$ such that
$D$ is given by $f_i \in \Rat(X)^{\times}_{\RR}$ on $U_i$.
A continuous function $g : X^{\an} \setminus \Supp_{\RR}(D)^{\an} \to \RR$
is called a {\em $D$-Green function of $C^0$-type on $X^{\an}$} 
\index{\AdelDivSubject}{Green function of C^0-type@Green function of $C^0$-type}%
if $g + \log \vert f_i \vert^2$ 
extends to a continuous function on $U_i^{\an}$ for each $i = 1, \ldots, N$.
\end{Definition}

For example, for $f \in \Rat(X)^{\times}_{\RR}$,
$-\log \vert f \vert^2$ is $(f)_{\RR}$-Green function of $C^0$-type on $X^{\an}$,
where $(f)_{\RR}$ is the $\RR$-principal divisor of $f$ (cf. 
\ifmonog Section~\ref{subsec:closedness:support:R:Cartier:divisor}). \fi
\ifpaper Subsection~\ref{subsec:closedness:support:R:Cartier:divisor}). \fi
We set
\[
C^0_{\eta}(X^{\an}) := \varinjlim_{\text{$U$: Zariski open set of $X$}} C^0(U^{\an}).
\]%
\index{\AdelDivSymbol}{0C:C^0_{eta}(X^{an})@$C^0_{\eta}(X^{\an})$}%
The space of all Green functions forms a subspace of $C^0_{\eta}(X^{\an})$ over $\RR$.
More precisely, we have the following proposition:

\begin{Proposition}
\label{prop:diff:two:D:Green:fun}
Let $D$ and $D'$ be $\RR$-Cartier divisors on $X$.
Let $g$ be a $D$-Green function of $C^0$-type on $X^{\an}$ and $g'$ 
a $D'$-Green function of $C^0$-type on $X^{\an}$.
Then we have the following:
\begin{enumerate}
\renewcommand{\labelenumi}{(\arabic{enumi})}
\item
For $a, b \in \RR$,
$a g + bg'$ is an $(aD + bD')$-Green function of $C^0$-type.

\item
If $D = D'$, then
$\max \{ g, g' \}$ and $\min \{ g, g' \}$ are $D$-Green functions of $C^0$-type.
\end{enumerate}
\end{Proposition}

\begin{proof}
(1)
Let $X = \bigcup_{i=1}^N U_i$ be an affine open covering of $X$ such that
$D$ and $D'$ are given by $f_i$ and $f'_i$ on $U_i$, respectively.
By our assumption, there are continuous functions $\varphi$ and $\varphi'$ on $U^{\an}_i$
such that $g = -\log \vert f_i \vert^2 + \varphi$ and $g' = -\log \vert f'_i \vert^2 + \varphi'$.
Thus 
\[
ag + bg' =  -\log \vert f_i^a {f'_i}^b \vert^2 + a \varphi + b \varphi'.
\]
Note that $f_i^a {f'_i}^b$ is a local equation of $a D + bD'$ on $U_i$.
Thus (1) follows.

(2) Note that 
\[
\max \{ g, g' \} = -\log \vert f_i \vert^2 + \max\{ \varphi, \varphi' \}
\quad\text{and}\quad
\min \{ g, g' \} = -\log \vert f_i \vert^2 + \min\{ \varphi, \varphi' \}
\]
on $U_i^{\an} \setminus \Supp_{\RR}(D)^{\an}$.
Moreover, $\max\{ \varphi, \varphi' \}$ and $\min\{ \varphi, \varphi' \}$
are continuous on $U_i^{\an}$, as required.
\end{proof}

Next let us consider a norm arising from a Green function.

\begin{Proposition}
\label{prop:ext:cont:norm}
We assume that $X$ is normal. 
We set 
\[
H^0(X, D) := \left\{ \phi \in \Rat(X)^{\times} \mid (\phi) + D \geq 0 \right\} \cup \{ 0 \}.
\]%
\index{\AdelDivSymbol}{0H:H^0(X,D)@$H^0(X, D)$}%
Let $g$ be a $D$-Green function of $C^0$-type on $X^{\an}$.
Then we have the following:
\begin{enumerate}
\renewcommand{\labelenumi}{(\arabic{enumi})}
\item
For $\phi \in H^0(X, D)$,
$\vert \phi \vert \exp(-g/2)$ extends to a continuous function $\vartheta$ on $X^{\an}$.
We denote $\Vert \vartheta \Vert_{\sup}$ by $\Vert \phi \Vert_{g}$.
\index{\AdelDivSymbol}{0n:Vert phi Vert_{g}@$\Vert \phi \Vert_{g}$}%

\item
The following formulae hold:
\begin{enumerate}
\renewcommand{\labelenumii}{(\arabic{enumi}.\arabic{enumii})}
\item
$\Vert a \phi \Vert_g = v(a) \Vert \phi \Vert_{g}$ for all $a \in k$ and $\phi \in H^0(X, D)$.
\item
$\Vert \phi_1 + \phi_2 \Vert_{g} \leq \max \left\{ \Vert \phi_1 \Vert_{g}, \Vert \phi_2 \Vert_{g} \right\}$
for all $\phi_1, \phi_2 \in H^0(X, D)$.
\end{enumerate}
\end{enumerate}
\end{Proposition}

\begin{proof}
(1) Clearly we may assume that $\phi \not= 0$.
Let $X = \bigcup_{i=1}^N \Spec(A_i)$ be an affine open covering of $X$ such that
$D$ is given by $h_i \in \Rat(X)^{\times}_{\RR}$ on $\Spec(A_i)$.
Since $D + (\phi)$ is effective as a Weil divisor, 
$\ord_{\Gamma}(\phi h_i) \geq 0$ for any prime divisor $\Gamma$ on $\Spec(A_i)$.
Thus, by virtue of Hartogs' lemma for $\RR$-rational functions (cf. Lemma~\ref{lem:Hartogs:R:rat:fun}),
there are $u_1, \ldots, u_r \in A_i \setminus \{ 0 \}$ and $a_1, \ldots, a_r \in \RR_{>0}$
with $\phi h_i = u_1^{a_1} \cdots u_r^{a_r}$.
In particular, $\vert \phi h_i \vert = \vert u_1 \vert^{a_1} \cdots \vert u_r \vert^{a_r}$ 
is continuous on $\Spec^{\an}_k(A_i)$.
On the other hand, there is a continuous function $\varphi_i$ on $\Spec^{\an}_k(A_i)$
such that $g = -\log \vert h_i \vert^2 + \varphi_i$ on $\Spec^{\an}_k(A_i)$.
Therefore,
\[
\vert \phi \vert \exp(-g/2) = \vert \phi h_i \vert \exp(-\varphi_i/2)
\]
is continuous on $\Spec^{\an}_k(A_i)$.

\medskip
(2.1) is obvious. (2.2) is also obvious 
because $\vert \phi_1 + \phi_2 \vert \leq \max \{ \vert \phi_1 \vert, \vert \phi_2 \vert \}$
on some dense open set.
\end{proof}

A pair $(\XXX, \DDD)$ is called a {\em model of $(X, D)$}
\index{\AdelDivSubject}{model of (X, D)@model of $(X, D)$}%
if $\XXX$ is a model of $X$ over $\Spec(k^{\circ})$
(cf. Conventions and terminology~\ref{CT:model})
and $\DDD$ is an $\RR$-Cartier divisor on $\XXX$ with $\DDD \cap X = D$.
The $\RR$-Cartier divisor $\DDD$ is often called a {\em model of $D$ on $\XXX$}.
\index{\AdelDivSubject}{model of D@model of $D$}%
For $x \in X^{\an} \setminus \Supp_{\RR}(D)^{\an}$, let $f$ be a local equation of $\DDD$ at 
$\xi = r_{\XXX}(x)$, where $r_{\XXX}$ is the reduction map $X^{\an} \to \XXX_{\circ}$
(cf.
\ifmonog Section~\ref{subsec:analytification:algebraic:schemes}). \fi
\ifpaper Subsection~\ref{subsec:analytification:algebraic:schemes}). \fi
As $p_x \in \Spec(\OOO_{\XXX,\xi})$ and $f \in (\OOO_{\XXX,p_x}^{\times})_{\RR}$,
we have $\vert f (x) \vert \not= 0$, so that we can define $g_{(\XXX,\, \DDD)}(x)$ to be
\[
g_{(\XXX,\,\DDD)}(x) := -\log \vert f (x) \vert^2.
\]%
\index{\AdelDivSymbol}{0g:g_{(XXX,DDD)}@$g_{(\XXX,\,\DDD)}$}%
Let $f'$ be another local equation of $\DDD$ at $\xi$.
Then there is $u \in (\OOO_{\XXX,\xi}^{\times})_{\RR}$ such that $f' = f u$, and hence
$\vert f(x) \vert = \vert f'(x) \vert$ because $\vert u(x) \vert = 1$ (cf. \eqref{eqn:norm:leq:1}).
Therefore, $g_{(\XXX,\, \DDD)}(x)$ does not depend on the choice of $f$.
Let us see the following proposition:

\begin{Proposition}
\label{prop:green:model}
\begin{enumerate}
\renewcommand{\labelenumi}{(\arabic{enumi})}
\item
$g_{(\XXX,\, \DDD)}$ is a $D$-Green function of $C^0$-type on $X^{\an}$.

\item
Let $Y$ be another projective and geometrically integral variety over $k$
and let $\nu : Y \to X$ be a morphism over $k$.
Let $\mathcal{Y}$ be a model of $Y$ such that 
there is a morphism $\tilde{\nu} : \mathcal{Y} \to \XXX$
over $\Spec(k^{\circ})$ as an extension of $\nu$.
Then $g_{(\mathcal{Y}, \tilde{\nu}^*(\DDD))} = \nu^{\an} \circ g_{(\XXX,\, \DDD)}$
on $Y^{\an} \setminus \Supp_{\RR}(\nu^*(D))^{\an}$.

\item
Let $\tilde{\mu} : \XXX' \to \XXX$ be the normalization of $\XXX$. Let $X'$ the generic fiber of $\XXX' \to \Spec(k^{\circ})$
and $\mu : X' \to X$ the induced morphism \rom{(}note that $X'$ is normal\rom{)}.
We assume that the associated $\RR$-Weil divisor of $\mu^*(D)$ is effective.
Then the associated $\RR$-Weil divisor of $\tilde{\mu}^*(\DDD)$ is effective
if and only if
$g_{(\XXX,\, \DDD)} \geq 0$ on $X^{\an} \setminus \Supp_{\RR}(D)^{\an}$.
\end{enumerate}
\end{Proposition}

\begin{proof}
(1)
Let $\XXX = \bigcup_{i=1}^N \Spec(\mathcal{A}_i)$ be an affine open covering of $\XXX$ such that
we have a local equation $f_i \in \Rat(X)^{\times}_{\RR}$ of $\DDD$ on $\mathcal{U}_i := \Spec(\mathcal{A}_i)$,
where $\mathcal{A}_i$ is a $k^{\circ}$-algebra  for each $i = 1, \ldots, N$.
We set $C_i = r_{\XXX}^{-1}(\mathcal{U}_i \cap \XXX_{\circ})$.
Then $C_i$ is closed (cf. \cite[Section~2.4]{Be}) and $\bigcup_{i=1}^N C_i = X^{\an}$.

First let us see that
$g_{(\XXX,\, \DDD)} : X^{\an} \setminus \Supp_{\RR}(D)^{\an} \to \RR$ is continuous.
By our construction, 
$g_{(\XXX,\, \DDD)} (x) = -\log \vert f_{i}(x) \vert^2$
for $x \in C_i \setminus \Supp_{\RR}(D)^{\an}$.
Thus $g_{(\XXX,\, \DDD)}$ is continuous on $C_i\setminus \Supp_{\RR}(D)^{\an}$.
Let $Z$ be a closed subset in $\RR_{\geq 0}$.
As 
\[
\left(\rest{g_{(\XXX,\, \DDD)}}{C_i \setminus \Supp_{\RR}(D)^{\an}}\right)^{-1}(Z)
= g_{(\XXX,\, \DDD)}^{-1}(Z) \cap (C_i \setminus \Supp_{\RR}(D)^{\an}), 
\]
$g_{(\XXX,\, \DDD)}^{-1}(Z) \cap (C_i \setminus \Supp_{\RR}(D)^{\an})$ is closed in $C_i \setminus \Supp_{\RR}(D)^{\an}$, and
hence 
\[
g_{(\XXX,\, \DDD)}^{-1}(Z) \cap (C_i \setminus \Supp_{\RR}(D)^{\an})
\]
is closed in $X^{\an} \setminus \Supp_{\RR}(D)^{\an}$.
Note that
\begin{multline*}
\qquad \bigcup_{i=1}^N \left(g_{(\XXX,\, \DDD)}^{-1}(Z) \cap (C_i \setminus \Supp_{\RR}(D)^{\an})\right)  \\
= 
g_{(\XXX,\, \DDD)}^{-1}(Z) \cap \bigcup_{i=1}^N (C_i \setminus \Supp_{\RR}(D)^{\an})
= g_{(\XXX,\, \DDD)}^{-1}(Z).\qquad
\end{multline*}
Thus $g_{(\XXX,\, \DDD)}^{-1}(Z)$ is closed in $X^{\an} \setminus \Supp_{\RR}(D)^{\an}$, 
so that $g_{(\XXX,\, \DDD)}$ is continuous on $X^{\an} \setminus \Supp_{\RR}(D)^{\an}$.

Since $f_i$ is a local equation of $D$ on $U_i = \mathcal{U}_i \cap X$, in order to see that
$g_{(\XXX,\, \DDD)}$ is a $D$-Green function of $C^0$-type, it is sufficient to see that
$\psi = g_{(\XXX,\, \DDD)} + \log \vert f_i \vert^2$ extends to a continuous function on $U_i^{\an}$,
which is obvious because $\psi = 0$ on $U_i^{\an}$.

\medskip
(2) First note that $\nu(\Supp_{\RR}(\nu^*(D))) \subseteq \Supp_{\RR}(D)$, so that
\[
(\nu^{\an})^{-1}(X^{\an} \setminus \Supp_{\RR}(D)^{\an}) \subseteq Y^{\an} \setminus \Supp_{\RR}(\nu^*(D))^{\an}.
\]
We set $C'_i = r_{\mathcal{Y}}^{-1}(\tilde{\nu}^{-1}(\mathcal{U}_i) \cap \mathcal{Y}_{\circ})$.
Let $y \in C'_i \setminus \Supp(\nu^*(D))^{\an}$, $\xi' = r_{\mathcal{Y}}(y)$
and $\xi = r_{\XXX}(\nu^{\an}(y))$.
Note that $\xi = \tilde{\nu}(\xi') \in \mathcal{U}_i \cap \XXX_{\circ}$ by \eqref{eqn:norm:leq:1}.
Then, as $\tilde{\nu}^*(f_i)$ is a local equation of $\tilde{\nu}^*(\DDD)$ at $\xi'$,
\[
g_{(\mathcal{Y},\, \tilde{\nu}^*(\DDD))}(y) = -\log \vert \tilde{\nu}^*(f_i) \vert^2_y
= -\log \vert f_i \vert^2_{\nu^{\an}(y)} = g_{(\XXX,\, \DDD)}(\nu^{\an}(y)),
\]
as required.

\medskip
(3)
By virtue of (2), $g_{(\XXX',\: \tilde{\mu}^*(\DDD))} = g_{(\XXX,\, \DDD)} \circ \mu^{\an}$ on
$(X')^{\an} \setminus \Supp_{\RR}(\mu^*(D))^{\an}$.
Moreover, $\mu^{\an} : (X')^{\an} \to X^{\an}$ is surjective by \cite[Proposition~3.4.6]{Be}.
Thus we may assume that $\XXX$ is normal.

First we assume that 
$\DDD$  is effective as a Weil divisor.
Then, $\ord_{\Gamma}(f_i) \geq 0$ for any prime divisor $\Gamma$ on $\mathcal{U}_i$.
Thus, by Hartogs' lemma for $\RR$-rational functions (cf. Lemma~\ref{lem:Hartogs:R:rat:fun}),
there are $h_1, \ldots, h_r \in \AAA_i \setminus \{ 0 \}$ and $a_1, \ldots, a_r \in \RR_{>0}$
with $f_i = h_1^{a_1} \cdots h_r^{a_r}$.
Note that $\vert h_j \vert_x \leq 1$ for $j=1, \ldots, r$ and $x \in C_i \setminus \Supp_{\RR}(D)^{\an}$,
and hence $\vert f_i \vert_x \leq 1$, as required.

\smallskip
Next we assume that $g_{(\XXX,\, \DDD)} \geq 0$ on $X^{\an} \setminus \Supp_{\RR}(D)^{\an}$.
Let us see that $\DDD$ is effective as a Weil divisor.
If $v$ is trivial, then $\DDD = D$ is effective,
so that we may assume that $v$ is non-trivial.
It is sufficient to show that the coefficient of $\DDD_W$ 
with respect to a vertical 
prime divisor $\Gamma$ is non-negative.
Let us consider a multiplicative seminorm $x_{\Gamma} = \vert\cdot\vert_{x_{\Gamma}}$ given by
\[
\vert f \vert_{x_{\Gamma}} := v(\varpi)^{\ord_{\Gamma}(f)/\ord_{\Gamma}(\varpi)}
\]
for $f \in \Rat(X)$,
where $\varpi$ is a uniformizing parameter of $k^{\circ}$.
As $x_{\Gamma} \in X^{\an} \setminus \Supp_{\RR}(D)^{\an}$,
if $x_{\Gamma} \in C_i$, then 
\[
0 \leq g_{(\XXX,\, \DDD)}(x_{\Gamma}) = -\log \vert f_i \vert^2_{x_{\Gamma}} = 
-2 \log v(\varpi) \frac{\ord_{\Gamma}(f_i)}{\ord_{\Gamma}(\varpi)},
\]
and hence $\ord_{\Gamma}(f_i) \geq 0$, as desired.
\end{proof}

Here we discuss a more sophisticated maximum problem of Green functions on a smooth projective curve
than (2) in Proposition~\ref{prop:diff:two:D:Green:fun}.

\begin{Proposition}
\label{prop:max:Green:adelic:01}
We assume that $X$ is a smooth projective curve over $k$. 
Let $D_1, \ldots, D_r$ be $\RR$-Cartier divisors on $X$ and
let $D := \max \{ D_1, \ldots, D_r \}$ \rom{(}Conventions and terminology~\rom{\ref{CT:max:min:divisor}}\rom{)}.
For each $i=1, \ldots, r$,
let $g_i$ be a $D_i$-Green function of $C^0$-type on $X^{\an}$.
We set 
\[
g := \max \{ g_1, \ldots, g_r \}
\]
on $X^{\an} \setminus (\Supp_{\RR}(D_1)^{\an} \cup \cdots \cup \Supp_{\RR}(D_r)^{\an})$.
Then $g$ extends to a continuous function on $X^{\an} \setminus \Supp_{\RR}(D)^{\an}$, and
$g$ yields a $D$-Green function of $C^0$-type on $X^{\an}$, which is also denoted by
\[
\max \{ g_1, \ldots, g_r \}
\]
by abuse of notation.
\end{Proposition}

\begin{proof}
Let $x_1, \ldots, x_N$ be closed points of $X$ such that, for each $i=1, \ldots, r$,
\[
D_i = a_{i1} x_1 + \cdots + a_{iN} x_N
\]
for some $a_{ij} \in \RR$. If we set $a_j = \max \{ a_{1j}, \ldots, a_{rj} \}$ for $j=1, \ldots, N$,
then 
\[
D = a_1 x_1 + \cdots + a_N x_N.
\]
Let $\tilde{x}_j$ be a unique point of $X^{\an}$ such that $\{ \tilde{x}_j \} = \{ x_j \}^{\an}$.

\begin{Claim}
If $D = 0$, then $g$ is continuous on $X^{\an}$.
\end{Claim}

\begin{proof}
Clearly $g$ is continuous on $X^{\an} \setminus \{ \tilde{x}_1, \ldots, \tilde{x}_N \}$, so that
we need to show that $g$ is continuous at each $\tilde{x}_j$.
Let us choose an affine open set $U$ of $X$ such that 
\[
\{ x_1, \ldots, x_N \} \cap U = \{ x_j \}.
\]
As $a_{ij} \leq 0$, we can see that $g_i$ is upper-semicontinuous on $U^{\an}$.
Moreover, $g_i$ is continuous on $U^{\an}$ if $a_{ij} = 0$.
We set 
\[
I := \{ i = 1, \ldots, r \mid a_{ij} = 0 \}
\quad\text{and}\quad
I' := \{ i = 1, \ldots, r \mid a_{ij} < 0 \}.
\]
Note that $I \not= \emptyset$ because $\max \{ a_{1j}, \ldots, a_{rj} \} = a_j = 0$, so that
we choose $i_0 \in I$.
Here we set
\[
V := \left\{ x \in U^{\an} \mid
\text{$g_{i_0}(x) > g_{i_0}(\tilde{x}_j) - 1$ and $g_{i'}(x) < g_{i_0}(\tilde{x}_j) - 1$ for all $i' \in I'$}
\right\}.
\]
Then $V$ is an open set and $\tilde{x}_j \in V$.
Further 
\[
\max \{ g_i(x) \mid i \in I \} \geq g_{i_0}(x) > g_{i_0}(\tilde{x}_j) - 1 > \max \{ g_{i'}(x) \mid i' \in I' \}
\]
on $V \setminus \{ \tilde{x}_j \}$, and hence $g(x) = \max \{ g_i(x) \mid i \in I \}$ on $V \setminus \{ \tilde{x}_j \}$.
Therefore, the claim follows because $g_i$ is continuous on $U^{\an}$ for $i \in I$.
\end{proof}

Let $g'$ be a $D$-Green function of $C^0$-type on $X^{\an}$.
It is sufficient to see that $g - g'$ extends to a continuous function on $X^{\an}$.
Clearly $g - g'$ is continuous on $X^{\an} \setminus \{ \tilde{x}_1, \ldots, \tilde{x}_N \}$ and
\[
g - g' = \max \{ g_1 - g', \ldots, g_r - g' \}
\]
on $X^{\an} \setminus \{ \tilde{x}_1, \ldots, \tilde{x}_N \}$.
Note that $g_i - g'$ is a $(D_i - D)$-Green function of $C^0$-type on $X^{\an}$ and
$\max \{ D_1 - D, \ldots, D_r - D \} = 0$,
so that the assertion follows from the above claim.
\end{proof}

Finally let us consider a Green function of $(C^0 \cap \Tpsh)$-type,
which is a counterpart of a semipositive metric.

\begin{Definition}
Let $g$ be a $D$-Green function of $C^0$-type on $X^{\an}$.
We say $g$ is {\em of $(C^0 \cap \Tpsh)$-type} 
\index{\AdelDivSubject}{of (C^0 \cap \Tpsh)-type@of $(C^0 \cap \Tpsh)$-type}%
if $D$ is nef and
there is a sequence $\{ (\XXX_n, \DDD_n) \}_{n=1}^{\infty}$ of
models  of $(X, D)$ with the following properties:
\begin{enumerate}
\renewcommand{\labelenumi}{(\arabic{enumi})}
\item
For each $n \geq 1$, $\DDD_n$ is relatively nef with respect to $\XXX_n \to \Spec(k^{\circ})$
(cf Conventions and terminology~\ref{CT:rel:nef}).

\item
If we set $\phi_n = g_{(\XXX_n,\, \DDD_n)} - g$, then $\lim_{n\to\infty} \Vert \phi_n \Vert_{\sup} = 0$.
\end{enumerate}
\end{Definition}

As an application of results in Appendix (cf. Corollary~\ref{cor:nef:approx:nef}),
we have the following characterization of relatively nef divisors.

\begin{Proposition}
\label{prop:PSH:imply:nef}
Let $\XXX$ be a normal model of $X$ and let $\DDD$ be an $\RR$-Cartier divisor on $\XXX$.
Then $g_{(\XXX, \DDD)}$ is of $(C^0 \cap \Tpsh)$-type if and only if
$\DDD$ is relatively nef.
\end{Proposition}

In addition,
we have the following propositions.

\begin{Proposition}
\label{prop:Green:psh:C0:approx}
We assume that $v$ is non-trivial.
Let $D$ be a nef $\RR$-Cartier divisor on $X$.
Let $g$ be a $D$-Green function of  $(C^0 \cap \Tpsh)$-type.
Then there are sequences $\{ (\XXX_n, \DDD_n) \}_{n=1}^{\infty}$ and $\{ (\XXX_n, \DDD'_n) \}_{n=1}^{\infty}$ of
models of $(X, D)$ with the following properties:
\begin{enumerate}
\renewcommand{\labelenumi}{(\arabic{enumi})}
\item
For all $n \geq 1$, $\DDD_n$ and $\DDD'_n$  are relatively nef 
with respect to $\XXX_n \to \Spec(k^{\circ})$.

\item
$g_{(\XXX_n,\, \DDD_n)} \leq g \leq g_{(\XXX_n,\, \DDD'_n)}$ for all $n \geq 1$.

\item
If we set $\phi_n = g_{(\XXX_n,\, \DDD_n)} - g$ and $\phi'_n = g_{(\XXX_n,\, \DDD'_n)} - g$,
then 
\[
\lim_{n\to\infty} \Vert \phi_n \Vert_{\sup} = \lim_{n\to\infty} \Vert \phi'_n \Vert_{\sup} = 0.
\]
\end{enumerate}
\end{Proposition}

\begin{proof}
By its definition,
there is a sequence $\{ (\XXX_n, \DDD''_n) \}_{n=1}^{\infty}$ of
models  of $(X, D)$ with the following properties:
\begin{enumerate}
\renewcommand{\labelenumi}{(\roman{enumi})}
\item
For all $n \geq 1$, $\DDD''_n$ is relatively nef with respect to $\XXX_n \to \Spec(k^{\circ})$.

\item
If we set $\phi''_n = g_{(\XXX_n,\, \DDD''_n)} - g$, then $\lim_{n\to\infty} \Vert \phi''_n \Vert_{\sup} = 0$.
\end{enumerate}
Here we set
\[
\DDD_n := \DDD''_n - \frac{\Vert \phi''_n \Vert_{\sup}}{-2\log v(\varpi)} (\XXX_n)_{\circ}
\quad\text{and}\quad
\DDD_n := \DDD''_n + \frac{\Vert \phi''_n \Vert_{\sup}}{-2\log v(\varpi)} (\XXX_n)_{\circ},
\]
where $(\XXX_n)_{\circ}$ is the central fiber of $\XXX_n \to \Spec(k^{\circ})$.
Then
\[
g_{(\XXX_n,\, \DDD_n)} = g_{(\XXX_n,\, \DDD''_n)} - \Vert \phi''_n \Vert_{\sup}
\quad\text{and}\quad
g_{(\XXX_n,\, \DDD'_n)} = g_{(\XXX_n,\, \DDD''_n)} + \Vert \phi''_n \Vert_{\sup}, 
\]
and hence
\[
g_{(\XXX_n,\, \DDD_n)} - g = \phi''_n - \Vert \phi''_n \Vert_{\sup} \leq 0
\quad\text{and}\quad
g_{(\XXX_n,\, \DDD'_n)} - g = \phi''_n + \Vert \phi''_n \Vert_{\sup} \geq 0,
\]
as required.
\end{proof}

\ifmonog\section{Definition of adelic $\RR$-Cartier divisors}\fi
\ifpaper\subsection{Definition of adelic $\RR$-Cartier divisors}\fi
\label{subsec:def:adelic:R:Cartier:divisor}

We assume that 
$k^{\circ}$ is excellent.
Let $X$ be a projective, smooth and geometrically integral variety over $k$.
Let $k_v$ be the completion of $k$ with respect to $v$.
By abuse of notation, the unique extension of $v$ to $k_v$ is also denoted by $v$.
We set 
\[
X_v := X \times_{\Spec(k)} \Spec(k_v), 
\]
which is also
a projective, smooth and geometrically integral variety over $k_v$.

A pair $\overline{D} = (D, g)$
is called an {\em adelic $\RR$-Cartier divisor of $C^0$-type on $X$} if
$D$ is an $\RR$-Cartier divisor on $X$ and $g$ is
a $D$-Green function of $C^0$-type on $X_v^{\an}$.
\index{\AdelDivSubject}{adelic R-Cartier divisor of $C^0$-type@adelic $\RR$-Cartier divisor of $C^0$-type}%
If $D$ is nef and $g$ is of $(\Tpsh \cap C^0)$-type,
then $\overline{D}$ is said to be {\em relatively nef}.
\index{\AdelDivSubject}{relatively nef adelic R-Cartier divisor@relatively nef adelic $\RR$-Cartier divisor}%
Moreover, we say $\overline{D}$ is {\em integrable} 
\index{\AdelDivSubject}{integrable adelic R-Cartier divisor@integrable adelic $\RR$-Cartier divisor}%
if there are relatively nef adelic $\RR$-Cartier divisors
$\overline{D}'$ and $\overline{D}''$ of $C^0$-type on $X$ such that
$\overline{D} = \overline{D}' - \overline{D}''$.
We say a continuous function $\phi$ on $X_v^{\an}$ is {\em integrable}
if $(0, \phi)$ is integrable as an adelic $\RR$-Cartier divisor.
Let $\overline{D}' = (D', g')$ be another adelic $\RR$-Cartier divisor of $C^0$-type on $X$.
For $a, a' \in \RR$, we define $a\overline{D} + a'\overline{D}'$ to be
\[
a\overline{D} + a'\overline{D}' : = (aD + a'D', ag + a'g').
\]
The space of all adelic $\RR$-Cartier divisors of $C^0$-type is denoted by
$\Div_{C^0}^{\ad}(X)_{\RR}$, which forms
a vector space over $\RR$ by the above formula.
\index{\AdelDivSymbol}{0Div:Div_{C^0}^{ad}(X)_{RR}@$\Div_{C^0}^{\ad}(X)_{\RR}$}%
For $\overline{D}_1 = (D_1, g_1),
\overline{D}_2 = (D_2, g_2) \in \Div_{C^0}^{\ad}(X)_{\RR}$,
we define $\overline{D}_1 \leq \overline{D}_2$ to be
\[
\overline{D}_1 \leq \overline{D}_2
\quad\overset{\operatorname{def}}{\Longleftrightarrow}\quad
\text{$D_1 \leq D_2$ and $g_{1} \leq g_{2}$}.
\]%
\index{\AdelDivSymbol}{0D:D_1 leq D_2@$D_1 \leq D_2$}%
Let $\XXX$ be a normal model of $X$ over $\Spec(k^{\circ})$ and let $\DDD$ 
be an $\RR$-Cartier divisor on $\XXX$.
The pair
$(\XXX, \DDD)$ gives rise to an adelic $\RR$-Cartier divisor of $C^0$-type on $X$,
that is, the pair $(\DDD \cap X, g_{(\XXX,\, \DDD)})$ of $\DDD \cap X$ and
$g_{(\XXX,\, \DDD)}$.
We denote it by $\DDD^{\ad}$ and it is
called the {\em associated adelic $\RR$-Cartier divisor with $\DDD$}.
\index{\AdelDivSubject}{associated adelic R-Cartier divisor@associated adelic $\RR$-Cartier divisor}%
By abuse of notation,
we often use the notations $\DDD \leq \overline{D}_2$ and $\overline{D}_1 \leq \DDD$ instead of
$\DDD^{\ad} \leq \overline{D}_2$ and $\overline{D}_1 \leq \DDD^{\ad}$, respectively

\begin{Proposition}
\label{prop:comp:adelic}
Let $\XXX$ be a normal model of $X$ over $\Spec(k^{\circ})$ and let $\Div(\XXX)_{\RR}$ be
the group of $\RR$-Cartier divisors on $\XXX$.
Let $\iota : \Div(\XXX)_{\RR} \to \Div^{\ad}_{C^0}(X)_{\RR}$
be the map given by $\DDD \mapsto \DDD^{\ad}$.
Then we have the following:
\begin{enumerate}
\renewcommand{\labelenumi}{(\arabic{enumi})}
\item
The map $\iota : \Div(\XXX)_{\RR} \to \Div^{\ad}_{C^0}(X)_{\RR}$ is an
injective homomorphism of $\RR$-vector spaces.

\item
$\DDD_1 \leq \DDD_2$ $\Longleftrightarrow$ 
$\DDD^{\ad}_1 \leq \DDD^{\ad}_2$.
\end{enumerate}
\end{Proposition}

\begin{proof}
Clearly $\iota$ is a homomorphism of $\RR$-vector spaces.
(2) is a consequence of Proposition~\ref{prop:green:model}.
The injectivity of $\iota$ follows from (2).
\end{proof}

\ifmonog\section{Local degree}\fi
\ifpaper\subsection{Local degree}\fi
\label{subsec:local:degree}

We assume that $k^{\circ}$ is excellent and $k$ is perfect.
We use the same notation as in 
\ifmonog Section~\ref{subsec:def:adelic:R:Cartier:divisor}. \fi
\ifpaper Subsection~\ref{subsec:def:adelic:R:Cartier:divisor}. \fi
Let $x$ be a closed point of $X$ with $x \not\in \Supp_{\RR}(D)$.
Let $k(x)$ be the residue field at $x$.
As $k(x)$ is separable over $k$,
we have $k(x) \otimes_{k} k_v = k_1 \oplus \cdots \oplus k_l$
for some finite separable extensions $k_1, \ldots. k_l$ over $k_v$.
Note that each $k_i$ has the unique extension $v_i$ of $v$.
The {\em local degree of $\overline{D}$ along $x$ over $v$} 
\index{\AdelDivSubject}{local degree@local degree}%
\index{\AdelDivSymbol}{0d:adeg_v(overline{D}{x})@$\adeg_v(\srest{\overline{D}}{x})$}%
is defined by
\[
\adeg_v(\srest{\overline{D}}{x}) := \sum_{i=1}^l \frac{[k_i : k_v]}{2} g(v_i).
\]

Here we assume that $k$ is a number field and $v(f) = \#(O_k/P)^{-\ord_P(f)}$,
where $O_k$ is the ring of integers in $k$ and $P$ is a maximal ideal of $O_k$.
Let $\XXX$ be a normal model of $X$ over $\Spec((O_k)_P)$.
Let $\DDD = a_1 \DDD_1 + \cdots + a_l \DDD_r$ be an $\RR$-Cartier divisor on $\XXX$
such that $\DDD \cap X = D$, $a_1, \ldots, a_r \in \RR$ and
$\DDD_1, \ldots, \DDD_r$ are effective Cartier divisors on $\XXX$.
We assume that $g = g_{(\XXX,\, \DDD)}$ and
$x \not\in \Supp_{\ZZ}(\DDD_1) \cup \cdots \cup \Supp_{\ZZ}(\DDD_r)$.
Let $O_{k(x)}$ be the ring of integers in $k(x)$.
Then it is easy to see that
\frontmatterforspececialeqn
\begin{equation}
\label{eqn:local:degree:arithmetic}
\adeg_v(\srest{\overline{D}}{x}) = \sum_{j=1}^r a_j \log \# \left ( \left(O_{k(x)}(\DDD_j)/O_{k(x)}\right)_P \right).
\end{equation}
\backmatterforspececialeqn

\ifmonog \section{Local intersection number}\fi
\ifpaper \subsection{Local intersection number}\fi
\label{subsec:local:intersection}

We assume that $k^{\circ}$ is excellent and $v$ is non-trivial.
We use the same notation as in 
\ifmonog Section~\ref{subsec:def:adelic:R:Cartier:divisor}. \fi
\ifpaper Subsection~\ref{subsec:def:adelic:R:Cartier:divisor}. \fi
Let $\phi$ be a continuous function on $X^{\an}_v$.
Let $\XXX$ be a normal model of $X$ and
let $\LLL_1, \ldots, \LLL_{d}$ be $\RR$-Cartier divisors on $\XXX$.
Let $\Gamma_1, \ldots, \Gamma_r$ be
irreducible components of the central fiber $\XXX_{\circ}$
of $\XXX \to \Spec(k^{\circ})$.
Let $v_j$ be the discrete valuation arising from $\Gamma_j$ such that $\rest{v_j}{k} = v$. 
By Lemma~\ref{lem:base:change:completion}, 
there is a unique $\tilde{v}_j \in X_v^{\an}$ such that
the restriction of $\tilde{v}_j$ to $\Rat(X)$ is $v_j$.
For each $i$ and $j$, we can choose a unique real number $\lambda_{ij}$ such that
$\Gamma_j \not\subseteq \Supp_{W} (\LLL_i + \lambda_{ij} \XXX_{\circ})$ (for $\Supp_W$, see Definition~\ref{def:Supp:RR:Cartier:div}).
Then the number given by
\[
\sum_{j=1}^r \frac{\phi(\tilde{v}_j)\ord_{\Gamma_j}(\varpi)}{-2 \log v(\varpi)} \deg_{k^{\circ}/k^{\circ\circ}} 
\left(\rest{(\LLL_1 + \lambda_{1j}\XXX_{\circ})}{\Gamma_j} \cdots \rest{(\LLL_d + \lambda_{dj}\XXX_{\circ})}{\Gamma_j}\right)
\]
is denoted by $\adeg_{v}(\LLL_1 \cdots \LLL_{d} ; \phi)$,
where $\deg_{k^{\circ}/k^{\circ\circ}}$ is the degree over $k^{\circ}/k^{\circ\circ}$.
\index{\AdelDivSymbol}{0d:adeg_v(LLL_1 cdots LLL_{d};phi)@$\adeg_v(\LLL_1 \cdots \LLL_{d} ; \phi)$}%
Obviously, $\adeg_{v}(\LLL_1 \cdots \LLL_{d} ; \phi)$ is multi-linear with respect to $\LLL_1, \ldots,  \LLL_{d}$.
In addition,
\[
\adeg_{v}(\LLL_1 \cdots \LLL_{d} ; a\phi + a'\phi') = a \adeg_{v}(\LLL_1 \cdots \LLL_{d} ; \phi) + 
a' \adeg_{v}(\LLL_1 \cdots \LLL_{d} ; \phi')
\]
for $a, a' \in \RR$ and $\phi, \phi' \in C^0(X^{\an}_v)$.
Let $\EEE$ be a vertical $\RR$-Cartier divisor on $\XXX$ and let $\EEE = a_1 \Gamma_1 + \cdots + a_r \Gamma_r$
be the irreducible decomposition of $\EEE$ as a Weil divisor.
Let $\phi_{\EEE}$ be the continuous function arising from $\EEE$, that is,
$\phi_{\EEE} = g_{(\XXX,\, \EEE)}$. Then
\begin{align*}
\adeg_{v}(\LLL_1 \cdots \LLL_{d} ; \phi_{\EEE}) & =
\sum_{j=1}^r a_j  \deg_{k^{\circ}/k^{\circ\circ}} 
\left(\rest{(\LLL_1 + \lambda_{1j}\XXX_{\circ})}{\Gamma_j} \cdots \rest{(\LLL_d + \lambda_{dj}\XXX_{\circ})}{\Gamma_j}\right) \\
& = \deg_{k^{\circ}/k^{\circ\circ}} 
\left( \LLL_1 \cdots \LLL_d \cdot \EEE \right).
\end{align*}
If $\LLL_1, \ldots, \LLL_d$ are relatively nef and
$\phi \leq \phi'$, then
\frontmatterforspececialeqn
\begin{equation}
\label{eqn:phi:leq:phi:prime:ineq}
\adeg_{v}(\LLL_1 \cdots \LLL_{d} ; \phi)
\leq \adeg_{v}(\LLL_1 \cdots \LLL_{d} ; \phi').
\end{equation}
\backmatterforspececialeqn
Let us begin with the following lemma:

\begin{Lemma}
\label{lem:rel:nef:diff:adelic}
Let $\LLL_1, \ldots, \LLL_{d+2}, \LLL'_1, \ldots, \LLL'_{d+2}$
be relatively nef
$\RR$-Cartier divisors on $\XXX$. 
We assume that there are $a_1, \ldots, a_d, a_{d+1} \in \RR_{\geq 0}$
with the following properties:
\begin{enumerate}
\renewcommand{\labelenumi}{(\arabic{enumi})}
\item
For all $i = 1, \ldots, d$, $\LLL_i \cap X = \LLL'_i \cap X$ and
$-a_{i} \XXX_{\circ} \leq
\LLL'_i - \LLL_i  \leq a_{i} \XXX_{\circ}$.

\item
$\LLL_{d+1} \cap X = \LLL_{d+2} \cap X$ and $\LLL'_{d+1} \cap X = \LLL'_{d+2} \cap X$.
Moreover, 
\[
-2a_{d+1}g_{\left(\XXX,\, \XXX_{\circ}\right)} \leq \psi' - \psi \leq
2a_{d+1}g_{\left(\XXX,\, \XXX_{\circ}\right)},
\]
where
$\psi := g_{\left(\XXX,\, \LLL_{d+1} - \LLL_{d+2}\right)}$ and
$\psi' := g_{(\XXX,\, \LLL'_{d+1} - \LLL'_{d+2})}$.
\end{enumerate}
Then we have the following inequality:
\[
\left\vert\adeg_{v} (\LLL_1 \cdots \LLL_{d};\psi) - 
\adeg_{v} (\LLL'_1 \cdots \LLL'_{d};\psi') \right\vert
\leq
\sum_{i=1}^{d+1}  2a_{i} \deg (L_1 \cdots L_{i-1}\cdot L_{i+1} \cdots L_{d+1}) ,
\]
where $L_i := \LLL_i \cap X$ for $i=1, \ldots, d+1$.
\end{Lemma}

\begin{proof}
As $\LLL'_1, \ldots, \LLL'_{d}$
are relatively nef and 
$-2a_{d+1}g_{\left(\XXX,\, \XXX_{\circ}\right)} \leq \psi' - \psi \leq
2a_{d+1}g_{\left(\XXX,\, \XXX_{\circ}\right)}$, by using \eqref{eqn:phi:leq:phi:prime:ineq},
we have
\[
\left\vert \adeg_{v} (\LLL'_1 \cdots \LLL'_{d};\psi') - \adeg_{v} (\LLL'_1 \cdots \LLL'_{d};\psi) \right\vert \leq
2a_{d+1} \deg(L_1 \cdots L_{d}),
\]
and hence, 
\begin{multline*}
\left\vert\adeg_{v} (\LLL_1 \cdots \LLL_{d};\psi) - 
\adeg_{v} (\LLL'_1 \cdots \LLL'_{d};\psi') \right\vert \\
\leq \left\vert\adeg_{v} (\LLL_1 \cdots \LLL_{d};\psi) - 
\adeg_{v} (\LLL'_1 \cdots \LLL'_{d};\psi) \right\vert
+ 2a_{d+1} \deg(L_1 \cdots L_{d}).
\end{multline*}
On the other hand, by Lemma~\ref{lem:mult:linear:diff:formula}, 
\begin{multline*}
\qquad\adeg_{v} (\LLL'_1 \cdots \LLL'_{d};\psi) =
\adeg_{v} (\LLL_1 \cdots \LLL_{d};\psi) \\
+
\sum_{i=1}^{d} \adeg_{v} (\LLL'_1 \cdots \LLL'_{i-1}
\cdot \EEE_i \cdot \LLL_{i+1} \cdots\LLL_{d}; \psi),\qquad
\end{multline*}
where $\EEE_i = \LLL'_i - \LLL_i$.
Let $\phi_i$ be the continuous function arising from $\EEE_i$.
Then, as 
\begin{multline*}
\qquad\adeg_{v} (\LLL'_1 \cdots \LLL'_{i-1}
\cdot \EEE_i \cdot \LLL_{i+1} \cdots\LLL_{d}; \psi) \\
= \adeg_{v} (\LLL'_1 \cdots \LLL'_{i-1}
\cdot (\LLL_{d+1} - \LLL_{d+2}) \cdot \LLL_{i+1} \cdots\LLL_{d}; \phi_i)\qquad
\end{multline*}
and $-a_{i} g_{\left(\XXX,\, \XXX_{\circ}\right)} \leq
\phi_i \leq a_{i} g_{\left(\XXX,\, \XXX_{\circ}\right)}$,
by using \eqref{eqn:phi:leq:phi:prime:ineq},
we can see that
\begin{multline*}
\left\vert \adeg_{v} (\LLL'_1 \cdots \LLL'_{i-1}
\cdot \EEE_i \cdot \LLL_{i+1} \cdots\LLL_{d}; \psi) \right\vert \\
\leq 
\left\vert \adeg_{v} (\LLL'_1 \cdots \LLL'_{i-1}
\cdot \LLL_{d+1}\cdot \LLL_{i+1} \cdots\LLL_{d}; \phi_i) \right\vert \\
\qquad\qquad\qquad\qquad+\left\vert \adeg_{v} (\LLL'_1 \cdots \LLL'_{i-1}
\cdot \LLL_{d+2} \cdot \LLL_{i+1} \cdots\LLL_{d}; \phi_i) \right\vert \\
\leq 2a_i \deg(L_1 \cdots L_{i-1} \cdot L_{i+1} \cdots L_{d} \cdot L_{d+1}),
\end{multline*}
as desired.
\end{proof}

The next proposition guarantees the intersection pairing of integrable
adelic $\RR$-Cartier divisors along an integrable continuous function.

\begin{PropDef}
\label{propdef:adelic:intersection}
Let $\overline{L}_1 = (L_1, g_1), \ldots, \overline{L}_{d} = (L_{d}, g_{d})$ be
relatively nef adelic $\RR$-Cartier divisors on $X$, and let $\phi$ be an integrable continuous function
on $X^{\an}_v$.
Then there are sequences
\begin{multline*}
\qquad\left\{ (\XXX_{1, n}, \LLL_{1, n}) \right\}_{n=1}^{\infty}, \ldots, \left\{ (\XXX_{d, n}, \LLL_{d, n}) \right\}_{n=1}^{\infty}, \\
\left\{ (\XXX_{d+1, n}, \LLL_{d+1, n}) \right\}_{n=1}^{\infty}, \left\{ (\XXX_{d+2, n}, \LLL_{d+2, n}) \right\}_{n=1}^{\infty}\qquad
\end{multline*}
with the following properties:
\begin{enumerate}
\renewcommand{\labelenumi}{(\arabic{enumi})}
\item
$\XXX_{i, n}$ is a normal model of $X$ over $\Spec(k^{\circ})$ and
$\LLL_{i,n}$ is a
relatively nef\ \  $\RR$-Cartier divisor on $\XXX_{i, n}$
for $i=1, \ldots, d+2$ and $n \geq 1$.

\item
$\LLL_{i,n} \cap X = L_i$ for $i=1, \ldots, d$ and $n \geq 1$.

\item
There is an $\RR$-Cartier divisor $L_{d+1}$ on $X$ such that
$L_{d+1} = \LLL_{d+1, n} \cap X = \LLL_{d+2, n} \cap X$ for all $n \geq 1$.

\item
If we set $\phi_{i, n} := g_i - g_{\left(\XXX_{i,n},\, \LLL_{i, n}\right)}$,
then $\lim_{n\to\infty} \Vert \phi_{i,n} \Vert_{\sup} = 0$ for $i=1, \ldots, d$.

\item
If we set $\psi_n := g_{(\XXX_{d+1, n},\, \LLL_{d+1, n})} - g_{(\XXX_{d+2, n},\, \LLL_{d+2, n})}$,
then 
\[
\lim_{n\to\infty} \Vert \psi_n - \phi \Vert_{\sup} = 0.
\]
\end{enumerate}
For sequences
\begin{multline*}
\qquad\left\{ (\XXX_{1, n}, \LLL_{1, n}) \right\}_{n=1}^{\infty}, \ldots, \left\{ (\XXX_{d, n}, \LLL_{d, n}) \right\}_{n=1}^{\infty}, \\
\left\{ (\XXX_{d+1, n}, \LLL_{d+1, n}) \right\}_{n=1}^{\infty}, \left\{ (\XXX_{d+2, n}, \LLL_{d+2, n}) \right\}_{n=1}^{\infty}\qquad
\end{multline*}
satisfying the above properties,
let $\YYY_n$ be a normal model of $X$ over $\Spec(k^{\circ})$ together with
birational morphisms 
\[
\mu_{i, n} : \YYY_n \to \XXX_{i, n}
\]
for $i=1, \ldots, d$.
Then the following limits
\[
\begin{cases}
\lim_{n\to\infty} \adeg_{v} 
\left( \mu_{1,n}^*(\LLL_{1, n}) \cdots \mu_{d, n}^*(\LLL_{d, n}) ; \phi \right), \\[2ex]
\lim_{n\to\infty} \adeg_{v} 
\left( \mu_{1,n}^*(\LLL_{1, n}) \cdots \mu_{d, n}^*(\LLL_{d, n}) ; \psi_n \right)
\end{cases}
\]
exist and
\begin{multline*}
\qquad\lim_{n\to\infty} \adeg_{v} 
\left( \mu_{1,n}^*(\LLL_{1, n}) \cdots \mu_{d, n}^*(\LLL_{d, n}) ; \phi \right) \\
=
\lim_{n\to\infty} \adeg_{v} 
\left( \mu_{1,n}^*(\LLL_{1, n}) \cdots \mu_{d, n}^*(\LLL_{d, n}) ; \psi_n \right).\qquad
\end{multline*}
Moreover, the above limits do not depend on the choice of
the sequences 
\begin{multline*}
\qquad\left\{ (\XXX_{1,n}, \LLL_{1, n}) \right\}_{n=1}^{\infty}, \ldots, \left\{ (\XXX_{d,n}, \LLL_{d, n}) \right\}_{n=1}^{\infty}, \\
\left\{ (\XXX_{d+1, n}, \LLL_{d+1, n}) \right\}_{n=1}^{\infty}, \left\{ (\XXX_{d+2, n}, \LLL_{d+2, n}) \right\}_{n=1}^{\infty},\qquad
\end{multline*}
so that
this limit is denoted by $\adeg_{v} \left( \overline{L}_1 \cdots \overline{L}_{d} ; \phi \right)$.
\index{\AdelDivSymbol}{0d:adeg_v(overline{L}_1 cdots overline{L}_{d};phi)@$\adeg_v(\overline{L}_1 \cdots \overline{L}_{d} ; \phi)$}%
\end{PropDef}

\begin{proof}
The existence of sequences are obvious by the relative nefness 
of $\overline{L}_1, \ldots, \overline{L}_{d}$ and the integrability of $\phi$.

We set
\[
\begin{cases}
A_n = \adeg_{v} \left( \mu_{1, n}^*(\LLL_{1, n}) \cdots \mu_{d, n}^*(\LLL_{d, n}) ; \phi \right), \\[1ex]
B_n = \adeg_{v} \left( \mu_{1, n}^*(\LLL_{1, n}) \cdots \mu_{d, n}^*(\LLL_{d, n}) ; \psi_n \right).
\end{cases}
\]
For a positive number $\epsilon$, there is $N$ such that
\[
\Vert \psi_n - \phi \Vert_{\sup} \leq 2\epsilon(-\log v(\varpi))
\quad\text{and}\quad
\Vert \phi_{i, n} \Vert_{\sup} \leq \epsilon(-\log v(\varpi))
\]
for all $n \geq N$ and
$i=1, \ldots, d$.
Then, for $n, m \geq N$,
\begin{multline*}
\qquad \left\vert g_{\left(\XXX_{i,n},\, \LLL_{i, n}\right)} - g_{\left(\XXX_{i,m},\, \LLL_{i, m}\right)} \right\vert 
\leq \left\vert g_i - g_{\left(\XXX_{i,n},\, \LLL_{i, n}\right)} \right\vert \\
+
\left\vert g_i - g_{\left(\XXX_{i,m},\, \LLL_{i, m}\right)} \right\vert
\leq 2 \epsilon(-\log v(\varpi))\qquad
\end{multline*}
for $i=1, \ldots, d$.
Let us choose  a normal model $\ZZZ_{n, m}$ of $X$ together with
birational morphisms $\tau_n : \ZZZ_{n, m} \to \YYY_n$ and
$\tau_m : \ZZZ_{n,m} \to \YYY_m$.
Then the above inequality implies
\[
-\epsilon g_{(\ZZZ_{n,m},\, (\ZZZ_{n,m})_{\circ})} \leq 
g_{\left(\ZZZ_{n,m},\, \tau_{n}^*(\mu_{i,n}^*(\LLL_{i, n})) - 
\tau_{m}^*(\mu_{i, m}^*(\LLL_{i, m}))\right)} \leq \epsilon g_{(\ZZZ_{n,m},\, (\ZZZ_{n,m})_{\circ})},
\]
so that, by Proposition~\ref{prop:comp:adelic},
\[
-\epsilon (\ZZZ_{n, m})_{\circ} \leq \tau_{n}^*(\mu_{i,n}^*(\LLL_{i, n})) - \tau_{m}^*(\mu_{i,m}^*(\LLL_{i, m}))
\leq \epsilon (\ZZZ_{n, m})_{\circ}
\]
for $i=1, \ldots, d$.
On the other hand,
\[
\Vert \psi_n - \psi_m \Vert_{\sup}  \leq \Vert \psi_n - \phi \Vert_{\sup} + \Vert \psi_m - \phi \Vert_{\sup}
\leq 4 \epsilon(-\log v(\varpi)).
\]
Therefore, by using Lemma~\ref{lem:rel:nef:diff:adelic},
we have
\[
\left\vert  B_n - B_m \right\vert
\leq 2 \epsilon 
\sum_{i=1}^{d+1} 
\deg (L_1 \cdots L_{i-1} \cdot L_{i+1} \cdots L_{d+1})
\]
for $n, m \geq N$, which shows that
the sequence $\{ B_n \}_{n=1}^{\infty}$ is a Cauchy sequence, so that
its limit exists. Further, as 
$0 \leq \vert \phi - \psi_n \vert \leq 2 \epsilon(-\log v(\varpi))$,
by \eqref{eqn:phi:leq:phi:prime:ineq},
we have 
\[
\vert A_n - B_n \vert \leq \epsilon \deg(L_1 \cdots L_d),
\]
so that $\lim_{n\to\infty} A_n$ exists and
$\lim_{n\to\infty} A_n =\lim_{n\to\infty} B_n$.

\medskip
Let $\left\{ (\XXX'_{1,n}, \LLL'_{1, n}) \right\}_{n=1}^{\infty}, \ldots, 
\left\{ (\XXX'_{d+2,n}, \LLL'_{d+2, n}) \right\}_{n=1}^{\infty}$
be another sequences satisfying the above properties
(1), (2), (3), (4) and (5).
For the above sequences, $L_{d+1}$ in the property (3),
$\phi_{i,n}$ in the property (4) and
$\psi_n$ in the property (5) are denoted by
$L'_{d+1}$, $\phi'_{i,n}$ and $\psi'_n$, respectively.
Replacing $\YYY_n$ by a suitable model of $X$,
we may assume that there are 
birational morphisms $\mu'_{i, n} : \YYY_n \to \XXX'_{i, n}$ ($i=1, \ldots, d$).
For a positive number $\epsilon$, there is $N$ such that
\[
\begin{cases}
\Vert \psi_n - \phi \Vert_{\sup} \leq 2\epsilon(-\log v(\varpi)),\quad
\Vert \phi_{i, n} \Vert_{\sup} \leq \epsilon(-\log v(\varpi)),\\
\Vert \psi'_n - \phi \Vert_{\sup} \leq 2\epsilon(-\log v(\varpi)),\quad
\Vert \phi'_{i, n} \Vert_{\sup} \leq \epsilon(-\log v(\varpi))
\end{cases}
\]
for all $n \geq N$ and
$i=1, \ldots, d$.
Then, 
\begin{multline*}
\qquad\left\vert g_{(\YYY_{n},\, \mu_{i,n}^*(\LLL_{i, n}))} - g_{(\YYY_{n},\, \mu^{\prime *}_{i,n}(\LLL'_{i, n}))} \right\vert \\
\leq \left\vert g_i - g_{(\YYY_{n},\, \mu_{i,n}^*(\LLL_{i, n}))} \right\vert +
\left\vert g_i - g_{(\YYY_{n},\, \mu^{\prime *}_{i,n}(\LLL'_{i, n}))} \right\vert
\leq 2 \epsilon(-\log v(\varpi))
\end{multline*}
for $n \geq N$ and $i=1, \ldots, d$.
Therefore, in the similar way as before,
\[
- \epsilon
 (\YYY_n)_{\circ} \leq \mu_{i, n}^*(\LLL_{i, n}) - \mu^{\prime *}_{i, n}(\LLL'_{i, n})
\leq \epsilon(\YYY_n)_{\circ}.
\]
Moreover, 
\[
\left\Vert \psi_n - \psi'_n
\right\Vert_{\sup} \leq \Vert \psi_n - \phi \Vert_{\sup} + \Vert \psi'_n - \phi \Vert_{\sup} 
\leq 4 \epsilon(-\log v(\varpi))
\]
for $n \geq N$, and hence
the uniqueness of the limit follows from Lemma~\ref{lem:rel:nef:diff:adelic}.
\end{proof}

Let $\phi$ be an integrable continuous function on $X_v^{\an}$.
Let $\overline{L}_1, \ldots, \overline{L}_i, \overline{L}'_i,
\ldots, \overline{L}_{d}$ be relatively nef adelic $\RR$-Cartier divisors
of $C^0$-type on $X$. Then it is easy to see that
\frontmatterforspececialeqn
\begin{multline}
\label{eqn:multlinear:cone:intersection:adelic}
\adeg_{v}(\overline{L}_1 \cdots (a \overline{L}_i + a' \overline{L}'_i) \cdots \overline{L}_{d};\phi) \\
=
a \adeg_{v}(\overline{L}_1 \cdots \overline{L}_i \cdots \overline{L}_{d};\phi) +
a' \adeg_{v}(\overline{L}_1 \cdots \overline{L}'_i \cdots \overline{L}_{d};\phi)
\end{multline}
\backmatterforspececialeqn
for $a, a' \in \RR_{\geq 0}$, and that
\frontmatterforspececialeqn
\begin{equation}
\label{eqn:com:cone:intersection:adelic}
\adeg_{v}(\overline{L}_1 \cdots \overline{L}_i \cdots \overline{L}_j \cdots \overline{L}_{d};\phi)
=
\adeg_{v}(\overline{L}_1 \cdots \overline{L}_j \cdots \overline{L}_i \cdots \overline{L}_{d};\phi).
\end{equation}
\backmatterforspececialeqn

Let $\overline{L}_1, \ldots, \overline{L}_{d}$ be integrable adelic $\RR$-Cartier divisors
of $C^0$-type on $X$. For each $i$, we choose 
relatively nef adelic $\RR$-Cartier divisors $\overline{L}_{i, +1}$ and
$\overline{L}_{i, -1}$ of $C^0$-type such that $\overline{L}_i = \overline{L}_{i, +1} - \overline{L}_{i, -1}$.
By using \eqref{eqn:multlinear:cone:intersection:adelic},
it is not difficult to see that the quantity 
\[
\sum_{\epsilon_1,  \dots, \epsilon_{d+1} \in \{ \pm 1 \}} \epsilon_1 \cdots \epsilon_{d+1}
\adeg_{v} (\overline{L}_{1, \epsilon_1} \cdots \overline{L}_{d, \epsilon_{d}};\phi )
\]
does not depend on the choice of
$\overline{L}_{1, +1}, \overline{L}_{1, -1}, \ldots, \overline{L}_{d, +1}, \overline{L}_{d, -1}$,
so that it is denoted by $\adeg(\overline{L}_1 \cdots \overline{L}_{d};\phi)$.
Further, \eqref{eqn:multlinear:cone:intersection:adelic} and \eqref{eqn:com:cone:intersection:adelic} extends to
the following formula:
\frontmatterforspececialeqn
\begin{multline}
\label{eqn:multlinear:adelic:intersection}
\adeg_{v}(\overline{L}_1 \cdots (a \overline{L}_i + a' \overline{L}'_i) \cdots \overline{L}_{d};\phi) \\
=
a \adeg_{v}(\overline{L}_1 \cdots \overline{L}_i \cdots \overline{L}_{d};\phi) +
a' \adeg_{v}(\overline{L}_1 \cdots \overline{L}'_i \cdots \overline{L}_{d};\phi)
\end{multline}
\backmatterforspececialeqn
and
\frontmatterforspececialeqn
\begin{equation}
\label{eqn:com:intersection:adelic}
\adeg_{v}(\overline{L}_1 \cdots \overline{L}_i \cdots \overline{L}_j \cdots \overline{L}_{d};\phi)
=
\adeg_{v}(\overline{L}_1 \cdots \overline{L}_j \cdots \overline{L}_i \cdots \overline{L}_{d};\phi)
\end{equation}
\backmatterforspececialeqn
for $a, a' \in \RR$ and integrable adelic $\RR$-Cartier divisors
$\overline{L}_1, \ldots, \overline{L}_i, \overline{L}'_i,\ldots,\overline{L}_{d}$
of $C^0$-type on $X$.
Here let us consider a consequence of Proposition~\ref{propdef:adelic:intersection}.

\begin{Proposition}
\label{prop:intersection:com}
Let $\overline{L}_1 =(L_1, g_1), \ldots, \overline{L}_{d} = (L_d, g_d)$ be integrable adelic $\RR$-Cartier divisors
of $C^0$-type on $X$, and let $\phi$ be an integrable continuous function
on $X^{\an}_v$.
If $L_i = 0$, then
\[
\adeg_v(\overline{L}_1 \cdots \overline{L}_i \cdots \overline{L}_d ; \phi) =
\adeg_v(\overline{L}_1 \cdots (0, \phi) \cdots \overline{L}_d ; g_i).
\]
\end{Proposition}

\begin{proof}
For $j=1, \ldots, d$,
let
\[
\overline{L}_{j, +1} = (L_{j, +1}, g_{j,+1})
\quad\text{and}\quad
\overline{L}_{j, -1}  = (L_{j, -1}, g_{j,-1})
\]
be relatively nef adelic $\RR$-Cartier divisors  
of $C^0$-type on $X$ such that $\overline{L}_j = \overline{L}_{j, +1} - \overline{L}_{j, -1}$.
Moreover, we choose relative nef adelic $\RR$-Cartier divisors 
\[
\overline{L}_{d+1, +1}= (L_{d+1, +1}, g_{d+1,+1})
\quad\text{and}\quad
\overline{L}_{d+1,-1} = (L_{d+1, -1}, g_{d+1,-1})
\]
of $C^0$-type
on $X$ such that $\overline{L}_{d+1, +1} - \overline{L}_{d+1, -1} = (0, \phi)$.
Then there are sequences
\begin{multline*}
\left\{ (\XXX_{1, +1, n}, \LLL_{1, +1, n}) \right\}_{n=1}^{\infty}, 
\left\{ (\XXX_{1, -1, n}, \LLL_{1, -1, n}) \right\}_{n=1}^{\infty}, \ldots, \\
\left\{ (\XXX_{d+1, +1, n}, \LLL_{d+1, +1, n}) \right\}_{n=1}^{\infty}, 
\left\{ (\XXX_{d+1, -1, n}, \LLL_{d+1, -1, n}) \right\}_{n=1}^{\infty}
\end{multline*}
satisfying the following conditions:
\begin{enumerate}
\renewcommand{\labelenumi}{(\alph{enumi})}
\item
$\XXX_{j, \epsilon, n}$ is a normal model of $X$ over $\Spec(k^{\circ})$ for $j=1, \ldots, d+1$,
$\epsilon = \pm 1$ and $n \geq 1$.

\item
$\LLL_{j,\epsilon, n}$ is a nef $\RR$-Cartier divisor on $\XXX_{j, \epsilon, n}$ such that
$\LLL_{j,\epsilon, n} \cap X = L_{j,\epsilon}$ for $j=1, \ldots, d+1$,
$\epsilon = \pm 1$ and $n \geq 1$. 

\item
If we set $\phi_{j, \epsilon, n} = g_{j,\epsilon} - g_{\left(\XXX_{j,\epsilon, n},\, \LLL_{j, \epsilon, n}\right)}$,
then $\lim_{n\to\infty} \Vert \phi_{j,\epsilon, n} \Vert_{\sup} = 0$ for $j=1, \ldots, d+1$ and
$\epsilon = \pm 1$.
\end{enumerate}
Here we set
\[
\begin{cases}
\psi_n = g_{(\XXX_{d+1, +1, n},\, \LLL_{d+1, +1, n})}  - g_{(\XXX_{d+1, -1, n},\, \LLL_{d+1, -1, n})}, \\
\theta_n = g_{(\XXX_{i, +1, n},\, \LLL_{i, +1, n})}  - g_{(\XXX_{i, -1, n},\, \LLL_{i, -1, n})}.
\end{cases}
\]
Then, by Proposition-Definition~\ref{propdef:adelic:intersection},
\begin{align*}
&\lim_{n\to\infty} \adeg_v(\LLL_{1,\epsilon_1, n} \cdots \LLL_{i, \epsilon_i, n} \cdots \LLL_{d, \epsilon_d, n} ; \psi_n) =
\adeg_v(\overline{L}_{1,\epsilon_1} \cdots \overline{L}_{i, \epsilon_i} \cdots \overline{L}_{d, \epsilon_d} ; \phi) \\
\intertext{and}
& \lim_{n\to\infty} \adeg_v(\LLL_{1,\epsilon_1, n} \cdots \LLL_{d+1, \epsilon_{d+1}, n} \cdots \LLL_{d, \epsilon_d, n} ; \theta_n) \\
& \qquad\qquad\qquad\qquad\qquad\qquad\qquad\qquad = 
\adeg_v(\overline{L}_{1,\epsilon_1} \cdots \overline{L}_{d+1, \epsilon_{d+1}} \cdots \overline{L}_{d, \epsilon_d} ; g_i).
\end{align*}
Therefore, if we set $\LLL_{i, n} = \LLL_{i, +1, n} - \LLL_{i, -1, n}$ for $i=1, \ldots, d+1$,
then
\[
\begin{cases}
\lim_{n\to\infty} \adeg_v(\LLL_{1, n} \cdots \LLL_{i, n} \cdots \LLL_{d, n} ; \psi_n) =
\adeg_v(\overline{L}_{1} \cdots \overline{L}_{i} \cdots \overline{L}_{d} ; \phi), \\[2ex]
\lim_{n\to\infty} \adeg_v(\LLL_{1, n} \cdots \LLL_{d+1, n} \cdots \LLL_{d, n} ; \theta_n) =
\adeg_v(\overline{L}_{1} \cdots (0, \phi) \cdots \overline{L}_{d} ; g_i).
\end{cases}
\]
On the other hand,
note that 
\[
 \adeg_v(\LLL_{1, n} \cdots \LLL_{i, n} \cdots \LLL_{d, n} ; \psi_n) =
 \adeg_v(\LLL_{1, n} \cdots \LLL_{d+1, n} \cdots \LLL_{d, n} ; \theta_n).
\]
Thus the assertion of the proposition follows.
\end{proof}

The results of this 
\ifmonog section \fi
\ifpaper subsection \fi
leads to the following definition:

\begin{Definition}
\label{def:intersection:integral:div}
Let $\overline{L}_1 =(L_1, g_1), \ldots, \overline{L}_{d+1} = (L_{d+1}, g_{d+1})$ be integrable adelic $\RR$-Cartier divisors
of $C^0$-type on $X$. 
By Proposition~\ref{prop:intersection:com} and \eqref{eqn:com:intersection:adelic}, if $L_i = 0$ for some $i$, then
\[
\adeg_v(\overline{L}_1 \cdots   \overline{L}_{d+1})
\]
is well-defined, that is,
\[
\adeg_v(\overline{L}_1 \cdots  \overline{L}_{d+1}) := \adeg_v(\overline{L}_1 \cdots \overline{L}_{i-1} \cdot \overline{L}_{i+1} \cdots
\overline{L}_{d+1} ; g_i).
\]%
\index{\AdelDivSymbol}{0d:adeg_v(overline{L}_1 cdots  overline{L}_{d+1})@$\adeg_v(\overline{L}_1 \cdots  \overline{L}_{d+1})$}%
Moreover, it is symmetric and multi-linear.
\end{Definition}

\begin{Proposition}
\label{prop:misc:intersection:adelic}
Let $\overline{L}_1, \ldots, \overline{L}_d$ be integrable adelic $\RR$-Cartier divisors of $C^0$-type on $X$, and
let $\phi$ and $\phi'$ be continuous functions on $X_v^{\an}$. Then we have the following:
\begin{enumerate}
\renewcommand{\labelenumi}{(\arabic{enumi})}
\item
For $f \in \Rat(X)^{\times}_{\RR}$, 
$\adeg_v(((f), -\log \vert f \vert^2) \cdot \overline{L}_2 \cdots \overline{L}_d ; \phi) = 0$.
 
\item
If  $\overline{L}_1, \ldots, \overline{L}_d$ are relatively nef and
$\phi' \leq \phi$ on $X_v^{\an}$, then
\[
\adeg_v(\overline{L}_1  \cdots \overline{L}_d ; \phi') \leq
\adeg_v(\overline{L}_1  \cdots \overline{L}_d ; \phi).
\]

\item
Let 
\begin{multline*}
\qquad\qquad\overline{L}_{1,1} = (L_{1,1}, g_{1,1}),\overline{L}_{1,-1} = (L_{1,-1}, g_{1,-1}), \ldots, \\
\overline{L}_{d,1} = (L_{d,1}, g_{d,1}),\overline{L}_{d,-1} = (L_{d,-1}, g_{d,-1})
\end{multline*}
be 
relatively nef adelic $\RR$-Cartier divisors of $C^0$-type on $X$ such that
$\overline{L}_i = \overline{L}_{i,1} - \overline{L}_{i,-1}$ for $i=1, \ldots, d$.
Then
\[
\left| \adeg_v(\overline{L}_1 \cdots \overline{L}_d ; \phi) \right|
\leq \frac{\Vert \phi \Vert_{\sup}}{-2 \log v(\varpi)} \sum_{\epsilon_1, \ldots, \epsilon_d \in \{ \pm 1 \}}
\deg(L_{1, \epsilon_1} \cdots L_{d, \epsilon_d}).
\]
\end{enumerate}
\end{Proposition}

\begin{proof}
Let $\XXX$ be a normal model of $X$ over $\Spec(k^{\circ})$ and let $\LLL_1, \LLL_2, \ldots, \LLL_d$ be
$\RR$-Cartier divisors on $\XXX$

(1) By the definition of $\adeg_v$,
we have $\adeg_v((f)_{\XXX} \cdot \LLL_2 \cdots \LLL_d ; \phi) = 0$.
Thus (1) follows.

(2) If $\LLL_1, \ldots, \LLL_d$ are relatively nef, then
$\adeg_v(\LLL_1  \cdots \LLL_d ; \phi') \leq
\adeg_v(\LLL_1  \cdots \LLL_d ; \phi)$ (cf. \eqref{eqn:phi:leq:phi:prime:ineq}), so that
we have (2).

(3) First of all, note that
\[
\adeg_v(\overline{L}_1 \cdots \overline{L}_d ; \phi) = \sum_{\epsilon_1, \ldots, \epsilon_d \in \{ \pm 1 \}}
\epsilon_1 \cdots \epsilon_d 
\adeg_v(\overline{L}_{1, \epsilon_1} \cdots \overline{L}_{d, \epsilon_d};\phi).
\]
Therefore, by using (2),
\begin{align*}
\left| \adeg_v(\overline{L}_1 \cdots \overline{L}_d ; \phi) \right| & \leq \sum_{\epsilon_1, \ldots, \epsilon_d \in \{ \pm 1 \}}
\left| \adeg_v(\overline{L}_{1, \epsilon_1} \cdots \overline{L}_{d, \epsilon_d};\phi) \right| \\
& \leq \sum_{\epsilon_1, \ldots, \epsilon_d \in \{ \pm 1 \}}
\adeg_v(\overline{L}_{1, \epsilon_1} \cdots \overline{L}_{d, \epsilon_d}; \Vert \phi \Vert_{\sup}) \\
& =  
\frac{\Vert \phi \Vert_{\sup}}{-2 \log v(\varpi)} \sum_{\epsilon_1, \ldots, \epsilon_d \in \{ \pm 1 \}}
\deg(L_{1, \epsilon_1} \cdots L_{d, \epsilon_d}),
\end{align*}
as required.
\end{proof}

\bigskip
Finally let us consider Zariski's lemma for integrable functions.
We assume that $\dim X = 1$ and $v$ is complete.
Let $C^0_{\integrable}(X^{\an})$ be the space of integrable continuous functions on $X^{\an}$.
Let  
\[
\langle\ ,\ \rangle : C^0_{\integrable}(X^{\an}) \times C^0_{\integrable}(X^{\an}) \to \RR
\]
be a map given by 
$\langle \varphi, \psi \rangle := \adeg_v((0,\varphi) ; \psi) = \adeg_v((0,\varphi) \cdot (0, \psi))$,
which is bilinear and symmetric (see Definition~\ref{def:intersection:integral:div}).

\begin{Lemma}[Zariski's lemma for integrable functions]
\label{lem:Zariski:adelic}
The above pairing $\langle\ ,\ \rangle$ is negative semi-definite.
Moreover, for $\varphi \in C^0_{\integrable}(X^{\an})$,
$\langle\varphi ,\varphi \rangle = 0$ if and only if $\varphi$ is a constant function.
\end{Lemma}

\begin{proof}
The proof can be found in \cite{YZ}. For reader's convenience, we prove it here.
Let $\MM(X^{\an})$ be the group of model functions on $X^{\an}$ (cf. 
\ifmonog Section~\ref{subsec:model:function}). \fi
\ifpaper Subsection~\ref{subsec:model:function}). \fi
Then, by virtue of Zariski's lemma (cf. \cite[Lemma~1.1.4]{MoD}),
for $\theta \in \MM(X^{\an}) \otimes_{\ZZ} \RR$,
$\langle \theta, \theta \rangle \leq 0$.
Therefore, by the density theorem (cf. Theorem~\ref{thm:density:local:model}) together with
Proposition-Definition~\ref{propdef:adelic:intersection}, we can see that
$\langle \varphi, \varphi \rangle \leq 0$ for all $\varphi \in C^0_{\integrable}(X^{\an})$.
Thus $\langle\ ,\ \rangle$ is negative semi-definite.
Clearly if $\varphi$ is a constant function, then $\langle\varphi ,\varphi \rangle = 0$.
Conversely, we assume $\langle\varphi ,\varphi \rangle = 0$.
Then, by \cite[Lemma~1.1.3]{MoD},
$\langle \psi, \varphi \rangle = 0$
for all $\psi \in C^0_{\integrable}(X^{\an})$.

Let $\XXX$ be a regular model of $X$ over $\Spec(k^{\circ})$ and
$\XXX_{\circ} = a_1 C_1 + \cdots + a_r C_r$ the irreducible decomposition of the central fiber $\XXX_{\circ}$ as a cycle.
Let $x_j$ be the point of $X^{\an}$ corresponding to $C_j$. 
For $i=1, \ldots, r$, if $\psi_i = g_{(\XXX, C_i)}$, then
\begin{align*}
\deg\left(C_i \cdot \sum\nolimits_{j=1}^r \varphi(x_j) a_j C_j \right)  & = 
-2 \log v(\varpi) \sum\nolimits_{j=1}^r \frac{\varphi(x_j) a_j }{-2 \log v(\varpi)} \deg(C_i \cdot C_j) \\
& = -2 \log v(\varpi) \langle \psi_i, \varphi \rangle = 0.
\end{align*}
Therefore, by the equality condition of Zariski's lemma (cf. \cite[Lemma~1.1.4]{MoD}),
we have
\[
\varphi(x_1) = \cdots = \varphi(x_r).
\]

Let $X^{\an}_{\zeros}$ be a subset of $X^{\an}$ consisting of valuations arising from irreducible components of
the central fiber of any regular model of $X$.
The above observation shows that $\rest{\varphi}{X^{\an}_{\zeros}}$ is a constant function $a$.
We set $\lambda := \varphi - a$ on $X^{\an}$.
By the density theorem, for any positive number $\epsilon$,
there is $\theta \in \MM(X^{\an}) \otimes_{\ZZ} \RR$ such that $\vert \lambda - \theta \vert \leq \epsilon$.
Thus $\vert \theta(x) \vert \leq \epsilon$ for all $x \in X^{\an}_{\zeros}$, which implies that 
$\vert \theta \vert \leq \epsilon$ on $X^{\an}$. 
Indeed, if $\theta$ is given by a vertical $\RR$-Cartier divisor 
$\Theta = c_1 C_1 + \cdots + c_r C_r$ on some regular model $\XXX$ as before (i.e.
$\theta = g_{(\XXX, \Theta)}$),
then $c_i = a_i \theta(x_i)/(-2\log v(\varpi))$, so that
\[
-(\epsilon/(-2\log v(\varpi))) \XXX_{\circ} \leq \Theta \leq (\epsilon/(-2\log v(\varpi))) \XXX_{\circ},
\]
which means that $\vert \theta \vert \leq \epsilon$ on $X^{\an}$.
Thus $\vert \lambda \vert \leq 2 \epsilon$, and hence
$\lambda = 0$.
\end{proof}

\ifmonog\chapter{Local and global density theorems}\fi
\ifpaper\section{Local and global density theorems}\fi
The density theorem in terms of global model functions was established
by Gubler \cite[Theorem~7.12]{GL} and X. Yuan \cite[Lemma~3.5]{YB}.
Recently an elementary proof was found by Boucksom, Favre and Jonsson \cite{BFJ}.
Unfortunately, they assume that the characteristic of the residue field is zero,
which is a strong restriction for our purpose.
In this 
\ifmonog chapter, \fi
\ifpaper section, \fi
we will give a proof of the density theorem in general by using their ideas.

Let $k$ be a field and $v$ a non-trivial discrete valuation of $k$.
Note that $v$ is multiplicative.
Let $\varpi$ be a uniformizing parameter of $k^{\circ}$, that is,
$k^{\circ\circ} = \varpi k^{\circ}$ (for the definition of $k^{\circ}$ and $k^{\circ\circ}$,
see Conventions and terminology~\ref{CT:nonArchimedean:value}). 
Let $X$ be a projective and geometrically integral variety over $k$ and
let $\Rat(X)$ be the rational function field of $X$.

\ifmonog\section{Vertical fractional ideal sheaves and birational system of models}\fi
\ifpaper\subsection{Vertical fractional ideal sheaves and birational system of models}\fi
\label{subsec:vertical:fractional:ideal:sheaf}

Let $\XXX$ be a model of $X$ over $\Spec(k^{\circ})$ (cf. Conventions and terminology~\ref{CT:model}).
The central fiber of $\XXX \to \Spec(k^{\circ})$ is denoted by $\XXX_{\circ}$,
that is, $\XXX_{\circ} := \XXX \times_{\Spec(k^{\circ})} \Spec(k^{\circ}/k^{\circ\circ})$.
A non-zero coherent subsheaf $\JJJ$ of $\Rat(X)$ on $\XXX$ is called a {\em fractional ideal sheaf on $\XXX$}.
\index{\AdelDivSubject}{fractional ideal sheaf@fractional ideal sheaf}%
It is said to be {\em vertical} 
\index{\AdelDivSubject}{vertical fractional ideal sheaf@vertical fractional ideal sheaf}%
if there is $m \in \ZZ_{\geq 0}$ such that
$\varpi^m \JJJ$ is an ideal sheaf 
and $\Supp(\OOO_{\XXX}/\varpi^m \JJJ) \subseteq \XXX_{\circ}$.
Let $\DDD$ be a vertical Cartier divisor on $\XXX$, that is,
$\Supp_{\ZZ}(\DDD) \subseteq \XXX_{\circ}$ (for the definition of $\Supp_{\ZZ}$, 
see 
\ifmonog Section~\ref{subsec:closedness:support:R:Cartier:divisor}). \fi
\ifpaper Subsection~\ref{subsec:closedness:support:R:Cartier:divisor}). \fi
Note that $\OOO_{\XXX}(\DDD)$ is a vertical fractional sheaf.
Indeed, let $\xi \in \XXX_{\circ}$ and $f$ a local equation of $\DDD$ at $\xi$. 
Then $f$ is a unit element of $((\OOO_{\XXX, \xi})_S)_p$ for all 
$p \in \Spec((\OOO_{\XXX, \xi})_S)$,
where $S$ is a multiplicative set given by $\{ 1, \varpi, \varpi^2 , \ldots \}$,
and hence $f \in (\OOO_{\XXX, \xi})_S$. Thus
we can find $m_{\xi} \in \ZZ_{\geq 0}$ such that $\varpi^{m_{\xi}} f \in \OOO_{\XXX, \xi}$.
This observation shows that $\varpi^m \OOO_{\XXX}(\DDD) \subseteq \OOO_{\XXX}$ for some $m \in \ZZ_{\geq 0}$, as required.

A set $\Psi$ of  models of $X$ is called a {\em birational system of  models of $X$}
\index{\AdelDivSubject}{birational system of models@birational system of models}%
if the following conditions are satisfied:
\begin{enumerate}
\renewcommand{\labelenumi}{(\arabic{enumi})}
\item
For any $\XXX \in \Psi$ and any vertical fractional ideal sheaf $\JJJ$ on $\XXX$,
there is $\XXX' \in \Psi$ together with a birational morphism $\nu : \XXX' \to \XXX$
such that $\JJJ \OOO_{\XXX'}$ is invertible.

\item
For any $\XXX_1, \XXX_2 \in \Psi$, there is $\XXX_3 \in \Psi$ together with
birational morphisms $\nu_1 : \XXX_3 \to \XXX_1$ and $\nu_2 : \XXX_3 \to \XXX_1$.
\end{enumerate}

\begin{Remark}
\label{rem:birational:system:integral:model}
The set of all  models of $X$ is a birational system of  models of $X$.
In addition, fixing a model $\mathcal{Z}$ of $X$,
the set of all  models $\XXX$ over $\mathcal{Z}$ (that is,
there is a birational morphism $\XXX \to \mathcal{Z}$)
forms a birational system of  models of $X$.
\end{Remark}

\ifmonog\section{Model functions}\fi
\ifpaper\subsection{Model functions}\fi
\label{subsec:model:function}

We assume that $v$ is complete.
Let $\XXX \to \Spec(k^{\circ})$ be a model of $X$.
Let $\JJJ$ be a vertical fractional ideal sheaf on $\XXX$.
According to \cite[Subsection~2.3]{BFJ}, we define a function $\log \vert \JJJ \vert$ on $X^{\an}$ to be
\[
\log \vert \JJJ \vert (x) := \log \max \left\{ \left\vert h(x) \right\vert \mid 
h \in \JJJ_{r_{\XXX}(x)} \right\},
\]
where $r_{\XXX} : X^{\an} \to \XXX_{\circ}$ is the reduction map induced by the model $\XXX$.
\index{\AdelDivSymbol}{0l:log vert JJJ vert@$\log \vert \JJJ \vert$}%
For example, $\log \vert \varpi \OOO_{\XXX} \vert = \log v(\varpi)$.
Note that $\log \vert \JJJ \vert = \log \vert \JJJ \OOO_{\XXX'} \vert$
for a birational morphism $\nu : \XXX' \to \XXX$ of  models of $X$ and a vertical fractional ideal sheaf $\JJJ$ on $\XXX$.
Moreover, if $\JJJ$ is invertible,
then $\log \vert \JJJ \vert = g_{(\XXX,\, \JJJ)}/2$ (cf. 
\ifmonog Section~\ref{subsec:green:function:analytic:space}). \fi
\ifpaper Subsection~\ref{subsec:green:function:analytic:space}). \fi
Let us fix a birational system $\Psi$ of  models of $X$.
A function $\varphi$ on $X^{\an}$ is called a {\em model function with respect to $\Psi$} 
\index{\AdelDivSubject}{model function@model function}%
if there are $\XXX \in \Psi$ and a vertical fractional ideal sheaf $\JJJ$ on $\XXX$ such that
$\varphi = \log \vert \JJJ \vert$.
The set of all model functions with respect to $\Psi$ is denoted by $\MM(X^{\an};\Psi)$.
\index{\AdelDivSymbol}{0M:MM(X^{an};Psi)@$\MM(X^{\an};\Psi)$}%
Then we have the following generalization of \cite[Proposition~2.2]{BFJ}.

\begin{Proposition}
\label{prop:Boolean:separate:points}
\begin{enumerate}
\renewcommand{\labelenumi}{(\arabic{enumi})}
\item
For $\varphi \in \MM(X^{\an};\Psi)$, there are $\XXX \in \Psi$ and a vertical
Cartier divisor $\DDD$ on $\XXX$ such that $\varphi = \log \vert \OOO_{\XXX}(\DDD)\vert$.
In particular, $\MM(X^{\an};\Psi)$ forms an abelian group and
$\MM(X^{\an};\Psi) \subseteq C^0(X^{\an})$ \rom{(}cf. Proposition~\rom{\ref{prop:green:model}}\rom{)}.

\item
For $\varphi_1, \varphi_2 \in \MM(X^{\an};\Psi)$, 
$\max \{ \varphi_1, \varphi_2 \} \in \MM(X^{\an};\Psi)$.

\item
For $x, y \in X^{\an}$ with $x \not= y$, there is $\varphi \in \MM(X^{\an};\Psi)$ such that
$\varphi(x) \not= \varphi(y)$
\end{enumerate}
\end{Proposition}

\begin{proof}
The assertions of the proposition can be proved by the ideas of \cite[Proposition~2.2]{BFJ}.

(1) We choose $\XXX \in \Psi$ and a vertical fractional ideal sheaf $\JJJ$ on $\XXX$
such that $\varphi = \log \vert \JJJ \vert$.
By our assumption, there is $\XXX' \in \Psi$ together with a birational morphism
$\nu : \XXX' \to \XXX$ such that $\JJJ \OOO_{X'}$ is invertible, that is,
there is a vertical Cartier divisor $\DDD'$ on $\XXX'$ such that
$\JJJ \OOO_{\mathcal{X'}} = \OOO_{\XXX'}(\DDD')$.
Clearly $\log \vert \JJJ \vert = \log \vert \OOO_{\XXX'}(\DDD')\vert$,
as desired.

\medskip
(2) By (1) and the property of $\Psi$,
we can find $\XXX \in \Psi$ and vertical Cartier divisors 
$\DDD_1$ and $\DDD_2$ on $\XXX$
such that $\varphi_1 = \log \vert \OOO_{\XXX}(\DDD_1) \vert$ and 
$\varphi_2 = \log \vert \OOO_{\XXX}(\DDD_2) \vert$.
If we set $\JJJ =\OOO_{\XXX}(\DDD_1) + \OOO_{\XXX}(\DDD_2)$ in $\Rat(X)$, then
$\JJJ$ is a vertical fractional ideal sheaf of $\XXX$ and
$\log \vert \JJJ \vert = \max \{ \varphi_1, \varphi_2 \}$.

\medskip
(3) Fix $\XXX \in \Psi$.
First we assume that $r_{\XXX}(x) \not= r_{\XXX}(y)$.
Let $\mathfrak{m}$ be the maximal ideal at $r_{\XXX}(x)$.
Then $\log \vert \mathfrak{m} \vert(x) < 0$ and $\log \vert \mathfrak{m} \vert(y) \geq 0$,
as desired. 
Next we assume that $\xi = r_{\XXX}(x) = r_{\XXX}(y)$.
Let $\mathcal{U} = \Spec(\mathcal{A})$ be an affine open neighborhood of $\xi$.
We can find $f \in \mathcal{A}$ such that $\vert f  \vert_x \not= \vert f \vert_y$.
Since $\XXX$ is noetherian, there is an ideal sheaf $\III$ on $\XXX$ such that
$\III_{\xi} = f \OOO_{\XXX, \xi}$. For each $m \in \ZZ_{> 0}$, we set
$\JJJ_m = \III + \varpi^m \OOO_{\XXX}$. Note that $\JJJ_m$ is a vertical ideal sheaf on $\XXX$.
Moreover,
\[
\begin{cases}
\log \vert \JJJ_m \vert(x) = \max \{ \log \vert f \vert_x, m \log v(\varpi) \}, \\
\log \vert \JJJ_m \vert(y) = \max \{ \log \vert f \vert_y, m \log v(\varpi) \},
\end{cases}
\]
where $\log (0) = -\infty$. As $\vert f \vert_x \not= \vert f\vert_y$ and $v(\varpi) < 1$,
if $m$ is sufficiently large, then 
\[
\log \vert \JJJ_m \vert(x) \not= \log \vert \JJJ_m \vert(y),
\]
as required.
\end{proof}

\ifmonog\section{Density theorems}\fi
\ifpaper\subsection{Density theorems}\fi
\label{subsec:density:theorem}

Let $k_v$ be the completion of $k$ with respect to $v$.
By abuse of notation, the unique extension of $v$ to $k_v$ is also denoted by $v$.
We set $X_v = X \times_{\Spec(k)} \Spec(k_v)$, which is also
a projective and geometrically integral variety over $k_v$.
For a model $\XXX \to \Spec(k^{\circ})$ of $X$, 
$\XXX_v := \XXX \times_{\Spec(k^{\circ})} \Spec(k_v^{\circ})$ 
is also a model of $X_v$ by (1.1) in Lemma~\ref{lem:base:change:completion}.
The projection $\XXX_v \to \XXX$ is denoted by $\pi_{\XXX}$.
Let $\Psi$ be a system of models of $X$.

\begin{Proposition}
$\Psi_v := \left\{ \XXX_v \mid \XXX \in \Psi \right\}$ forms a system of models of $X_v$.
\end{Proposition}

\begin{proof}
The assertion follows from (2) in the following lemma.
\end{proof}

\begin{Lemma}
\label{lem:vertical:ideal:descent}
\begin{enumerate}
\renewcommand{\labelenumi}{(\arabic{enumi})}
\item
Let $\JJJ_v$ be a vertical ideal sheaf on $\XXX_v$.
Then there is $n \in \ZZ_{\geq 0}$ such that $\varpi^n \OOO_{\XXX_v} \subseteq \JJJ_v$.

\item
Let $\JJJ_v$ be a vertical fractional ideal sheaf on $\XXX_v$.
Then there is a vertical fractional ideal sheaf $\JJJ$ on $\XXX$ such that
$\JJJ \OOO_{\XXX_v} = \JJJ_v$.
\end{enumerate}
\end{Lemma}

\begin{proof}
(1)
For $\xi \in (\XXX_v)_{\circ}$, 
$(\OOO_{\XXX_v, \xi}/(\JJJ_v)_{\xi})_S = 0$ because $\Supp(\OOO_{\XXX_v}/\JJJ_v) \subseteq (\XXX_v)_{\circ}$,
where $S$ is a multiplicative set given by $\{ 1, \varpi, \varpi^2 , \ldots \}$. Therefore,
\[
\varpi^{m_{\xi}} (\OOO_{\XXX_v, \xi}/(\JJJ_v)_{\xi}) = 0
\]
for some $m_{\xi} \in \ZZ_{\geq 0}$, and hence
$\varpi^{m_{\xi}} \OOO_{\XXX_v, \xi} \subseteq (\JJJ_v)_{\xi}$.
Thus the assertion follows.

(2)
Clearly we may assume that $\JJJ_v$ is an ideal sheaf.
Then, by (1), there is $n \in \ZZ_{\geq 0}$ such that $\varpi^n \OOO_{\XXX_v} \subseteq \JJJ_v$.
As 
\[
\begin{CD}
\JJJ_v/\varpi^n \OOO_{\XXX_v} \subseteq \OOO_{\XXX_v}/\varpi^n \OOO_{\XXX_v} @<{\sim}<{\pi^*_{\XXX}}< \OOO_{\XXX}/\varpi^n \OOO_{\XXX},
\end{CD}
\]
we can find an ideal sheaf $\JJJ$ on $\XXX$ such that $\varpi^n \OOO_{\XXX} \subseteq \JJJ$ and
$\JJJ_v/\varpi^n \OOO_{\XXX_v} \simeq \JJJ/\varpi^n \OOO_{\XXX}$, so that we can easily see that
$\JJJ \OOO_{\XXX_v} = \JJJ_v$.
\end{proof}

\begin{Theorem}[Local density theorem]
\label{thm:density:local:model}
$\MM(X_v^{\an};\Psi_v) \otimes_{\ZZ} \QQ$ is dense in $C^0(X_v^{\an})$
with respect to the supremum norm $\Vert\cdot\Vert_{\sup}$.
\end{Theorem}

\begin{proof}
We set $T = X_v^{\an}$ and
$\Sigma = \MM(X_v^{\an};\Psi_v) \otimes_{\ZZ} \QQ$.
It is well known that $T$ is a compact Hausdorff space (cf. \cite[Theorem~3.4.8]{Be}).
Thus, if we can check the conditions (1) -- (4) in Lemma~\ref{lem:Boolean:Stone:Weierstrass},
we have the assertion.

(1) follows from (3) in Proposition~\ref{prop:Boolean:separate:points}.

(2) Note that $\log v(\varpi) \in \MM(X_v^{\an};\Psi_v) \subseteq \Sigma$.

(3) is obvious. 

(4)
Let us check that $\max\{ \psi_1, \psi_2 \} \in \Sigma$ for $\psi_1, \psi_2 \in \Sigma$.
We choose $n \in \ZZ_{>0}$ such that $n \varphi_1, n \varphi_2 \in \MM(X_v^{\an};\Psi_v)$.
Then, by (2) in Proposition~\ref{prop:Boolean:separate:points},
\[
n \max \{\psi_1, \psi_2 \} = \max \{n \psi_1, n \psi_2 \} \in \Sigma,
\]
and hence $\max \{\psi_1, \psi_2 \} \in \Sigma$.
\end{proof}

\begin{Lemma}
\label{lem:Boolean:Stone:Weierstrass}
Let $T$ be a compact Hausdorff space.
Let $\Sigma$ be a subset of $C^0(T)$ with the following properties:
\begin{enumerate}
\renewcommand{\labelenumi}{(\arabic{enumi})}
\item
For any $x, y \in T$ with $x \not= y$, there is $f \in \Sigma$ such that $f(x) \not= f(y)$.

\item
$\RR^{\times} \cap \Sigma \not= \emptyset$.

\item
$\Sigma$ forms a $\QQ$-vector space.

\item
For $f, g \in \Sigma$, $\max \{ f, g \} \in \Sigma$.
\end{enumerate}
Then $\Sigma$ is dense in $C^0(T)$ with respect to the supremum norm $\Vert\cdot\Vert_{\sup}$.
\end{Lemma}

\begin{proof}
Let $\overline{\Sigma}$ be the closure of $\Sigma$ with respect to $\Vert\cdot\Vert_{\sup}$.
Then it is easy to see that $\overline{\Sigma}$ has the following properties:
\begin{enumerate}
\renewcommand{\labelenumi}{(\arabic{enumi})'}
\item
For any $x, y \in T$ with $x \not= y$, there is $f \in \overline{\Sigma}$ such that $f(x) \not= f(y)$.

\item
$1 \in \overline{\Sigma}$.

\item
$\Sigma$ forms an $\RR$-vector space.

\item
For $f, g \in \overline{\Sigma}$, $\max \{ f, g \} \in \overline{\Sigma}$.
\end{enumerate}
Thus, by \cite[Theorem~7.29]{HS}, $\overline{\Sigma}$ is dense in $C^0(T)$.
Note that $\overline{\Sigma} = \overline{\overline{\Sigma}}$.
Therefore the assertion follows.
\end{proof}

\begin{Definition}
A function $\varphi$ on $X_v^{\an}$ is called a {\em global model function with respect to $\Psi$} 
\index{\AdelDivSubject}{global model function@global model function}%
if there are $\XXX \in \Psi$ and a vertical fractional ideal sheaf $\JJJ$ on $\XXX$ such that
$\varphi = \log \vert \JJJ \OOO_{\XXX_v} \vert$.
The set of all global model functions with respect to $\Psi$ is denoted by $\MM(X;\Psi)$.
\index{\AdelDivSymbol}{0M:MM(X;Psi)@$\MM(X;\Psi)$}%
\end{Definition}

By using (2) in Lemma~\ref{lem:vertical:ideal:descent},
we can see 
\[
\MM(X;\Psi) = \MM(X_v^{\an};\Psi_v),
\]
Thus the local density theorem (cf. Theorem~\ref{thm:density:local:model}) implies
the following main result of this 
\ifmonog chapter. \fi
\ifpaper section. \fi

\begin{Theorem}[Global density theorem]
\label{thm:density:global:model}
$\MM(X;\Psi) \otimes_{\ZZ} \QQ$ is dense in $C^0(X_v^{\an})$
with respect to the supremum norm $\Vert\cdot\Vert_{\sup}$.
\end{Theorem}

The following theorem is an application of the global density theorem.

\begin{Theorem}[Approximation theorem of adelic $\RR$-divisors]
\label{thm:approx:adelic}
We assume that $k^{\circ}$ is excellent and $X$ is normal.
Let $\overline{D} = (D, g)$ be an adelic $\RR$-Cartier divisor of $C^0$-type
on $X$.
For any positive number $\epsilon > 0$, 
there is a normal 
model $\XXX$ of $X$ over $\Spec(k^{\circ})$, and $\RR$-Cartier divisors $\DDD_1$ and $\DDD_2$ on 
$\XXX$ such that
\[
\overline{D} - (0, \epsilon) \leq \DDD_1 \leq \overline{D} \leq \DDD_2 \leq 
\overline{D} + (0, \epsilon).
\]
\end{Theorem}

\begin{proof}
Let us begin with the following claim:

\begin{Claim}
\label{claim:thm:approx:adelic:arith:01}
For a positive number $\epsilon$,
there is a normal 
model $\XXX$ of $X$ over $\Spec(k^{\circ})$ and an $\RR$-Cartier divisor $\DDD$ on $\XXX$
such that $\DDD \cap X = D$ and
\[
\Vert g_{(\XXX, \DDD)} - g \Vert_{\sup} \leq \epsilon/2.
\]
\end{Claim}

\begin{proof}
First let us choose a model $\XXX_0$ of $X$ over $\Spec(k^{\circ})$ and
an $\RR$-Cartier divisor $\DDD_0$ on $\XXX_0$ such that $D = \DDD_0 \cap X$.
Let $g_{0}$ be the $D$-Green function of $C^0$-type on $X_v^{\an}$
arising from the model $(\XXX_0, \DDD_0)$.
We set $\phi := g - g_{0}$. Then $\phi$ is a continuous function on $X^{\an}_v$. 
Let $\Psi$ be the set of all models of $X$ over $\XXX_0$
(cf. Remark~\ref{rem:birational:system:integral:model}).
By the global density theorem (cf. Theorem~\ref{thm:density:global:model}), 
there is a global model function $\varphi$ on $X^{\an}_v$ with respect to $\Psi$ such that
$\Vert \phi - \varphi \Vert_{\sup} \leq \epsilon/2$, that is,
there are a model $\XXX'$ in $\Psi$ together with a birational morphism
$\mu : \XXX' \to \XXX_0$, a vertical Cartier divisor $\EEE'$ on $\XXX'$ and
$a \in \QQ$ such that $\varphi = a g_{(\XXX', \EEE')}$.
Let $\pi : \XXX \to \XXX'$ be the normalization of $\XXX'$ and $\DDD := a \pi^*(\EEE') + \pi^*(\mu^*(\DDD_0))$.
As $k^{\circ}$ is an excellent, $\pi$ is a finite morphism, that is,
$\XXX \in \Psi$. In addition, $g_{(\XXX, \DDD)} = \varphi + g_0$.
Therefore, we have the assertion of the claim.
\end{proof}

We set $g_{1} = g_{(\XXX, \DDD)} - \epsilon/2$ and $g_{2} = g_{(\XXX, \DDD)} + \epsilon/2$.
Then
\[
g - \epsilon \leq g_{1} \leq g \leq g_{2} \leq g + \epsilon
\]
on $X_v^{\an}$.
Moreover, note that the global model function arising from the central fiber of $\XXX \to \Spec(k^{\circ})$ 
is a constant function. Thus the assertion follows.
\end{proof}

\ifmonog\chapter{Adelic arithmetic $\RR$-Cartier divisors}\fi
\ifpaper\section{Adelic arithmetic $\RR$-Cartier divisors}\fi
In this 
\ifmonog chapter, \fi
\ifpaper section, \fi
we will introduce adelic arithmetic $\RR$-Cartier divisors on
arithmetic varieties and 
investigate their several basic properties.

Throughout this 
\ifmonog chapter, \fi
\ifpaper section, \fi
let $K$ be a number field and let $X$ be a $d$-dimensional,
projective, smooth and geometrically integral
variety over $K$.

\ifmonog\section{Definition and basic properties}\fi
\ifpaper\subsection{Definition and basic properties}\fi
\label{subsec:def:basic:prop}

Let $O_K$ be the ring of integers in $K$.
We denote the set of all maximal ideals of $O_K$ by $M_K$.
For $P \in M_K$, 
the valuation $v_P$ of $K$ at $P$ is given by
\[
v_P(x) = \#(O_K/P)^{-\ord_P(x)}.
\]%
\index{\AdelDivSymbol}{0v:v_P@$v_P$}%
Let $K_P$ be the completion of $K$ with respect to $v_P$ and let
$X_P := X \times_{\Spec(K)} \Spec(K_P)$, which is also
a projective, smooth and geometrically integral
variety over $K_P$.
\index{\AdelDivSymbol}{0K:K_P@$K_P$}%
Let $X(\CC)$ be the set of all $\CC$-valued points of $X$, that is,
\[
X(\CC) := \left\{ x : \Spec(\CC) \to X \mid \text{$x$ is a morphism as schemes} \right\}.
\]%
\index{\AdelDivSymbol}{0X:X(CC)@$X(\CC)$}%
Note that $X(\CC)$ is a projective complex manifold and
$X(\CC)$ is not necessarily connected.
Let $F_{\infty} : X(\CC) \to X(\CC)$ be the complex conjugation map, that is,
for $x \in X(\CC)$, $F_{\infty}(x)$ is given by the composition of
morphisms 
\[
\Spec(\CC) \overset{-^{a}}{\longrightarrow} \Spec(\CC)\quad\text{and}\quad
\Spec(\CC) \overset{x}{\to} X,
\]
where $\Spec(\CC) \overset{-^{a}}{\to} \Spec(\CC)$ is the morphism arising from
the complex conjugation. 
\index{\AdelDivSymbol}{0F:F_{infty}@$F_{\infty}$}%
The complex conjugation map
$F_{\infty} : X(\CC) \to X(\CC)$ is an anti-holomorphic isomorphism.
The space of $F_{\infty}$-invariant real valued continuous functions on $X(\CC)$
is denoted by $C^0_{F_{\infty}}(X(\CC))$, that is,
\[
C^0_{F_{\infty}}(X(\CC)) := \{ f \in C^0(X(\CC)) \mid f \circ F_{\infty} = f \}.
\]%
\index{\AdelDivSymbol}{0C:C^0_{F_{infty}}(X(CC))@$C^0_{F_{\infty}}(X(\CC))$}%

\begin{Definition}
\label{def:adelic:arithmetic:div}
A pair $\overline{D} = (D, g)$ of an $\RR$-Cartier divisor $D$ on $X$ and a collection
of Green functions $g = \{ g_P \}_{P \in M_K} 
\cup \{ g_{\infty} \}$ is called an 
{\em adelic arithmetic $\RR$-Cartier divisor of $C^0$-type on $X$} 
\index{\AdelDivSubject}{adelic arithmetic R-Cartier divisor of C^0-type@adelic arithmetic $\RR$-Cartier divisor of $C^0$-type}%
if the following conditions (1) and (2) hold:
\begin{enumerate}
\renewcommand{\labelenumi}{(\arabic{enumi})}
\item
For all $P \in M_K$,
$g_P$ is a $D$-Green function of $C^0$-type on $X^{\an}_P$.
In addition, there exist a non-empty open set $U$ of $\Spec(O_K)$,
a normal model $\XXX_U$ of $X$ over $U$ and an $\RR$-Cartier divisor $\DDD_U$ on $\XXX_U$
such that $\DDD_U \cap X = D$ and $g_P$ is a $D$-Green function induced by
$(\XXX_U, \DDD_U)$ for all $P \in U \cap M_K$. 

\item
The Green function $g_{\infty}$ is a $D$-Green function of $C^0$-type on $X(\CC)$ such that
$g_{\infty}  = g_{\infty} \circ F_{\infty}$ (cf. \cite[Section~5]{MoArZariski}).
\end{enumerate}
Moreover, a pair $\overline{D}' = (D, g')$ of an $\RR$-Cartier divisor $D$ on $X$ and a collection
of Green functions $g' = \{ g'_P \}_{P \in M_K}$ is called a 
{\em global adelic $\RR$-Cartier divisor of $C^0$-type on $X$} 
\index{\AdelDivSubject}{global adelic R-Cartier divisor of C^0-type@global adelic $\RR$-Cartier divisor of $C^0$-type}%
if the above condition (1) holds
for $\{ g'_P \}_{P \in M_K}$, that is,
$g'_P$ is a $D$-Green function of $C^0$-type on $X^{\an}_P$ for all $P \in M_K$, and
there exist a non-empty open set $U'$ of $\Spec(O_K)$,
a normal model $\XXX'_{U'}$ of $X$ over $U'$ and an $\RR$-Cartier divisor $\DDD'_{U'}$ on $\XXX'_{U'}$
such that $\DDD'_{U'} \cap X = D$ and $g'_P$ is a $D$-Green function induced by
$(\XXX'_{U'}, \DDD'_{U'})$ for all $P \in U' \cap M_K$. 

The pair $(\XXX_U, \DDD_U)$ in the condition (1) is called a {\em defining model of $\overline{D}$ over $U$}.
\index{\AdelDivSubject}{defining model@defining model}%
If we forget the Green function $g_{\infty}$ on $X(\CC)$ from $\overline{D}$,
we have a global adelic $\RR$-Cartier divisor on $X$, which is denoted by 
$\overline{D}^{\tau}$ and is called the {\em truncation of $\overline{D}$}.
\index{\AdelDivSubject}{truncation of the adelic arithmetic R-Cartier divisor@truncation of the adelic arithmetic $\RR$-Cartier divisor}%
\index{\AdelDivSymbol}{0D:overline{D}^{tau}@$\overline{D}^{\tau}$}%
For simplicity, a collection
of Green functions $g = \{ g_P \}_{P \in M_K} \cup \{ g_{\infty} \}$
is often denoted by the following symbol:
\[
g = \sum_{P \in M_K} g_P [P] + g_{\infty} [\infty].
\]
Let $\Rat(X)$ be the rational function field of $X$.
For $\varphi \in \Rat(X)^{\times}$,
we define $\widehat{(\varphi)}$ to be
\[
\widehat{(\varphi)} := \left( (\varphi), \sum_{P \in M_K} (-\log \vert \varphi_P \vert^2)[P] +
(-\log \vert \varphi_{\infty} \vert^2) [\infty] \right),
\]
where $\varphi_P$ and $\varphi_{\infty}$ are the rational functions on $X_P^{\an}$ and
$X(\CC)$ induced by $\varphi$, respectively.
The adelic arithmetic divisor $\widehat{(\varphi)}$ is called an {\em adelic arithmetic principal divisor}.
\index{\AdelDivSubject}{adelic arithmetic principal divisor@adelic arithmetic principal divisor}%
\index{\AdelDivSymbol}{0p:widehat{(varphi)}@$\widehat{(\varphi)}$}%
Let $\overline{D}_1 = (D_1, g_1)$ and $\overline{D}_2 = (D_2, g_2)$ be adelic
arithmetic $\RR$-Cartier divisors of $C^0$-type on $X$.
For $a_1, a_2 \in \RR$, we define $a_1\overline{D}_1 + a_2 \overline{D}_2$ to be
\[
a_1\overline{D}_1 + a_2\overline{D}_2 : = (a_1D_1 + a_2D_2 , a_1g_1 + a_1g_2),
\]
where $a_1g_1 + a_2g_2 = \sum_{P \in M_K} (a_1 (g_1)_P + a_2 (g_2)_P)[P] + 
(a_1 (g_1)_{\infty} + a_2 (g_2)_{\infty})[\infty]$.
The space of all adelic arithmetic $\RR$-Cartier divisors of $C^0$-type on $X$ is denoted by
$\aDiv_{C^0}^{\ad}(X)_{\RR}$, which forms an $\RR$-vector space by the previous definition.
\index{\AdelDivSymbol}{0Div:aDiv_{C^0}^{ad}(X)_{RR}@$\aDiv_{C^0}^{\ad}(X)_{\RR}$}%
Moreover,
we define $\overline{D}_1 \leq \overline{D}_2$ by the following conditions:
\index{\AdelDivSymbol}{0D:overline{D}_1 leq overline{D}_2@$\overline{D}_1 \leq \overline{D}_2$}%
\begin{enumerate}
\renewcommand{\labelenumi}{(\alph{enumi})}
\item
$D_1 \leq D_2$.

\item
$(g_{1})_P \leq (g_{2})_P$ for all $P \in M_K$ and
$(g_{1})_{\infty} \leq (g_{2})_{\infty}$.
\end{enumerate}
Similarly, for global adelic $\RR$-Cartier divisors 
\[
(D_1, \{ (g_1)_P \}_{P \in M_K})
\quad\text{and}\quad
(D_2, \{ (g_2)_P \}_{P \in M_K}),
\] 
\[
(D_1, \{ (g_1)_P \}_{P \in M_K}) \leq (D_2, \{ (g_2)_P \}_{P \in M_K})
\]
is defined by $D_1 \leq D_2$ and $(g_{1})_P \leq (g_{2})_P$ for all $P \in M_K$.
\end{Definition}

Let $\XXX$ be a normal model of $X$ over $\Spec(O_K)$ and let $\overline{\DDD} = (\DDD, g_{\infty})$ 
be an arithmetic $\RR$-Cartier divisor of $C^0$-type on $\XXX$
(cf. \cite[Section~5]{MoArZariski}).
The pair
$(\XXX, \overline{\DDD})$ gives rise to an adelic arithmetic $\RR$-Cartier divisor of $C^0$-type on $X$,
that is,
\[
\left( \DDD \cap X, \sum_{P \in M_K} g_{(\XXX_{(P)},\, \DDD_{(P)})}[P] + g_{\infty} [\infty] \right),
\]
where $\XXX_{(P)}$ is the localization of $\XXX \to \Spec(O_K)$ at $P$ and
$\DDD_{(P)}$ is the resection of $\DDD$
to $\XXX_{(P)}$.
\index{\AdelDivSymbol}{0X:XXX_{(P)}@$\XXX_{(P)}$}%
\index{\AdelDivSymbol}{0D:DDD_{(P)}@$\DDD_{(P)}$}%
We use the symbol $\XXX_{(P)}$ to distinguish it from $X_P$ at the beginning of this 
\ifmonog section. \fi
\ifpaper subsection. \fi
We denote it by $\overline{\DDD}^{\ad}$ and it is
called the {\em associated adelic arithmetic $\RR$-Cartier divisor with $\overline{\DDD}$}.
\index{\AdelDivSubject}{associated adelic arithmetic R-Cartier divisor@associated adelic arithmetic $\RR$-Cartier divisor}%
\index{\AdelDivSymbol}{0D:overline{DDD}^{ad}@$\overline{\DDD}^{\ad}$}%
Note that $\widehat{(\varphi)} = \left(\widehat{(\varphi)}_{\XXX}\right)^{\ad}$ for $\varphi \in \Rat(\XXX)^{\times}$,
where $\widehat{(\varphi)}_{\XXX}$ is the arithmetic principal divisor of $\varphi$ on $\XXX$.
Similarly, the {\em associated global adelic $\RR$-Cartier divisor $\DDD^{\ad}$ with $\DDD$}
\index{\AdelDivSubject}{associated global adelic R-Cartier divisor@associated global adelic $\RR$-Cartier divisor}%
\index{\AdelDivSymbol}{0D:DDD^{ad}@$\DDD^{\ad}$}%
is defined by
\[
\DDD^{\ad} := \left( \DDD \cap X, \left\{ g_{(\XXX_{(P)},\, \DDD_{(P)})} \right\}_{P \in M_K} \right).
\]
By abuse of notation,
we often use the notations $\overline{\DDD} \leq \overline{D}_2$ and $\overline{D}_1 \leq \overline{\DDD}$ instead of
$\overline{\DDD}^{\ad} \leq \overline{D}_2$ and $\overline{D}_1 \leq \overline{\DDD}^{\ad}$ respectively.
The following proposition is the arithmetic version of
Proposition~\ref{prop:comp:adelic} and it follows from
Proposition~\ref{prop:comp:adelic}.

\begin{Proposition}
\label{prop:comp:adelic:arith}
Let $\XXX$ be a normal model of $X$ over $\Spec(O_K)$ and let $\aDiv_{C^0}(\XXX)_{\RR}$ be
the space of arithmetic $\RR$-Cartier divisors of $C^0$-type on $\XXX$.
Let 
\[
\iota : \aDiv_{C^0}(\XXX)_{\RR} \to \aDiv^{\ad}_{C^0}(X)_{\RR}
\]
be the map given by $\overline{\DDD} \mapsto \overline{\DDD}^{\ad}$.
Then we have the following:
\begin{enumerate}
\renewcommand{\labelenumi}{(\arabic{enumi})}
\item
The map $\iota : \aDiv_{C^0}(\XXX)_{\RR} \to \aDiv^{\ad}_{C^0}(X)_{\RR}$ is an
injective homomorphism of $\RR$-vector spaces.

\item
$\overline{\DDD}_1 \leq \overline{\DDD}_2$ $\Longleftrightarrow$ 
$\overline{\DDD}^{\ad}_1 \leq \overline{\DDD}^{\ad}_2$.
\end{enumerate}
\end{Proposition}

The following theorem is a consequence of Theorem~\ref{thm:approx:adelic}.

\begin{Theorem}[Approximation theorem of adelic arithmetic $\RR$-divisors]
\label{thm:approx:adelic:arith}
Let $\overline{D} = (D, g)$ be an adelic arithmetic $\RR$-Cartier divisor of $C^0$-type
on $X$ and let $(\XXX_U, \DDD_U)$ be a defining 
model of $\overline{D}$ over a non-empty open set $U \subseteq \Spec(O_K)$.
For any positive number $\epsilon > 0$, 
there exist a normal 
model $\XXX_{\epsilon}$ over $\Spec(O_K)$, and $\RR$-Cartier divisors $\DDD_1$ and $\DDD_2$ on 
$\XXX_{\epsilon}$
with the following properties:
\begin{enumerate}
\renewcommand{\labelenumi}{(\arabic{enumi})}
\item
$\rest{\XXX_{\epsilon}}{U} = \XXX_U$, $\rest{\DDD_1}{U} = \DDD_U$ and $\rest{\DDD_2}{U} = \DDD_U$.

\item
If we set $S = M_K \setminus U$, $\overline{\DDD}_1 = (\DDD_1, g_{\infty})$ and
$\overline{\DDD}_2 = (\DDD_2, g_{\infty})$, then
\[
\overline{D} - \left(0, \sum_{P \in S} \epsilon [P]\right) \leq \overline{\DDD}^{\ad}_1 \leq \overline{D} \leq \overline{\DDD}^{\ad}_2 \leq 
\overline{D} + \left(0, \sum_{P \in S} \epsilon [P]\right).
\]
\end{enumerate}
\end{Theorem}

\ifmonog\section{Global degree}\fi
\ifpaper\subsection{Global degree}\fi
\label{subsec:global:degree}

Let $\overline{D} = (D, g)$ be an adelic arithmetic $\RR$-Cartier divisor of $C^0$-type on $X$.
Let $x$ be a closed point of $X$. First we assume that $x \not\in \Supp_{\RR}(D)$.
For $P \in M_K$,
the local degree of $\overline{D}$ over the valuation $v_P$
is denoted by $\adeg_P(\srest{\overline{D}}{x})$ (cf. 
\ifmonog Section~\ref{subsec:local:degree}). \fi
\ifpaper Subsection~\ref{subsec:local:degree}). \fi
\index{\AdelDivSymbol}{0d:adeg_P(overline{D}{x})@$\adeg_P(\srest{\overline{D}}{x})$}%
Moreover, $\adeg_{\infty}(\srest{\overline{D}}{x})$ is defined by
\[
\adeg_{\infty}(\srest{\overline{D}}{x}) := \frac{1}{2} \sum_{\sigma : K(x) \hookrightarrow \CC} 
g_{\infty}(x_{\sigma}),
\]
where $K(x)$ is the residue field at $x$ and $x_{\sigma}$ is the $\CC$-value point
given by 
\[
\OOO_{X, x} \to K(x) \overset{\sigma}{\hookrightarrow} \CC.
\]%
\index{\AdelDivSymbol}{0d:adeg_{infty}(overline{D}{x})@$\adeg_{\infty}(\srest{\overline{D}}{x})$}%
Let $U$ be a non-empty Zariski open set of $\Spec(O_K)$ such that
$\overline{D}$ has a defining model $(\XXX_U, \DDD_U)$ over $U$.
Let $\Delta_x$ be the closure of $x$ in $\XXX_U$.
Shrinking $U$ if necessarily,
we may assume that $\Delta_x \cap \Supp(\DDD_U) = \emptyset$,
which implies that $\adeg_P(\srest{\overline{D}}{x}) = 0$ for $P \in U$.
Therefore, 
\index{\AdelDivSymbol}{0d:adeg(overline{D}{x})@$\adeg(\srest{\overline{D}}{x})$}%
we can define $\adeg(\srest{\overline{D}}{x})$ to be
\[
\adeg(\srest{\overline{D}}{x}) = \sum_{P \in M_K} \adeg_P(\srest{\overline{D}}{x})
+ \adeg_{\infty}(\srest{\overline{D}}{x}).
\]
Note that
\frontmatterforspececialeqn
\begin{equation}
\label{eqn:arith:degree:principal:01}
\adeg(\srest{\widehat{(\varphi)}}{x}) = 0
\end{equation}
\backmatterforspececialeqn
for $\varphi \in \Rat(X)^{\times}$ with $x \not\in \Supp((\varphi))$.
In general, we can find $\phi \in \Rat(X)^{\times}_{\RR}$ such that
$x \not\in \Supp(D + (\phi))$ (cf. \cite[Lemma~5.2.3]{MoArZariski}).
By using \eqref{eqn:arith:degree:principal:01},
we can see that the quantity 
$\adeg(\srest{\overline{D} + \widehat{(\phi)}}{x})$ does not depend on the choice of
$\phi$, so that it is denoted  by $\adeg(\srest{\overline{D}}{x})$
and is called the {\em global degree of $\overline{D}$ along $x$}.
\index{\AdelDivSubject}{global degree@global degree}%
The equation \eqref{eqn:arith:degree:principal:01} can be generalized as follows:
\frontmatterforspececialeqn
\begin{equation}
\label{eqn:arith:degree:principal:02}
\adeg(\srest{\widehat{(\psi)}}{x}) = 0
\end{equation}
\backmatterforspececialeqn
for all $\psi \in \Rat(X)^{\times}$.

\begin{Lemma}
\label{lem:global:degree:comp}
Let $\overline{D}_1 = (D_1, g_1)$ and $\overline{D}_2 = (D_2, g_2)$ be 
adelic arithmetic $\RR$-Cartier divisors of $C^0$-type on $X$.
If $D_1 = D_2$ and $g_1 \leq g_2$, then
$\adeg(\srest{\overline{D}_1}{x}) \leq \adeg(\srest{\overline{D}_2}{x}) $
for all closed points $x$ of $X$.
\end{Lemma}

\begin{proof}
As $D_1 = D_2$ and $g_1 \leq g_2$,
there are non-negative continuous functions $\phi_P$ on $X_P^{\an}$ and $\phi_{\infty}$ on
$X(\CC)$ such that $(g_2)_P = (g_1)_P + \phi_P$ and
$(g_2)_{\infty} = (g_1)_{\infty} + \phi_{\infty}$, respectively. 
We set $\phi = \sum_{P \in M_K} \phi_P[P] + \phi_{\infty} [\infty]$.
Then, as $\adeg(\srest{(0, \phi)}{x}) \geq 0$,
\[
\adeg(\srest{\overline{D}_2}{x}) = \adeg(\srest{\overline{D}_1}{x}) +
\adeg(\srest{(0, \phi)}{x}) \geq \adeg(\srest{\overline{D}_1}{x}),
\] 
as required.
\end{proof}

\ifmonog\section{Volume of adelic arithmetic $\RR$-Cartier divisors}\fi
\ifpaper\subsection{Volume of adelic arithmetic $\RR$-Cartier divisors}\fi
\label{subsec:volume:adelic:arithmetic:divisor}

Let $D$ be an $\RR$-Cartier divisor on $X$,
$\overline{D}' = (D, g')$ a global adelic $\RR$-Cartier divisor of $C^0$-type on $X$, and
$\overline{D} = (D, g)$ an adelic arithmetic $\RR$-Cartier divisor of $C^0$-type on $X$.
We define $H^0(X, D)$, $\aH(X, \overline{D}')$ and $\aH(X, \overline{D})$ to be
\[
\begin{cases}
H^0(X, D) := \left\{ \varphi \in \Rat(X)^{\times} \mid D + (\varphi) \geq 0 \right\} \cup \{ 0 \}, \\
\aH(X, \overline{D}') := \left\{ \varphi \in H^0(X, D) \mid 
\text{$\Vert \varphi \Vert_{g'_P} \leq 1$ for all $P \in M_K$} \right\},\\
\aH(X, \overline{D}) := \left\{ \varphi \in H^0(X, D) \mid 
\text{$\Vert \varphi \Vert_{g_{\wp}} \leq 1$ for all $\wp \in M_K \cup \{\infty\}$} \right\}.
\end{cases}
\]%
\index{\AdelDivSymbol}{0H:H^0(X,D)@$H^0(X,D)$}%
\index{\AdelDivSymbol}{0H:aH(X,overline{D}')@$\aH(X, \overline{D}')$}%
\index{\AdelDivSymbol}{0H:aH(X,overline{D})@$\aH(X, \overline{D})$}%
Note that 
$\aH(X, \overline{D}')$ is 
a submodule of $H^0(X, D)$ by using Proposition~\ref{prop:ext:cont:norm}.
Let us check the following proposition:

\begin{Proposition}
\label{prop:basic:prop:H:0}
\begin{enumerate}
\renewcommand{\labelenumi}{(\arabic{enumi})}
\item 
$\aH(X, \overline{D}')$ and $\aH(X, \overline{D})$ are given in the following ways:
\[
\begin{cases}
\aH(X, \overline{D}') = \left\{ \varphi \in \Rat(X)^{\times} \mid 
\overline{D}' + \widehat{(\varphi)}^{\tau} \geq 0 \right\} \cup \{ 0 \}, \\[1ex]
\aH(X, \overline{D}) = \left\{ \varphi \in \Rat(X)^{\times} \mid 
\overline{D} + \widehat{(\varphi)} \geq 0 \right\} \cup \{ 0 \}.
\end{cases}
\]

\item
We assume that $\overline{D}^{\tau} = \overline{D}'$.
If there are a normal model $\XXX$ of $X$ over $\Spec(O_K)$
and
an $\RR$-Cartier divisor $\DDD$ on $\XXX$
such that $\DDD \cap X = D$ and $g_P$ is the Green function arising from $(\XXX, \DDD)$ for each $P \in M_K$,
then
\[
\aH(X, \overline{D}') = H^0(\XXX, \DDD)\quad\text{and}\quad
\aH(X, \overline{D}) = \aH(\XXX, (\DDD, g_{\infty})).
\]

\item
$\aH(X, \overline{D}')$ is a finitely generated free $\ZZ$-module and
$\aH(X, \overline{D})$ is a finite set.
We denote $\log \# (\aH(X, \overline{D}))$ by $\ah(X, \overline{D})$.
\index{\AdelDivSymbol}{0h:ah(X,overline{D})@$\ah(X, \overline{D})$}%
\end{enumerate}
\end{Proposition}

\begin{proof}
(1) Note that
\[
\text{$\Vert \varphi \Vert_{g_{\infty}} \leq 1$}
\quad\Longleftrightarrow\quad
\text{$g_{\infty} - \log \vert \varphi \vert^2 \geq 0$ on $X(\CC)$}
\]
and
\[
\text{$\Vert \varphi \Vert_{g_P} \leq 1$}
\quad\Longleftrightarrow\quad
\text{$g_P - \log \vert \varphi \vert^2 \geq 0$ on $X^{\an}_P$}
\]
for $P \in M_K$.
Thus (1) follows.

\medskip
(2) The assertions of (2) follow from (1) and Proposition~\ref{prop:comp:adelic:arith}.

\medskip
(3) Clearly we may assume that $\overline{D}^{\tau} = \overline{D}'$.
We can find a normal model $\XXX$ of $X$ over $\Spec(O_K)$ and
an arithmetic $\RR$-Cartier divisor $\overline{\DDD} =(\DDD, h)$ of $C^0$-type on $\XXX$ such that
$\overline{D} \leq \overline{\DDD}^{\ad}$.
Thus 
\[
\aH(X, \overline{D}') \subseteq \aH(X, \DDD^{\ad}) = H^0(\XXX, \DDD)
\]
by (2).
Note that $H^0(\XXX, \DDD)$ is a finitely generated free $\ZZ$-module, so that
$\aH(X, \overline{D}')$ is also a finitely generated free $\ZZ$-module.
Since $\aH(X, \overline{D}) \subseteq \aH(\XXX, \overline{\DDD})$, the last assertion is obvious.
\end{proof}

\begin{Definition}
\label{def:volume:chi:volume}
Let $\overline{D} = (D, g)$ be an adelic arithmetic $\RR$-Cartier divisor of $C^0$-type on $X$.
The quantity $\achi(X, \overline{D})$ is defined by
\[
\achi(X, \overline{D}) := \achi\left( \aH(X, \overline{D}^{\tau}), \Vert\cdot\Vert_{g_{\infty}}\right)
\]
(cf. Conventions and terminology~\ref{CT:norm:module}).
\index{\AdelDivSymbol}{0c:achi(X,overline{D})@$\achi(X, \overline{D})$}%
Note that 
\[
\ah(X, \overline{D}) = \ah\left( \aH(X, \overline{D}^{\tau}), \Vert\cdot\Vert_{g_{\infty}}\right).
\]
Moreover, 
we define the {\em volume $\avol(\overline{D})$ of $\overline{D}$} and the
{\em $\achi$-volume $\avol_{\chi}(\overline{D})$ of $\overline{D}$} 
\index{\AdelDivSubject}{volume of adelic arithmetic R-Cartier divisor@volume of adelic arithmetic $\RR$-Cartier divisor}%
\index{\AdelDivSymbol}{0v:avol(overline{D})@$\avol(\overline{D})$}%
\index{\AdelDivSubject}{chi-volume of adelic arithmetic R-Cartier divisor@$\achi$-volume of adelic arithmetic $\RR$-Cartier divisor}%
\index{\AdelDivSymbol}{0v:avol_{chi}(overline{D})@$\avol_{\chi}(\overline{D})$}%
to be
\[
\avol(\overline{D}) := \limsup_{n\to\infty} \frac{\ah(X, n \overline{D})}{n^{d + 1}/(d + 1)!}
\]
and
\[
\avol_{\chi}(\overline{D}) := \limsup_{n\to\infty} \frac{\achi(X, n \overline{D})}{n^{d + 1}/(d + 1)!},
\]
respectively, where $d = \dim X$.
By Minkowski's theorem, $\avol_{\chi}(\overline{D}) \leq \avol(\overline{D})$.
Let $\overline{L} = (L, h)$ be
another adelic arithmetic $\RR$-Cartier divisor of $C^0$-type on $X$.
Clearly, if $\overline{L} \leq \overline{D}$, then $\ah(X,\overline{L}) \leq \ah(X, \overline{D})$
and $\avol(\overline{L}) \leq \avol(\overline{D})$.
Further, by \eqref{eqn:lem:h:0:chi:01},
if $\overline{L} \leq \overline{D}$ and $L = D$, then $\achi(X,\overline{L}) \leq \achi(X, \overline{D})$
and $\avol_{\chi}(\overline{L}) \leq \avol_{\chi}(\overline{D})$.
For the symbol $\natural$, 
$\avol_{\natural}(\overline{D})$ stands for
either $\avol(\overline{D})$ or $\avol_{\chi}(\overline{D})$,
that is,
\[
\avol_{\natural}(\overline{D}) = \begin{cases}
\avol(\overline{D}) & \text{if $\natural$ is blank},\\
\avol_{\chi}(\overline{D}) & \text{if $\natural$ is $\chi$}.
\end{cases}
\]
\end{Definition}

\ifmonog\section{Positivity of adelic arithmetic $\RR$-Cartier divisors}\fi
\ifpaper\subsection{Positivity of adelic arithmetic $\RR$-Cartier divisors}\fi
\label{subsec:positivity:adelic:arithmetic:divisor}

Here let us introduce several kinds of the positivity of adelic arithmetic divisors.

\begin{Definition}
\label{def:positive:adelic:arithmetic:divisor}
Let $\overline{D} = (D, g)$ be an adelic arithmetic $\RR$-Cartier divisor of $C^0$-type on $X$.

$\bullet$ {\bf Big}:
We say $\overline{D}$ is {\em big} 
\index{\AdelDivSubject}{big adelic arithmetic R-Cartier divisor@big adelic arithmetic $\RR$-Cartier divisor}%
if $\avol(\overline{D}) > 0$.
According as \cite{MoARH}, we can give an alternative definition, that is,
for any adelic arithmetic $\RR$-Cartier divisor $\overline{L}$ of $C^0$-type on $X$,
$\aH(X, n \overline{D} + \overline{L}) \not= \{ 0 \}$ for some positive integer $n$.
Actually two definitions are equivalent by the continuity of the volume function.

$\bullet$ {\bf Pseudo-effective}:
$\overline{D}$ is said to be {\em pseudo-effective} 
\index{\AdelDivSubject}{pseudo-effective arithmetic R-Cartier divisor@pseudo-effective arithmetic $\RR$-Cartier divisor}%
if $\overline{D} + \overline{A}$ is big for any big adelic arithmetic $\RR$-Cartier divisor 
$\overline{A}$ of $C^0$-type on $X$.

$\bullet$ {\bf Relatively nef}:  
$\overline{D}$ is said to be {\em relatively nef}
\index{\AdelDivSubject}{relatively nef adelic arithmetic R-Cartier divisor@relatively nef adelic arithmetic $\RR$-Cartier divisor}%
if the following conditions are satisfied:
\begin{enumerate}
\renewcommand{\labelenumi}{(\arabic{enumi})}
\item
For $P \in M_K$, $g_P$ is of $(C^0 \cap \Tpsh)$-type.

\item
The Green function $g_{\infty}$ on $X(\CC)$ is of $(C^0 \cap \Tpsh)$-type, that is,
the first Chern current $c_1(D, g_{\infty})$ is positive.
\end{enumerate}
If $(\XXX_U, \DDD_U)$ is a defining model of $\overline{D}$,
then $\DDD_U$ is relatively nef with respect to $\XXX_U \to U$
by Proposition~\ref{prop:PSH:imply:nef}.

$\bullet$ {\bf Nef}:
We say $\overline{D}$ is {\em nef} 
\index{\AdelDivSubject}{nef adelic arithmetic R-Cartier divisor@nef adelic arithmetic $\RR$-Cartier divisor}%
if $\overline{D}$ is relatively nef and
$\adeg(\srest{\overline{D}}{x}) \geq 0$ for all closed point $x$ of $X$.
For example, if $\phi \in \Rat(X)^{\times}$, then
the adelic arithmetic principal divisor $\widehat{(\phi)}$ of $\phi$ is nef.
\end{Definition}

Let us see the following proposition.

\begin{Proposition}
\label{prop:rel:nef:pseudo:effective}
Let $\overline{D} = (D, \{ g_P \}_{P\in M_K} \cup \{ g_{\infty} \})$ be an adelic arithmetic $\RR$-Cartier divisor of $C^0$-type on $X$.
Then we have the following:
\begin{enumerate}
\renewcommand{\labelenumi}{(\arabic{enumi})}
\item
Let $U$ be a non-empty open set of $\Spec(O_K)$ such that there is a defining model $(\XXX_U, \DDD_U)$
of $\overline{D}$ over $U$.
If $\overline{D}$ is relatively nef, then there are 
sequences $\{ (\XXX_n, \DDD_n) \}_{n=1}^{\infty}$ and $\{ (\XXX_n, \DDD'_n) \}_{n=1}^{\infty}$
with the following properties:
\begin{enumerate}
\renewcommand{\labelenumii}{(\arabic{enumi}.\arabic{enumii})}
\item
For every $n \geq 1$, $\XXX_n$ is a normal model of $X$ over $\Spec(O_K)$ such that $\rest{\XXX_n}{U} = \XXX_U$.

\item
For every $n \geq 1$, $\DDD_n$ and $\DDD'_n$ are relatively nef $\RR$-Cartier divisors on $\XXX_n$
such that $\srest{\DDD_n}{U} = \srest{\DDD'_n}{U} = \DDD_U$.

\item
$(\DDD'_n, g_{\infty})^{\ad} \leq \overline{D} \leq (\DDD_n, g_{\infty})^{\ad}$ for all $n \geq 1$.

\item
If we set 
\[
\qquad\qquad\qquad\phi_{n,P} := g_P - g_{((\XXX_n)_{(P)},\ (\DDD_n)_{(P)})}
\quad\text{and}\quad
\phi'_{n, P} := g_P - g_{((\XXX_n)_{(P)},\ (\DDD'_n)_{(P)})},
\]
then 
\[
\lim_{n\to\infty} \Vert \phi_{n, P} \Vert_{\sup} = \lim_{n\to\infty} \Vert \phi'_{n, P} \Vert_{\sup} = 0,
\]
where $(\XXX_n)_{(P)}$ is the localization of $\XXX_n \to \Spec(O_K)$ at $P$ and
$(\DDD_n)_{(P)}$ and $(\DDD'_n)_{(P)}$ are the restrictions of $\DDD_n$ and $\DDD'_n$ to $(\XXX_n)_{(P)}$,
respectively.
\end{enumerate}

\item
If $\overline{D}$ is nef, then $\overline{D}$ is pseudo-effective.

\item
If $D$ is big on $X$ and $\overline{D}$ is pseudo-effective, then
$\overline{D} + (0, \epsilon [\infty])$ is big for any positive number $\epsilon$.
\end{enumerate}
\end{Proposition}

\begin{proof}
(1) By Proposition~\ref{prop:PSH:imply:nef}, $\DDD_U$ is relatively nef with respect to $\XXX_U \to U$.
Thus the assertion follows from Proposition~\ref{prop:Green:psh:C0:approx}.

(2) 
Let us choose a non-empty open set $U$ of $\Spec(O_K)$ such that $\overline{D}$ has
a defining model over $U$.
Let $\XXX$ be a normal model of $X$ over $\Spec(O_K)$ and let
$\overline{\AAA}$ be an arithmetic Cartier divisor of $C^{\infty}$-type on $\XXX$ such that
\[
\overline{\AAA} - \left(\sum\nolimits_{P \in M_K \setminus U} F_P, 0 \right)
\]
is ample,
where $F_P$ is the fiber of $\XXX \to \Spec(O_K)$ over $P$.
It is sufficient to show that $\overline{D} + \epsilon \overline{\AAA}^{\ad}$ is big
for all $\epsilon \in \RR_{>0}$.
By (1), we can choose a normal model $\XXX'$ of $X$ over $\Spec(O_K)$ and a relatively nef $\RR$-Cartier divisor $\DDD$ on $\XXX'$ such that
\[
(\DDD, g_{\infty})^{\ad} - \left(0, \sum\nolimits_{P \in M_K \setminus U} \epsilon [P]\right) \leq \overline{D} 
\leq (\DDD, g_{\infty})^{\ad}.
\]
We may assume that there is a birational morphism $\mu : \XXX' \to \XXX$.
Then $(\DDD, g_{\infty})$ is nef by Lemma~\ref{lem:global:degree:comp}
and
\[ 
\left(  (\DDD, g_{\infty}) + \epsilon \mu^*\left(\overline{\AAA} - \left(\sum\nolimits_{P \in M_K \setminus U} F_P, 0 \right)\right)\right)^{\ad}
\leq \overline{D} + \epsilon \overline{\AAA}^{\ad}.
\]
Note that $(\DDD, g_{\infty}) + \epsilon \mu^*\left(\overline{\AAA} - \left(\sum\nolimits_{P \in M_K \setminus U} F_P, 0 \right)\right)$
is nef and big by \cite[Proposition~6.2.2]{MoArZariski}, as required.

(3) Let $\overline{A}$ be a big adelic arithmetic $\RR$-Cartier divisor of $C^0$-type on $X$.
Let us see the following claim:

\begin{Claim}
There are $m \in \ZZ_{>0}$, $\phi \in \Rat(X)^{\times}$ and $\lambda \in \RR$ such that
\[
m \overline{D} - \overline{A} + \widehat{(\phi)} + (0, \lambda [\infty]) \geq 0.
\]
\end{Claim}

\begin{proof}
Since $D$ is big on $X$, there are a positive integer $m$ and
a non-zero rational function $\psi$ on $X$ such that $mD - A + (\psi)$ is effective.
We set $(L, h) := m \overline{D} - \overline{A} + \widehat{(\psi)}$.
Let $U$ be a non-empty open set of $\Spec(O_K)$ such that $\overline{L}$ has a defining model $\LLL_U$ over $U$.
As $L \geq 0$, shrinking $U$ if necessarily, we may assume that $\LLL_U$ is effective.
In particular, $h_P \geq 0$ for $P \in M_K \cap U$.
Thus there is $\lambda' \in \RR$ such that 
\[
(L, h) \geq \left(0, \sum\nolimits_{P \in M_K \setminus U} (-\lambda') [P] + (- \lambda') [\infty]\right).
\]
We choose $N \in \ZZ_{>0}$ such that $-\lambda' + 2 \ord_P(N)\#(O_K/P) \geq 0$ for all $P \in M_K \setminus U$.
If we set $\Delta = m \overline{D} - \overline{A} + \widehat{(N\psi)}$,
then
\begin{align*}
\Delta & = (L, h) + \widehat{(N)} \\
& \geq
(L, h) + \left( 0, \sum_{P \in M_K \setminus U} 2 \ord_P(N)\#(O_K/P) [P] -\log N^2[\infty] \right) \\
& \geq
\left( 0, \sum_{P \in M_K \setminus U} (-\lambda ' + 2 \ord_P(N)\#(O_K/P)) [P] - (\lambda' + \log N^2)[\infty] \right) \\
& \geq \left( 0,  -(\lambda' + \log N^2)[\infty] \right),
\end{align*}
as required.
\end{proof}
Let $n$ be a positive integer such that $\lambda/(n + m) \leq \epsilon$. Then
\begin{align*}
\overline{D} + (1/n)(\overline{A} - \widehat{(\phi)}) & \leq \overline{D} + (1/n)(m \overline{D} + (0, \lambda[\infty])) \\
& = ((n+m)/n)(\overline{D} + (0, \lambda/(n + m)[\infty])) \\
& \leq ((n+m)/n))(\overline{D} + (0, \epsilon [\infty])),
\end{align*}
so that we have the assertion.
\end{proof}

In addition to the above positivity, an adelic arithmetic $\RR$-Cartier divisor 
$\overline{D}$ of $C^0$-type on $X$ is said to be {\em integrable} 
\index{\AdelDivSubject}{integrable adelic arithmetic R-Cartier divisor@integrable adelic arithmetic $\RR$-Cartier divisor}%
if there are relatively nef adelic arithmetic $\RR$-Cartier divisors $\overline{D}'$ and
$\overline{D}''$ of $C^0$-type on $X$
such that $\overline{D} = \overline{D}' - \overline{D}''$.
The set of all integrable adelic arithmetic $\RR$-Cartier divisors of $C^0$-type on $X$
is denoted by $\aDiv^{\ad}_{\integrable}(X)_{\RR}$.
\index{\AdelDivSymbol}{0Div:aDiv^{ad}_{integrable}(X)_{RR}@$\aDiv^{\ad}_{\integrable}(X)_{\RR}$}%
Note that $\aDiv^{\ad}_{\integrable}(X)_{\RR}$ forms a subspace of $\aDiv^{\ad}_{C^0}(X)_{\RR}$ over $\RR$.

\begin{Remark}
Let $\XXX$ be a normal model of $X$ over $\Spec(O_K)$.
We recall that an arithmetic $\RR$-Cartier divisor $\overline{\DDD}$ of $C^0$-type on $\XXX$ is said to be
{\em integrable} if there are relatively nef arithmetic $\RR$-Cartier divisors $\overline{\DDD}'$ and
$\overline{\DDD}''$ of $C^0$-type on $\XXX$ such that $\overline{\DDD} = \overline{\DDD}' - \overline{\DDD}''$
(cf. \cite[Subsection~2.1]{MoD}).
\end{Remark}

Finally let us introduce the relative nefness of a global adelic $\RR$-Cartier divisor.

\begin{Definition}
Let $\overline{D} = (D, \{ g_P \}_{P \in M_K})$ be a global adelic $\RR$-Cartier divisor of $C^0$-type on $X$.
We say $\overline{D}$ is {\em relatively nef}\ \ 
\index{\AdelDivSubject}{relatively nef global adelic R-Cartier divisor@relatively nef global adelic $\RR$-Cartier divisor}%
if $g_P$ is of $(C^0 \cap \Tpsh)$-type for all $P \in M_K$.
\end{Definition}

\ifmonog\section{Global intersection number}\fi
\ifpaper\subsection{Global intersection number}\fi
\label{subsec:global:int:number}

The purpose of this 
\ifmonog section \fi
\ifpaper subsection \fi
is to construct the intersection pairing
\[
\left(\aDiv^{\ad}_{\integrable}(X)_{\RR}\right)^{d+1} \to \RR
\qquad\left((\overline{D}_1, \ldots, \overline{D}_{d+1}) \mapsto \adeg(\overline{D} \cdots \overline{D}_{d+1})\right)
\]
by using the local intersection number (cf. 
\ifmonog Section~\ref{subsec:local:intersection}). \fi
\ifpaper Subsection~\ref{subsec:local:intersection}). \fi
For this, let us begin with the following lemma.

\begin{Lemma}
\label{lem:approx:integrable}
Let $\overline{D} = (D, g)$ be an integrable adelic arithmetic $\RR$-Cartier divisor of $C^0$-type on $X$.
Then there are a normal model $\XXX$ of $X$, an integrable arithmetic $\RR$-Cartier divisor $\overline{\DDD}$ of $C^0$-type on $\XXX$,
a finite subset $\{ P_1, \ldots, P_r \}$ of $M_K$ and integrable continuous functions $\phi_1, \ldots, \phi_r$ on $X_{P_1}^{\an},
\ldots, X_{P_r}^{\an}$, respectively such that
\[
\overline{D} = \overline{\DDD}^{\ad} + \left(0, \sum_{i=1}^{r} \phi_i [P_i]\right).
\]
\end{Lemma}

\begin{proof}
By definition,
there are relatively nef adelic arithmetic $\RR$-Cartier divisors $\overline{L}_1$ and
$\overline{L}_2$ of $C^0$-type on $X$
such that $\overline{D} = \overline{L}_1 - \overline{L}_2$, so that,
by using Proposition~\ref{prop:rel:nef:pseudo:effective},
we can find
a normal model $\XXX$ of $X$, relatively nef arithmetic $\RR$-Cartier divisor $\overline{\LLL}_1$ and $\overline{\LLL}_2$ of $C^0$-type on $\XXX$ and
a finite subset $\{ P_1, \ldots, P_r \}$ of $M_K$ such that
\[
\overline{L}_1 = \overline{\LLL}_1^{\ad} + \left(0, \sum_{i=1}^r \varphi_i [P_i]\right)
\quad\text{and}\quad
\overline{L}_2 = \overline{\LLL}_2^{\ad} + \left(0, \sum_{i=1}^r \psi_i [P_i]\right),
\]
where $\varphi_i$ and $\psi_i$ are continuous functions on $X_{P_i}^{\an}$.
Note that $\varphi_i$ and $\psi_i$ are integrable.
Thus, if we set $\overline{\DDD} = \overline{\LLL}_1 - \overline{\LLL}_2$ and
$\phi_i = \varphi_i - \psi_i$, then we have the assertion.
\end{proof}

Let $\overline{D}_1 = (D_1, g_1), \ldots, \overline{D}_{d+1} = (D_{d+1}, g_{d+1})$
be integrable adelic arithmetic $\RR$-Cartier divisors of $C^0$-type on $X$.
Then, by Lemma~\ref{lem:approx:integrable},
there are a normal model $\XXX$ of $X$, integrable arithmetic $\RR$-Cartier divisors $\overline{\DDD}_1, \ldots, \overline{\DDD}_{d+1}$ of $C^0$-type on $\XXX$,
and a finite subset $S$ of $M_K$ such that
\[
\overline{D}_i = \overline{\DDD}_i^{\ad} + \left(0, \sum_{P \in S} \phi_{i, P}[P] \right),
\]
where $\phi_{i, P}$'s are integrable continuous functions $X_{P}^{\an}$.
We would like to define the intersection number $\adeg(\overline{D}_1 \cdots \overline{D}_{d+1})$ to be
\begin{multline*}
\adeg(\overline{D}_1 \cdots \overline{D}_{d+1}) := \adeg(\overline{\DDD}_1 \cdots \overline{\DDD}_{d+1}) \\
+ \sum_{P \in S} \sum_{\substack{I \subseteq \{ 1, \ldots, d+1\} \\ I \not= \emptyset}}
  \log \#(O_K/P) \adeg_P\left(\prod_{i \in I} (0, \phi_{i, P}) \cdot \prod_{j \not\in I} (\DDD_j)_{(P)} \right),
\end{multline*}
where $\adeg_P$ is the local intersection number at $P$ (cf. 
\ifmonog Section~\ref{subsec:local:intersection}) \fi
\ifpaper Subsection~\ref{subsec:local:intersection}) \fi
and $(\DDD_j)_{(P)}$ means the restriction of $\DDD_j$ to the localization of $\XXX \to \Spec(O_K)$ at $P$.
\index{\AdelDivSymbol}{0d:adeg(overline{D}_1 cdots overline{D}_{d+1})@$\adeg(\overline{D}_1 \cdots \overline{D}_{d+1})$}%
For this purpose, we need to see that the above formula does not depend on the choice of $\XXX$,
$\overline{\DDD}_1, \ldots, \overline{\DDD}_{d+1}$ and $S$.
We denote the right hand side of the above by 
$\Delta(\XXX, \overline{\DDD}_1, \ldots, \overline{\DDD}_{d+1}, S)$.
Let $\XXX'$, $\overline{\DDD}'_1, \ldots, \overline{\DDD}'_{d+1}$ and $S'$ be another choice.
In order to check 
\[
\Delta(\XXX, \overline{\DDD}_1, \ldots, \overline{\DDD}_{d+1}, S) =
\Delta(\XXX', \overline{\DDD}'_1, \ldots, \overline{\DDD}'_{d+1}, S'),
\]
we may assume that $\XXX' = \XXX$ and $S' = S$.
Note that there are vertical $\RR$-Cartier divisors $\EEE_1, \ldots, \EEE_{d+1}$ on $\XXX$
such that $\Supp_{\RR}(\EEE_1), \ldots, \Supp_{\RR}(\EEE_{d+1}) \subseteq \sum_{P \in S} F_P$ and
$\DDD'_i = \DDD_i + \EEE_i$ for $i=1, \ldots, d+1$,
where $F_P$ is the fiber of $\XXX \to \Spec(O_K)$ over $P$.
Thus it is sufficient to show that
\[
\Delta(\XXX, \overline{\DDD}_1, \ldots, \overline{\DDD}_{d+1}, S) =
\Delta(\XXX, \overline{\DDD}_1 + (\EEE_1, 0) , \ldots, \overline{\DDD}_{d+1} + (\EEE_{d+1}, 0), S)
\]
for all vertical $\RR$-Cartier divisors $\EEE_1, \ldots, \EEE_{d+1}$ on $\XXX$
with 
\[
\Supp_{\RR}(\EEE_1), \ldots, \Supp_{\RR}(\EEE_{d+1}) \subseteq \sum_{P \in S} F_P.
\]
If we can show
\frontmatterforspececialeqn
\begin{equation}
\label{eqn:Delta:induction}
\Delta(\XXX, \overline{\DDD}_1, \ldots, \overline{\DDD}_l, \ldots, \overline{\DDD}_{d+1}, S) =
\Delta(\XXX, \overline{\DDD}_1, \ldots, \overline{\DDD}_l + (\EEE, 0), \ldots,
\overline{\DDD}_{d+1}, S)
\end{equation}
\backmatterforspececialeqn
for a vertical $\RR$-Cartier divisor $\EEE$ on $\XXX$
with $\Supp_{\RR}(\EEE) \subseteq \sum_{P \in S} F_P$, then
\begin{align*}
\Delta(\XXX, \overline{\DDD}_1, \ldots, \overline{\DDD}_{d+1}, S) &
= \Delta(\XXX, \overline{\DDD}_1 + (\EEE_1, 0), \overline{\DDD}_2, \ldots, \overline{\DDD}_{d+1}, S) \\
& = \Delta(\XXX, \overline{\DDD}_1 + (\EEE_1, 0), \overline{\DDD}_2 + (\EEE_2, 0), \overline{\DDD}_3, 
\ldots, \overline{\DDD}_{d+1}, S) \\
& = \cdots = \Delta(\XXX, \overline{\DDD}_1 + (\EEE_1, 0) , \ldots, \overline{\DDD}_{d+1} + (\EEE_{d+1}, 0), S).
\end{align*}
Therefore, it suffices to check \eqref{eqn:Delta:induction}.
We set $e_P = g_{(\XXX_{(P)},\, \EEE_{(P)})}$. Then
\[
\overline{D}_l = \left(\overline{\DDD}_l + (\EEE, 0)\right)^{\ad} + \left(0, \sum_{P \in S} (\phi_{l, P}-e_P)[P] \right),
\]
so that
\begin{multline*}
\Delta(\XXX, \overline{\DDD}_1, \ldots, \overline{\DDD}_l + (\EEE, 0), \ldots,
\overline{\DDD}_{d+1}, S)= 
\adeg(\overline{\DDD}_1 \cdots (\overline{\DDD}_l + (\EEE, 0)) \cdots \overline{\DDD}_{d+1}) + \\
\sum_{P \in S} \sum_{\substack{I \subseteq \{ 1, \ldots, d+1\} \\ I \not= \emptyset}}
  \log \#(O_K/P) \adeg_P\left(\prod_{i \in I} (0, \phi_{i, P} - \delta_{il}e_P) \cdot 
  \prod_{j \not\in I} (\DDD_j + \delta_{jl} \EEE)_{(P)} \right).
\end{multline*}
Note that
\begin{multline*}
\adeg(\overline{\DDD}_1 \cdots (\overline{\DDD}_l + (\EEE, 0)) \cdots \overline{\DDD}_{d+1}) =
\adeg(\overline{\DDD}_1 \cdots \overline{\DDD}_{d+1}) \\
+ \sum_{P \in S} \log \#(O_K/P) \adeg_P((\DDD_1)_{(P)} \cdots (\DDD_{l-1})_{(P)} \cdot (\DDD_{l+1})_{(P)} \cdots (\DDD_{d+1})_{(P)} \cdot \EEE_{(P)}).
\end{multline*}
Moreover,
\begin{multline*}
 \adeg_P\left(\prod_{i \in I} (0, \phi_{i, P} - \delta_{il}e_P) \cdot 
  \prod_{j \not\in I} (\DDD_j + \delta_{jl} \EEE)_{(P)} \right) \\
= \begin{cases}
{\displaystyle \adeg_P\left(\prod_{i \in I} (0, \phi_{i, P}) \cdot 
  \prod_{j \not\in I} (\DDD_j)_{(P)} \right) } \\
\qquad\qquad\qquad {\displaystyle - \adeg_P\left(\prod_{i \in I \setminus \{ l \}} 
  (0, \phi_{i, P}) \cdot 
  \prod_{j \not\in I} (\DDD_j)_{(P)} \cdot e_P \right)} & \text{if $l \in I$}, \\
{\displaystyle   \adeg_P\left(\prod_{i \in I} (0, \phi_{i, P}) \cdot 
  \prod_{j \not\in I} (\DDD_j)_{(P)} \right)} \\
\qquad\qquad\qquad {\displaystyle + \adeg_P\left(\prod_{i \in I} 
  (0, \phi_{i, P}) \cdot 
  \prod_{j \not\in I \cup \{ l \}} (\DDD_j)_{(P)} \cdot e_P \right)} & \text{if $l \not\in I$},
  \end{cases}
\end{multline*}
and hence
\begin{multline*}
\sum_{P \in S} \sum_{\substack{I \subseteq \{ 1, \ldots, d+1\} \\ I \not= \emptyset}}
  \log \#(O_K/P) \adeg_P\left(\prod_{i \in I} (0, \phi_{i, P} - \delta_{il}e_P) \cdot 
  \prod_{j \not\in I} (\DDD_j + \delta_{jl} \EEE)_{(P)} \right) \\
  =
  \sum_{P \in S} \sum_{\substack{I \subseteq \{ 1, \ldots, d+1\} \\ I \not= \emptyset}}
  \log \#(O_K/P) \adeg_P\left(\prod_{i \in I} (0, \phi_{i, P}) \cdot 
  \prod_{j \not\in I} (\DDD_j)_{(P)} \right) \\
  -
  \sum_{P \in S} \log \#(O_K/P) \adeg_P((\DDD_1)_{(P)} \cdots (\DDD_{l-1})_{(P)} \cdot (\DDD_{l+1})_{(P)} \cdots (\DDD_{d+1})_{(P)} \cdot e_P).
\end{multline*}
Therefore, we have \eqref{eqn:Delta:induction}.
By our construction, it is easy to see that
the map 
\[
(\overline{D}_1, \ldots, \overline{D}_{d+1}) \mapsto \adeg(\overline{D} \cdots \overline{D}_{d+1})
\]
is multi-linear and symmetric (cf. \eqref{eqn:multlinear:adelic:intersection},
\eqref{eqn:com:intersection:adelic} and Proposition~\ref{prop:intersection:com}).

Let $\overline{D}_1, \ldots, \overline{D}_{d+1}$, $\overline{D}'_1, \ldots, \overline{D}'_{d+1}$
be integrable adelic arithmetic $\RR$-Cartier divisors of $C^0$-type on $X$.
Let $T$ be a finite set of $M_K$ and let $\varphi_{1, P}, \ldots, \varphi_{d,P}$ be
integrable continuous functions on $X_P^{\an}$ for $P \in T$.
By using Lemma~\ref{lem:multi:symmetric:partition}, we can see that if
\[
\overline{D}'_i = \overline{D}_i + \left( 0, \sum\nolimits_{P \in T} \varphi_{i, P}[P] \right) 
\]
for $i=1, \ldots, d+1$, then
\frontmatterforspececialeqn
\begin{multline}
\label{eqn:formula:intersection:adelic}
\adeg(\overline{D}'_1 \cdots \overline{D}'_{d+1}) = \adeg(\overline{D}_1 \cdots \overline{D}_{d+1}) \\
+ \sum_{P \in T} \sum_{\substack{I \subseteq \{ 1, \ldots, d+1\} \\ I \not= \emptyset}}
  \log \#(O_K/P) \adeg_P\left(\prod_{i \in I} (0, \varphi_{i, P}) \cdot \prod_{j \not\in I} (D_j, (g_j)_P) \right),
\end{multline}
\backmatterforspececialeqn
where $\overline{D}_j = \left(D_j, \sum\nolimits_{P} (g_j)_P[P] + (g_j)_{\infty} [\infty]\right)$
for $j=1, \ldots, d+1$.

\begin{Proposition}
\label{prop:intersection:nef:pseudo:effective}
Let $\overline{D}_1, \ldots, \overline{D}_d, \overline{D}_{d+1}$ be integrable
adelic arithmetic $\RR$-Cartier divisors of $C^0$-type on $X$.
Then we have the following:
\begin{enumerate}
\renewcommand{\labelenumi}{(\arabic{enumi})}
\item
For $\phi \in \Rat(X)^{\times}_{\RR}$,
$\adeg(\overline{D}_1 \cdots \overline{D}_d \cdot \widehat{(\phi)}) = 0$.

\item
If $\overline{D}_1, \ldots, \overline{D}_d$ are nef and $\overline{D}_{d+1}$ is pseudo-effective, then
\[
\adeg(\overline{D}_1 \cdots \overline{D}_d \cdot \overline{D}_{d+1}) \geq 0.
\]
\end{enumerate}
\end{Proposition}

\begin{proof}
(1) Let us begin with the following claim:

\begin{Claim}
\label{claim:prop:intersection:nef:pseudo:effective:01}
Let $\XXX$ be a normal model of $X$ and let $\overline{\DDD}_1, \ldots, \overline{\DDD}_d$ be
integrable arithmetic $\RR$-Cartier divisors of $C^0$-type on $\XXX$. 
Then $\adeg(\overline{\DDD}_1 \cdots \overline{\DDD}_{d} \cdot \widehat{(\phi)}) = 0$.
\end{Claim}

\begin{proof}
{\bf Step~1} (the case where $\overline{\DDD}_1, \ldots, \overline{\DDD}_d$ are of $C^{\infty}$-type) : 
If $\overline{\DDD}_1, \ldots, \overline{\DDD}_d$ are arithmetic Cartier divisors of $C^{\infty}$-type on $\XXX$
and $\phi \in \Rat(X)^{\times}$, then the assertion is well-known.
On the other hand, by \cite[Proposition~2.4.2]{MoArZariski},
we can find arithmetic Cartier divisors $\overline{\EEE}_1, \ldots, \overline{\EEE}_r$ of $C^{\infty}$-type
on $\XXX$, $\phi_1, \ldots, \phi_l \in \Rat(X)^{\times}$,
$a_{ij} \in \RR$ ($i=1, \ldots, d$, $j=1, \ldots, r$) and $b_1, \ldots, b_l \in \RR$ such that
$\overline{\DDD}_i = \sum_{j=1}^r a_{ij} \overline{\EEE}_j$ and $\phi = \phi_1^{b_1} \cdots \phi_l^{b_l}$.
Thus, using the linearity of the intersection pairing, we have the assertion.

{\bf Step~2}  (the case where $\overline{\DDD}_1, \ldots, \overline{\DDD}_d$ are relatively nef) : 
Let $\overline{\AAA}$ be an ample arithmetic Cartier divisor of $C^{\infty}$-type on $\XXX$.
As 
\[
\lim_{n \to \infty} \adeg((\overline{\DDD}_1 + (1/n) \overline{\AAA} ) \cdots (\overline{\DDD}_{d} + (1/n) \overline{\AAA}) \cdot \widehat{(\phi)}) = \adeg(\overline{\DDD}_1 \cdots \overline{\DDD}_{d} \cdot \widehat{(\phi)}),
\]
we may assume that $\DDD_1, \ldots, \DDD_{d}$ is ample on $X$.
Then, by \cite[Theorem~4.6]{MoArZariski}, there are sequences $\{ f_{1, n} \}_{n=1}^{\infty}, \ldots
\{ f_{d, n} \}_{n=1}^{\infty}$ of $F_{\infty}$-invariant continuous functions on $X(\CC)$
such that 
\[
\lim_{n \to \infty} \Vert f_{i, n} \Vert_{\sup} = 0
\]
and
$\overline{\DDD}_i + (0, f_{i,n})$ is relatively nef and of $C^{\infty}$-type for $i=1, \ldots, d$ and $n \geq 1$.
Therefore, by using \cite[Lemma~1.2.1]{MoD} together with Step~1, we have
\[
\adeg(\overline{\DDD}_1 \cdots \overline{\DDD}_{d} \cdot \widehat{(\phi)}) = 
\lim_{n \to \infty} \adeg((\overline{\DDD}_1 + (0, f_{1, n})) \cdots (\overline{\DDD}_{d} + (0, f_{d,n})) \cdot \widehat{(\phi)}) = 0.
\]

{\bf Step~3} (general case) : 
Since $\overline{\DDD}_i$ is integrable, there are relatively nef arithmetic $\RR$-Cartier divisors
$\overline{\LLL}_i$ and $\overline{\MMM}_i$ of $C^0$-type on $\XXX$ such that $\overline{\DDD}_i = \overline{\LLL}_i - \overline{\MMM}_i$.
Thus the assertion follows from Step~2.
\end{proof}

Let us start the proof of (1).
By Lemma~\ref{lem:approx:integrable},
we can find a normal model $\XXX$ of $X$, 
integrable arithmetic $\RR$-Cartier divisors  $\overline{\DDD}_1, \ldots, \overline{\DDD}_{d}$ 
of $C^0$-type on $\XXX$,
and  a finite subset $S$ of $M_K$ such that
\[
\overline{D}_i = \overline{\DDD}_i^{\ad} + \left(0, \sum_{P \in S} \phi_{i, P}[P] \right),
\]
where $\phi_{i, P}$'s are integrable continuous functions $X_{P}^{\an}$. Then, by \eqref{eqn:formula:intersection:adelic},
\begin{multline*}
\adeg(\overline{D}_1 \cdots \overline{D}_{d} \cdot \widehat{(\phi)}) = \adeg(\overline{\DDD}_1 \cdots \overline{\DDD}_{d} \cdot \widehat{(\phi)}) \\
+ \sum_{P \in S} \sum_{\substack{I \subseteq \{ 1, \ldots, d+1\} \\ I \not= \emptyset}}
  \log \#(O_K/P) \adeg_P\left(\prod_{i \in I} (0, \phi_{i, P}) \cdot \prod_{j \not\in I} (\DDD_j)_{(P)} \cdot (\phi)_{(P)}\right).
\end{multline*}
Therefore, (1) follows from Claim~\ref{claim:prop:intersection:nef:pseudo:effective:01} and 
(1) in Proposition~\ref{prop:misc:intersection:adelic}.

\bigskip
(2) First let us see the following claim:

\begin{Claim}
\label{claim:prop:intersection:nef:pseudo:effective:02}
Let $\XXX$ be a normal model of $X$ and $\overline{\DDD}_1, \ldots, \overline{\DDD}_d, \overline{\DDD}_{d+1}$ be
integrable arithmetic $\RR$-Cartier divisors of $C^0$-type on $\XXX$.
\begin{enumerate}
\renewcommand{\labelenumi}{(\alph{enumi})}
\item
If $\overline{\DDD}_1, \ldots, \overline{\DDD}_d$ are nef and
$\overline{\DDD}_{d+1}$ is effective, then
\[
\adeg(\overline{\DDD}_1 \cdots \overline{\DDD}_{d} \cdot \overline{\DDD}_{d+1}) \geq 0.
\]

\item
Let $\overline{E}$ be an effective and integrable adelic arithmetic $\RR$-Cartier divisor of $C^0$-type on $X$.
If $\overline{\DDD}_1, \ldots, \overline{\DDD}_d$ are nef, then
\[
\adeg(\overline{\DDD}^{\ad}_1 \cdots \overline{\DDD}^{\ad}_{d} \cdot \overline{E}) \geq 0.
\]
\end{enumerate}
\end{Claim}

\begin{proof}
(a)
Let $\overline{\AAA}$ be an ample arithmetic Cartier divisor of $C^{\infty}$-type on $\XXX$.
It is sufficient to show that
\[
\adeg\left( \left( \overline{\DDD}_1 + \epsilon \overline{\AAA}\right) \cdots \left(\overline{\DDD}_d + \epsilon \overline{\AAA} \right)
\cdot \overline{\DDD}_{d+1} \right) \geq 0
\]
for $\epsilon > 0$.
First we assume that $\overline{\DDD}_1, \ldots, \overline{\DDD}_{d}$ are of $C^{\infty}$-type.
Then,  
by \cite[Proposition~6.2.2]{MoArZariski},
$\overline{\DDD}_i + \epsilon \overline{\AAA}$ is ample for every $i$, that is,
there are ample arithmetic Cartier divisors $\overline{\AAA}_1, \ldots, \overline{\AAA}_r$  of $C^{\infty}$-type on $\XXX$
such that $\overline{\DDD}_i + \epsilon \overline{\AAA} = \sum_{j=1}^r a_{ij} \overline{\AAA}_{j}$ for some
$a_{ij} \in \RR_{\geq 0}$.
On the other hand, by \cite[Proposition~2.4.2]{MoArZariski},
there are effective arithmetic Cartier divisors $\overline{\EEE}_1, \ldots, \overline{\EEE}_r$ of $C^{\infty}$-type on $\XXX$
such that $\overline{\DDD}_{d+1} = b_1 \overline{\EEE}_1 + \cdots +  b_l \overline{\EEE}_l$ for
some  $b_1, \ldots, b_l \in \RR_{\geq 0}$,
and hence the assertion follows from \cite[Proposition~2.3]{MoARH}.

In general, as before,
by \cite[Theorem~4.6]{MoArZariski}, there are sequences 
\[
\{ f_{1, n} \}_{n=1}^{\infty}, \ldots,
\{ f_{d, n} \}_{n=1}^{\infty}
\]
of $F_{\infty}$-invariant non-negative continuous functions on $X(\CC)$
such that 
\[
\lim_{n \to \infty} \Vert f_{i, n} \Vert_{\sup} = 0
\]
and
$\overline{\DDD}_i + \epsilon \overline{\AAA} + (0, f_{i,.n})$ is 
nef and of $C^{\infty}$-type for $i=1, \ldots, d$ and $n \geq 1$.
Therefore, by \cite[Lemma~1.2.1]{MoD}, we have
\begin{multline*}
\adeg\left(\left( \overline{\DDD}_1 + \epsilon \overline{\AAA}\right) \cdots \left( \overline{\DDD}_1 + \epsilon \overline{\AAA}\right)
 \cdot \overline{\DDD}_{d+1} \right) \\
 = 
\lim_{n \to \infty} \adeg\left(\left(\overline{\DDD}_1 + \epsilon \overline{\AAA} + (0, f_{1, n})\right) \cdots \left(\overline{\DDD}_{d} + \epsilon \overline{\AAA} + (0, f_{d,n})\right) \cdot \overline{\DDD}_{d+1}\right) \geq 0,
\end{multline*}
as required.

\medskip
(b) By Theorem~\ref{thm:approx:adelic:arith}, there is a finite subset $S$ of $M_K$ with the following property:
for any $n \in \ZZ_{>0}$, there are a normal model $\XXX_n$ of $X$ together with
a birational morphism $\mu_n : \XXX_n \to \XXX$, and
an integrable arithmetic $\RR$-Cartier divisor $\overline{\EEE}_n$ of $C^0$-type on $\XXX_n$ such that
\[
\overline{E} \leq \overline{\EEE}_n^{\ad} \leq \overline{E} + \left( 0, \sum\nolimits_{P \in S} (1/n)[P] \right).
\]
Then $\overline{\EEE}_n$ is effective by (2) in Proposition~\ref{prop:comp:adelic:arith}
and, if we set
\[
\overline{\EEE}_n^{\ad} = \overline{E} + \left( 0, \sum\nolimits_{P \in S} \psi_{n, P}[P] \right),
\]
then $\psi_{n,P}$ is a continuous function on $X_P^{\an}$ with $0 \leq \psi_{n, P} \leq 1/n$ for
every $n$ and $P \in S$. Therefore, by (a)
\begin{multline*}
\adeg(\overline{\DDD}^{\ad}_1 \cdots \overline{\DDD}^{\ad}_{d} \cdot \overline{E}) =
\adeg(\mu_n^*(\overline{\DDD}_1) \cdots \mu_n^*(\overline{\DDD}_{d}) \cdot \overline{\EEE}_n) \\
\qquad\qquad\qquad\qquad - \sum_{P \in S} \log \#(O_K/P) \adeg_P(\LLL_1 \cdots \LLL_d ; \psi_{n, P}) \\
\geq - \sum_{P \in S} \log \#(O_K/P) \adeg_P(\LLL_1 \cdots \LLL_d ; \psi_{n, P})
\end{multline*}
On the other hand, by (3) in Proposition~\ref{prop:misc:intersection:adelic},
\[
\left| \adeg_P(\LLL_1 \cdots \LLL_d ; \psi_{n, P}) \right| \leq \frac{1}{-2n \log v(\varpi)} 
\deg((\LLL_1 \cap X) \cdots (\LLL_d \cap X)).
\]
Thus we have the assertion of (b).
\end{proof}

\begin{Claim}
\label{claim:prop:intersection:nef:pseudo:effective:03}
If $\overline{D}_1, \ldots, \overline{D}_d$ are nef and $\overline{D}_{d+1}$ is effective, then
\[
\adeg(\overline{D}_1 \cdots \overline{D}_d \cdot \overline{D}_{d+1}) \geq 0.
\]
\end{Claim}

\begin{proof}
By (1) in Proposition~\ref{prop:rel:nef:pseudo:effective}, as before,
we can find a finite subset $T$ of $M_K$ with the following property:
for any $n \in \ZZ_{>0}$, there are a normal model $\XXX_n$ and
relatively nef arithmetic $\RR$-Cartier divisors 
\[
\overline{\DDD}_{1,n}, \ldots, \overline{\DDD}_{d,n}
\]
of $C^0$-type on $\XXX_n$ such that
\[
\overline{D}_i \leq \overline{\DDD}_{i, n}^{\ad} \leq \overline{D}_i + \left( 0, \sum\nolimits_{P \in T} (1/n)[P] \right)
\]
for $i=1, \ldots, d$. Note that $\overline{\DDD}_{i, n}$ is nef for every $i$ and $n$ by Lemma~\ref{lem:global:degree:comp},
and if we set 
\[
\overline{\DDD}_{i,n}^{\ad} = \overline{D}_i + \left( 0, \sum\nolimits_{P \in T} \varphi_{i, n, P}[P] \right),
\]
then $\varphi_{i, n,P}$ is a continuous function on $X_P^{\an}$ with $0 \leq \varphi_{i, n, P} \leq 1/n$ for
every $n$ and $P \in T$. Then, by \eqref{eqn:formula:intersection:adelic},
\begin{multline*}
\adeg(\overline{D}_1 \cdots \overline{D}_d \cdot \overline{D}_{d+1}) =
\adeg(\overline{\DDD}^{\ad}_{1,n} \cdots \overline{\DDD}^{\ad}_{d,n} \cdot \overline{D}_{d+1}) + 
\sum_{P \in T} \sum_{\substack{I \subseteq \{ 1, \ldots, d+1\} \\ I \not= \emptyset}} \\
  \log \#(O_K/P) \adeg_P\left(\prod_{i \in I} (0, -\varphi_{i,n, P}) \cdot \prod_{j \not\in I} (\DDD_j)_{(P)} \cdot (D_{d+1}, (g_{d+1})_P)
  \right),
\end{multline*}
where $\overline{D}_{d+1} = (D_{d+1}, \sum_{P} (g_{d+1})_P[P] + (g_{d+1})_{\infty} [\infty])$.
By using (3) in Proposition~\ref{prop:misc:intersection:adelic}, it is easy to see that
\[
\lim_{n\to\infty} \adeg_P\left(\prod_{i \in I} (0, -\varphi_{i,n, P}) \cdot \prod_{j \not\in I} (\DDD_j)_{(P)} \cdot (D_{d+1}, (g_{d+1})_P)
  \right) = 0
\]
for all $P \in T$ and $I \subseteq \{ 1, \ldots, d\}$ with $I \not= \emptyset$ (cf. the proof of (2) in
Claim~\ref{claim:prop:intersection:nef:pseudo:effective:02}).
Therefore, the assertion follows from (2) in Claim~\ref{claim:prop:intersection:nef:pseudo:effective:02}.
\end{proof}

Let $\overline{B}$ be a nef and big adelic arithmetic $\RR$-Cartier divisors of $C^0$-type on $X$.
For $\epsilon > 0$,
as $\overline{D}_{d+1} + \epsilon \overline{B}$ is big, there are $n \in \ZZ_{>0}$ and
$\psi \in \Rat(X)^{\times}$ such that 
\[
\overline{D}_{d+1} + \epsilon \overline{B} + (1/n)\widehat{(\psi)} \geq 0,
\]
so that,
by using (1) and Claim~\ref{claim:prop:intersection:nef:pseudo:effective:03},
\begin{multline*}
\adeg(\overline{D}_1 \cdots \overline{D}_d \cdot \overline{D}_{d+1}) + \epsilon
\adeg(\overline{D}_1 \cdots \overline{D}_d \cdot \overline{B}) \\
=
\adeg\left(\overline{D}_1 \cdots \overline{D}_d \cdot 
\left(\overline{D}_{d+1} + \epsilon \overline{B} + (1/n)\widehat{(\psi)}\right)\right) \geq 0.
\end{multline*}
Thus the assertion follows.
\end{proof}

\ifmonog\chapter{Continuity of the volume function}\fi
\ifpaper\section{Continuity of the volume function}\fi
The purpose of this 
\ifmonog chapter \fi
\ifpaper section \fi
is to consider the continuity of the volume function and its applications.
The continuity of the volume function is a very fundamental result in
the theory of birational Arakelov geometry.
It actually has a lot of applications by using perturbation methods.
The generalized Hodge index theorem is one of significant examples,
which is a generalization of results due to
Faltings-Gillet-Soul\'{e}-Zhang (cf. \cite{FaCAS}, \cite{GSRR} and \cite{ZhPos}).

Throughout this 
\ifmonog chapter, \fi
\ifpaper section, \fi
let $K$ be a number field and let $X$ be a $d$-dimensional,
projective, smooth and geometrically integral
variety over $K$.

\ifmonog\section{Basic properties of the volume}\fi
\ifpaper\subsection{Basic properties of the volume}\fi
\label{subsec:basic:prop:volume}

In this 
\ifmonog section, \fi
\ifpaper subsection, \fi
we investigate several basic properties of the volume function.
First of all, let us begin with the finiteness of $\avol$,
the limit theorem and the positive homogeneity of $\avol$.

\begin{Theorem}
\label{thm:avol:lim}
\begin{enumerate}
\renewcommand{\labelenumi}{(\arabic{enumi})}
\item \rom{(Finiteness)}
$\avol(\overline{D}) < \infty$.

\item \rom{(Limit theorem)}
The `` $\limsup$'' in the definition of $\avol$ \rom{(}cf. Definition~\rom{\ref{def:volume:chi:volume}}\rom{)}
can be replaced by `` $\lim$'', that is, 
\[
\avol(\overline{D}) 
= \lim_{t\to\infty} \frac{\ah(X, t \overline{D})}{t^{d + 1}/(d + 1)!},
\]
where $t$ is a positive real number.

\item \rom{(Positive homogeneity)}
$\avol(a \overline{D}) = a^{d + 1} \avol(\overline{D})$ for $a \in \RR_{\geq 0}$.
\end{enumerate}
\end{Theorem}

\begin{proof}
(1) is obvious because we can find a normal model $\XXX$ of $X$ over $\Spec(O_K)$ and
an arithmetic $\RR$-Cartier divisor $\overline{\DDD}$ of $C^0$-type on $\XXX$ such that
$\overline{D} \leq \overline{\DDD}^{\ad}$.

\medskip
(2) Let $(\XXX_U, \DDD_U)$ be a defining model of $\overline{D}$ over a non-empty
open set $U$ of $\Spec(O_K)$. By Theorem~\ref{thm:approx:adelic:arith},
for any $\epsilon > 0$,
there is a normal model $\XXX_{\epsilon}$ of $X$ and an arithmetic
$\RR$-Cartier divisor $\overline{\DDD}_{\epsilon}$ of $C^0$-type on $\XXX_{\epsilon}$
such that $\rest{\XXX_{\epsilon}}{U} = \XXX_U$,
$\rest{\DDD_{\epsilon}}{U} = \DDD_U$ and
\[
\overline{\DDD}_{\epsilon}^{\ad}
\leq \overline{D} \leq \overline{\DDD}_{\epsilon}^{\ad} + \left(0, \sum\nolimits_{P \in S} 2 \epsilon \log \#(O_K/P)
[P] \right).
\]
where $S = M_K \setminus U$. 
Note that the fiber $F_P$ of $\XXX_{\epsilon} \to \Spec(O_K)$ over $P$ yields
a constant function $2 \log \#(O_K/P)$ on $X_{P}^{\an}$.
Thus the above inequalities mean that
\[
\overline{\DDD}_{\epsilon}^{\ad}
\leq \overline{D} \leq \left( \overline{\DDD}_{\epsilon} + 
\epsilon \left(\sum\nolimits_{P \in S} F_P, 0 \right)\right)^{\ad}.
\]
As $\avol\left( \overline{\DDD}_{\epsilon} \right)$ and
$\avol\left( \overline{\DDD}_{\epsilon} + \epsilon \left(\sum_{P}  F_P, 0 \right) \right)$ can be
expressed by ``$\lim$'' (cf. \cite[Theorem~5.1]{MoContExt} and \cite[Theorem~5.2.2]{MoArZariski}),
if we set
\[
\Delta = \limsup_{t\to\infty} \frac{\ah(X, t \overline{D})}{t^{d + 1}/(d + 1)!} -
\liminf_{t\to\infty} \frac{\ah(X, t \overline{D})}{t^{d + 1}/(d + 1)!},
\]
then
\[
0 \leq \Delta \leq \avol\left( \overline{\DDD}_{\epsilon} + \epsilon \left(\sum\nolimits_{P \in S}  F_P, 0 \right) \right)
- \avol\left( \overline{\DDD}_{\epsilon} \right).
\]
Let us choose a positive number $N$ such that $N \in P$ for all $P \in S$.
Then $\sum_{P \in S} F_P \leq (N)$. Thus, 
by using \cite[Proposition~4.6]{MoContExt} and \cite[Theorem~5.2.2]{MoArZariski}, 
\begin{align*}
\Delta & \leq \avol\left( \overline{\DDD}_{\epsilon} + \epsilon \left(\sum\nolimits_{P}  F_P, 0 \right) \right)
- \avol\left( \overline{\DDD}_{\epsilon}\right) \\
& \leq \avol\left( \overline{\DDD}_{\epsilon} + \epsilon ((N),0) \right)
- \avol\left( \overline{\DDD}_{\epsilon}\right) \\
& = \epsilon^{d + 1}\avol\left((1/\epsilon)\overline{\DDD}_{\epsilon} + ((N), 0)\right) - \avol\left( \overline{\DDD}_{\epsilon} \right) \\
& = \epsilon^{d + 1}\avol\left((1/\epsilon)\overline{\DDD}_{\epsilon} + ((N), 0) - \widehat{(N)}\right) - \avol\left( \overline{\DDD}_{\epsilon} \right) \\
& = \avol\left(\overline{\DDD}_{\epsilon} + \epsilon((N), 0) - \epsilon \widehat{(N)}\right) - \avol\left( \overline{\DDD}_{\epsilon}  \right) \\
& = \avol\left(\overline{\DDD}_{\epsilon} + (0, 2\epsilon\log N) \right) - \avol\left( \overline{\DDD}_{\epsilon} \right)\\
&\leq
\epsilon \log(N) (d + 1) [K : \QQ] \vol(X, D).
\end{align*}
Note that $\log(N) (d + 1) [K : \QQ] \vol(X, D)$ does not depend on $\epsilon$,
so that $\Delta = 0$.

\medskip
(3) If $a= 0$, then the assertion is obvious. Otherwise, by using (2),
\begin{align*}
\avol(a \overline{D}) & = \lim_{t\to\infty} \frac{\ah(X, t a \overline{D})}{t^{d + 1}/(d + 1)!}
= a^{d +1} \lim_{t\to\infty} \frac{\ah(X, t a \overline{D})}{(ta)^{d + 1}/(d + 1)!} \\
& = a^{d +1} \avol(\overline{D}),
\end{align*}
as desired.
\end{proof}

Next let us consider the following estimate for the proof of the continuity of $\avol$ and $\avol_{\chi}$.

\begin{Proposition}
\label{prop:vol:comp:C:0}
Let $\overline{D} = (D, g)$ be an adelic arithmetic $\RR$-Cartier divisor on $X$.
Let $\varphi_1 \in C^0(X_{P_1}^{\an}),
\ldots, \varphi_r \in C^0(X_{P_r}^{\an}), \varphi_{\infty} \in C^0_{F_{\infty}}(X(\CC))$, where 
$P_1, \ldots, P_r \in M_K$.
Then 
\begin{multline*}
\left \vert \avol_{\natural}\left(\overline{D} + \left(0,\sum_{i=1}^r \varphi_i [P_i] + \varphi_{\infty} [\infty] \right)\right) - \avol_{\natural}(\overline{D}) \right\vert \\
\leq  \frac{(d + 1) [K : \QQ] \vol(X, D)}{2} \left( \sum_{i=1}^r \Vert \varphi_i \Vert_{\sup} + 
\Vert \varphi_{\infty} \Vert_{\sup} \right),
\end{multline*}
where $\avol_{\natural}$ is either $\avol$ or $\avol_{\chi}$
\rom{(}see Definition~\rom{\ref{def:volume:chi:volume}}\rom{)}.
\end{Proposition}

Let us begin with the following lemma:

\begin{Lemma}
\label{lem:vol:comp:a}
Let $\overline{D} = (D, g)$ be an adelic arithmetic $\RR$-Cartier divisor of $C^0$-type on $X$
and $P \in M_K$.
Then we have the following:
\begin{enumerate}
\renewcommand{\labelenumi}{(\arabic{enumi})}
\item
${\displaystyle \avol_{\natural}(\overline{D} + (0, a [P])) \leq \avol_{\natural}(\overline{D}) +
\frac{(d + 1) [K : \QQ]\vol(X, D)}{2[K_P: \QQ_p]}a}$ for $a \in \RR_{\geq 0}$,
where $p$ is the prime number with $p\ZZ = \ZZ \cap P$.

\item
${\displaystyle \avol_{\natural}(\overline{D} + (0, a [\infty])) \leq \avol_{\natural}(\overline{D}) +
\frac{(d + 1) [K : \QQ]\vol(X, D)}{2}a}$ for $a \in \RR_{\geq 0}$.
\end{enumerate}
\end{Lemma}

\begin{proof}
(1)
For each $n \in \ZZ_{>0}$, let $a_n$ be the round down of 
$\frac{na}{2 \ord_P(p)\log \#(O_K/P)}$, that is,
\[
a_n = \left\lceil \frac{n a }{2 \ord_P(p)\log \#(O_K/P)} \right\rceil.
\]
Then we have the following claim:

\begin{Claim}
$p^{a_n} \aH(X, n(\overline{D}^{\tau} + (0, a  [P]))) \subseteq \aH(X, n\overline{D}^{\tau})$.
\end{Claim}

\begin{proof}
Let $\phi \in \aH(X, n(\overline{D}^{\tau} + (0, a  [P])))$.
Let $Q \in \Spec(O_K)$. If $Q \not= P$, then
\[
\Vert p^{a_n} \phi \Vert_{n g_Q} = v_{Q}(p)^{a_n} \Vert \phi \Vert_{n g_Q} \leq 1.
\]
Otherwise (i.e. $Q = P$), 
\begin{align*}
\Vert p^{a_n} \phi \Vert_{n g_P} & = v_{P}(p)^{a_n} \Vert \phi \Vert_{n g_P}
= \exp\left(-a_n \ord_P(p) \log \#(O_K/P) + \frac{na}{2}\right) \Vert \phi \Vert_{n (g_P +a)} \\
& \leq \exp\left( \frac{- n a }{2\ord_P(p)\log \#(O_K/P)}  \ord_P(p) \log \#(O_K/P) + \frac{na}{2} \right)\Vert \phi \Vert_{n (g_P +a)} \\
& = \Vert \phi \Vert_{n (g_P +a)} \leq 1,
\end{align*}
as required.
\end{proof}

We set 
\[
\begin{cases}
Q_n = \Coker\left(\aH\left(X, n\overline{D}^{\tau}\right) \to \aH\left(X, n(\overline{D}^{\tau} + (0, a  [P])\right)\right), \\[1ex]
r_n = \dim_\QQ H^0(X, nD) = [K : \QQ]\dim_K H^0(X, nD).
\end{cases}
\]
Then, by \eqref{eqn:lem:h:0:chi:04} and \eqref{eqn:lem:h:0:chi:05},
\[
\begin{cases}
\ah(X, n(\overline{D} + (0, a  [P]))) \leq
\ah(X, n\overline{D}) + \log \#(Q_n) + \log(6) r_n, \\
\achi(X, n(\overline{D} + (0, a  [P]))) =
\achi(X, n\overline{D}) + \log \#(Q_n).
\end{cases}
\]
By the above claim,
\begin{align*}
\log \#(Q_n) & \leq
\log \# \left(  \aH\left(X, n(\overline{D}^{\tau} + (0, a  [P]))\right) \big/p^{a_n} \aH\left(X, n(\overline{D}^{\tau} + (0, a  [P]))\right)\right) \\
& = a_n r_n \log(p).
\end{align*}
Therefore, (1) follows because
\begin{align*}
\lim_{n\to\infty} \frac{a_n r_n \log(p)}{n^{d+1}/(d+1)!} & = \frac{a(d+1)[K: \QQ] \log(p)}{2 \ord_P(p) \log \#(O_K/P)}
\lim_{n\to\infty} \frac{\dim_K H^0(X, nD)}{n^d/d!} \\
& = \frac{a(d+1)[K: \QQ]\vol(X, D)}{2 \ord_P(p) [O_K/P : \ZZ/p\ZZ]}  = \frac{a(d+1)[K: \QQ]\vol(X, D)}{2 [K_P: \QQ_p]}.
\end{align*}

\bigskip
(2) By \eqref{eqn:lem:h:0:chi:02} and \eqref{eqn:lem:h:0:chi:03},
\[
\begin{cases}
\ah(X, n(\overline{D} + (0, a  [\infty]))) \leq
\ah(X, n\overline{D}) + (na/2) r_n + \log(3) r_n, \\
\achi(X, n(\overline{D} + (0, a  [\infty]))) =
\achi(X, n\overline{D}) + (na/2) r_n .
\end{cases}
\]
Therefore, (2) follows.
\end{proof}

\begin{proof}[Proof of Proposition~\rom{\ref{prop:vol:comp:C:0}}]
\addtocounter{Theorem}{1}
\setcounter{Claim}{0}
Let us start the proof of Proposition~\ref{prop:vol:comp:C:0}.
First we check the following special cases:

\begin{Claim}
\begin{enumerate}
\renewcommand{\labelenumi}{(\arabic{enumi})}
The following inequalities hold:
\item
${\displaystyle \left \vert \avol_{\natural}\left(\overline{D} + \left(0,\varphi_1 [P_1] \right)\right) - 
\avol_{\natural}(\overline{D}) \right\vert
\leq  \frac{(d + 1) [K : \QQ] \vol(X, D)}{2} \Vert \varphi_1 \Vert_{\sup}}$.

\item
${\displaystyle \left \vert \avol_{\natural}\left(\overline{D} + \left(0,\varphi_{\infty} [\infty] \right)\right) - 
\avol_{\natural}(\overline{D}) \right\vert
\leq  \frac{(d + 1) [K : \QQ] \vol(X, D)}{2} \Vert \varphi_{\infty} \Vert_{\sup}}$.
\end{enumerate}
\end{Claim}

\begin{proof}
(1)
By using (1) in Lemma~\ref{lem:vol:comp:a},
\begin{align*}
\avol_{\natural}(\overline{D} + (0, \varphi_1  [P_1])) - \avol_{\natural}(\overline{D}) & \leq
\avol_{\natural}(\overline{D} + (0, \Vert \varphi_1 \Vert_{\sup} [P_1])) -\avol_{\natural}(\overline{D}) \\
& \leq \frac{(d + 1) [K : \QQ] \vol(X, D)}{2} \Vert \varphi_1 \Vert_{\sup}.
\end{align*}
Applying (1) in Lemma~\ref{lem:vol:comp:a} to the case where $\overline{D}$ is
$\overline{D} - (0, a [P_1])$, we have
\[
\avol_{\natural}(\overline{D})
\leq  \avol_{\natural}(\overline{D}  - (0, a [P_1])) + (a/2)  (d + 1) [K : \QQ] \vol(X, D).
\]
Therefore,
\begin{align*}
\avol_{\natural}(\overline{D}) - \avol_{\natural}(\overline{D} + (0, \varphi_1  [P_1])) & \leq
\avol_{\natural}(\overline{D}) - \avol_{\natural}(\overline{D} - (0, \Vert \varphi_1 \Vert_{\sup} [P_1])) \\
& \leq \frac{(d + 1) [K : \QQ] \vol(X, D)}{2} \Vert \varphi_1 \Vert_{\sup}.
\end{align*}
Thus we have (1).

\medskip
(2) can be shown in the same way as (1) by using (2) in Lemma~\ref{lem:vol:comp:a}.
\end{proof}

\medskip
In general, we set
\[
\overline{D}_j =
\begin{cases}
\overline{D} + (0, \varphi_{\infty} [\infty])& \text{if $j=0$}, \\[1ex]
\overline{D} + \left( 0,\sum_{i=1}^j \varphi_i [P_i] + \varphi_{\infty} [\infty] \right)& \text{if $j\geq 1$}.
\end{cases}
\]
Then, as
\begin{multline*}
\left \vert \avol_{\natural}\left(\overline{D} + \left(0,\sum_{i=1}^r \varphi_i [P_i] + \varphi_{\infty} [\infty]\right)\right) - \avol_{\natural}(\overline{D}) \right\vert \\
\leq
\sum_{j=1}^r 
\left \vert \avol_{\natural}\left( \overline{D}_j \right) -  \avol_{\natural}\left( \overline{D}_{j-1} \right) \right\vert + 
\left\vert \avol_{\natural}\left(\overline{D}_0\right)
- \avol_{\natural}\left(\overline{D}\right) \right\vert,
\end{multline*}
using the previous claim,
we have the assertion.
\end{proof}

As a consequence of the above estimate,
we have the following proposition, so that
we can deduce Fujita's approximation theorem for adelic arithmetic $\RR$-Cartier divisors
(cf. Theorem~\ref{thm:Fujita:approx:adel}).

\begin{Proposition}
\label{prop:vol:approx}
Let $\overline{D}$ be an
adelic arithmetic $\RR$-Cartier divisor of $C^0$-type on $X$.
Then, for a positive number $\epsilon$, there are
a normal model of $\XXX$ over $\Spec(O_K)$ and
arithmetic $\RR$-Cartier divisors $\overline{\DDD}$ and $\overline{\DDD}'$
of $C^0$-type on $\XXX$ such that
\[
\overline{\DDD}^{\ad} \leq \overline{D} \leq {\overline{\DDD}'}^{\ad},\quad
0 \leq \avol(\overline{D}) - \avol(\overline{\DDD}) \leq \epsilon
\quad\text{and}\quad
0 \leq \avol(\overline{\DDD}') - \avol(\overline{D}) \leq \epsilon.
\]
\end{Proposition}

\begin{proof}
Let $U$ be a non-empty open set of $\Spec(O_K)$ such that
$\overline{D}$
has defining models $\DDD_{U}$ over $U$. We set $S = M_K \setminus U$.
We choose a positive number $\epsilon'$ such that
\[
\epsilon' (d + 1) [K : \QQ] \avol(X,D) \#(S)  \leq 2 \epsilon.
\]
By Theorem~\ref{thm:approx:adelic:arith},
there are a normal model of $\XXX$ and arithmetic $\RR$-Cartier divisors 
$\overline{\DDD}$ and $\overline{\DDD}'$ of $C^0$-type on $\XXX$ such that
\[
\overline{D} -\left(0, \sum\nolimits_{P \in S} \epsilon' [P]\right) \leq \overline{\DDD}
\leq \overline{D} \leq \overline{\DDD}' \leq
\overline{D} + \left(0, \sum\nolimits_{P \in S} \epsilon' [P]\right).
\]
Then
\begin{align*}
0 & \leq \avol(\overline{D}) - \avol(\overline{\DDD}) \leq \avol(\overline{D}) -
\avol\left( \overline{D} -\left(0, \sum\nolimits_{P \in S} \epsilon' [P]\right) \right) \\
\intertext{and}
0 & \leq \avol(\overline{\DDD}') - \avol(\overline{D}) \leq 
\avol\left( \overline{D} +\left(0, \sum\nolimits_{P \in S} \epsilon' [P]\right) \right) - \avol(\overline{D}).
\end{align*}
On the other hand, by Proposition~\ref{prop:vol:comp:C:0},
\begin{align*}
\avol(\overline{D}) -
\avol\left( \overline{D} -\left(0, \sum\nolimits_{P \in S} \epsilon' [P]\right) \right) & \leq
(\epsilon'/2) (d + 1) [K : \QQ] \avol(X,D) \#(S)\\
\intertext{and} 
\avol\left( \overline{D} +\left(0, \sum\nolimits_{P \in S} \epsilon' [P]\right) \right) - \avol(\overline{D})
& \leq
(\epsilon'/2) (d + 1) [K : \QQ] \avol(X,D) \#(S).
\end{align*}
Thus the assertion follows.
\end{proof}

The following theorem is the adelic version of arithmetic Fujita's approximation theorem.
It has been already established by
Boucksom and H. Chen \cite{BC}.
Here we give another proof of it and generalize it to $\RR$-divisors.

\begin{Theorem}[Fujita's approximation theorem for adelic arithmetic divisors]
\label{thm:Fujita:approx:adel}
Let $\overline{D}$ be a big adelic arithmetic $\RR$-Cartier divisor of $C^0$-type on $X$.
Then, for any positive number $\epsilon$, there are a birational morphism $\mu : Y \to X$
of smooth, projective and geometrically integral varieties over $K$ and
a nef adelic arithmetic $\RR$-Cartier divisor $\overline{Q}$ of $C^0$-type on $Y$ such that
$\overline{Q} \leq \mu^*(\overline{D})$ and $\avol(\overline{Q}) \geq \avol(\overline{D}) - \epsilon$.
\end{Theorem}

\begin{proof}
By Proposition~\ref{prop:vol:approx}, we can find a normal model $\XXX$ of $X$ over $\Spec(O_K)$ and
an arithmetic $\RR$-Cartier divisor $\overline{\DDD}$ of $C^0$-type on $\XXX$ such that
\[
\overline{\DDD} \leq \overline{D}
\quad\text{and}\quad
\avol(\overline{\DDD}) \geq \avol(\overline{D}) - \epsilon/2.
\]
Moreover, by virtue of
Fujita's approximation theorem for arithmetic $\RR$-Cartier divisors due to Chen-Yuan 
(cf. \cite{HChenFujita},
\cite{YuanVol} and \cite[Theorem~5.2.2]{MoArZariski}),
there are a birational morphism $\tilde{\mu} : \mathcal{Y} \to \XXX$ of
generically smooth, normal and projective  arithmetic varieties and a nef 
arithmetic $\RR$-Cartier divisor $\overline{\QQQ}$ of $C^0$-type on $\mathcal{Y}$
such that 
\[
\overline{\QQQ} \leq \tilde{\mu}^*(\overline{\DDD})
\quad\text{and}\quad
\avol(\overline{\QQQ}) \geq \avol(\overline{\DDD}) - \epsilon/2.
\]
Thus if we set $Y = \mathcal{Y} \times_{\Spec(O_K)} \Spec(K)$, $\mu = \rest{\tilde{\mu}}{Y}$ and
$\overline{Q} = \overline{\QQQ}^{\ad}$,
then we have the assertion.
\end{proof}

\ifmonog\section{Proof of the continuity of the volume function}\fi
\ifpaper\subsection{Proof of the continuity of the volume function}\fi
\label{subsec:proof:cont:vol}

The purpose of this 
\ifmonog section \fi
\ifpaper subsection \fi
is to prove the continuity of the volume function for
adelic arithmetic $\RR$-Cartier divisors of $C^0$-type.
Namely we have the following theorem:

\begin{Theorem}[Continuity of the volume functions for adelic arithmetic divisors]
\label{thm:cont:volume}
The volume function $\avol : \aDiv_{C^0}^{\ad}(X)_{\RR} \to \RR$ is continuous in the following sense:
Let $\overline{D}_1, \ldots, \overline{D}_r$, $\overline{A}_1, \ldots, \overline{A}_{r'}$ be 
adelic arithmetic $\RR$-Cartier divisors of $C^0$-type on $X$.
Let $\{ P_1, \ldots, P_s \}$ be a finite subset of $M_K$.
For a compact subset $B$ in $\RR^r$ and a positive number $\epsilon$, 
there are positive numbers $\delta$ and $\delta'$ such that
\[
\left\vert \avol\left(\sum_{i=1}^r a_i \overline{D}_i + \sum_{j=1}^{r'} \delta_j \overline{A}_j + 
\left(0, \sum_{l=1}^{s} \varphi_{P_l} [P_l] + \varphi_{\infty} [\infty] \right) \right) -
\avol\left(\sum_{i=1}^r a_i \overline{D}_i \right) \right\vert
\leq \epsilon
\]
holds for all $a_1, \ldots, a_r, \delta_1, \ldots, \delta_{r'} \in \RR$, 
$\varphi_{P_1} \in C^0(X_{P_1}^{\an}), \ldots, \varphi_{P_s} \in C^0(X_{P_s}^{\an})$
and $\varphi_{\infty} \in C^0_{F_{\infty}}(X(\CC))$
with $(a_1, \ldots, a_r) \in B$, $\sum_{j=1}^{r'} \vert \delta_j \vert \leq \delta$
and $\sum_{l=1}^s \Vert \varphi_{P_l} \Vert_{\sup} + \Vert \varphi_{\infty} \Vert_{\sup}
\leq \delta'$.
\end{Theorem}

\begin{proof}
Let us choose a non-empty open set $U$ of $\Spec(O_K)$ such that
$\overline{D}_i$ ($i=1, \ldots, r$) and $\overline{A}_j$ ($j=1,\ldots, r'$)
have defining models $\DDD_{i,U}$ and
$\AAA_{j,U}$ over $U$, respectively. 
We set $T = M_K \setminus U$ and
\[
C = \max \left\{\left. \vol\left(X, \sum_{i=1}^{r} a_iD_i +
\sum_{j=1}^{r'} \delta_j A_j\right) \left( \sum_{i=1}^r \vert a_i \vert + 1
 \right) \ \right|\  
\begin{array}{l}
 (a_1, \ldots, a_r) \in B\\
 \sum_{j=1}^{r'} \vert \delta_j \vert \leq 1 
\end{array}
\right\}.
\]
We choose a positive number $\epsilon'$ such that
\[
\epsilon' (d + 1) [K : \QQ] \#(T) C \leq \epsilon/3.
\]
By Theorem~\ref{thm:approx:adelic:arith},
there are a normal model of $\XXX$ and arithmetic $\RR$-Cartier divisors 
$\overline{\DDD}_1, \ldots, \overline{\DDD}_r, \overline{\AAA}_1, \ldots,
\overline{\AAA}_{r'}$ of $C^0$-type on $\XXX$ such that
\[
\overline{\DDD}_i^{\ad}
\leq \overline{D}_i \leq
\overline{\DDD}_i^{\ad} + \left(0, \sum_{P \in T} \epsilon' [P]\right)
\quad\text{and}\quad
\overline{\AAA}_j^{\ad} 
\leq \overline{A}_j \leq
\overline{\AAA}_j^{\ad} + \left(0, \sum_{P \in T} \epsilon' [P]\right)
\]
for all $i=1,\ldots,r$ and $j=1,\ldots,r'$.
Then,
\begin{multline*}
\sum_{i=1}^r a_i \overline{\DDD}_i^{\ad} + \sum_{j=1}^{r'} \delta_j \overline{\AAA}_j^{\ad} -
\left(0, \sum_{P \in T} \epsilon'\left( \sum_{i=1}^r \vert a_i \vert + \sum_{j=1}^{r'} \vert \delta_j \vert \right) [P]\right) \\
\hskip-15em 
\leq \sum_{i=1}^r a_i \overline{D}_i + \sum_{j=1}^{r'} \delta_j \overline{A}_j \\
\leq
\sum_{i=1}^r a_i \overline{\DDD}^{\ad}_i + \sum_{j=1}^{r'} \delta_j \overline{\AAA}^{\ad}_j +
\left(0, \sum_{P \in T} \epsilon'\left( \sum_{i=1}^r \vert a_i \vert + \sum_{j=1}^{r'} \vert \delta_j \vert \right) [P]\right).
\end{multline*}
Therefore, by using Proposition~\ref{prop:vol:comp:C:0},
\addtocounter{Claim}{1}
\begin{multline}
\label{eqn:thm:cont:volume:01}
\left| \avol\left( \sum_{i=1}^r a_i \overline{\DDD}_i + \sum_{j=1}^{r'} \delta_j \overline{\AAA}_j \right)
- \avol\left(  \sum_{i=1}^r a_i \overline{D}_i + \sum_{j=1}^{r'} \delta_j \overline{A}_j \right)
\right| \\
\leq \frac{\epsilon' (d + 1) [K : \QQ] \#(T)}{2}
\vol\left(X, \sum_{i=1}^{r} a_iD_i +
\sum_{j=1}^{r'} \delta_j A_j\right) \left( \sum_{i=1}^r \vert a_i \vert + \sum_{j=1}^{r'} \vert \delta_j \vert
 \right)
\end{multline}
for all $a_1, \ldots, a_r, \delta_1, \ldots, \delta_{r'} \in \RR$.
Moreover, by \cite[Theorem~5.2.2]{MoArZariski},
there is a positive number $\delta$ such that $\delta \leq 1$ and
\[
\left| 
\avol\left(  \sum_{i=1}^r a_i \overline{\DDD}_i + \sum_{j=1}^{r'} \delta_j \overline{\AAA}_j \right)
- \avol\left(  \sum_{i=1}^r a_i \overline{\DDD}_i \right)
\right| \leq \epsilon/6
\]
for all $(a_1, \ldots, a_r) \in B$ and $\delta_1, \ldots, \delta_{r'} \in \RR$
with 
\[
\vert \delta_1 \vert + \cdots + \vert \delta_{r'} \vert \leq \delta.
\]
Therefore, if we set
\[
\Delta = \left| 
\avol\left(  \sum_{i=1}^r a_i \overline{D}_i + \sum_{j=1}^{r'} \delta_j \overline{A}_j \right)
- \avol\left(  \sum_{i=1}^r a_i \overline{D}_i \right)
\right|,
\]
then
\begin{multline*}
\Delta \leq 
\left| 
\avol\left(  \sum_{i=1}^r a_i \overline{\DDD}_i + \sum_{j=1}^{r'} \delta_j \overline{\AAA}_j \right)
- \avol\left(  \sum_{i=1}^r a_i \overline{\DDD}_i \right)
\right| \\
\hskip-10em
+ \left| 
\avol\left(  \sum_{i=1}^r a_i \overline{\DDD}_i  \right)
- \avol\left(  \sum_{i=1}^r a_i \overline{D}_i \right)
\right| \\
+ \left| 
\avol\left(  \sum_{i=1}^r a_i \overline{D}_i + \sum_{j=1}^{r'} \delta_j \overline{A}_j \right)
- \avol\left(  \sum_{i=1}^r a_i \overline{\DDD}_i + \sum_{j=1}^{r'} \delta_j \overline{\AAA}_j\right)
\right|.
\end{multline*}
Thus, by using \eqref{eqn:thm:cont:volume:01},
 for $(a_1, \ldots, a_r) \in B$ and $\delta_1, \ldots, \delta_{r'} \in \RR$ with
\[
\vert \delta_1 \vert + \cdots + \vert \delta_{r'} \vert \leq \delta,
\]
we have
\addtocounter{Claim}{1}
\begin{equation}
\label{eqn:thm:cont:volume:02}
\Delta \leq \epsilon/6 + \epsilon/6 + \epsilon/6 = \epsilon/2.
\end{equation}
On the other hand, by Proposition~\ref{prop:vol:comp:C:0},
{\allowdisplaybreaks
\begin{multline*}
\left\vert \avol\left(\sum_{i=1}^r a_i \overline{D}_i + \sum_{j=1}^{r'} \delta_j \overline{A}_j + 
\left(0, \sum_{l=1}^{s} \varphi_{P_l} [P_l] + \varphi_{\infty} [\infty] \right) \right) \right. \\
\qquad\qquad\qquad\qquad\qquad\qquad\qquad \left. -
\avol\left(\sum_{i=1}^r a_i \overline{D}_i + \sum_{j=1}^{r'} \delta_j \overline{A}_j
\right) \right\vert \\
\leq
\frac{(d + 1) [K : \QQ] \vol\left(X, \sum_{i=1}^{r} a_iD_i +
\sum_{j=1}^{r'} \delta_j A_j\right)}{2} \left( \sum_{l=1}^s \Vert \varphi_{P_i} \Vert_{\sup} + \Vert \varphi_{\infty} \Vert_{\sup} \right).
\end{multline*}}
Here we set
\[
C' = \max \left\{\left. \vol\left(X, \sum_{i=1}^{r} a_iD_i +
\sum_{j=1}^{r'} \delta_j A_j\right) \ \right|\  (a_1, \ldots, a_r) \in B,\ \sum_{j=1}^{r'} \vert \delta_j \vert \leq \delta \right\}
\]
and choose a positive number $\delta'$ such that
\[
(d + 1) [K : \QQ] C' \delta' \leq \epsilon.
\]
Then
\begin{multline*}
\left\vert \avol\left(\sum_{i=1}^r a_i \overline{D}_i + \sum_{j=1}^{r'} \delta_j \overline{A}_j + 
\left(0, \sum_{l=1}^{s} \varphi_{P_l} [P_l] + \varphi_{\infty} [\infty] \right) \right) \right. \\
\left. -
\avol\left(\sum_{i=1}^r a_i \overline{D}_i + \sum_{j=1}^{r'} \delta_j \overline{A}_j
\right) \right\vert \leq \epsilon/2
\end{multline*}
for all $a_1, \ldots, a_r, \delta_1, \ldots, \delta_{r'} \in \RR$, 
$\varphi_{P_1} \in C^0(X_{P_1}^{\an}), \ldots, \varphi_{P_s} \in C^0(X_{P_s}^{\an})$
and $\varphi_{\infty} \in C^0_{F_{\infty}}(X(\CC))$
with $(a_1, \ldots, a_r) \in B$, $\sum_{j=1}^{r'} \vert \delta_j \vert \leq \delta$
and $\sum_{l=1}^s \Vert \varphi_{P_l} \Vert_{\sup} + \Vert \varphi_{\infty} \Vert_{\sup}
\leq \delta'$.
Thus, by the above estimate together with \eqref{eqn:thm:cont:volume:02}, 
we have the assertion.
\end{proof}

\begin{Theorem}
\label{thm:cont:chi:volume}
The $\hat{\chi}$-volume function $\avol_{\chi} : \aDiv_{C^0}^{\ad}(X)_{\RR} \to \RR$ is continuous in the following sense:
Let $\overline{D}_1, \ldots, \overline{D}_r$, $\overline{A}_1, \ldots, \overline{A}_{r'}$ be 
adelic arithmetic $\RR$-Cartier divisors of $C^0$-type on $X$.
Let $\{ P_1, \ldots, P_s \}$ be a finite subset of $M_K$.
For $a_1, \ldots, a_r \in \RR$ and $\epsilon \in \RR_{>0}$, 
there are positive numbers $\delta$ and $\delta'$ such that
\[
\left\vert \avol_{\chi}\left(\sum_{i=1}^r a_i \overline{D}_i + \sum_{j=1}^{r'} \delta_j \overline{A}_j + 
\left(0, \sum_{l=1}^{s} \varphi_{P_l} [P_l] + \varphi_{\infty} [\infty] \right) \right) -
\avol_{\chi}\left(\sum_{i=1}^r a_i \overline{D}_i \right) \right\vert
\leq \epsilon
\]
holds for all $\delta_1, \ldots, \delta_{r'} \in \RR$, 
$\varphi_{P_1} \in C^0(X_{P_1}^{\an}), \ldots, \varphi_{P_s} \in C^0(X_{P_s}^{\an})$
and $\varphi_{\infty} \in C^0_{F_{\infty}}(X(\CC))$
with $\sum_{j=1}^{r'} \vert \delta_j \vert \leq \delta$
and $\sum_{l=1}^s \Vert \varphi_{P_l} \Vert_{\sup} + \Vert \varphi_{\infty} \Vert_{\sup}
\leq \delta'$.
\end{Theorem}

\begin{proof}
The proof is almost same as the proof of Theorem~\ref{thm:cont:volume}
in the case where $B = \{ (a_1, \ldots, a_r) \}$.
We use the same notation as in the proof of Theorem~\ref{thm:cont:volume}.
Then, by \cite[Corollary~3.4.4]{Ikoma},
there is a positive number $\delta$ such that $\delta \leq 1$ and
\[
\left| 
\avol_{\chi}\left(  \sum_{i=1}^r a_i \overline{\DDD}_i + \sum_{j=1}^{r'} \delta_j \overline{\AAA}_j \right)
- \avol_{\chi}\left(  \sum_{i=1}^r a_i \overline{\DDD}_i \right)
\right| \leq \epsilon/6
\]
for $\delta_1, \ldots, \delta_{r'} \in \RR$
with $\vert \delta_1 \vert + \cdots + \vert \delta_{r'} \vert \leq \delta$.
Thus, by virtue of Proposition~\ref{prop:vol:comp:C:0},
we can show the similar estimate
\[
\left| 
\avol_{\chi}\left(  \sum_{i=1}^r a_i \overline{D}_i + \sum_{j=1}^{r'} \delta_j \overline{A}_j \right)
- \avol_{\chi}\left(  \sum_{i=1}^r a_i \overline{D}_i \right)
\right| \leq \epsilon/2
\]
as \eqref{eqn:thm:cont:volume:02}.
The remaining part is exactly same as the proof of Theorem~\ref{thm:cont:volume}
by using Proposition~\ref{prop:vol:comp:C:0}.
\end{proof}

\ifmonog\section{Applications}\fi
\ifpaper\subsection{Applications}\fi
\label{subsec:app:cont:vol}

Here we would like to give two applications of the continuity of the volume function, that is,
the log-concavity of $\avol$ and the generalized Hodge index theorem for adelic arithmetic divisors.

\begin{Theorem}
Let $\overline{D}_1$ and $\overline{D}_2$ be pseudo-effective 
adelic arithmetic $\RR$-Cartier divisors of
$C^0$-type on $X$.
Then
\[
\avol(\overline{D}_1 + \overline{D}_2)^{1/(d + 1)}
\geq \avol(\overline{D}_1)^{1/(d + 1)} +
\avol(\overline{D}_2)^{1/(d + 1)}.
\]
\end{Theorem}

\begin{proof}
Let $U$ be a non-empty open set of $\Spec(O_K)$ such that
$\overline{D}_1$ and $\overline{D}_2$ have defining models $\DDD_{1,U}$ and
$\DDD_{2,U}$ over $U$. We set $T = M_K \setminus U$.
For a positive number $\epsilon$,
by Theorem~\ref{thm:approx:adelic:arith},
there are a normal model of $\XXX$ of $X$ and arithmetic $\RR$-Cartier divisors 
$\overline{\DDD}_1$ and $\overline{\DDD}_2$ of $C^0$-type on $\XXX$ such that
\[
\overline{D}_1 \leq
\overline{\DDD}_1 \leq \overline{D}_1 + \left(0, \sum_{P \in T} \epsilon [P], 0 \right)
\quad\text{and}\quad
\overline{D}_2 \leq
\overline{\DDD}_2 \leq \overline{D}_2 + \left(0, \sum_{P \in T} \epsilon [P], 0 \right).
\]
As $\overline{\DDD}_1$ and $\overline{\DDD}_2$ are pseudo-effective,
by \cite[Theorem~B]{YuanVol} or \cite[Theorem~5.2.2]{MoArZariski}, we have
\[
\avol(\overline{\DDD}_1 + \overline{\DDD}_2)^{1/(d + 1)}
\geq \avol(\overline{\DDD}_1)^{1/(d + 1)} +
\avol(\overline{\DDD}_2)^{1/(d + 1)},
\]
and hence
\[
\avol\left(\overline{D}_1 + \overline{D}_2 + 2 \epsilon 
\left(0, \sum_{P \in T} [P], 0 \right)\right)^{1/(d + 1)}
\geq \avol(\overline{D}_1)^{1/(d + 1)} +
\avol(\overline{D}_2)^{1/(d + 1)}.
\]
Thus the assertion follows from the continuity of $\avol$ on
$\aDiv_{C^0}^{\ad}(X)_{\RR}$ (cf. Theorem~\ref{thm:cont:volume}).
\end{proof}

\begin{Theorem}[Generalized Hodge index theorem for adelic arithmetic divisors]
\label{thm:G:H:I:T:adelic:arith}
Let $\overline{D} = (D, g)$ be a relatively nef adelic arithmetic $\RR$-Cartier divisor of $C^0$-type on $X$.
Then $\adeg(\overline{D}^{d+1}) = \avol_{\chi}(\overline{D})$.
In particular, $\adeg(\overline{D}^{d+1}) \leq \avol(\overline{D})$.
Moreover, if $\overline{D}$ is nef, then
$\adeg(\overline{D}^{d+1}) = \avol(\overline{D})$.
\end{Theorem}

\begin{proof}
Let $\XXX$ be a normal model of $X$ over $\Spec(O_K)$ and
$\overline{\DDD} = (\DDD, g_{\infty})$ an arithmetic $\RR$-Cartier divisor of $C^0$-type on $\XXX$.
First let us see the following claim:

\begin{Claim}
If $\overline{\DDD}$ is relatively nef, then
$\adeg(\overline{\DDD}^{d+1}) = \avol_{\chi}(\overline{\DDD})$.
\end{Claim}

\begin{proof}
We divide the proof into four steps:

{\bf Step 1} (the case where $\overline{\DDD}$ is an arithmetic $\QQ$-Cartier divisor of $C^{\infty}$-type,
$\DDD$ is ample on $\XXX$ and
$c_1(\overline{\DDD})$ is a positive form) :
This is a classic case. For example, it follows from the arithmetic Riemann-Roch theorem
due to Gillet-Soul\'{e}
(cf. \cite{GSRR}).

{\bf Step 2} (the case where
$\overline{\DDD}$ is of $C^{\infty}$-type, $\DDD$ is relatively nef,
$c_1(\overline{\DDD})$ is a semi-positive form) :
As any arithmetic Cartier divisor of $C^{\infty}$-type can be written by
a difference of ample arithmetic Cartier divisors of $C^{\infty}$-type,
we can find ample
arithmetic Cartier divisors $\overline{\AAA}_1, \ldots, \overline{\AAA}_l$
of $C^{\infty}$-type and real numbers $a_1, \ldots, a_l$ such that
\[
\overline{\DDD} = a_1 \overline{\AAA}_1 + \cdots + a_l \overline{\AAA}_l.
\]
Then, for any rational numbers $b_1, \ldots, b_l$ with $a_i < b_i$ for all $i$,
$b_1 \AAA_1 + \cdots + b_l \AAA_l$ is ample and
$c_1(b_1 \overline{\AAA}_1 + \cdots + b_l \overline{\AAA}_l)$ is positive because
\[
b_1 \overline{\AAA}_1 + \cdots + b_l \overline{\AAA}_l =
\overline{\DDD} + (b_1 -a_1) \overline{\AAA}_1 + \cdots + (b_l-a_l) \overline{\AAA}_l.
\]
Thus, by Step~1,
\[
\adeg ((b_1 \overline{\AAA}_1 + \cdots + b_l \overline{\AAA}_l)^{d+1}) = 
\avol_{\chi}(b_1 \overline{\AAA}_1 + \cdots + b_l \overline{\AAA}_l).
\]
Therefore,
the assertion follows from Theorem~\ref{thm:cont:chi:volume}.

{\bf Step 3} (the case where $\DDD \cap X$ is ample on $X$) :
Let $h$ be an $F_{\infty}$-invariant $\DDD$-Green function of $C^{\infty}$-type on $X(\CC)$  such that
$c_1(\DDD, h)$ is a positive form.
Then there is a continuous function $\phi$ on $X(\CC)$ such that
$g_{\infty} = h + \phi$, and hence
$c_1(\DDD, h) + dd^c([\phi]) \geq 0$.
Thus, by \cite[Lemma~4.2]{MoArZariski}, 
there is a sequence $\{ \phi_n \}_{n=1}^{\infty}$ 
of $F_{\infty}$-invariant $C^{\infty}$-functions on $X(\CC)$ with the following properties:
\begin{enumerate}
\renewcommand{\labelenumi}{(\alph{enumi})}
\item
$\lim_{n\to\infty} \Vert \phi_n - \phi \Vert_{\sup} = 0$.

\item
If we set $\overline{\EEE}_n = (\DDD, h + \phi_n)$, then
$c_1(\overline{\EEE}_n)$ is a semipositive form.
\end{enumerate}
Then, by Step~2, 
$\adeg(\overline{\EEE}_n^{d+1}) =
\avol_{\chi}(\overline{\EEE}_n)$ for all $n$.
As $\overline{\EEE}_n = \overline{\DDD} + (0, \phi_n - \phi)$, by Theorem~\ref{thm:cont:chi:volume},
\[
\lim_{n\to\infty} \acvol(\overline{\EEE}_n) = \acvol(\overline{\DDD}).
\]
Moreover, by using \cite[Lemma~1.2.1]{MoD},
\[
\lim_{n\to\infty} \adeg(\overline{\EEE}_n^{d+1}) = \adeg(\overline{\DDD}^{d+1}),
\]
as required.

{\bf Step 4} (general case) :
Finally we prove the assertion of the claim.
Let $\overline{\AAA}$ be an ample arithmetic Cartier divisor 
of $C^{\infty}$-type on $\XXX$.
Then,  
since $\DDD + \epsilon \AAA$ is ample on $X$ for any positive number $\epsilon$,
we have $\adeg((\overline{\DDD} + \epsilon\overline{\AAA})^{d+1}) =
\avol_{\chi}(\overline{\DDD} + \epsilon\overline{\AAA})$ by Step~3.
Thus, the assertion follows from Theorem~\ref{thm:cont:chi:volume}.
\end{proof}

We assume that $\overline{D}$ is relatively nef.
Let us choose a non-empty open set $U$ of $\Spec(O_K)$, 
a normal model $\XXX_U$ over $U$ and
a relatively nef $\RR$-Cartier divisor $\DDD_U$ on $\XXX_U$ such that
$\DDD_U \cap X = D$ and $g_P$ is the Green function arising from  $(\XXX_U, \DDD_U)$ for all $P \in U \cap M_K$.
Moreover, by Proposition~\ref{prop:rel:nef:pseudo:effective},
there is a sequence $\{ (\XXX_n, \DDD_n) \}_{n=1}^{\infty}$ with the following properties:
\begin{enumerate}
\renewcommand{\labelenumi}{(\arabic{enumi})}
\item
$\XXX_n$ is a normal model of $X$ over $\Spec(O_K)$
such that $\rest{\XXX_n}{U} = \XXX_U$.

\item
$\DDD_n$ is relatively nef $\RR$-Cartier divisor on $\XXX_n$ and $\rest{\DDD_n}{U} = \DDD_U$.

\item
$\overline{D} \leq (\DDD_n, g_{\infty})^{\ad}$.

\item
If we set $\phi_{n, P} = g_{((\XXX_n)_{(P)},\,  (\DDD_n)_{(P)})} - g_P$ for $P \in M_K \setminus U$, 
then
\[
\lim_{n\to\infty} \Vert \phi_{n, P} \Vert_{\sup} = 0.
\]
\end{enumerate}
As $(\DDD_n, g_{\infty})^a = \overline{D} + \left(0, \sum_{P \in M_K \setminus U} \phi_{n,P}[P] \right)$,
by Theorem~\ref{thm:cont:volume} and Theorem~\ref{thm:cont:chi:volume}, 
\addtocounter{Claim}{1}
\begin{equation}
\label{eqn:thm:G:H:I:T:adelic:arith:01}
\lim_{n\to\infty} \avol((\DDD_n, g_{\infty})) = \avol(\overline{D})
\quad\text{and}\quad
\lim_{n\to\infty} \avol_{\chi}((\DDD_n, g_{\infty})) = \avol_{\chi}(\overline{D}).
\end{equation}
Here let us see 
\addtocounter{Claim}{1}
\begin{equation}
\label{eqn:thm:G:H:I:T:adelic:arith:02}
\lim_{n\to\infty} \adeg ((\DDD_n, g_{\infty})^{d+1}) = \adeg(\overline{D}^{d+1}).
\end{equation}
Indeed, we set 
\[
\psi_P = g_P - g_{((\XXX_1)_{(P)},\,  (\DDD_1)_{(P)})}
\quad\text{and}\quad
\psi_{n, P} = g_{((\XXX_n)_{(P)},\,  (\DDD_n)_{(P)})} - g_{((\XXX_1)_{(P)},\,  (\DDD_1)_{(P)})}
\]
for $P \in M_K \setminus U$.
Note that $\phi_{n, P} = \psi_P - \psi_{n,P}$.
Then, by using Lemma~\ref{lem:mult:linear:diff:formula}, we can see
\begin{multline*}
\adeg((\DDD_n, g_{\infty})^{d+1}) =
\adeg((\DDD_1, g_{\infty})^{d+1}) \\
+ \sum_{i=1}^{d+1} \sum_{P \in M_K \setminus U} \log \#(O_K/P)
\adeg_P((\DDD_n)_{(P)}^{i-1} \cdot (\DDD_1)_{(P)}^{d+1-i} ; \psi_{n,P}).
\end{multline*}
Thus, by virtue of Proposition~\ref{propdef:adelic:intersection},
\begin{multline*}
\lim_{n\to\infty} \adeg((\DDD_n, g_{\infty})^{d+1})
=
\adeg((\DDD_1, g_{\infty})^{d+1}) \\
+ \sum_{i=1}^{d+1} \sum_{P \in M_K \setminus U} \log \#(O_K/P)
\adeg_P((D, g_P)^{i-1} \cdot (\DDD_1)_{(P)}^{d+1-i} ; \psi_{P})
= \adeg(\overline{D}^{d+1}),
\end{multline*}
as desired.

By \eqref{eqn:thm:G:H:I:T:adelic:arith:01} and \eqref{eqn:thm:G:H:I:T:adelic:arith:02}
together with the above claim, 
we have the first assertion.
If $\overline{D}$ is nef, then $(\DDD_n, g_{\infty})$ is also nef by the property (3) and
Lemma~\ref{lem:global:degree:comp}, and
hence the second assertion follows from \eqref{eqn:thm:G:H:I:T:adelic:arith:01} and \eqref{eqn:thm:G:H:I:T:adelic:arith:02}
by using \cite[Proposition~6.4.2]{MoArZariski} and 
\cite[Proposition~2.1.1]{MoD}.
\end{proof}

\ifmonog\chapter[Zariski decompositions of adelic arithmetic divisors]{Zariski decompositions of adelic arithmetic divisors on arithmetic surfaces}\fi
\ifpaper\section{Zariski decompositions of adelic arithmetic divisors on arithmetic surfaces}\fi

Let $\XXX$ be a regular projective arithmetic surface and
let $\overline{\DDD}$ be an arithmetic $\RR$-Cartier divisor of $C^0$-type on $\XXX$.
The set of all nef arithmetic $\RR$-Cartier divisors $\overline{\LLL}$ 
of $C^0$-type on $\XXX$ with
$\overline{\LLL} \leq \overline{\DDD}$ is denoted
by $\Upsilon(\overline{\DDD})$.
In \cite[Theorem~9.2.1]{MoArZariski}, it is shown that if
$\Upsilon(\overline{D}) \not= \emptyset$, then
$\Upsilon(\overline{D})$ has the greatest element $\overline{\QQQ}$, that is,
$\overline{\QQQ} \in \Upsilon(\overline{D})$ and
$\overline{\LLL} \leq \overline{\QQQ}$ for all
$\overline{\LLL} \in \Upsilon(\overline{D})$.
If we set $\overline{\NNN} := \overline{\DDD} - \overline{\QQQ}$,
then $\overline{\DDD} = \overline{\QQQ} + \overline{\NNN}$ yields the
Zariski decomposition of $\overline{\DDD}$.
For example, we can see that the natural map
\[
\aH(\XXX, n \overline{\QQQ}) \to \aH(\XXX, n \overline{\DDD})
\]
is bijective for every $n \in  \ZZ_{> 0}$.
In particular, $\avol(\overline{\QQQ}) = \avol(\overline{\DDD})$, so that it gives rise to
a refinement of
Fujita's approximation theorem for arithmetic divisors.
In this 
\ifmonog chapter, \fi
\ifpaper section, \fi
we consider a generalization of the above result to
an adelic arithmetic $\RR$-Cartier divisor.

\ifmonog\section[Local Zariski decompositions of adelic divisors]{Local Zariski decompositions of adelic divisors on algebraic curves}\fi
\ifpaper\subsection{Local Zariski decompositions of adelic divisors on algebraic curves}\fi
\label{subsec:Zariski:decomp:adelic:divisor}

Let $k$ be a field and $v$ a non-trivial discrete valuation of $k$.
We assume that $k^{\circ}$ is excellent.
Let $\varpi$ be a uniformizing parameter of $k^{\circ}$.
Let $X$ be a projective, smooth and geometrically integral
curve over $k$. 
Let $k_v$ be the completion of $k$ with respect to $v$ and
$X_v := X \times_{\Spec(k)} \Spec(k_v)$.
The purpose of this 
\ifmonog section \fi
\ifpaper subsection \fi
is to prove the following theorem:

\begin{Theorem}[Local Zariski decomposition]
\label{thm:vertical:Zariski:decomp}
Let $\overline{D} = (D, g)$ be an adelic $\RR$-Cartier divisor on $X$ and
let $Q$ be an $\RR$-Cartier divisor on $X$ with $Q \leq D$.
Here we set
\[
\Sigma(\overline{D}; Q) := \left\{ \overline{L} \ \left| \ 
\begin{array}{l}
\text{$\overline{L}$ is a relatively nef adelic $\RR$-Cartier divisor on $X$} \\
\text{such that $L \leq Q$ and
$\overline{L} \leq \overline{D}$}
\end{array}
\right\}\right..
\]%
\index{\AdelDivSymbol}{0S:Sigma(overline{D};Q)@$\Sigma(\overline{D}; Q)$}%
We assume that $\deg(Q) \geq 0$. Then
there exists a $Q$-Green function $q$ of $(C^0 \cap \Tpsh)$-type on $X_v^{\an}$
such that $\overline{Q} := (Q, q)$ gives rise to the greatest element of $\Sigma(\overline{D}; Q)$,
that is, $\overline{Q}\in \Sigma(\overline{D}; Q)$ and $\overline{L} \leq \overline{Q}$ for all
$\overline{L} \in \Sigma(\overline{D}; Q)$. 
Moreover, we have the following:
\begin{enumerate}
\renewcommand{\labelenumi}{(\arabic{enumi})}
\item
If $\overline{D}$ is given by an $\RR$-Cartier divisor $\DDD$ on a regular model $\XXX$ of $X$
over $\Spec(k^{\circ})$,
then $\overline{Q}$ is given by a relatively nef
$\RR$-Cartier divisor $\QQQ$ on $\XXX$. 

\item
If $Q = D$, then $\adeg_v(\overline{Q};g-q) = 0$
\rom{(}cf. Proposition-Definition~\rom{\ref{propdef:adelic:intersection}}\rom{)}.
\end{enumerate}
\end{Theorem}

Before starting the proof of Theorem~\ref{thm:vertical:Zariski:decomp},
we need to prepare two lemmas.

\begin{Lemma}
\label{lem:for:vertical:Zariski:01}
Let $\XXX$ be a regular model of $X$ over $\Spec(k^{\circ})$.
Then we have the following:
\begin{enumerate}
\renewcommand{\labelenumi}{(\arabic{enumi})}
\item
Let $\pi : \XXX' \to \XXX$ be a birational morphism of regular models of $X$ over $\Spec(k^{\circ})$, 
and let $\LLL'$ be a relatively nef $\RR$-Cartier divisor on $\XXX'$.
Then $\pi_*(\LLL')$ is relatively nef and 
$\pi^*(\pi_*(\LLL')) - \LLL'$ is effective.

\item
If $\LLL_1, \ldots, \LLL_l$ are relatively nef $\RR$-Cartier divisors on $\XXX$, then
\[
\max \{ \LLL_1, \ldots, \LLL_l \}
\]
is also relatively nef
\rom{(}for the definition of $\max \{ \LLL_1, \ldots, \LLL_l \}$, see
Conventions and terminology~\rom{\ref{CT:max:min:divisor}}\rom{)}.

\item
Let $\DDD$ be an $\RR$-Cartier divisor on $\XXX$ and $g = g_{(\XXX,\, \DDD)}$.
Then $g$ is of $(C^0 \cap \Tpsh)$-type if and only $\DDD$ is relatively nef.
\end{enumerate}
\end{Lemma}

\begin{proof}
(1) Let $C$ be an irreducible component of the central fiber $\XXX_{\circ}$
of $\XXX \to \Spec(k^{\circ})$.
Then 
\[
0 \leq (\LLL' \cdot \pi^*(C)) = (\pi_*(\LLL') \cdot C).
\]
Thus $\pi_*(\LLL')$ is relatively nef.

Let us consider the second assertion.
By \cite[Theorem~9.2.2]{Liu},
$\pi$ can be obtained by a succession of blowing-ups at closed points.
We prove it by induction on the number of blowing-ups.
First we consider the case where $\pi$ is a blowing-up at a closed point.
Let $C$ be the exceptional curve of $\pi$.
Then 
\[
\pi^*(\pi_*(\LLL')) - \LLL' = a C
\]
for some $a \in \RR$.
As
\[
((\pi^*(\pi_*(\LLL')) - \LLL') \cdot C) = -(\LLL' \cdot C) \leq 0
\quad\text{and}\quad
(C \cdot C) < 0,
\]
we have $a \geq 0$, as required.
In general, we decompose $\pi$ into two birational morphisms
$\pi_1 : \XXX' \to \XXX_1$ and $\pi_2 : \XXX_1 \to \XXX$ of
regular models of $X$, that is, $\pi = \pi_2 \circ \pi_1$.
Note that $(\pi_1)_*(\LLL')$ is relatively nef by the previous observation.
Thus, by the induction hypothesis,
\[
\pi_1^* (\pi_1)_*(\LLL') - \LLL'
\quad\text{and}\quad
\pi_2^*(\pi_2)_*((\pi_1)_*(\LLL')) - (\pi_1)_*(\LLL')
\]
are effective, so that
\[
\pi^*(\pi_*(\LLL')) - \pi_1^* (\pi_1)_*(\LLL') =
\pi_1^*\left( \pi_2^*(\pi_2)_*((\pi_1)_*(\LLL')) - (\pi_1)_*(\LLL') \right)
\]
is also effective. Therefore, as
\[
\pi^*(\pi_*(\LLL')) - \LLL' = 
\left( \pi^*(\pi_*(\LLL')) - \pi_1^* (\pi_1)_*(\LLL') \right)
+ \left( \pi_1^* (\pi_1)_*(\LLL') - \LLL' \right),
\]
we have the assertion.

\medskip
(2) We set $\LLL'_i := \max \{ \LLL_1, \ldots, \LLL_l \} - \LLL_i$ for each $i$.
Let $C$ be an irreducible component of $\XXX_{\circ}$. Then there is $i$ such that
$C \not\subseteq \Supp_{\RR}(\LLL'_i)$.
As $\LLL'_i$ is effective, we have $\deg(\rest{\LLL'_i}{C}) \geq 0$, so that
\[
\deg(\rest{\max \{ \LLL_1, \ldots, \LLL_l \}}{C}) = \deg(\rest{\LLL_i}{C}) + \deg(\rest{\LLL'_i}{C}) \geq 0.
\]

\medskip
(3) This is a special case of Proposition~\ref{prop:PSH:imply:nef}.
\end{proof}

\begin{Lemma}
\label{lem:vertical:Zariski:decomp}
Let $\XXX$ be a regular model of $X$ and $\DDD$ an $\RR$-Cartier divisor on $\XXX$.
Let $Q$ be an $\RR$-Cartier divisor on $X$ with $Q \leq D := \DDD \cap X$.
Here we set
\[
\Sigma_{\XXX}(\DDD; Q) := \left\{ \LLL \ \left| \ 
\begin{array}{l}
\text{$\LLL$ is a relatively nef $\RR$-Cartier divisor on $\XXX$} \\
\text{such that $\LLL \cap X \leq Q$ and
$\LLL \leq \DDD$}
\end{array}
\right\}\right..
\]%
\index{\AdelDivSymbol}{0S:Sigma_{XXX}(DDD;Q)@$\Sigma_{\XXX}(\DDD; Q)$}%
\begin{enumerate}
\renewcommand{\labelenumi}{(\arabic{enumi})}
\item
We assume that $\deg(Q) \geq 0$. Then
there is a relatively nef $\RR$-Cartier divisor $\QQQ$ on $\XXX$ such that
$\QQQ \cap X = Q$ and
$\QQQ$ gives rise to
the greatest element of $\Sigma_{\XXX}(\DDD; Q)$,
that is, $\QQQ \in \Sigma_{\XXX}(\DDD; Q)$ and $\LLL \leq \QQQ$ for all
$\LLL \in \Sigma_{\XXX}(\DDD; Q)$. 
Moreover, if $Q = D$, then $(\QQQ \cdot \DDD - \QQQ) = 0$, that is,
$\adeg_v(\QQQ^{\ad} ; g_{(\XXX,\ \DDD - \QQQ)}) = 0$.

\item
Let $\pi : \XXX' \to \XXX$ be a birational morphism of regular models of $X$.
If $\QQQ$ is the greatest element of $\Sigma_{\XXX}(\DDD; Q)$, then
$\pi^*(\QQQ)$ is also the greatest element of
\[
\Sigma_{\XXX'}(\pi^*(\DDD); Q).
\]
\end{enumerate}
\end{Lemma}

\begin{proof}
(1)
Let us begin with following claim:

\begin{Claim}
\label{claim:lem:vertical:Zariski:decomp:01}
\begin{enumerate}
\renewcommand{\labelenumi}{(\roman{enumi})}
\item
There is a relatively nef\ \  $\RR$-Cartier divisor $\PPP_0$ on $\XXX$ with
$\PPP_0 \cap X = Q$.

\item
There is $\PPP \in \Sigma_{\XXX}(\DDD; Q)$ with $\PPP \cap X = Q$.
\end{enumerate}
\end{Claim}

\begin{proof}
(i) First we assume that $\deg(Q) = 0$. Let $\PPP'$ be the closure of $Q$ in $\XXX$.
Let $C_1, \ldots, C_r$ be irreducible components of $\XXX_{\circ}$.
As $(\PPP' \cdot \XXX_{\circ}) = 0$, by Zariski's lemma,
we can find $a_1, \ldots, a_r \in \RR$ such that
\[
\left( \sum_{i=1}^r a_i C_i \cdot C_j\right) = (\PPP' \cdot C_j)
\]
for all $j=1, \ldots, r$.
Thus, if we set $\PPP_0 = \PPP' - \sum_{i=1}^r a_i C_i$,
then $\PPP_0$ is relatively nef and $\PPP_0 \cap X = Q$.

Next we assume that $\deg(Q) > 0$.
Then there is $\phi \in \Rat(\XXX)^{\times}_{\QQ}$ such that
$Q + (\phi)_X \geq 0$, where $(\phi)_X$ is the $\QQ$-principal divisor of $\phi$ on $X$.
Let $\PPP'$ be the closure of $Q + (\phi)_X$ in $\XXX$.
As $Q + (\phi)_X$ is effective, $\PPP'$ is relatively nef.
Here we set $\PPP_0 = \PPP' - (\phi)$ on $\XXX$.
Then $\PPP_0$ is relatively nef and 
\[
\PPP_0 \cap X = \PPP' \cap X - (\phi) \cap X =
Q + (\phi)_X - (\phi)_X = Q.
\]

(ii) follows from (i) because
$\PPP_0 - n \XXX_{\circ} \leq \DDD$ for a sufficiently large $n$ and
$\PPP_0 - n \XXX_{\circ}$ is relatively nef.
\end{proof}

For a prime divisor $C$ on $\XXX$ (that is, $C$ is a reduced and irreducible curve on $\XXX$), 
we set
\[
q_C := \sup \left\{ \mult_C(\LLL) \mid \LLL \in \Sigma_{\XXX}(\DDD; Q) \right\},
\]
which exists in $\RR$ because $\mult_C(\LLL) \leq \mult_C(\DDD)$ for all $\LLL \in \Sigma_{\XXX}(\DDD; Q)$.
We fix $\PPP \in \Sigma_{\XXX}(\DDD; Q)$ with $\PPP \cap X = Q$ by using Claim~\ref{claim:lem:vertical:Zariski:decomp:01}.

\begin{Claim}
\label{claim:lem:vertical:Zariski:decomp:03}
There is a sequence $\{ \LLL_n \}_{n=1}^{\infty}$ of $\RR$-Cartier divisors in $\Sigma_{\XXX}(\DDD; Q)$ such that
$\PPP \leq \LLL_n$ for all $n \geq 1$ and $\lim_{n\to\infty} \mult_{C}(\LLL_n) = q_C$ for 
all prime divisors
$C$ in $\Supp_{\RR}(\DDD) \cup \Supp_{\RR}(\PPP)$.
\end{Claim}

\begin{proof}
For each prime divisor
$C$ in $\Supp_{\RR}(\DDD) \cup \Supp_{\RR}(\PPP)$, 
there is a sequence $\{ \LLL_{C, n} \}_{n=1}^{\infty}$ in $\Sigma_{\XXX}(\DDD; Q)$ such that
\[
\lim_{n\to\infty} \mult_{C}(\LLL_{C,n}) = q_{C}.
\]
If we set 
\[
\LLL_n = \max \left( \{ \LLL_{C, n} \}_{C \subseteq
\Supp_{\RR}(\DDD) \cup \Supp_{\RR}(\PPP)} \cup \{ \PPP \} \right),
\]
then
$\PPP \leq \LLL_n$ and $\LLL_n \in \Sigma_{\XXX}(\DDD; Q)$ 
by (2) in Lemma~\ref{lem:for:vertical:Zariski:01}.
Moreover, as
\[
\mult_{C}(\LLL_{C, n}) \leq \mult_{C}(\LLL_n) \leq q_{C},
\]
$\lim_{n\to\infty} \mult_{C}(\LLL_n) = q_{C}$.
\end{proof}

Since $\max \{ \PPP, \LLL \} \in \Sigma_{\XXX}(\DDD; Q)$ for all $\LLL \in \Sigma_{\XXX}(\DDD; Q)$ 
by (2) in Lemma~\ref{lem:for:vertical:Zariski:01},
we have 
\[
\mult_C(\PPP) \leq q_C \leq \mult_C(\DDD).
\]
In particular, if $C \not\subseteq \Supp_{\RR}(\DDD) \cup \Supp_{\RR}(\PPP)$, then $q_C = 0$, so that we can set
$\QQQ := \sum_C q_C C$.
Clearly $\QQQ \cap X = Q$ because $\PPP \leq \QQQ$.
Moreover, $\QQQ \in \Sigma_{\XXX}(\DDD; Q)$ by  Claim~\ref{claim:lem:vertical:Zariski:decomp:03}, and
$\LLL \leq \QQQ$ for all $\LLL \in \Sigma_{\XXX}(\DDD; Q)$ by our construction.

Here we assume that $Q = D$.
Then $\DDD - \QQQ$ is vertical.
Let $C$ be an irreducible component of $\Supp_{\RR}(\DDD - \QQQ)$.
If $(\QQQ \cdot C) > 0$, then $\QQQ + \epsilon C$ is relatively nef for a sufficiently small positive number
$\epsilon$. Moreover, $\QQQ + \epsilon C \leq \DDD$.
This is a contradiction, so that  $(\QQQ \cdot C) = 0$.
Therefore, $(\QQQ \cdot \DDD - \QQQ) = 0$.

\bigskip
(2) Clearly $\pi^*(\QQQ) \in \Sigma_{\XXX'}(\pi^*(\DDD); Q)$.
Let $\LLL' \in \Sigma_{\XXX'}(\pi^*(\DDD); Q)$.
As $\pi_*(\LLL')$ is relatively nef by (1) in Lemma~\ref{lem:for:vertical:Zariski:01},
we have $\pi_*(\LLL') \in \Sigma_{\XXX}(\DDD; Q)$, and hence
$\pi_*(\LLL') \leq \QQQ$. Thus, by using (1) in Lemma~\ref{lem:for:vertical:Zariski:01},
\[
\LLL' \leq \pi^*(\pi_*(\LLL')) \leq \pi^*(\QQQ),
\]
as required.
\end{proof}

\begin{proof}[Proof of Theorem~\rom{\ref{thm:vertical:Zariski:decomp}}]
Let us start the proof of Theorem~\ref{thm:vertical:Zariski:decomp}.
We fix a regular model $\XXX_0$ with the following properties:

\begin{enumerate}
\renewcommand{\labelenumi}{(\alph{enumi})}
\item
If $\overline{D}$ is given by an $\RR$-Cartier divisor $\DDD$ on a regular model $\XXX$,
then $\XXX_0 = \XXX$.

\item
There is a relatively nef  $\RR$-Cartier divisor $\QQQ_0$ on $\XXX_0$ with
$\QQQ_0 \cap X = Q$ (for details, see Claim~\ref{claim:lem:vertical:Zariski:decomp:01}).
\end{enumerate}
By Theorem~\ref{thm:approx:adelic},
for each $n \geq 1$,
we can find a regular model $\XXX_n$ and 
an $\RR$-Cartier divisor $\DDD_n$ on $\XXX_n$ such that
\[
\overline{D} - \frac{1}{n(n+1)} (\XXX_n)_{\circ}^{\ad} \leq \DDD_n^{\ad} \leq  \overline{D} + \frac{1}{n(n+1)} (\XXX_n)_{\circ}^{\ad}.
\]
Replacing $\XXX_n$ by a suitable regular model of $X$ if necessarily,
we may assume that there is a birational morphism $\pi_{n+1} : \XXX_{n+1} \to \XXX_n$
for every $n \geq 0$.
Note that if $\overline{D}$ is given by an $\RR$-Cartier divisor $\DDD$ on $\XXX$,
then $\XXX_n = \XXX$ and $\DDD_n = \DDD$ for all $n \geq 1$.
By using (1) in Lemma~\ref{lem:vertical:Zariski:decomp},
let $\QQQ_n$ be the greatest element
of $\Sigma_{\XXX_n}(\DDD_n; Q)$.
Let us check the following claim:

\frontmatterforspececialclaim
\begin{Claim}
\begin{enumerate}
\renewcommand{\labelenumi}{(\roman{enumi})}
\item The following inequalities
\[
\qquad\qquad
\DDD_{n+1} - \frac{2}{n(n+1)} (\XXX_{n+1})_{\circ} \leq \pi^*_{n+1}(\DDD_n) \leq \DDD_{n+1} + \frac{2}{n(n+1)} (\XXX_{n+1})_{\circ}
\]
hold for all $n \geq 1$.

\item Moreover,
\[
\qquad\qquad
\QQQ_{n+1} - \frac{2}{n(n+1)} (\XXX_{n+1})_{\circ} \leq \pi^*_{n+1}(\QQQ_n) \leq \QQQ_{n+1} + \frac{2}{n(n+1)} (\XXX_{n+1})_{\circ}
\]
hold for all $n \geq 1$.
\end{enumerate}
\end{Claim}
\backmatterforspececialclaim

\begin{proof}
(i) The first inequality follows from the following observation:
\begin{align*}
\pi^*_{n+1}(\DDD_n)^{\ad} & \geq \overline{D} - \frac{1}{n(n+1)} \pi_{n+1}^*((\XXX_n)_{\circ})^{\ad} \\
& \geq
\left( \DDD_{n+1}^{\ad} - \frac{1}{(n+1)(n+2)} (\XXX_{n+1})_{\circ}^{\ad} \right) - \frac{1}{n(n+1)} (\XXX_{n+1})_{\circ}^{\ad} \\
& \geq \left( \DDD_{n+1} - \frac{2}{n(n+1)} (\XXX_{n+1})_{\circ}\right) ^{\ad}.
\end{align*}
The second inequality is similar.

(ii) Note that $\pi_{n+1}^*(\QQQ_{n}) - (2/n(n+2)) (\XXX_{n+1})_{\circ} \in \Sigma_{\XXX_{n+1}}(\DDD_{n+1}; Q)$ because
\[
\pi_{n+1}^*(\DDD_{n}) - (2/n(n+2)) (\XXX_{n+1})_{\circ} \leq \DDD_{n+1}
\]
by (i),
so that $\pi_{n+1}^*(\QQQ_{n}) - (2/n(n+2)) (\XXX_{n+1})_{\circ} \leq \QQQ_{n+1}$.
Similarly 
\[
\QQQ_{n+1} - (2/n(n+2)) (\XXX_{n+1})_{\circ} \in \Sigma_{\XXX_{n+1}}(\pi_{n+1}^*(\DDD_{n}); Q)
\]
by using (i), and hence $\QQQ_{n+1} - (2/n(n+2)) (\XXX_{n+1})_{\circ} \leq \pi_{n+1}^*(\QQQ_n)$
by (2) in Lemma~\ref{lem:vertical:Zariski:decomp}.
\end{proof}

Let $\EEE_n := \QQQ_n - \rho_n^*(\QQQ_0)$,
where $\rho_{n} := \pi_1 \circ \cdots \circ \pi_{n} : \XXX_n \to \XXX_0$.
Then $\EEE_n$ is vertical and
\[
\EEE_{n+1} - \frac{2}{n(n+1)} (\XXX_{n+1})_{\circ} \leq \pi^*_{n+1}(\EEE_n) 
\leq \EEE_{n+1} + \frac{2}{n(n+1)} (\XXX_{n+1})_{\circ}
\]
by (ii) of the previous claim,
so that
\[
\varphi_{n+1} - 4\left( \frac{1}{n} - \frac{1}{n+1} \right)(-\log v(\varpi)) \leq \varphi_n \leq \varphi_{n+1} + 
4\left( \frac{1}{n} - \frac{1}{n+1} \right) (-\log v(\varpi)),
\]
where $\varphi_n := g_{(\XXX_n,\, \EEE_n)}$.
Therefore, if we set 
\[
\varphi'_n = \varphi_n - \frac{4(-\log v(\varpi))}{n}
\quad\text{and}\quad
\varphi''_n = \varphi_n  + \frac{4(-\log v(\varpi))}{n},
\]
then
\[
\varphi'_1 \leq \cdots \leq \varphi'_n \leq \varphi'_{n+1} \leq
\cdots \leq \varphi''_{n+1} \leq \varphi''_n \leq \cdots \leq \varphi''_1,
\]
and hence $\varphi(x) := \lim_{n\to\infty} \varphi_n(x)$ exists for each $x \in X_v^{\an}$.
Moreover, as
\[
\vert \varphi_n(x) - \varphi(x) \vert \leq \varphi''_n(x) - \varphi'_n(x)  \leq (8/n)(-\log v(\varpi)),
\]
the sequence $\{ \varphi_n \}_{n=1}^{\infty}$ converges to $\varphi$ uniformly.
In particular, $\varphi$ is continuous on $X_v^{\an}$.
We set $q := g_{(\XXX_0,\, \QQQ_0)} + \varphi$.
As $\QQQ_n$ is relatively nef, $q$ is a $Q$-Green function of $(\Tpsh \cap C^0)$-type.
Note that in the case where
$\overline{D}$ is given by a relatively nef $\RR$-Cartier divisor $\DDD$ on $\XXX$,
then $q = g_{(\XXX,\, \QQQ_1)}$.

\bigskip
Let us see that $\overline{Q} := (Q, q)$ is the greatest element of
$\Sigma(\overline{D}; Q)$.
As $\{ \varphi_n \}_{n=1}^{\infty}$ converges $\varphi$ uniformly
and
\[
g_{(\XXX_0,\, \QQQ_0)} + \varphi_n = g_{(\XXX_n,\, \QQQ_n)}  \leq g_{(\XXX_n,\, \DDD_n)} \leq g + 
\frac{-2\log v(\varpi)}{n(n+1)},
\]
we can see that $\overline{Q} \leq \overline{D}$,
that is, $\overline{Q} \in \Sigma(\overline{D};Q)$, so that we need to see that
$\overline{L} \leq \overline{Q}$ for all $\overline{L} = (L, g_L) \in \Sigma(\overline{D};Q)$.

First we assume that $\overline{L}$ is given by an $\RR$-Cartier divisor $\LLL$ on a regular model $\YYY$.
By (3) in Lemma~\ref{lem:for:vertical:Zariski:01}, $\LLL$ is relatively nef.
For each $n \geq 1$,
we choose a regular model $\YYY_n$ of $X$ such that
there are birational morphisms $\nu_n : \YYY_n \to \YYY$ and $\mu_n : \YYY_n \to \XXX_n$.
If we set $\FFF_n = (\mu_n)_*(\nu_n^*(\LLL))$, then
$\FFF_n$ is relatively nef by (1) in Lemma~\ref{lem:for:vertical:Zariski:01}.
Moreover, as 
\[
\nu_n^*(\LLL)^{\ad} \leq \overline{D} \leq \left( \mu_n^*(\DDD_n) + \frac{1}{n(n+1)} (\YYY_n)_{\circ}\right)^{\ad}, 
\]
we have
$\FFF_n \leq \DDD_n + (1/n(n+1)) (\XXX_{n})_{\circ}$ by using Proposition~\ref{prop:comp:adelic}, 
and hence 
\[
\FFF_n \leq \QQQ_n + (1/n(n+1)) (\XXX_n)_{\circ}
\]
because $\FFF_n - (1/n(n+1)) (\XXX_n)_{\circ} \in \Sigma_{\XXX_n}(\DDD_n ; Q)$.
Therefore, by (1) in Lemma~\ref{lem:for:vertical:Zariski:01},
\[
\nu_n^*(\LLL) \leq \mu_n^*(\FFF_n) \leq \mu_n^*(\QQQ_n) +  \frac{1}{n(n+1)} (\YYY_n)_{\circ}.
\]
In particular, 
\[
g_{(\YYY,\, \LLL)} \leq g_{(\XXX_n,\, \QQQ_n)} + \frac{-2\log v (\varpi)}{n(n+1)} =
g_{(\XXX_0,\, \QQQ_0)} + \varphi_n + \frac{-2\log v (\varpi)}{n(n+1)}.
\]
Note that $\{ \varphi_n \}_{n=1}^{\infty}$ converges $\varphi$ uniformly, so that
we have $g_{(\YYY,\, \LLL)} \leq q$.

In general, by Proposition~\ref{prop:Green:psh:C0:approx},
there are a sequence $\{ \YYY_l \}_{l=1}^{\infty}$ of regular models of $X$ and
a sequence $\{ \LLL_l \}_{l=1}^{\infty}$ of relatively nef $\RR$-Cartier divisors  such that
$\LLL_l$ is defined on $\YYY_l$, $\LLL_l \cap X = L$,
$(\LLL_l)^{\ad} \leq \overline{L}$ and $g_L = \lim_{l\to\infty} g_{(\YYY_l,\LLL_l)}$ uniformly.
By the previous observation, $g_{(\YYY_l,\LLL_l)} \leq q$ for all $l$, so that
$g_{L} \leq q$

Let us see the additional assertions (1) and (2) in the theorem.
The assertion (1) is obvious by our construction. 
Let us consider (2).
We assume that $Q = D$.
If we set $\theta_n := g_{(\XXX_n,\ \DDD_n)}  - g_{(\XXX_n, \QQQ_n)}$ and $\theta = g - q$, then
$\{ \theta_n \}_{n=1}^{\infty}$ converges to $\theta$ uniformly because
\[
\theta_n  = (g_{(\XXX_n,\ \DDD_n)} - g) +  g - (g_{(\XXX_0, \QQQ_0)} + \varphi_n).
\]
Thus, by Proposition-Definition~\ref{propdef:adelic:intersection},
\[
\lim_{n\to\infty} \adeg_v(\QQQ_n;\theta_n) = \adeg_v(\overline{Q}; \theta).
\]
On the other hand, $\adeg_v(\QQQ_n; \theta_n) = 0$ by Lemma~\ref{lem:vertical:Zariski:decomp},
so that the assertion (2) follows.
\end{proof}

Finally we consider the maximal of two Green functions,
which will be used in the next 
\ifmonog section. \fi
\ifpaper subsection. \fi

\begin{Proposition}
\label{prop:max:Green:adelic}
Let $D_1$ and $D_2$ be $\RR$-Cartier divisors on $X$ and
let $D_3 := \max \{ D_1, D_2 \}$.
For $i=1, 2$,
let $g_i$ be a $D_i$-Green function of $C^0$-type on $X_v^{\an}$.
\begin{enumerate}
\renewcommand{\labelenumi}{(\arabic{enumi})}
\item
$\max \{ g_1, g_2 \}$ is a $D_3$-Green function of $C^0$-type on $X_v^{\an}$.

\item
If $g_1$ and $g_2$ are of $(C^0 \cap \Tpsh)$-type, then
$\max \{ g_1, g_2 \}$ is also of $(C^0 \cap \Tpsh)$-type.
\end{enumerate}
\end{Proposition}

\begin{proof}
(1) Let $\pi : X_v \to X$ be the canonical morphism.
It is easy to see that 
\[
\max \{ \pi^*(D_1), \pi^*(D_2) \} = \pi^*(\max \{ D_1, D_2 \}),
\]
so that the assertion follows from Proposition~\ref{prop:max:Green:adelic:01}.

(2)
For each $n \geq 1$, by Proposition~\ref{prop:Green:psh:C0:approx},
there are a regular model $\XXX'_n$ of $X$ and relatively nef 
$\RR$-Cartier divisors $\DDD'_{1,n}$ and $\DDD'_{2,n}$ on $\XXX'_n$
such that
\[
0 \leq g_1 - g_{(\XXX'_n,\, \DDD'_{1,n})} \leq 1/n
\quad\text{and}\quad
0 \leq g_2 - g_{(\XXX'_n,\, \DDD'_{2,n})} \leq 1/n.
\]
On the other hand, by Theorem~\ref{thm:approx:adelic},
we can find a
regular model $\XXX''_n$ of $X$ and an $\RR$-Cartier divisors $\EEE''_n$ on $\XXX''_n$
such that $\EEE_n'' \cap X = D_3$ and
\[
\max \{ g_1, g_2 \} \leq g_{(\XXX''_n,\, \EEE''_n)} \leq  \max \{ g_1, g_2 \} + 1/n.
\]
We choose birational morphisms
$\nu_n : \XXX_n \to \XXX'_n$ and $\mu_n : \XXX_n \to \XXX''_n$ and we set
\[
\DDD_{1,n} := \nu_n^*(\DDD'_{1,n}), \ 
\DDD_{2,n} := \nu_n^*(\DDD'_{2,n}), \ 
\DDD_{3,n} := \max \{\DDD_{1,n}, \DDD_{2,n} \}
\ \text{and}\ 
\EEE_n := \mu_n^*(\EEE''_n).
\]
As $\DDD_{1,n} \leq \DDD_{3,n}$ and $\DDD_{2,n} \leq \DDD_{3,n}$,
we have 
\[
\max \{ g_{(\XXX_n,\, \DDD_{1,n})} , g_{(\XXX_n,\, \DDD_{2,n})} \}
\leq g_{(\XXX_n,\, \DDD_{3,n})},
\]
so that
\begin{align*}
\max \{ g_1, g_2 \}  - 1/n & =
\max \{ g_1 -1/n ,\ g_2  -1/n \} \\
& \leq \max \{ g_{(\XXX_n,\, \DDD_{1,n})} , g_{(\XXX_n,\, \DDD_{2,n})} \}
\leq g_{(\XXX_n,\, \DDD_{3,n})}.
\end{align*}
Moreover,
\[
g_{(\XXX_n,\, \DDD_{1,n})} \leq g_1 \leq \max \{ g_1, g_2 \} \leq g_{(\XXX_n,\, \EEE_n)}
\]
and
\[
g_{(\XXX_n,\, \DDD_{2,n})} \leq g_2 \leq \max \{ g_1, g_2 \} \leq g_{(\XXX_n,\, \EEE_n)},
\]
and hence, by Proposition~\ref{prop:comp:adelic},
\[
\DDD_{1,n} \leq \EEE_n
\quad\text{and}\quad
\DDD_{2,n} \leq \EEE_n,
\]
so that
$\DDD_{3,n} \leq \EEE_n$.
Therefore,
\[
\max \{ g_1, g_2 \}  - 1/n \leq g_{(\XXX_n,\, \DDD_{3,n})} \leq \max \{ g_1, g_2 \} + 1/n.
\]
Note that $\DDD_{3,n}$ is relatively nef by (2) in Lemma~\ref{lem:for:vertical:Zariski:01}, 
and hence
$\max \{ g_1, g_2 \}$ is of $(C^0 \cap \Tpsh)$-type.
\end{proof}

\ifmonog\section[Proof of Zariski decompositions]{Proof of Zariski decompositions for adelic arithmetic divisors}\fi
\ifpaper\subsection{Proof of Zariski decompositions for adelic arithmetic divisors}\fi
\label{subsec:proof:Zariski:decomp:adelic:arithmetic:divisor}

In this 
\ifmonog section, \fi
\ifpaper subsection, \fi
we give the proof of Zariski decompositions for adelic arithmetic divisors.
Let $X$ be a projective, smooth and geometrically integral
curve over a number field $K$.
Let us begin with decompositions for global adelic divisors.

\begin{Theorem}
\label{thm:vertical:Zariski:decomp:global}
Let $\overline{D} = (D, \{ g_P \}_{P \in M_K})$ be a global adelic $\RR$-Cartier divisor on $X$ 
\rom{(}cf. Definition~\rom{\ref{def:adelic:arithmetic:div}}\rom{)} and
let $Q$ be an $\RR$-Cartier divisor on $X$ with $Q \leq D$.
Here we set
\[
\Sigma\left(\overline{D}; Q\right) := \left\{ \overline{L} = (L, \{ l_P \}_{P \in M_K}) \left|
\begin{array}{l}
\text{$\overline{L}$ is a relatively nef global adelic $\RR$-Cartier} \\
\text{divisor on $X$ such that $L \leq Q$ and
$\overline{L} \leq \overline{D}$}
\end{array}
\right\}\right..
\]%
\index{\AdelDivSymbol}{0S:Sigma(overline{D};Q)@$\Sigma(\overline{D}; Q)$}%
If $\deg(Q) \geq 0$, then
there exists a $Q$-Green function $q_P$ of $(\Tpsh \cap C^0)$-type on $X_P^{\an}$
for each $P \in M_K$
such that $\overline{Q} := (Q, \{ q_P \}_{P \in M_K})$ 
gives rise to the greatest element of $\Sigma(\overline{D}; Q)$,
that is, $\overline{Q}\in \Sigma(\overline{D}; Q)$ and 
$\overline{L} \leq \overline{Q}$ for all
$\overline{L} \in \Sigma(\overline{D}; Q)$. 
Moreover, if there are a non-empty Zariski open set $U$ of\ \  $\Spec(O_K)$, 
a regular  model $\XXX_U$ of $X$ over $U$ and
an $\RR$-Cartier divisor $\DDD_U$ on $\XXX_U$ such that
$g_P$ is the Green function arising from $\DDD_U$ for all $P \in U \cap M_K$, then
there is a relatively nef $\RR$-Cartier divisor $\QQQ_U$ on $\XXX_U$ such that
$q_P$ is given by $\QQQ_U$ for all $P \in U \cap M_K$.
\end{Theorem}

\begin{proof}
Let us choose a non-empty Zariski open set $U$ of $\Spec(O_K)$,
a regular  model $\XXX_U$ of $X$ over $U$ and
an $\RR$-Cartier divisor $\DDD_U$ on $\XXX_U$ such that
$g_P$ is given by $\DDD_U$ for all $P \in U \cap M_K$.
Moreover, we set
\[
U' = \left\{ P \in U \mid 
\text{$\XXX_U \to U$ is smooth over $P$ and $\DDD_U$ is horizontal over $P$} \right\}.
\]
By Theorem~\ref{thm:vertical:Zariski:decomp},
for each $P \in M_K$,
we can find a $Q$-Green function $q_P$ of $(C^0 \cap \Tpsh)$-type on $X_P^{\an}$
such that
$(Q, q_P)$ yields the greatest element of
\[
\left\{ (L, l_P) \ \left| \ 
\begin{array}{l}
\text{$L$ is a nef $\RR$-Cartier divisor on $X$, $l_P$ is an $L$-Green function
} \\
\text{of $(C^0 \cap \Tpsh)$-type on $X_P^{\an}$, $L \leq Q$ and
$(L, l_P) \leq (D, g_P)$}
\end{array}
\right\}\right..
\]
For each $P \in U$, we set
$(\XXX_U)_{(P)} := \XXX_U \times_U \Spec((O_K)_P)$.
Then, by Theorem~\ref{thm:vertical:Zariski:decomp} again,
$q_P$ is obtain by a relatively nef $\RR$-Cartier divisor $\QQQ_{(P)}$ on $(\XXX_U)_{(P)}$.
Note that if $P \in U'$, then
$\QQQ_{(P)}$ is actually given by the Zariski closure of $Q$ in $(\XXX_U)_{(P)}$.
Therefore, we can find a relatively nef $\RR$-Cartier divisor $\QQQ_U$
on $\XXX_U$ such that $\QQQ_U \cap (\XXX_U)_{(P)} = \QQQ_{(P)}$.
Therefore, $\overline{Q} := (Q, \{ q_P \}_{P \in M_K})$ forms a global
adelic $\RR$-Cartier divisor on $X$.
By our construction, it is obvious that
$\overline{Q}$ is the greatest element of $\Sigma\left(\overline{D}; Q\right)$.
Further, the second assertion of the theorem is
also obvious by our construction.
\end{proof}

As a corollary,
we have the relative version of Corollary~\ref{cor:Zariski:decomp:adelic:arithmetic:divisor}.

\begin{Corollary}
\label{cor:Zariski:rel:decomp:adelic:arithmetic:divisor}
Let $\overline{D} = (D, g)$ be an adelic arithmetic $\RR$-Cartier divisor of $C^0$-type on $X$.
Let $\Upsilon_{rel}(\overline{D})$ be the set of all
relatively nef adelic arithmetic $\RR$-Cartier divisors $\overline{L}$ of
$C^0$-type on $X$ with $\overline{L} \leq \overline{D}$.
\index{\AdelDivSymbol}{0U:Upsilon_{rel}(overline{D})@$\Upsilon_{rel}(\overline{D})$}%
If $\deg(D) \geq 0$,
then there is the greatest element $\overline{Q} = (Q, q)$ of
$\Upsilon_{rel}(\overline{D})$, that is,
$\overline{Q} \in \Upsilon_{rel}(\overline{D})$ and
$\overline{L} \leq \overline{Q}$ for all $\overline{L} \in \Upsilon_{rel}(\overline{D})$.
Moreover, we have the following properties:
\begin{enumerate}
\renewcommand{\labelenumi}{(\arabic{enumi})}
\item
$\overline{D} - \overline{Q}$ is vertical, that is, $D = Q$.

\item
For every $a \in \RR_{>0}$, the natural homomorphism
$\aH(X, a \overline{Q}^{\tau}) \to \aH(X, a\overline{D}^{\tau})$ is bijective.
Further, $\Vert \phi \Vert_{aq_{\infty}} = \Vert \phi \Vert_{ag_{\infty}}$ for all
$\phi \in H^0(X(\CC), a D)$.
In particular, 
\[
\hat{\chi}(X, a \overline{Q}) = \hat{\chi}(X, a \overline{D})
\quad\text{and}\quad
\avol_{\chi}(\overline{Q}) = \avol_{\chi}(\overline{D}).
\]

\item
$\overline{Q}$ is perpendicular to $\overline{D} - \overline{Q}$, that is,
$\adeg (\overline{Q} \cdot \overline{D} - \overline{Q}) = 0$.
\end{enumerate}
\end{Corollary}

\begin{proof}
Applying Theorem~\ref{thm:vertical:Zariski:decomp:global} to the case $Q = D$,
we have the greatest element 
\[
\left(D, \sum\nolimits_{P \in M_K} q_P[P]\right)
\]
of
$\Sigma\left(\overline{D}^{\tau}; D\right)$.
Moreover, by using \cite[Theorem~4.6]{MoArZariski},
there is a $D$-Green function $q_{\infty}$ of $(C^0 \cap \Tpsh)$-type on $X(\CC)$
such that $q_{\infty}$ yields 
the greatest element of
\[
\left\{ h_{\infty} \mid \text{$h_{\infty}$ is a $D$-Green function of $(C^0 \cap \Tpsh)$-type 
on $X(\CC)$ and $h_{\infty} \leq g_{\infty}$} \right\}.
\]
Thus $\left(D, \sum\nolimits_{P \in M_K} q_P[P] + q_{\infty}[\infty]\right)$
is our desired adelic arithmetic $\RR$-Cartier divisor.

The property (1) is obvious. 
For (2), we suppose $\phi \in \aH(X, a\overline{D}^{\tau})$,
that is, 
\[
-(1/a) (\phi)^{\ad} \leq \overline{D}^{\tau}
\]
by Proposition~\ref{prop:basic:prop:H:0}.
Note that $-(1/a) (\phi)^{\ad}$ is relatively nef, so that
$-(1/a) (\phi)^{\ad} \leq \overline{Q}^{\tau}$.
Therefore, $\phi \in \aH(X, a\overline{Q}^{\tau})$ by Proposition~\ref{prop:basic:prop:H:0}.
The assertion $\Vert \cdot \Vert_{aq_{\infty}} = \Vert \cdot \Vert_{ag_{\infty}}$ on
$H^0(X(\CC), a D)$ follows from \cite[Lemma~1.3]{MoCharNef}.
Further, (3) is a consequence of (2) in Theorem~\ref{thm:vertical:Zariski:decomp} and
\cite[Lemma~1.3]{MoCharNef}.
\end{proof}

The following theorem is one of the main results of this article.

\begin{Theorem}
\label{thm:Zariski:decomp:adelic:arithmetic:divisor}
Let $\overline{D}$ be an adelic arithmetic $\RR$-Cartier divisor of $C^0$-type on $X$.
Let $R$ be an $\RR$-Cartier divisor on $X$ with $R \leq D$.
Let $\Upsilon(\overline{D};R)$ be the set of all
nef adelic arithmetic $\RR$-Cartier divisors $\overline{L} =(L,l)$ of
$C^0$-type on $X$ with $L \leq R$ and $\overline{L} \leq \overline{D}$.
\index{\AdelDivSymbol}{0U:Upsilon(overline{D};R)@$\Upsilon(\overline{D};R)$}%
If\ \ $\Upsilon(\overline{D};R) \not= \emptyset$,
then there is the greatest element $\overline{Q}$ of
$\Upsilon(\overline{D};R)$, that is,
$\overline{Q} \in \Upsilon(\overline{D};R)$ and
$\overline{L} \leq \overline{Q}$ for all $\overline{L} \in \Upsilon(\overline{D};R)$.
\end{Theorem}

First let us consider two lemmas.

\begin{Lemma}
\label{lem:nef:seq:approx:nef}
Let $\overline{M}$ and
$\overline{Q}$ be adelic arithmetic $\RR$-Cartier divisors of $C^0$-type  on $X$ and
let $\{ \overline{L}_n \}_{n=1}^{\infty}$ be a sequence of 
adelic arithmetic $\RR$-Cartier divisors of $C^0$-type on $X$
with the following properties:
\begin{enumerate}
\renewcommand{\labelenumi}{(\arabic{enumi})}
\item
$\overline{L}_n$ is nef for all $n \geq 1$.

\item
$\overline{M} \leq \overline{L}_n \leq \overline{Q}$ for all $n \geq 1$.

\item
For all closed closed points $x$ of $X$,
$\lim_{n\to\infty} \mult_x(L_n) = \mult_x(Q)$.
\end{enumerate}
Then $\adeg(\srest{\overline{Q}}{x}) \geq 0$ for all closed points $x \in X$.
\end{Lemma}

\begin{proof}
We set $\overline{M} = (M, m)$,
$\overline{Q} = (Q, q)$ and $\overline{L}_n = (L_n, g_n)$.
First we assume that there are a regular  model $\XXX$ of $X$ over $\Spec(O_K)$ and
an $\RR$-Cartier divisor $\QQQ$ on $\XXX$ such that
$\overline{Q} = (\QQQ, q_{\infty})^{\ad}$.
Let us see the following claim:

\begin{Claim}
\begin{enumerate}
\renewcommand{\labelenumi}{(\roman{enumi})}
\item
For all $n \geq 1$, there is an $\RR$-Cartier divisor $\LLL_n$ on $\XXX$ such that
$\LLL_n \cap X = L_n$,
$(\LLL_n, (g_n)_{\infty})^{\ad}$ is nef and
$\overline{L}_n \leq (\LLL_n, (g_n)_{\infty})^{\ad} \leq (\QQQ, q_{\infty})^{\ad}$.

\item
There is an $\RR$-Cartier divisor $\MMM$ on $\XXX$ such that $\MMM \cap X = M$ and
$(\MMM, m_{\infty})^{\ad} \leq \overline{M}$.
\end{enumerate}
\end{Claim}

\begin{proof}
(i)
For each $n \geq 1$, we consider the following set:
\[
\left.\left\{ \overline{D} = \left( D, \{ g_P \}_{P \in M_K}\right) \ \right| \ 
\text{$\overline{D}$ is relatively nef, $D \leq L_n$ and
$\overline{D} \leq \QQQ^{\ad}$}
\right\}.
\]
Then, by using Theorem~\ref{thm:vertical:Zariski:decomp:global},
there is a relatively nef $\RR$-Cartier divisor $\LLL_n$ on $\XXX$ such that
$(\LLL_n)^{\ad}$ gives rise to the greatest
element of the above set. As $\overline{L}_n^{\tau}$ belongs to the above set,
we can see that
$\overline{L}_n \leq (\LLL_n, (g_n)_{\infty})^{\ad}$.
Moreover,  as $L_n = \LLL_n \cap X$,
we have $(\LLL_n, (g_n)_{\infty})^{\ad}$ is nef by Lemma~\ref{lem:global:degree:comp},
so that (i) follows.

(ii)
There are an $\RR$-Cartier divisor $\MMM'$ on $\XXX$ and
a non-empty open set $U$ of $\Spec(O_K)$ such that
$m_P$ is defined by $\MMM'$ for all $P \in U \cap M_K$.
For each $P \in M_K \setminus U$,
let $m'_P$ be the $M$-Green function arising from $\MMM'$.
As $m_P - m'_P$ is a continuous function on $X_P$, there is a constant $\theta_P$ such that
$m_P \geq m'_P + \theta_P$ for all $P \in M_K \setminus U$.
Let $F_P$ be the fiber of $\XXX \to \Spec(O_K)$ over $P$.
If we set 
\[
\overline{\MMM} = \left(\MMM' + \sum\nolimits_{P \in M_K \setminus U} \frac{\theta_P}{2\log\#(O_K/P)} F_P, m_{\infty} 
\right),
\]
then $\overline{M} \geq \overline{\MMM}^{\ad}$,
as required.
\end{proof}

By the above claim, $(\MMM, m_{\infty}) \leq (\LLL_n, (g_n)_{\infty}) \leq (\QQQ, q_{\infty})$ and
$\LLL_n \cap X = L_n$
for all $ n \geq 1$.
Thus
$\lim_{n\to\infty} \mult_C(\LLL_n)$ exists for all prime divisors $C$ on $\XXX$ except
finitely many fiber components, so that
if we choose a subsequence $\{ (\LLL_{n_i}, (g_{n_i})_{\infty}) \}_{i=1}^{\infty}$ of
$\{ (\LLL_{n}, (g_{n})_{\infty}) \}_{n=1}^{\infty}$,
then $\lim_{i\to\infty} \mult_C(\LLL_{n_i})$ exists for all prime divisors $C$ on $\XXX$.
Therefore, by using \cite[Theorem~7.1]{MoArZariski},
there are an $\RR$-Cartier divisor $\LLL$ on $\XXX$ and
a $Q$-Green function $g_{\infty}$ of $\Tpsh_{\RR}$-type on $X(\CC)$
with the following properties:
\begin{enumerate}
\renewcommand{\labelenumi}{(\alph{enumi})}
\item
$\mult_{C}(\LLL) = \lim_{i\to\infty} \mult_{C}(\LLL_{n_i})$ for all
prime divisors $C$ on $\XXX$. In particular, $\LLL \cap X = Q$.

\item
$\adeg(\rest{(\LLL, g_{\infty})}{C}) \geq \limsup_{i\to\infty} 
\adeg(\rest{(\LLL_{n_i}, (g_{n_i})_{\infty})}{C})$
for all
prime divisors $C$ on $\XXX$.

\item
$(\LLL, g_{\infty}) \leq (\QQQ, q_{\infty})$.
\end{enumerate}
Let $x$ be a closed point of $X$ and $\Delta_x$ the closure of $x$ in $\XXX$.
Then, as $\LLL \cap X = \QQQ \cap X$ and $(\LLL, g_{\infty}) \leq (\QQQ, q_{\infty})$,
\begin{multline*}
\qquad\adeg(\srest{\overline{Q}}{x}) = \adeg(\srest{(\QQQ, q_{\infty})}{\Delta_x}) \geq \adeg(\srest{(\LLL, g_{\infty})}{\Delta_x}) \\
\geq \limsup_{i\to\infty} 
\adeg(\srest{(\LLL_{n_i}, (g_{n_i})_{\infty})}{\Delta_x}) \geq 0.\qquad
\end{multline*}

In general,  let $U$ be a non-empty open set of $\Spec(O_K)$ such that
$\overline{Q}$ has a defining model over $U$.
For a positive number $\epsilon$, 
by Theorem~\ref{thm:approx:adelic:arith}, 
there are a regular  model $\XXX_{\epsilon}$ 
over $\Spec(O_K)$ and
an $\RR$-Cartier divisor $\QQQ_{\epsilon}$ on $\XXX_{\epsilon}$
such that
\[
\overline{Q} \leq (\QQQ_{\epsilon}, q_{\infty})^{\ad} \leq \overline{Q} + 
\left(0, \sum\nolimits_{P \in M_K \setminus U} \epsilon [P]\right).
\]
Then, by the previous observation, $\adeg(\rest{(\QQQ_{\epsilon}, q_{\infty})^{\ad}}{x}) \geq 0$, 
so that, by Lemma~\ref{lem:global:degree:comp},
\[
\adeg\left(\rest{\overline{Q} + 
\left(0, \sum\nolimits_{P \in M_K \setminus U} \epsilon [P]\right)}{x}\right) \geq 0,
\]
and hence
\[
\adeg(\srest{\overline{Q}}{x}) \geq -\epsilon [K(x) : K] \#(M_K \setminus U),
\]
where $K(x)$ is the residue field at $x$.
Thus the assertion follows.
\end{proof}

\begin{Lemma}
\label{lem:max:nef:adelic:arith}
\begin{enumerate}
\renewcommand{\labelenumi}{(\arabic{enumi})}
\item
Let 
\[
\overline{L}_1 = (L_1, \{ (l_1)_P \}_{P \in M_K})
\quad\text{and}\quad
\overline{L}_2 = (L_2, \{ (l_2)_P \}_{P \in M_K})
\]
be global adelic $\RR$-Cartier divisors of $C^0$-type on $X$.
If $\overline{L}_1$ and $\overline{L}_2$ are relatively nef,
then
\[
\max \{ \overline{L}_1, \overline{L}_2 \}
: = \left( \max \{ L_1, L_2 \}, \big\{ \max \{ (l_1)_P, (l_2)_P \} \big\}_{P \in M_K} \right)
\]
is also relatively nef.

\item
Let $\overline{Q}_1 = (Q_1, q_1)$ and $\overline{Q}_2 = (Q_2, q_2)$ 
be adelic arithmetic $\RR$-Cartier divisors of $C^0$-type
on $X$. If $\overline{Q}_1$ and $\overline{Q}_2$ 
are nef, then
\begin{multline*}
\qquad\qquad
\max \{ \overline{Q}_1, \overline{Q}_2 \}
:= \left( \max \{ Q_1, Q_2 \}, \left\{ \max \{ (q_1)_P, (q_2)_P \} \right\}_{P \in M_K} \right. \\
\left. \cup 
\left\{ \max \{ (q_1)_{\infty}, (q_2)_{\infty} \} \right\} \right)
\end{multline*}
is also nef.
\end{enumerate}
\end{Lemma}

\begin{proof}
(1)
We set $L_1 = a_{11} x_1 + \cdots + a_{1r} x_r$
and $L_2 = a_{21} x_1 + \cdots + a_{2r} x_r$,
where $x_1, \ldots, x_r$ are closed points on $X$ and
$a_{11}, \ldots, a_{1r}, a_{21}, \ldots, a_{2r} \in \RR$.
Let us choose an non-empty Zariski open set $U$ of $\Spec(O_K)$,
a regular  model $\XXX_U$ over $U$ and nef $\RR$-Cartier divisors $\LLL_1$ and $\LLL_2$ on $\XXX_U$
such that $l_1$ and $l_2$ are given by $\LLL_1$ and $\LLL_2$ over $U$, respectively.
For $i=1,\ldots, r$, let $C_i$ be the Zariski closure of $x_i$ in $\XXX_U$.
Shrinking $U$ if necessarily,
we may assume the following:
\begin{enumerate}
\renewcommand{\labelenumi}{(\alph{enumi})}
\item
$\LLL_1 = a_{11} C_1 + \cdots + a_{1r} C_r$ and
$\LLL_2 = a_{21} C_1 + \cdots + a_{2r} C_r$.

\item
$C_i \cap C_j = \emptyset$ for all $i \not= j$.
\end{enumerate}
Then, by the properties (a) and (b),
for $P \in U$, it is easy to see that
$\max \{ (l_1)_{P}, (l_2)_{P} \}$ is given by
$g_{((\XXX_U)_{(P)},\, \max \{ \LLL_1, \LLL_2\}_{(P)})}$, 
where $(\XXX_U)_{(P)}$ is the localization of $\XXX_U \to U$ at $P$.
Note that $\max \{ \LLL_1, \LLL_2\}$ is relatively nef by (2) in Lemma~\ref{lem:for:vertical:Zariski:01}.
Moreover, for $P \in M_K \setminus U$,
by (2) in Proposition~\ref{prop:max:Green:adelic},
$\max \{ (l_1)_{P}, (l_2)_{P} \}$ is of $(C^0 \cap \Tpsh)$-type.
Thus the assertion follows.

\medskip
(2)
By \cite[Lemma~9.1.1]{MoArZariski},
$\max \{ (q_1)_{\infty}, (q_2)_{\infty} \}$ is of $(C^0 \cap \Tpsh)$-type, so that,
by virtue of (1), it is sufficient to show that
$\adeg(\srest{\max \{ \overline{Q}_1, \overline{Q}_2 \}}{x}) \geq 0$
for all closed points $x$ of $X$.
As 
\[
\Supp_{\RR}(\max \{ Q_1, Q_2 \} - Q_1) \cap \Supp_{\RR}(\max \{ Q_1, Q_2 \} - Q_2)
= \emptyset,
\]
we may assume that $x \not\in \Supp_{\RR}(\max \{ Q_1, Q_2 \} - Q_1)$.
If we set 
\[
\overline{Q} := \max \{ \overline{Q}_1, \overline{Q}_2 \} - \overline{Q}_1,
\]
then $\overline{Q}$ is effective and
$x \not\in \Supp_{\RR}(Q)$, so that $\adeg(\srest{\overline{Q}}{x}) \geq 0$.
Therefore,
\[
\adeg(\srest{\max \{ \overline{Q}_1, \overline{Q}_2 \}}{x}) = \adeg(\srest{\overline{Q}_1}{x}) +
\adeg(\srest{\overline{Q}}{x}) \geq 0,
\]
as required.
\end{proof}

\begin{proof}[Proof of Theorem~\rom{\ref{thm:Zariski:decomp:adelic:arithmetic:divisor}}]
Let us start the proof of Theorem~\ref{thm:Zariski:decomp:adelic:arithmetic:divisor}.
We choose 
\[
\overline{M} = (M, m) \in \Upsilon(\overline{D};R).
\]
Let $\Upsilon([\overline{M}, \overline{D}];R)$ be the set of all
nef adelic arithmetic $\RR$-Cartier divisors $\overline{L} = (L, l)$ of
$C^0$-type on $X$ with $L \leq R$ and $\overline{M} \leq \overline{L} \leq \overline{D}$.
For each closed point $x \in X$,
we set 
\[
a_x = \sup \left\{ \mult_x(L) \mid (L, l) \in\Upsilon([\overline{M},\overline{D}];R) \right\}.
\]
Note that, if $x \not\in \Supp_{\RR}(D) \cup \Supp_{\RR}(M)$, then 
$\mult_{x}(L) = 0$ for all $\overline{L} = (L, l) \in \Upsilon([\overline{M},\overline{D}];R)$.
In particular, $a_x = 0$, so that
we set $Q = \sum_{x} a_x x$, which is an $\RR$-Cartier divisor on $X$ with $Q \leq R$.
By using Theorem~\ref{thm:vertical:Zariski:decomp:global},
let $(Q,  \{ q_P \}_{P \in M_K})$ be the greatest element of
\[
A := \left\{ \overline{L} = (L, \{ l_P \}_{P \in M_K}) \ \left| \ 
\begin{array}{l}
\text{$\overline{L}$ is a relatively nef global adelic $\RR$-Cartier} \\
\text{divisor on $X$ such that $L \leq Q$ and
$\overline{L} \leq \overline{D}^{\tau}$}
\end{array}
\right\}\right..
\]
Note that $(Q,  \{ q_P \}_{P \in M_K}) \geq \overline{M}^{\tau}$ because
$\overline{M}^{\tau} \in A$.
Moreover, by \cite[Theorem~4.6]{MoArZariski},
there is a $Q$-Green function $q_{\infty}$ of $(C^0 \cap \Tpsh)$-type on $X(\CC)$
such that $q_{\infty}$ yields 
the greatest element of
\[
B := \left\{ h_{\infty} \ \left|\ 
\begin{array}{l}
\text{$h_{\infty}$ is a $Q$-Green function of $(C^0 \cap \Tpsh)$-type} \\ 
\text{on $X(\CC)$ and $h_{\infty} \leq g_{\infty}$}
\end{array}
\right\}\right..
\]
As $\max \{ m_{\infty}, q_{\infty} \}$ is an element of $B$
(cf. \cite[Lemma~9.1.1]{MoArZariski}), 
we have $m_{\infty} \leq q_{\infty}$.
Here we set
\[
\overline{Q} := (Q, \{ q_P \}_{P \in M_K} \cup \{ q_{\infty} \}).
\]
Clearly $\overline{M} \leq \overline{Q} \leq \overline{D}$ and $Q \leq R$.

\frontmatterforspececialclaim
\begin{Claim}
\label{claim:thm:Zariski:decomp:adelic:arithmetic:divisor:01}
There is a sequence $\left\{ \overline{L}_n = (L_n, l_n)\right\}_{n=1}^{\infty}$ in $\Upsilon([\overline{M}, \overline{D}]; R)$
such that
\[
a_x = \lim_{n\to\infty} \mult_x(L_n)
\]
for all closed points $x \in X$.
\end{Claim}
\backmatterforspececialclaim

\begin{proof}
Let $\{ x_1, \ldots, x_N \} := \Supp_{\RR}(D) \cup \Supp_{\RR}(M)$.
For each $i = 1, \ldots, N$, we can find
a sequence $\left\{ \overline{L}_{i,n} = (L_{i,n}, l_{i, n})\right\}_{n=1}^{\infty}$ in $\Upsilon([\overline{M}, \overline{D}]; R)$
such that
\[
a_{x_i} = \lim_{n\to\infty} \mult_{x_i}(L_{i, n}).
\]
We set $\overline{L}_n = \max \{ \overline{L}_{1,n}, \ldots, \overline{L}_{N,n} \}$ for $n \geq 1$.
Then, by (2) in Lemma~\ref{lem:max:nef:adelic:arith}, $\overline{L}_n \in \Upsilon([\overline{M}, \overline{D}]; R)$, and hence
\[
a_{x_i} = \lim_{n\to\infty} \mult_{x_i}(L_{n})
\]
for $i=1, \ldots, N$.
If $x \not\in \{ x_1, \ldots, x_N \}$, then $a_x = 0$ and $\mult_{x}(L_{n}) = 0$ for all $n$.
Thus we have the claim.
\end{proof}

As $\max \{ q_{\infty}, (l_n)_{\infty} \} \in B$ by \cite[Lemma~9.1.1]{MoArZariski}, 
we have $(l_n)_{\infty} \leq q_{\infty}$.
Moreover, $\overline{L}^{\tau}_n \leq \overline{Q}^{\tau}$ because $\overline{L}^{\tau}_n  \in A$.
Therefore, $\overline{L}_n \leq \overline{Q}$, so that,
by Lemma~\ref{lem:nef:seq:approx:nef}, $\overline{Q}$ is nef.
In particular, $\overline{Q} \in \Upsilon(\overline{D};R)$.
We need to check that $\overline{Q}$ is the greatest element of $\Upsilon(\overline{D};R)$.
Indeed, let $\overline{L} = (L, l)\in \Upsilon(\overline{D};R)$.
Then $\overline{L}^{\tau} \in A$ and $l_{\infty} \in B$, and
hence $\overline{L}^{\tau} \leq \overline{Q}^{\tau}$ and $l_{\infty} \leq q_{\infty}$,
as required.
\end{proof}

\begin{Corollary}[Zariski decomposition for adelic arithmetic divisors]
\label{cor:Zariski:decomp:adelic:arithmetic:divisor}
Let $\overline{D}$ be an adelic arithmetic $\RR$-Cartier divisor of $C^0$-type on $X$.
Let $\Upsilon(\overline{D})$ be the set of all
nef adelic arithmetic $\RR$-Cartier divisors $\overline{L}$ of
$C^0$-type on $X$ with $\overline{L} \leq \overline{D}$.
\index{\AdelDivSymbol}{0U:Upsilon(overline{D})@$\Upsilon(\overline{D})$}%
If $\Upsilon(\overline{D}) \not= \emptyset$,
then there is the greatest element $\overline{Q}$ of
$\Upsilon(\overline{D})$, that is,
$\overline{Q} \in \Upsilon(\overline{D})$ and
$\overline{L} \leq \overline{Q}$ for all $\overline{L} \in \Upsilon(\overline{D})$.
Moreover, 
the natural map 
$\aH(X, a \overline{Q}) \to \aH(X, a \overline{D})$
is bijective
for $a \in \RR_{>0}$.
In particular, $\avol(\overline{Q}) = \avol(\overline{D})$.
\end{Corollary}

\begin{proof}
Applying Theorem~\ref{thm:Zariski:decomp:adelic:arithmetic:divisor} to the case where $R = D$,
we have the first assertion.
Let us see the second assertion. Let $\phi \in \aH(X, a\overline{D})$, that is,
$a \overline{D} + \widehat{(\phi)} \geq 0$ by Proposition~\ref{prop:basic:prop:H:0}, so that
$\overline{D} \geq (1/a)\widehat{(\phi^{-1})}$.
Note that $(1/a)\widehat{(\phi^{-1})}$ is nef, and hence $(1/a)\widehat{(\phi^{-1})} \in \Upsilon(\overline{D})$.
Therefore, $\overline{Q} \geq (1/a)\widehat{(\phi^{-1})}$, that is,
$a \overline{Q} + \widehat{(\phi)} \geq 0$, 
which means that $\phi \in \aH(X, a\overline{Q})$ by Proposition~\ref{prop:basic:prop:H:0}.
\end{proof}

\begin{Remark}
There are several conditions to guarantee $\Upsilon(\overline{D}) \not= \emptyset$.
For example, if $\aH(X, a\overline{D}) \not= \{ 0 \}$ for some $a \in \RR_{>0}$,
then $\Upsilon(\overline{D}) \not= \emptyset$.
Indeed, $\aH(X, a\overline{D}) \not= \{ 0 \}$ implies that
$a\overline{D} + \widehat{(\phi)} \geq 0$ for some $\phi \in \Rat(X)^{\times}$,
so that $(1/a)  \widehat{(\phi^{-1})} \leq \overline{D}$.
Note that $(1/a)  \widehat{(\phi^{-1})}$ is nef, and hence $(1/a)  \widehat{(\phi^{-1})} \in \Upsilon(\overline{D})$.
In particular, if $\overline{D}$ is big, then $\Upsilon(\overline{D}) \not= \emptyset$.
As a conjecture, we expect that if  $\overline{D}$ is pseudo-effective, then $\Upsilon(\overline{D}) \not= \emptyset$.
\end{Remark}

\ifmonog
\chapter[Characterization of nef adelic arithmetic divisors]{Characterization of nef adelic arithmetic divisors on arithmetic surfaces}\fi
\ifpaper\section{Characterization of nef adelic arithmetic divisors on arithmetic surfaces}\fi
In this 
\ifmonog chapter, \fi
\ifpaper section, \fi
we consider a generalization of the numerical 
characterization of nef arithmetic divisors proved in \cite{MoCharNef} to adelic arithmetic divisors.
Namely, we will prove that an integrable adelic arithmetic $\RR$-Cartier divisor $\overline{D}$ of $C^0$-type
on a projective smooth curve over a number field is nef if and only if
$\overline{D}$ is pseudo-effective and $\adeg(\overline{D}^2) = \avol(\overline{D})$.
Throughout this 
\ifmonog chapter, \fi
\ifpaper section, \fi
let $X$ be a projective, smooth and geometrically integral
variety over a number field $K$. 

\ifmonog\section{Hodge index theorem for adelic arithmetic divisors}\fi
\ifpaper\subsection{Hodge index theorem for adelic arithmetic divisors}\fi
\label{subsec:Hodge:index:adelic:arith:div}

We assume that $\dim X = 1$.
Let us start with a refinement of the generalized Hodge index theorem on an arithmetic surface.

\begin{Theorem}
\label{thm:GHI:adelic}
Let $\overline{D} = (D, g)$ be an integrable adelic arithmetic $\RR$-Cartier divisor of $C^0$-type on $X$.
If $\deg(D) \geq 0$, then $\deg(\overline{D}^2) \leq \avol_{\chi}(\overline{D})$.
Moreover, the equality holds if and only if $\overline{D}$ is relatively nef.
\end{Theorem}

\begin{proof}
Let $\Upsilon_{rel}(\overline{D})$ be the set of all
relatively nef adelic arithmetic $\RR$-Cartier divisors $\overline{L}$ of
$C^0$-type on $X$ with $\overline{L} \leq \overline{D}$ (cf. Corollary~\ref{cor:Zariski:rel:decomp:adelic:arithmetic:divisor}).
Then, by Corollary~\ref{cor:Zariski:rel:decomp:adelic:arithmetic:divisor},
$\Upsilon_{rel}(\overline{D})$ has the greatest element $\overline{Q} = (Q, q)$, that is,
$\overline{Q} \in \Upsilon_{rel}(\overline{D})$ and
$\overline{L} \leq \overline{Q}$ for all $\overline{L} \in \Upsilon_{rel}(\overline{D})$.
Further, we have the following:
\begin{enumerate}
\renewcommand{\labelenumi}{(\arabic{enumi})}
\item
$\avol_{\chi}(\overline{Q}) = \avol_{\chi}(\overline{D})$.

\item
If we set $\overline{N} := \overline{D} - \overline{Q}$, then
$\overline{N} = (0, \phi)$ and
$\adeg(\overline{Q} \cdot \overline{N}) = 0$, where
$\phi = \{ \phi_P \}_{P \in M_K} \cup \{ \phi_{\infty} \}$ 
is a collection of integrable continuous functions. 
\end{enumerate}
By Theorem~\ref{thm:G:H:I:T:adelic:arith},
\[
\adeg(\overline{Q}^2) = \avol_{\chi}(\overline{Q}) = \avol_{\chi}(\overline{D}).
\]
Note that
\[
\adeg(\overline{N}^2) = \sum_{P \in M_K} \log \#(O_K/P) \adeg_P((0, \phi_P) ;\phi_P) +
\frac{1}{2} \int_{X(\CC)} \phi_{\infty} dd^c(\phi_{\infty}),
\]
so that, by Lemma~\ref{lem:Zariski:adelic} and \cite[Proposition~1.2.3]{MoD}, we have
\[
\adeg(\overline{D}^2) = \adeg(\overline{Q}^2) + \adeg(\overline{N}^2) \leq \adeg(\overline{Q}^2).
\]
Thus the first assertion follows. In addition,
by using the equality conditions in Lemma~\ref{lem:Zariski:adelic} and \cite[Proposition~1.2.3]{MoD},
\begin{align*}
\adeg(\overline{D}^2) = \avol_{\chi}(\overline{D}) & \quad\Longleftrightarrow\quad  \deg(\overline{N}^2) = 0 \\
& \quad\Longleftrightarrow\quad  \text{$\phi_P$ ($\forall P \in M_K$) and $\phi_{\infty}$ are constant functions} \\
& \quad\Longleftrightarrow\quad  \text{$\overline{D}$ is relatively nef},
\end{align*}
as required.
\end{proof}

As a consequence of the above theorem, we have the Hodge index theorem for adelic arithmetic divisors.

\begin{Corollary}
\label{cor:HI:adelic}
Let $\overline{D} = (D, g)$ be an integrable adelic arithmetic $\RR$-Cartier divisor of $C^0$-type on $X$.
If $\deg(D) = 0$, then $\deg(\overline{D}^2) \leq 0$.
Moreover, the equality holds if and only if $\overline{D} = \widehat{(\psi)}_{\RR} + (0, \lambda [\infty])$
for some $\psi \in \Rat(X)^{\times}_{\RR}$ and $\lambda \in \RR$.
\end{Corollary}

\begin{proof}
By Theorem~\ref{thm:GHI:adelic}, $\deg(\overline{D}^2) \leq \avol_{\chi}(\overline{D}) \leq
\avol(\overline{D})$, so that it is sufficient to show that $\avol(\overline{D}) = 0$ for the first assertion.
Indeed, if $\avol(\overline{D}) > 0$, then, by the continuity of the volume function
(cf. Theorem~\ref{thm:cont:volume}), $\avol(\overline{D} - (1/n) \overline{\AAA}^{\ad}) > 0$
for a sufficiently large $n$,
where $\overline{\AAA}$ is an ample arithmetic Cartier divisor on some regular model of $X$.
In particular, $\deg(D) \geq (1/n)\deg(\AAA \cap X) > 0$, which is a contradiction.

Next we consider the equality condition.
Clearly if $\overline{D} = \widehat{(\psi)}_{\RR} + (0, \lambda [\infty])$
for some $\psi \in \Rat(X)^{\times}_{\RR}$ and $\lambda \in \RR$, then $\adeg(\overline{D}^2) = 0$, so that
we assume that $\deg(\overline{D}^2) = 0$.
Let $\XXX$ be a regular model of $X$ over $\Spec(O_K)$.

\begin{Claim}
There is an $\RR$-Cartier divisor $\DDD$ on $\XXX$ such that $\DDD \cap X = D$ and
$(\DDD \cdot C) = 0$ for any vertical curves $C$ on $\XXX$.
\end{Claim}

\begin{proof}
Let $U$ be a non-empty Zariski open set of $\Spec(O_K)$ such that $\XXX \to \Spec(O_K)$ is smooth over $U$.
Let $\DDD_0$ be the Zariski closure of $D$ in $\XXX$. Then the degree of $\DDD_0$ along every smooth fiber of $\XXX \to \Spec(O_K)$ is
zero. Moreover, using Zariski's lemma, for each $P \in M_K \setminus U$, there is a vertical $\RR$-Cartier 
divisor $Z_P$ in the fiber $F_P$
over $P$ such that the degree of $\DDD_0 + Z_P$ along every irreducible component of $F_P$ is zero.
Therefore, if we set $\DDD = \DDD_0 + \sum_{P \in M_K \setminus U} Z_P$, we have our desired $\RR$-Cartier divisor.
\end{proof}

We set $\overline{D}' := (\DDD, g_{\infty})^{\ad}$, $\phi_P := g_P - g_{(\XXX_{(P)},\ \DDD_{(P)})}$ and
$\phi := \sum_{P \in M_K} \phi_P [P]$,
where $\XXX_{(P)}$ is the localization of $\XXX \to \Spec(O_K)$ at $P$ and $\DDD_{(P)}$ is the restriction of $\DDD$ to $\XXX_{(P)}$.
Then $\overline{D} = \overline{D}' + (0, \phi)$.
Note that $\adeg(\overline{D}' \cdot (0, \phi)) = 0$ because $(\DDD \cdot C) = 0$ for any vertical curve $C$ in $\XXX$, so that
\[
0 = \adeg(\overline{D}^2) = \adeg({\overline{D}'}^2) + \adeg((0, \phi)^2).
\]
Therefore, 
\[
\adeg({\overline{D}'}^2) = \adeg((0, \phi)^2) = 0
\]
because
$\adeg({\overline{D}'}^2) \leq 0$ and $\adeg((0, \phi)^2)  \leq 0$ by the first assertion.
As 
\[
\adeg({\overline{D}'}^2) = \adeg((\DDD, g_{\infty})^2) = 0,
\]
by virtue of \cite[Lemma~4.1]{MoCharNef}, 
\[
\overline{D}' =  \widehat{(\psi')}_{\RR} + (0, \eta' [\infty])
\]
for some $\psi' \in \Rat(X)^{\times}_{\RR}$ and an $F_{\infty}$-invariant locally constant function $\eta'$ on $X(\CC)$.
Moreover, by using the fact $\adeg((0, \phi)^2) = 0$ together with
Lemma~\ref{lem:Zariski:adelic}, $\phi_P$ is a constant for all $P \in M_K$.
Note that $\phi_P = 0$ expect finitely many $P \in M_K$. In addition, for each $P \in M_K$, there is
$f_P \in K^{\times} \otimes_{\ZZ} \RR$ such that $(f_P) = P$ on $\Spec(O_K)$.
Therefore, 
\[
\overline{D} =  \widehat{(\psi'')}_{\RR} + (0, \eta'' [\infty])
\]
for some $\psi'' \in \Rat(X)^{\times}_{\RR}$ and an $F_{\infty}$-invariant locally constant function $\eta''$ on $X(\CC)$.
Thus, by using Dirichlet's unit theorem,
we have the second assertion of the corollary (cf. \cite[Proof of Lemma~4.1]{MoCharNef}).
\end{proof}

\ifmonog\section{Arithmetic asymptotic multiplicity}\fi
\ifpaper\subsection{Arithmetic asymptotic multiplicity}\fi
\label{subsec:arith:asymptotic:multiplicity}

Let $\KK$ be either $\QQ$ or $\RR$.
Let 
\[
\Rat(X)^{\times}_{\KK} := \Rat(X)^{\times} \otimes_{\ZZ} \KK,
\]
and let 
\[
(\ )_{\KK} : \Rat(X)^{\times}_{\KK} \to \Div(X)_{\KK}\quad\text{and}\quad
\widehat{(\ )}_{\KK} : \Rat(X)^{\times}_{\KK} \to \aDiv^{\ad}_{C^{0}}(X)_{\RR}
\]
be the natural extensions of the homomorphisms 
\[
\Rat(X)^{\times} \to \Div(X)\quad\text{and}\quad
\Rat(X)^{\times} \to \aDiv^{\ad}_{C^{0}}(X)_{\RR}
\]
given by
$\phi \mapsto (\phi)$ and 
$\phi \mapsto \widehat{(\phi)}$, respectively.
\index{\AdelDivSymbol}{0R:Rat(X)^{times}_{KK}@$\Rat(X)^{\times}_{\KK}$}%
\index{\AdelDivSymbol}{0p:(varphi)_{KK}@$(\varphi)_{\KK}$}%
\index{\AdelDivSymbol}{0p:widehat{(varphi)}_{KK}@$\widehat{(\varphi)}_{\KK}$}%
Let $\overline{D}$ be an adelic arithmetic $\RR$-Cartier divisor of $C^0$-type.
We define 
$\widehat{\Gamma}^{\times}_{\KK}(X, \overline{D})$
to be
\[
\widehat{\Gamma}^{\times}_{\KK}(X, \overline{D}) := \left\{ \phi \in \Rat(X)^{\times}_{\KK} \mid
\overline{D} + \widehat{(\phi)}_{\KK} \geq (0,0) \right\}.
\]%
\index{\AdelDivSymbol}{0G:widehat{Gamma}^{times}_{KK}(X,overline{D})@$\widehat{\Gamma}^{\times}_{\KK}(X, \overline{D})$}%
For $\xi \in X$,
the {\em $\KK$-asymptotic multiplicity of $\overline{D}$ at $\xi$} 
\index{\AdelDivSubject}{K-asymptotic multiplicity@$\KK$-asymptotic multiplicity}%
\index{\AdelDivSymbol}{0m:mu_{KK,xi}(overline{D})@$\mu_{\KK,\xi}(\overline{D})$}%
is defined to be
\[
\mu_{\KK,\xi}(\overline{D}) :=
\begin{cases}
\inf \left\{ \mult_{\xi}(D + (\phi)_{\KK}) \mid \phi \in \widehat{\Gamma}_{\KK}^{\times}(X, \overline{D}) \right\} &
\text{if $\widehat{\Gamma}_{\KK}^{\times}(X, \overline{D}) \not= \emptyset$}, \\
\infty & \text{otherwise}.
\end{cases}
\]

\begin{Proposition}
\label{prop:ar:mu:basic}
Let $\overline{D}$ and $\overline{E}$ be adelic arithmetic $\RR$-Cartier divisors of $C^0$-type on $X$.
Then we have the following:
\begin{enumerate}
\renewcommand{\labelenumi}{(\arabic{enumi})}
\item
$\mu_{\KK,\xi}(\overline{D} + \overline{E}) \leq 
\mu_{\KK,\xi}(\overline{D}) + \mu_{\KK,\xi}(\overline{E})$.

\item
If $\overline{D} \leq \overline{E}$, then 
$\mu_{\KK,\xi}(\overline{E}) \leq \mu_{\KK,\xi}(\overline{D}) + \mult_{\xi}(E - D)$.

\item
$\mu_{\KK,\xi}(\overline{D} + \widehat{(\phi)}_{\KK}) = \mu_{\KK,\xi}(\overline{D})$ for 
$\phi \in \Rat(X)^{\times}_{\KK}$.

\item
$\mu_{\KK,\xi}(a\overline{D}) = a \mu_{\KK,\xi}(\overline{D})$ for $a \in \KK_{\geq 0}$.

\item
$0 \leq \mu_{\RR,\xi}(\overline{D}) \leq \mu_{\QQ,\xi}(\overline{D})$.

\item
If $\overline{D}$ is big, then $\mu_{\RR,\xi}(\overline{D}) = \mu_{\QQ,\xi}(\overline{D})$.

\item
If $\overline{D}$ is nef and big, then $\mu_{\KK,\xi}(\overline{D}) = 0$.
\end{enumerate}
\end{Proposition}

\begin{proof}
(1), (2), (3), (4) and (5) can be proved in the same way as \cite[Proposition~2.1]{MoArLinB}.
For the proofs for (6) and (7), let us begin with the following claim:

\begin{Claim}
\label{claim:rop:ar:mu:basic:01}
Let $\overline{D}_1, \ldots, \overline{D}_r$ be adelic arithmetic $\RR$-Cartier divisors of $C^0$-type on $X$.
If $\overline{D}$ is big, then
\[
\lim_{\substack{(x_1, \ldots, x_r) \to (0,\ldots, 0) \\ (x_1, \ldots, x_r) \in \QQ^r}} 
\mu_{\QQ,\xi}(\overline{D} + x_1 \overline{D}_1 + \cdots + x_r \overline{D}_r)
= \mu_{\QQ,\xi}(\overline{D}).
\]
\end{Claim}

\begin{proof}
Here we define $f : \QQ^r \to \RR \cup \{ \infty \}$ to be 
$f(x) := \mu_{\QQ,\xi}(\overline{D}_x )$
for $x = (x_1, \ldots, x_r) \in \QQ^r$, where 
\[
\overline{D}_x = \overline{D} + x_1 \overline{D}_1 + \cdots + x_r \overline{D}_r.
\]
Note that $f$ is a convex function over $\QQ$, that is,
\[
f(tx + (1-t)y) \leq t f(x) + (1-t)f(y)
\]
for all $x, y \in \QQ^r$ and $t \in [0, 1] \cap \QQ$.
Indeed, by using (1) and (4),
\begin{align*}
f(tx + (1-t)y) & = \mu_{\QQ,\xi}(t\overline{D}_x + (1-t)\overline{D}_y ) 
\leq \mu_{\QQ,\xi}(t\overline{D}_x)
+ \mu_{\QQ,\xi}( (1-t)\overline{D}_y) \\
& = t f(x) + (1-t)f(y).
\end{align*}
Moreover, by virtue of the continuity of the volume function,
there is a positive rational number $c$ such that $\overline{D}_x$ is big for all 
$x \in (-c, \infty)^r \cap \QQ^r$. 
Therefore, by \cite[Proposition~1.3.1]{MoArLin},
there is a continuous function $\tilde{f} : (-c, \infty)^r \to \RR$ such that $\tilde{f} = f$
on $(-c, \infty)^r \cap \QQ^r$.
Thus the assertion follows.
\end{proof}

(6) Let $\phi \in \widehat{\Gamma}^{\times}_{\RR}(X, \overline{D})$, that is,
$\phi = \phi_1^{a_1} \cdots \phi_r^{a_r}$
and $\overline{D} + a_1 \widehat{(\phi_1)} + \cdots + a_r \widehat{(\phi_r)} \geq 0$
for some $\phi_1, \ldots, \phi_r \in \Rat(X)^{\times}$ and $a_1, \ldots, a_r \in \RR$.
By using \cite[Lemma~5.2.3 and Lemma~5.2.4]{MoArZariski}, for each $i$,
we can find effective adelic arithmetic $\RR$-Cartier divisors $\overline{A}_i$ and $\overline{B}_i$ of $C^0$-type
on $X$ such that $\widehat{(\phi_i)} = \overline{A}_i - \overline{B}_i$.
For $n \in \ZZ_{>0}$, we choose $x_{i,n} \in \RR$ and $x'_{i,n} \in \QQ$ such that
$a_i + x_{i,n} \in \QQ$ and $0 \leq x_{i,n} \leq x'_{i,n} \leq 1/n$.
Then 
\begin{multline*}
\overline{D} + \sum\nolimits_{i} x'_{i,n} \overline{B}_i + \sum\nolimits_{i} (a_i + x_{i,n}) \widehat{(\phi_i)} \geq
\overline{D} + \sum\nolimits_{i} x_{i,n} \overline{B}_i + \sum\nolimits_{i} (a_i + x_{i,n}) \widehat{(\phi_i)} \\
= \overline{D} + \sum\nolimits_{i} a_i \widehat{(\phi_i)} + \sum\nolimits_{i} x_{i,n} \overline{A}_i \geq 0,
\end{multline*}
and hence
\[
\mu_{\QQ,\xi}\left(\overline{D} + \sum\nolimits_{i} x'_{i,n} \overline{B}_i\right) \leq
\mult_{\xi}\left(D + \sum\nolimits_{i} x'_{i,n} B_i + \sum\nolimits_{i} (a_i + x_{i,n}) (\phi_i)\right).
\]
On the other hand, by using Claim~\ref{claim:rop:ar:mu:basic:01},
\[
\lim_{n\to\infty} \mu_{\QQ,\xi}\left(\overline{D} + \sum_{i} x'_{i,n} \overline{B}_i\right) = \mu_{\QQ,\xi}(\overline{D})
\]
and
\[
\lim_{n\to\infty} \mult_{\xi}\left(D + \sum\nolimits_{i} x'_{i,n} B_i + \sum\nolimits_{i} (a_i + x_{i,n}) (\phi_i)\right)
= \mult_{\xi}\left(D + \sum\nolimits_{i} a_i(\phi_i)\right).
\]
Thus 
\[
\mu_{\QQ,\xi}(\overline{D}) \leq \mult_{\xi}\left(D + \sum\nolimits_{i} a_i(\phi_i)\right),
\]
which yields $\mu_{\QQ,\xi}(\overline{D}) \leq \mu_{\RR,\xi}(\overline{D})$, so that
(6) follows from (5).

\medskip
(7) By (5), it is sufficient to see that $\mu_{\QQ,\xi}(\overline{D}) = 0$.
We set 
\[
\overline{D} = (D, \{ g_P \}_{P \in M_K} \cup \{ g_{\infty} \}).
\]
Let $U$ be a non-empty open set $\Spec(O_K)$ such that
$\overline{D}$ has a defining model over $U$.
For each $n \in \ZZ_{>0}$, by Proposition~\ref{prop:rel:nef:pseudo:effective},
there is a normal model $\XXX$ of $X$ and a relatively nef $\RR$-Cartier divisor
$\DDD$ on $\XXX$ such that $\DDD \cap X = D$ and
\[
(\DDD, g_{\infty})^{\ad} - (1/n) \left(0,  \sum\nolimits_{P \in M_K \setminus U} [P]
\right)\leq \overline{D} \leq (\DDD, g_{\infty})^{\ad} .
\]
Note that $(\DDD, g_{\infty})$ is nef and big, so that, by (2) and \cite[Proposition~2.1, (6)]{MoArLinB},
\[
0 \leq \mu_{\QQ,\xi}\left(\overline{D} + (1/n) \left(0,  \sum\nolimits_{P \in M_K \setminus U}[P]
\right) \right)  \leq \mu_{\QQ,\xi}((\DDD, g_{\infty})) = 0,
\]
and hence $\mu_{\QQ,\xi}\left(\overline{D} + (1/n) \left(0,  \sum\nolimits_{P \in M_K \setminus U} [P]
\right) \right) = 0$.
Further, by Claim~\ref{claim:rop:ar:mu:basic:01},
\[
\mu_{\QQ,\xi}(\overline{D}) = 
\lim_{n\to\infty} \mu_{\QQ,\xi}\left(\overline{D} + (1/n) \left(0,  \sum\nolimits_{P \in M_K \setminus U} [P]
\right) \right),
\]
so that (7) follows.
\end{proof}

\ifmonog\section{Necessary condition for the equality $\avol = \acvol$}\fi
\ifpaper\subsection{Necessary condition for the equality $\avol = \acvol$}\fi
\label{subset:necessary:condition:equality:avol:cvol}

Let $\overline{K}$ be an algebraic closure of $K$ and
$X_{\overline{K}} := X \times_{\Spec(K)} \Spec(\overline{K})$.

We fix a monomial order $\precsim$ on $\ZZ^d_{\geq 0}$, that is,
$\precsim$ is a total ordering relation on $\ZZ^d_{\geq 0}$ with the following properties:
\begin{enumerate}
\renewcommand{\labelenumi}{(\alph{enumi})}
\item $(0,\ldots,0) \precsim A$ for all $A \in \ZZ^d_{\geq 0}$.

\item If $A \precsim B$ for $A, B \in \ZZ_{\geq 0}^d$, then $A + C \precsim B + C$ for all $C \in \ZZ_{\geq 0}^d$.
\end{enumerate}
The monomial order $\precsim$ on $\ZZ^d_{\geq 0}$ extends uniquely to a totally ordering relation $\precsim$ on $\ZZ^d$
such that $A + C \precsim B + C$ holds for all $A, B, C \in \ZZ^d$ with $A \precsim B$.

Let $z_P = (z_1, \ldots, z_d)$ be a local system of parameters of $\OOO_{X_{\overline{K}}, P}$ at $P \in X(\overline{K})$.
Note that the completion $\widehat{\OOO}_{X_{\overline{K}}, P}$ of 
$\OOO_{X_{\overline{K}}, P}$ with respect to
the maximal ideal of $\OOO_{X_{\overline{K}}, P}$ is naturally isomorphic to
$\overline{K}[\![z_1, \ldots, z_d]\!]$.
Thus, for $f \in \OOO_{X_{\overline{K}}, P}$, we can put
\[
f = \sum_{(a_1, \ldots, a_d) \in \ZZ_{\geq 0}^d} c_{(a_1, \ldots, a_d)} z_1^{a_1} \cdots z_d^{a_d},\qquad(c_{(a_1, \ldots, a_d)} \in \overline{K}).
\]
We define $\ord_{z_P}^{\precsim}(f)$ to be
\[
\ord_{z_P}^{\precsim}(f) := \begin{cases}
\min\limits_{\precsim} \left\{ (a_1, \ldots, a_d) \mid c_{(a_1, \ldots, a_d)} \not= 0 \right\} & \text{if $f \not= 0$},\\
\infty & \text{otherwise},
\end{cases}
\]
which gives rise to a rank $d$ valuation, that is,
\index{\AdelDivSymbol}{0o:ord_{z_P}^{precsim}(f)@$\ord_{z_P}^{\precsim}(f)$}%
the following properties are satisfied:
\begin{enumerate}
\renewcommand{\labelenumi}{(\roman{enumi})}
\item
$\ord_{z_P}^{\precsim}(fg) = \ord_{z_P}^{\precsim}(f) + \ord_{z_P}^{\precsim}(g)$ for $f, g \in \OOO_{X_{\overline{K}},P}$.

\item
$\min \left\{ \ord_{z_P}^{\precsim}(f), \ord_{z_P}^{\precsim}(g) \right\} \precsim \ord_{z_P}^{\precsim}(f + g)$ for $f, g \in \OOO_{X_{\overline{K}},P}$.
\end{enumerate}
By the property (i), $\ord_{z_P}^{\precsim} : \OOO_{X_{\overline{K}},P} \setminus \{ 0 \} \to \ZZ_{\geq 0}^d$ has the natural extension
\[
\ord_{z_P}^{\precsim} : \Rat(X_{\overline{K}})^{\times} \to \ZZ^d
\]
given by $\ord_{z_P}^{\precsim}(f/g) = \ord_{z_P}^{\precsim}(f) - \ord_{z_P}^{\precsim}(g)$.
Note that this extension also satisfies the same properties (i) and (ii) as before.
Since $\ord_{z_P}^{\precsim}(u) = (0,\ldots,0)$ for all $u \in \OOO^{\times}_{X_{\overline{K}},P}$, 
$\ord_{z_P}^{\precsim}$ induces
$\Rat(X_{\overline{K}})^{\times}/\OOO^{\times}_{X_{\overline{K}},P} \to \ZZ^d$. 
The composition of homomorphisms
\[
\Div(X_{\overline{K}}) \overset{\alpha_P}{\longrightarrow} \Rat^{\times}(X_{\overline{K}})/\OOO^{\times}_{X_{\overline{K}},P} \overset{\ord_{z_P}^{\precsim}}{\longrightarrow} \ZZ^d
\]
is
denoted by $\mult_{z_P}^{\precsim}$, 
\index{\AdelDivSymbol}{0m:mult_{z_P}^{precsim}@$\mult_{z_P}^{\precsim}$}%
where 
$\alpha_P : \Div(X_{\overline{K}}) \to \Rat(X_{\overline{K}})^{\times}/\OOO^{\times}_{X_{\overline{K}},P}$ is the natural
homomorphism. Moreover, the homomorphism $\mult_{z_P}^{\precsim} : \Div(X_{\overline{K}}) \to \ZZ^d$ gives rise to
the natural extension
$\Div(X_{\overline{K}}) \otimes_{\ZZ} \RR \to \RR^d$
over $\RR$.
By abuse of notation, the above extension is also denoted by $\mult_{z_P}^{\precsim}$.

Let $\overline{D} = (D, g)$ be an adelic arithmetic $\RR$-Cartier divisor of $C^0$-type on $X$.
Let $V_{\bullet} =  \bigoplus_{m \geq 0} V_m$ be a graded subalgebra of $R(D) := \bigoplus_{m \geq 0} H^0(X, mD)$ over $K$.
The Okounkov body $\Delta(V_{\bullet})$ of $V_{\bullet}$ is defined by
the closed convex hull of
\[
\bigcup_{m > 0} \left\{ \mult_{z_P}^{\precsim}(D_{\overline{K}} + (1/m) (\phi)) \in \RR_{\geq 0}^d \mid \phi \in (V_{m} \otimes_K \overline{K}) \setminus \{ 0 \} \right\}.
\]
For $t \in \RR$, let $V_{\bullet}^t$ be a graded subalgebra of $V_{\bullet}$ given by
\[
V_{\bullet}^t := \bigoplus_{m \geq 0}  \left\langle V_m \cap \aH(X, m (\overline{D} + (0, -2t[\infty]))) \right\rangle_K,
\]
where $\left\langle V_m \cap \aH(X, m (\overline{D} + (0, -2t[\infty]))) \right\rangle_K$ means the subspace of $V_m$ generated by
\[
V_m \cap \aH(X, m (\overline{D} + (0, -2t[\infty])))
\]
over $K$.
We define $G_{(\overline{D};V_{\bullet})} : \Delta(V_{\bullet}) \to \RR \cup \{ -\infty \}$ 
to be 
\[
G_{(\overline{D};V_{\bullet})}(x) := \begin{cases}
\sup \left\{ t \in \RR \mid x \in \Delta(V_{\bullet}^t) \right\} & \text{if $x \in \Delta(V_{\bullet}^t)$ for some $t$}, \\
-\infty & \text{otherwise}.
\end{cases}
\]%
\index{\AdelDivSymbol}{0G:G_{(overline{D};V_{bullet})}@$G_{(\overline{D};V_{\bullet})}$}%
Note that $G_{(\overline{D};V_{\bullet})}$ is an upper
semicontinuous concave function (cf. \cite[SubSection~1.3]{BC}). 
We also define $\avol(\overline{D};V_{\bullet})$ and $\acvol(\overline{D};V_{\bullet})$ 
to be
\[
\begin{cases}
{\displaystyle \avol(\overline{D};V_{\bullet}) := \limsup_{m\to\infty} 
\frac{\# \log \left( V_m \cap \aH(X, m\overline{D}) \right)}{m^{d+1}/(d+1)!}}, \\[3ex]
{\displaystyle  \acvol(\overline{D};V_{\bullet}) := \limsup_{m\to\infty} 
\frac{\hat{\chi} \left(V_m \cap H^0(X, mD), \Vert\cdot\Vert_{m\overline{D}} \right)}{m^{d+1}/(d+1)!}}. 
\end{cases}
\]%
\index{\AdelDivSymbol}{0v:avol(overline{D};V_{bullet})@$\avol(\overline{D};V_{\bullet})$}%
\index{\AdelDivSymbol}{0v:avol_{chi}(overline{D};V_{bullet})@$\avol_{\chi}(\overline{D};V_{\bullet})$}%
Let $\Theta(\overline{D}; V_{\bullet})$ be the closure of  
\[
\left\{ x \in \Delta(V_{\bullet}) \mid 
G_{(\overline{D};V_{\bullet})}(x) > 0 \right\}.
\]%
\index{\AdelDivSymbol}{0T:Theta(overline{D};V_{bullet})@$\Theta(\overline{D}; V_{\bullet})$}%
We assume that $V_{\bullet}$ contains an ample series, that is,
$V_m \not= \{ 0\}$ for $m \gg 1$ and 
there is an ample $\QQ$-Cartier divisor $A$ on $X$ with 
the following properties:
\[
\begin{cases}
\bullet\ \text{$A \leq D$.} \\
\bullet\ \text{There is a positive integer $m_0$ such that
$H^0(X, mm_0A) \subseteq V_{mm_0}$ for all $m \geq 1$.}
\end{cases}
\]
Then, in the similar way as  \cite[Theorem~2.8]{BC}, \cite[Theorem~3.1]{BC} and \cite[Section~3]{MoCharNef}, 
we have the following integral formulae:
\frontmatterforspececialeqn
\begin{equation}
\label{eqn:integral:formula:vol}
\avol(\overline{D};V_{\bullet}) = (d+1)! [K : \QQ] \int_{\Theta(\overline{D}; V_{\bullet})}
G_{(\overline{D};V_{\bullet})}(x) dx
\end{equation}
\backmatterforspececialeqn
and
\frontmatterforspececialeqn
\begin{equation}
\label{eqn:integral:formula:chivol}
\acvol(\overline{D};V_{\bullet}) = (d+1)! [K : \QQ] \int_{\Delta(V_{\bullet})}
G_{(\overline{D};V_{\bullet})}(x) dx.
\end{equation}
\backmatterforspececialeqn

Therefore, in the same way as \cite[Section~3]{MoCharNef},
we have the following theorem (cf. \cite[Theorem~3.4 and Corollary~3.5]{MoCharNef}).

\begin{Theorem}
\label{thm:vol:cvol:mu:zero}
We assume that $D$ is nef and big and
$\avol(\overline{D}) = \acvol(\overline{D}) > 0$. Then
$\mu_{\QQ, \xi}(\overline{D}) = 0$ for all $\xi \in X$.
\end{Theorem}

Besides it, the following theorem is also obtained:

\begin{Theorem}
\label{thm:zariski:decomp:mu}
Let $\overline{Q} = (Q, q)$ be a nef adelic arithmetic $\RR$-Cartier divisor of $C^0$-type on $X$.
We assume that $\overline{D}$ is big, $\overline{Q} \leq \overline{D}$ and $\avol(\overline{Q}) = \avol(\overline{D})$.
If we set $N = D - Q$, then $\mu_{\QQ, \xi}(\overline{D}) = \mult_{\xi}(N)$
for all $\xi \in X$.
\end{Theorem}

\begin{proof}
By using (2) and (7) in Proposition~\ref{prop:ar:mu:basic}, we have
\[
\mu_{\QQ,\xi}(\overline{D}) \leq \mu_{\QQ,\xi}(\overline{Q}) + \mult_{\xi}(N) = \mult_{\xi}(N).
\]

Let us consider the converse inequality.
Let $B$ be the Zariski closure of $\{ \xi \}$ and $P$ a regular closed point of $B$.
Let $z_P = (z_1, \ldots, z_d)$ be a local system of parameters of $\OOO_{X, P}$ such that
$B$ is given by $z_1 = \cdots = z_r = 0$.
We choose a monomial order $\precsim$ of $\ZZ^{d}_{\geq 0}$ such that
$\ell(a) \leq \ell(b)$ for all $a, b \in \ZZ_{\geq 0}^d$ with $a \precsim b$,
where $\ell(x_1, \ldots, x_d) = x_1 + \cdots + x_r$.
We set $\nu := \mult_{z_P}^{\precsim}(N)$.
For simplicity, 
in the case $V_{\bullet} = R(D)$, we denote $\Delta(V_{\bullet})$, $\Delta(V_{\bullet}^t)$, $G_{(\overline{D}; V_{\bullet})}$ and
$\Theta(\overline{D}; V_{\bullet})$ by
$\Delta_D$, $\Delta_D^t$, $G_{\overline{D}}$ and
$\Theta_{\overline{D}}$, respectively.
Let us see the following claim:

\begin{Claim}
\begin{enumerate}
\renewcommand{\labelenumi}{(\arabic{enumi})}
\item
$\Delta_{Q}^t + \nu \subseteq \Delta_{D}^t$ for $t \in \RR$.

\item
$G_{\overline{Q}}(x) \leq G_{\overline{D}}(x + \nu)$ for $x \in \Delta_Q$.

\item
$\Theta_{\overline{Q}} + \nu \subseteq \Theta_{\overline{D}}$.

\item 
$\min \{ \ell(x) \mid x \in \Theta_{\overline{D}} \} \leq \mu_{\QQ,\xi}(\overline{D})$.
\end{enumerate}
\end{Claim}

\begin{proof}
(1) Let $\phi \in \left\langle \aH(X, m (\overline{Q} + (0, -2t[\infty]))) \right\rangle_{\overline{K}} \setminus \{ 0 \}$.
Then
\[
\mult_{z_P}^{\precsim}(Q + (1/m)(\phi)) + \nu = \mult_{z_P}^{\precsim}(D + (1/m)(\phi)),
\]
which shows (1).

(2) Let $t$ be a real number with $t < G_{\overline{Q}}(x)$.
Then $x \in \Delta^t_{Q} \subseteq \Delta_{D}^t - \nu$ by (1), and hence $x + \nu \in \Delta^t_D$.
Thus $t \leq G_{\overline{D}}(x + \nu)$, as required.

(3) follows because $G_{\overline{Q}}(x) > 0$ implies $G_{\overline{D}}(x + \nu) > 0$ by (2).

(4) Let $\phi \in \aH(X, m\overline{D}) \setminus \{ 0 \}$.
Note that $\mult_{z_P}^{\precsim}(D + (1/m) (\phi)) \in \Theta_{\overline{D}}$ and
\[
\ell(\mult_{z_P}^{\precsim}(D + (1/m) (\phi))) = \mult_{\xi}(D + (1/m) (\phi)).
\]
Thus $\min \{ \ell(x) \mid x \in \Theta_{\overline{D}} \} \leq \mult_{\xi}(D + (1/m) (\phi))$.
Therefore, we have (4).
\end{proof}

Since $\avol(\overline{Q}) = \avol(\overline{D})$,
by using the integral formula \eqref{eqn:integral:formula:vol} together with (2) and (3) in the above claim, we can see that
$\Theta_{\overline{Q}} + \nu = \Theta_{\overline{D}}$.
We choose $x_0 \in \Theta_{\overline{D}}$ such that $\ell(x_0) = \min \{ \ell(x) \mid x \in \Theta_{\overline{D}} \}$.
Then there is $y_0 \in \Theta_{\overline{Q}}$ such that $y_0 + \nu = x_0$. As $\ell(y_0) \geq 0$ and $\ell(\nu) = \mult_{\xi}(N)$,
by using (4) in the above claim,
\[
\mu_{\QQ,\xi}(\overline{D}) \geq \min \{ \ell(x) \mid x \in \Theta_{\overline{D}} \} = \ell(x_0) = \ell(y_0) + \ell(\nu) \geq \ell(\nu) = \mult_{\xi}(N),
\]
as required.
\end{proof}

\begin{Remark}
By virtue of Theorem~\ref{thm:zariski:decomp:mu},
we can generalize the necessary and sufficient condition for the existence of Zariski decompositions
on arithmetic toric varieties proved in \cite[Theorem~8.2]{BMPS} to the case of adelic arithmetic $\RR$-divisors.
\end{Remark}

\ifmonog\section{Numerical characterization}\fi
\ifpaper\subsection{Numerical characterization}\fi
\label{subset:numerical:characterization}

We assume that $\dim X = 1$.
The following theorem is the main result of this 
\ifmonog chapter. \fi
\ifpaper section. \fi

\begin{Theorem}
\label{thm:char:nef:general}
Let $\overline{D}$ be an integrable adelic arithmetic $\RR$-Cartier divisor on $X$.
Then $\overline{D}$ is nef if and only if $\overline{D}$ is pseudo-effective and
$\adeg(\overline{D}^2) = \avol(\overline{D})$.
\end{Theorem}

\begin{proof}
We need to show that if $\overline{D}$ is pseudo-effective and
$\adeg(\overline{D}^2) = \avol(\overline{D})$, then $\overline{D}$ is nef because
the converse follows from Proposition~\ref{prop:rel:nef:pseudo:effective} and
Theorem~\ref{thm:GHI:adelic}.

First we assume that $\overline{D}$ is big.
Since
\[
\deg(\overline{D}^2) \leq \avol_{\chi}(\overline{D}) \leq \avol(\overline{D})
\]
by Theorem~\ref{thm:GHI:adelic},
we have $\deg(\overline{D}^2) = \avol_{\chi}(\overline{D})$ and
$\avol_{\chi}(\overline{D}) = \avol(\overline{D})$.
Thus, by Theorem~\ref{thm:GHI:adelic} and Theorem~\ref{thm:vol:cvol:mu:zero},
$\overline{D}$ is relatively nef and $\mu_{\QQ,\xi}(\overline{D}) = 0$ for all $\xi \in X$.
On the other hand, by Corollary~\ref{cor:Zariski:decomp:adelic:arithmetic:divisor}, 
there is the  greatest element $\overline{Q}$ of
$\Upsilon(\overline{D})$. Thus, if we set $\overline{N} := \overline{D} - \overline{Q}$,
then $\mult_{\xi}(N) = \mu_{\QQ,\xi}(\overline{D}) = 0$ for all $\xi \in X$ by Theorem~\ref{thm:zariski:decomp:mu},
which means that $N = 0$. Therefore $\adeg(\rest{D}{x}) \geq 0$ for all closed point $x \in X$, and hence
$\overline{D}$ is nef.

Next we suppose that $\deg(D) > 0$ and $\overline{D}$ is not big. In this case,
$\adeg(\overline{D}^2) = \avol(\overline{D}) = 0$. Thus, for $\epsilon > 0$,
by using Proposition~\ref{prop:vol:comp:C:0},
\[
\epsilon[K : \QQ]\deg(D) = \adeg((\overline{D} + (0, \epsilon[\infty]))^2) \leq
\avol(\overline{D} + (0, \epsilon[\infty])) \leq \epsilon[K : \QQ]\deg(D).
\]
Therefore, $\adeg((\overline{D} + (0, \epsilon[\infty]))^2) = 
\avol(\overline{D} + (0, \epsilon[\infty])) > 0$, so that, by the previous observation,
$\overline{D} + (0, \epsilon[\infty])$ is nef, and hence $\overline{D}$ is also nef.

Finally we consider the case where $\deg(D) = 0$.
Then $\avol(\overline{D}) = 0$, so that $\adeg(\overline{D}^2) = 0$.
By Corollary~\ref{cor:HI:adelic}, $\overline{D} = \widehat{(\psi)}_{\RR} + (0, \lambda [\infty])$
for some $\psi \in \Rat(X)^{\times}_{\RR}$ and $\lambda \in \RR$.
As $\overline{D}$ is pseudo-effective, by (2) in Proposition~\ref{prop:intersection:nef:pseudo:effective},
$\adeg(\overline{A} \cdot \overline{D}) \geq 0$
for any nef adelic arithmetic $\RR$-Cartier divisor $\overline{A}$ of $C^0$-type. Thus
we can see that $\lambda \geq 0$, and hence $\overline{D}$ is nef.
\end{proof}

\begin{Corollary}
\label{cor:characterization:Zariski:decomp}
Let $\overline{D}$ and
$\overline{Q}$ be adelic arithmetic $\RR$-Cartier divisors of $C^0$-type on $X$.
Then the following are equivalent:
\begin{enumerate}
\renewcommand{\labelenumi}{(\arabic{enumi})}
\item
$\overline{Q}$ is the greatest element of $\Upsilon(\overline{D})$, that is,
$\overline{Q} \in \Upsilon(\overline{D})$ and $\overline{L} \leq \overline{Q}$ 
for all $\overline{L} \in \Upsilon(\overline{D})$.

\item
$\overline{Q}$ is an element of $\Upsilon(\overline{D})$ with the following property:
\[
\adeg(\overline{Q} \cdot \overline{B}) = 0\quad\text{and}\quad
\adeg(\overline{B}^2) < 0
\]
for all
integrable adelic arithmetic $\RR$-Cartier divisors $\overline{B}$ of $C^0$-type
with
$0 \lneqq \overline{B} \leq \overline{D} - \overline{Q}$.
\end{enumerate}
\end{Corollary}

\begin{proof}
(1) $\Longrightarrow$ (2) :
By Corollary~\ref{cor:Zariski:decomp:adelic:arithmetic:divisor},
$\avol(\overline{D}) = \avol(\overline{Q})$.
Let $\overline{B}$ be an integrable adelic arithmetic $\RR$-Cartier divisor $\overline{B}$ of $C^0$-type
with $0 \lneqq \overline{B} \leq \overline{D} - \overline{Q}$.
For $0 < \epsilon \leq 1$, 
\[
\adeg((\overline{Q} + \epsilon \overline{B})^2) \leq \avol(\overline{Q} + \epsilon \overline{B})
\]
by Theorem~\ref{thm:GHI:adelic}.
On the other hand, by using Theorem~\ref{thm:G:H:I:T:adelic:arith},
\[
\adeg((\overline{Q} + \epsilon \overline{B})^2) \leq \avol(\overline{Q} + \epsilon \overline{B}) \leq
\avol(\overline{D}) = \avol(\overline{Q}) = \adeg(\overline{Q}^2),
\]
so that $\adeg((\overline{Q} + \epsilon \overline{B})^2) \leq \adeg(\overline{Q}^2)$.
Therefore, $2 \adeg(\overline{Q} \cdot \overline{B}) + \epsilon \adeg(\overline{B}^2) \leq 0$.
In particular, $\adeg(\overline{Q} \cdot \overline{B}) \leq 0$.
Moreover, as $\overline{Q}$ is nef and $\overline{B}$ is effective,
by (2) in Proposition~\ref{prop:intersection:nef:pseudo:effective},
we have $\adeg(\overline{Q} \cdot \overline{B}) \geq 0$, and hence
$\adeg(\overline{Q} \cdot \overline{B}) = 0$.

$\overline{Q} + \overline{B}$ is not nef because $\overline{B} \gneqq 0$, 
so that, by Theorem~\ref{thm:char:nef:general},
\[
\adeg((\overline{Q} + \overline{B})^2) < \avol(\overline{Q} + \overline{B}) = \avol(\overline{Q}) = \adeg(\overline{Q}^2).
\]
Therefore, $\adeg(\overline{B}^2) < 0$.

\medskip
(2) $\Longrightarrow$ (1) :
Let $\overline{L}$ be an element of $\Upsilon(\overline{D})$.
If we set $\overline{A} := \max \{ \overline{Q}, \overline{L} \}$
and $\overline{B} := \overline{A} - \overline{Q}$,
then $\overline{B}$ is effective, $\overline{A} \leq \overline{D}$ and $\overline{A}$ is nef by 
Lemma~\ref{lem:max:nef:adelic:arith}.
Moreover, 
\[
\overline{B} = \overline{A} - \overline{Q} \leq \overline{D} - \overline{Q}.
\]
If we assume $\overline{B} \gneqq 0$, then,
by the property, $\adeg(\overline{Q} \cdot \overline{B}) = 0$ and $\adeg(\overline{B}^2) < 0$.
On the other hand, as $\overline{A}$ is nef and $\overline{B}$ is effective,
\[
0 \leq \adeg(\overline{A} \cdot \overline{B}) = \adeg(\overline{Q} + \overline{B} \cdot \overline{B })
= \adeg(\overline{B}^2),
\]
which is a contradiction, so that $\overline{B} = 0$, that is,
$\overline{Q} = \overline{A}$, which means that $\overline{L} \leq \overline{Q}$, as required.
\end{proof}

\ifmonog
\appendix
\renewcommand{\thesection}{\Alph{chapter}.\arabic{section}}
\renewcommand{\theTheorem}{\Alph{chapter}.\arabic{section}.\arabic{Theorem}}
\renewcommand{\theClaim}{\Alph{chapter}.\arabic{section}.\arabic{Theorem}.\arabic{Claim}}
\renewcommand{\theequation}{\Alph{chapter}.\arabic{section}.\arabic{Theorem}.\arabic{Claim}}
\fi
\ifpaper
\renewcommand{\thesection}{Appendix~\Alph{section}}
\renewcommand{\thesubsection}{\Alph{section}.\arabic{subsection}}
\renewcommand{\theTheorem}{\Alph{section}.\arabic{subsection}.\arabic{Theorem}}
\renewcommand{\theClaim}{\Alph{section}.\arabic{subsection}.\arabic{Theorem}.\arabic{Claim}}
\renewcommand{\theequation}{\Alph{section}.\arabic{subsection}.\arabic{Theorem}.\arabic{Claim}}
\setcounter{section}{0}
\fi

\ifmonog\chapter{Characterization of relatively nef Cartier divisors}\fi
\ifpaper\section{Characterization of relatively nef Cartier divisors}\fi
In this appendix, we consider a characterization of relatively nef Cartier divisors
in terms of asymptotic multiplicities.
Let $k$ be a field and $v$ a complete discrete valuation of $k$. 
Let $\varpi$ be a uniformizing parameter of $k^{\circ}$. Note that
the valuation $v$ is not necessarily non-trivial.

\ifmonog\section{Asymptotic multiplicity}\fi
\ifpaper\subsection{Asymptotic multiplicity}\fi

Let $\XXX$ be a $(d+1)$-dimensional, proper and normal variety over $k^{\circ}$
(cf. Conventions and terminology~\ref{CT:S:variety}), that is,
the Krull dimension of $\XXX$ is $d+1$, $\XXX$ is proper over $\Spec(k^{\circ})$ and
$\XXX$ is integral and normal.
We denote  the rational function field of $\XXX$ by $\Rat(\XXX)$.
\index{\AdelDivSymbol}{0R:Rat(XXX)@$\Rat(\XXX)$}%
Let $\WDiv(\XXX)$ and $\Div(\XXX)$ denote the group of Weil divisors on $\XXX$
and the group of Cartier divisors on $\XXX$, respectively.
\index{\AdelDivSymbol}{0W:WDiv(XXX)@$\WDiv(\XXX)$}%
\index{\AdelDivSymbol}{0Div:Div(XXX)@$\Div(\XXX)$}%
In addition,
for a point $x \in \XXX$,
let $\Div(\XXX;x)$ be the subgroup of $\WDiv(\XXX)$ consisting
of Weil divisors $\DDD$ on $\XXX$ such that
$\DDD = (\phi)$ around $x$ for some $\phi \in \Rat(\XXX)^{\times}$, that is,
$\DDD$ is a Cartier divisor around $x$.
\index{\AdelDivSymbol}{0Div:Div(XXX;x)@$\Div(\XXX;x)$}%
Note that 
\[
\Div(\XXX) \subseteq \Div(\XXX;x) \subseteq \WDiv(\XXX).
\]
For example, if $x$ is a regular point of $\XXX$,
then $\Div(\XXX;x) = \WDiv(\XXX)$.
We set
\[
\begin{cases}
\WDiv(\XXX)_{\RR} := \WDiv(\XXX) \otimes_{\ZZ} \RR, \\
\Div(\XXX)_{\RR} := \Div(\XXX) \otimes_{\ZZ} \RR, \\
\Div(\XXX;x)_{\RR} := \Div(\XXX;x) \otimes_{\ZZ} \RR.
\end{cases}
\]%
\index{\AdelDivSymbol}{0W:WDiv(XXX)_{RR}@$\WDiv(\XXX)_{\RR}$}%
\index{\AdelDivSymbol}{0Div:Div(XXX)_{RR}@$\Div(\XXX)_{\RR}$}%
\index{\AdelDivSymbol}{0Div:Div(XXX;x)_{RR}@$\Div(\XXX;x)_{\RR}$}%
Let $\KK$ be either $\QQ$ or $\RR$ and
let $\Rat(\XXX)^{\times}_{\KK} := (\Rat(\XXX)^{\times}, \times) \otimes_{\ZZ} \KK$.
\index{\AdelDivSymbol}{0R:Rat(XXX)^{times}_{KK}@$\Rat(\XXX)^{\times}_{\KK}$}%
Note that the homomorphism
\[
(\ \cdot\ ) : \Rat(\XXX)^{\times} \to \Div(\XXX)
\]
given by $f \mapsto (f)$ has the natural
extension
\[
(\ \cdot\ ) : \Rat(\XXX)^{\times}_{\KK} \to \Div(\XXX)_{\RR},
\]
that is, for $\phi = f_1^{\otimes a_1} \cdots f_r^{\otimes a_r} \in \Rat(\XXX)^{\times}_{\KK}$ 
($f_1, \ldots, f_r \in \Rat(\XXX)^{\times}$,
$a_1, \ldots, a_r \in \KK$),
\[
(\phi) = a_1(f_1) + \cdots + a_r (f_r).
\]
Let $\DDD$ be an $\RR$-Weil divisor on $\XXX$, that is, $\DDD \in \WDiv(\XXX)_{\RR}$.
We define $\Gamma^{\times}(\XXX, \DDD)$ and $\Gamma^{\times}_{\KK}(\XXX, \DDD)$ to be
\[
\begin{cases}
\Gamma^{\times}(\XXX, \DDD) := \left\{ \phi \in \Rat(\XXX)^{\times} \mid \DDD + (\phi) \geq 0 \right\}, \\[1ex]
\Gamma^{\times}_{\KK}(\XXX, \DDD) := \left\{ \phi \in \Rat(\XXX)^{\times}_{\KK} \mid \DDD + (\phi) \geq 0 \right\}.
\end{cases}
\]%
\index{\AdelDivSymbol}{0G:Gamma^{times}(XXX,DDD)@$\Gamma^{\times}(\XXX, \DDD)$}%
\index{\AdelDivSymbol}{0G:Gamma^{times}_{KK}(XXX,DDD)@$\Gamma^{\times}_{\KK}(\XXX, \DDD)$}%
Let $w : \Rat(\XXX) \to \RR \cup \{ \infty \}$ be an additive discrete valuation over $k$. Namely,
$w$ satisfies the following  conditions:
\begin{enumerate}
\renewcommand{\labelenumi}{(\arabic{enumi})}
\item
$w(f\cdot g) = w(f) + w(g)$ for all $f, g \in \Rat(\XXX)$.

\item
$w(f + g) \geq \min \{ w(f), w(g) \}$ for all $f, g \in \Rat(\XXX)$.

\item
$f = 0$ if and only if $w(f) = \infty$.

\item
$w(a) = -\log v(a) $ for all $a \in k^{\times}$.
\end{enumerate}
Let $\OOO_w$ be the valuation ring of $w$ and $m_w$ its maximal ideal, that is,
\[
\OOO_w = \left\{ f \in \Rat(\XXX) \mid w(f) \geq 0 \right\}\quad\text{and}\quad
m_w = \left\{ f \in \Rat(\XXX) \mid w(f) > 0 \right\}.
\]
Note that $k^{\circ} \subseteq \OOO_w$ and
$k^{\circ\circ} \subseteq m_w$, so that $\OOO_w/m_w$ is a $k^{\circ}/k^{\circ\circ}$-algebra.
We say $w$ is a {\em divisorial valuation of $\Rat(\XXX)$ over $k$} 
\index{\AdelDivSubject}{divisorial valuation@divisorial valuation}%
if $\trdeg_{k^{\circ}/k^{\circ\circ}} \OOO_w/m_w = d$.
For a divisorial valuation $w$ of $\Rat(\XXX)$ over $k$,
there are  a normal variety $\VVV$ over $k^{\circ}$, a vertical prime divisor $\Gamma$ on $\VVV$ and a birational morphism
$\mu : \VVV \to \XXX$ over $\Spec(k^{\circ})$
such that  
$w = a \ord_{\Gamma}$ for some $a \in \RR_{>0}$.
Indeed, it can be shown as follows:
We may assume that $v$ is non-trivial. Otherwise the assertion follows from \cite[Chapter~VI, \S~14, Theorem~31]{ZS}.
We choose $x_1, \ldots, x_d \in \OOO_w$ such that $x_1, \ldots, x_d$ form a transcendental basis
of $\OOO_w/m_w$ over $k^{\circ}/k^{\circ\circ}$. Then $\Rat(\XXX)$ is a finite extension of $k(x_1, \ldots, x_d)$ and
the transcendental degree of $k^{\circ}[x_1, \ldots, x_d]/k^{\circ}[x_1, \ldots, x_d] \cap m_w$ over 
$k^{\circ}/k^{\circ\circ}$ is $d$.
Let $R$ be the normalization of $k^{\circ}[x_1, \ldots, x_d]$ in $\Rat(\XXX)$.
Note that $R$ is finite over $k^{\circ}[x_1, \ldots, x_d]$ because $k^{\circ}$ is excellent.
In addition, $R \subseteq \OOO_w$, 
$R \cap m_w$ is a prime ideal of $R$ and 
$\trdeg_{k^{\circ}/k^{\circ}}(R/R \cap m_w) = d$, 
which prove the assertion.

We denote the set of all divisorial valuations of $\Rat(\XXX)$ over $k$ by $\DVal_k(\XXX)$.
\index{\AdelDivSymbol}{0DVal:DVal_k(XXX)@$\DVal_k(\XXX)$}%
As $\XXX$ is proper and separated over $\Spec(k^{\circ})$, there is a unique morphism $t : \Spec(\OOO_w) \to X$ such that
the following diagram is commutative:
\[
\xymatrix{
\Spec(\Rat(\XXX)) \ar[r] \ar[d] &  \XXX \ar[d] \\
\Spec(\OOO_w) \ar[ru]^t \ar[r] & \Spec(k^{\circ}) \\
}
\]
Let $x$ be the image of the closed point $m_w$ by $t$.
The point $x$ is called the {\em center of $w$ on $\XXX$}.
\index{\AdelDivSubject}{center of valuation@center of valuation}%
Note that $x \in \XXX_{\circ}$ (the central fiber of $\XXX \to \Spec(k^{\circ})$).
For $\DDD \in \Div(\XXX;x)$, $\mult_w(\DDD)$ is defined by $w(f)$,
where $f$ is a local equation of $\DDD$ at $x$.
In this way, we have a map 
\[
\mult_w : \Div(\XXX;x) \to \ZZ.
\]%
\index{\AdelDivSymbol}{0m:mult_w@$\mult_w$}%
It is easy to see that $\mult_w$ is a homomorphism, so that we have the natural extension
\[
\mult_w : \Div(\XXX;x)_{\RR} \to \RR,
\]
that is,
\[
\mult_w(a_1 \DDD_1 + \cdots + a_r \DDD_r) = a_1 \mult_w(\DDD_1) + \cdots + a_r \mult_w(\DDD_r),
\]
where $\DDD_1, \ldots, \DDD_r \in \Div(\XXX;x)$ and $a_1, \ldots, a_r \in \RR$.

For $\DDD \in \Div(\XXX;x)_{\RR}$,
we define $\mu_{\KK, w}(\DDD)$ to be
\[
\mu_{\KK, w}(\DDD) :=
\begin{cases}
\inf \left\{ \mult_w(\DDD + (\phi)) \mid \phi \in \Gamma^{\times}_{\KK}(\XXX,\DDD) \right\} 
& \text{if $\Gamma^{\times}_{\KK}(\XXX,\DDD) \not= \emptyset$}, \\
\infty & \text{otherwise},
\end{cases}
\]
which is called the {\em $\KK$-asymptotic multiplicity of $\DDD$ at $w$}.
\index{\AdelDivSubject}{K-asymptotic multiplicity@$\KK$-asymptotic multiplicity}%
\index{\AdelDivSymbol}{0m:mu_{KK,w}(DDD)@$\mu_{\KK, w}(\DDD)$}%
Here we give one additional definition.
An $\RR$-Cartier divisor $\DDD$ (i.e. $\DDD \in \Div(\XXX)_{\RR}$)
is said to be {\em big}
\index{\AdelDivSubject}{big R-Cartier divisor@big $\RR$-Cartier divisor}%
if $\DDD$ is big on the generic fiber $\XXX \to \Spec(k^{\circ})$.
First let us observe elementary properties of the asymptotic multiplicity.
The arithmetic version can be found in \cite[Proposition~2.1 and Theorem~2.5]{MoArLinB}
and Proposition~\ref{prop:ar:mu:basic}.

\begin{Proposition}
\label{prop:mu:basic}
Let $w$ be a divisorial valuation of $\Rat(\XXX)$ over $k$ and $x$ the center of $w$ on $\XXX$.
For $\DDD, \EEE \in \Div(\XXX;x)_{\RR}$, we have the following:
\begin{enumerate}
\renewcommand{\labelenumi}{(\arabic{enumi})}
\item
$\mu_{\KK, w}(\DDD + \EEE) \leq 
\mu_{\KK, w}(\DDD) + \mu_{\KK, w}(\EEE)$.

\item
If $\DDD \leq \EEE$, then 
$\mu_{\KK, w}(\EEE) \leq \mu_{\KK, w}(\DDD) + \mult_{w}(\EEE - \DDD)$.

\item
$\mu_{\KK, w}(\DDD + (\phi)) = \mu_{\KK, w}(\DDD)$ for 
$\phi \in \Rat(\XXX)^{\times}_{\KK}$.

\item
$\mu_{\KK, w}(a\DDD) = a \mu_{\KK, w}(\DDD)$ for $a \in \KK_{>0}$.

\item
$0 \leq \mu_{\RR, w}(\DDD) \leq \mu_{\QQ, w}(\DDD)$.

\item
Let $\nu : \YYY \to \XXX$ be a birational morphism of proper and normal varieties over $k^{\circ}$.
\begin{enumerate}
\renewcommand{\labelenumii}{(\arabic{enumi}.\arabic{enumii})}
\item
If $\DDD$ is an $\RR$-Cartier divisor on $\XXX$, then
$\mu_{\KK,w}(\nu^*(\DDD)) = \mu_{\KK,w}(\DDD)$.

\item
Let $x$ and $y$ be the centers of $w$ on $\XXX$ and $\YYY$, respectively \rom{(}note that $\nu(y) = x$\rom{)}.
We assume that $\nu$ is an isomorphism over $x$.
Then, for $\DDD' \in \Div(\YYY; y)$,
\[
\mu_{\KK,w}(\nu_*(\DDD')) \leq \mu_{\KK,w}(\DDD').
\]
\end{enumerate}

\item
If $\DDD$ is an $\RR$-Cartier divisor on $\XXX$ and
$\DDD$ is relatively nef with respect to
$\XXX \to \Spec(k^{\circ})$
\rom{(}cf. Conventions and terminology~\rom{\ref{CT:rel:nef}}\rom{)}
and big, then $\mu_{\KK, w}(\DDD) = 0$.

\item
If $\DDD$ is an $\RR$-Cartier divisor on $\XXX$  and
$\DDD$ is big, then $\mu_{\QQ,w}(\DDD) = \mu_{\RR,w}(\DDD)$.
\end{enumerate}
\end{Proposition}

\begin{proof}
(1)  
If $\Gamma^{\times}_{\KK}(\XXX, \DDD+\EEE) = \emptyset$,
then either $\Gamma^{\times}_{\KK}(\XXX, \DDD) = \emptyset$ or
$\Gamma^{\times}_{\KK}(\XXX, \EEE) = \emptyset$, so that
we may assume that $\Gamma^{\times}_{\KK}(\XXX, \DDD+\EEE) \not= \emptyset$.
Thus we may also assume that $\Gamma^{\times}_{\KK}(\XXX, \DDD) \not= \emptyset$ and
$\Gamma^{\times}_{\KK}(\XXX, \EEE) \not= \emptyset$.
Therefore, the assertion follows because $\phi \psi \in \Gamma^{\times}_{\KK}(\XXX, \DDD+\EEE)$ for all 
$\phi \in \Gamma^{\times}_{\KK}(\XXX,\DDD)$ and
$\psi \in \Gamma^{\times}_{\KK}(\XXX,\EEE)$.

(2) is derived from (1).

(3) 
The assertion follows from the following:
\[
\psi \in \Gamma^{\times}_{\KK}(\XXX, \DDD)\quad\Longleftrightarrow \quad
\psi\phi^{-1} \in \Gamma^{\times}_{\KK}(\XXX, \DDD+ (\phi)).
\]

(4) Note that $\psi \in \Gamma^{\times}_{\KK}(\XXX,\DDD)$ if and only if
$\psi^a \in \Gamma^{\times}_{\KK}(\XXX, a\DDD)$, and that
\[
\mult_{w}(a \DDD + (\psi^a)) = a \mult_{v}(\DDD + (\psi)),
\]
which implies (4).

(5) is obvious.

\medskip
(6.1)
For $\phi \in \Rat(\XXX)^{\times}_{\KK}$,
$\DDD + (\phi)_{\XXX} \geq 0$ if and only if $\nu^*(\DDD) + (\phi)_{\YYY} \geq 0$.
Thus 
\[
\Gamma^{\times}_{\KK}(\XXX, \DDD) = \Gamma^{\times}_{\KK}(\YYY, \nu^*(\DDD)).
\]
Moreover,
\[
\mult_w(\DDD + (\phi)_{\XXX}) = \mult_w(\nu^*(\DDD) + (\phi)_{\YYY}).
\]
Therefore, we have (6.1).

(6.2) Let $\phi \in \Gamma^{\times}_{\KK}(\YYY, \DDD')$, that is,
$\DDD' + (\phi)_{\YYY} \geq 0$. Then
\[
0 \leq \nu_*(\DDD' + (\phi)_{\YYY}) = \nu_*(\DDD') + (\phi)_{\XXX}.
\]
The above observation means that 
$\Gamma^{\times}_{\KK}(\YYY, \DDD') \subseteq \Gamma^{\times}_{\KK}(\XXX, \nu_*(\DDD'))$.
Moreover, by our assumption,
\[
\mult_w(\DDD' + (\phi)_{\YYY}) = \mult_w(\nu_*(\DDD') + (\phi)_{\XXX})
\]
for $\phi \in \Gamma^{\times}_{\KK}(\YYY, \DDD')$.
Thus the assertion follows.

\medskip
(7) Let us begin with the following claim:

\begin{Claim}
\label{claim:prop:mu:basic:01}
If $\XXX$ is projective over $\Spec(k^{\circ})$, then, for any Cartier divisor $\EEE$ on $\XXX$,
there are effective Cartier divisors $\AAA$ and $\BBB$ on $\XXX$ such that
$\EEE = \AAA - \BBB$.
\end{Claim}

\begin{proof}
Let $\HHH$ be an ample Cartier divisor on $\XXX$.
Let $\EEE = e_1 \Gamma_1 + \cdots + e_r \Gamma_r$ be the decomposition as a Weil divisor.
As $\HHH$ is ample, for a sufficiently large $l$,
there is $\phi \in H^0(\XXX, l\HHH) \setminus \{ 0 \}$
such that $\AAA := l\HHH + (\phi)$ is effective
and $\mult_{\Gamma_i}(\AAA)  \geq  e_i$ for $i=1, \ldots, r$.
Thus $\BBB := \AAA - \EEE$ is effective and $\EEE = \AAA - \BBB$.
\end{proof}

Let us go back to the proof of (7).
By (5), it is sufficient to show that $\mu_{\QQ,w}(\DDD) = 0$.
By using Chow's lemma together with (6.1),
we may assume that $\XXX$ is projective over $\Spec(k^{\circ})$.
First we assume that $\DDD$ is an ample $\QQ$-Cartier divisor.
Then there is $\phi \in \Gamma^{\times}_{\QQ}(\XXX, \DDD)$ such that
$\mult_w(\DDD + (\phi)) = 0$, and hence $\mu_{\QQ, w}(\DDD) = 0$.

Next we assume that $\DDD$ is ample. By using Claim~\ref{claim:prop:mu:basic:01},
we can set 
\[
\DDD = a_1 \DDD_1 + \cdots + a_r \DDD_r,
\]
where $\DDD_1, \ldots, \DDD_r$ are effective Cartier divisors and $a_1, \ldots, a_r \in \RR$.
For any $n > 0$, there are $\delta_1, \ldots, \delta_r \in \RR$ such that
$0 < \delta_i < 1/n$ and $a_i - \delta_i \in \QQ$ for all $i$ and that
$(a_1 - \delta_1)\DDD_1 + \cdots + (a_r - \delta_r)\DDD_r$ is ample.
Then, by (2) and the previous case,
\begin{align*}
\mu_{\QQ, w}(\DDD) & 
\leq \mu_{\QQ, w}((a_1 - \delta_1)\DDD_1 + \cdots + (a_r - \delta_r)\DDD_r) + \mult_{\QQ, w}(\delta_1\DDD_1 + \cdots + \delta_r\DDD_r) \\
& \leq \delta_1\mult_w(\DDD_1) + \cdots + \delta_r\mult_w(\DDD_r) \\
& \leq (1/n)(\mult_w(\DDD_1) + \cdots + \mult_w(\DDD_r)),
\end{align*}
which proves the assertion in this case.

Let us consider a general case. 

\begin{Claim}
\label{claim:prop:mu:basic:02}
There are an ample $\QQ$-Cartier divisor $\AAA$ on $\XXX$ and
$\phi \in \Rat(\XXX)^{\times}_{\QQ}$ such that
$\EEE := \DDD - \AAA + (\phi)$ is effective.
\end{Claim}

\begin{proof}
If $v$ is trivial, then $k^{\circ} = k$, so that the assertion is obvious.
We assume that $v$ is non-trivial.
Let $\AAA'$ be an ample Cartier divisor on $\XXX$. Let $X$ be the generic fiber
of $\XXX \to \Spec(k^{\circ})$, $D := \DDD \cap X$ and $A' := \AAA' \cap X$.
Then, as $D$ is big, there are $n \in \ZZ_{>0}$ and $\phi_1 \in \Rat(\XXX)^{\times}$ such that
\[
nD - A' + (\phi_1) \geq 0
\]
on $X$.
Therefore, we can find $m \in \ZZ_{> 0}$ such that $n\DDD - \AAA' + (\phi_1) + m(\varpi) \geq 0$,
and hence,
$\DDD - (1/n)\AAA' + (\phi_1^{1/n} \varpi^{m/n}) \geq 0$, as required.
\end{proof}

As $\AAA + (1-\epsilon)\EEE = \epsilon \AAA + (1- \epsilon)(\DDD + (\phi))$ is ample for $0 < \epsilon < 1$,
by using (2), (3) and the previous assertion in the case where $\DDD$ is ample,
\begin{align*}
\mu_{\QQ, w}(\DDD) & = \mu_{\QQ, w}(\AAA + \EEE + (\phi^{-1})) 
= \mu_{\QQ, w}(\AAA + \EEE) = \mu_{\QQ, w}(\AAA + (1-\epsilon)\EEE + \epsilon \EEE) \\
& \leq \mu_{\QQ, w}(\AAA + (1-\epsilon)\EEE) + \epsilon\mult_w(\EEE) \leq \epsilon\mult_w(\EEE),
\end{align*}
and hence $\mu_{\QQ, w}(\DDD) = 0$.

\medskip
(8) In the same way as (7), we may assume that $\XXX$ is projective over $\Spec(k^{\circ})$.
Let $\phi = \phi_1^{a_1} \cdots \phi_r^{a_r} \in \Gamma^{\times}_{\RR}(\XXX, \DDD)$,
where $\phi_1, \ldots, \phi_r \in \Rat(\XXX)^{\times}$ and $a_1, \ldots, a_r \in \RR$.
By Claim~\ref{claim:prop:mu:basic:01}, for each $i$, there are effective Cartier divisors $\AAA_i$ and $\BBB_i$ on $\XXX$ such that
$(\phi_i) = \AAA_i - \BBB_i$.
Here we consider a map $f : \QQ^r \to \RR \cup \{ \infty \}$ given by
\[
f(t_1, \ldots, t_r) = \mu_{\QQ, w}(\DDD + t_1 \BBB_1 + \cdots + t_r \BBB_r).
\]

\begin{Claim}
\label{claim:prop:mu:basic:03}
$\lim\limits_{\substack{(t_1, \ldots, t_r) \to (0,\ldots, 0) \\ (t_1, \ldots, t_r) \in \QQ^r}} 
f(t_1, \ldots, t_r)
= f(0,\ldots, 0) = \mu_{\QQ,w}(\DDD)$.
\end{Claim}

\begin{proof}
First note that there is a positive rational number $c$ such that
$\DDD + t_1 \BBB_1 + \cdots + t_r \BBB_r$ is big for $(t_1, \ldots, t_r) \in (-c, \infty)^{r} \cap \QQ^r$.
Moreover, by using (1) and (4), we can see that $f$ is a convex function over $\QQ$, that is,
$f(\lambda t + (1-\lambda) t') \leq \lambda f(t) + (1-\lambda) f(t')$ for $t, t' \in \QQ^r$ and
$\lambda \in [0, 1] \cap \QQ$. Therefore, by virtue of \cite[Proposition~1.3.1]{MoArLin},
there is a continuous function $\tilde{f} : (-c, \infty)^r \to \RR$ such that $\tilde{f} = f$
on $(-c, \infty)^r \cap \QQ^r$, which shows the assertion of the claim.
\end{proof}

For each $i=1, \ldots, r$ and $n \in \ZZ_{>0}$, we choose $t_{i,n} \in \RR$ and $t'_{i,n} \in \QQ$ such that
$a_i + t_{i,n} \in \QQ$ and $0 \leq t_{i,n} \leq t'_{i,n} \leq 1/n$.
Then 
\begin{multline*}
\DDD + \sum\nolimits_{i} t'_{i,n} \BBB_i + \sum\nolimits_{i} (a_i + t_{i,n}) (\phi_i) \geq
\DDD + \sum\nolimits_{i} t_{i,n} \BBB_i + \sum\nolimits_{i} (a_i + t_{i,n}) (\phi_i) \\
= \DDD + \sum\nolimits_{i} a_i (\phi_i) + \sum\nolimits_{i} t_{i,n} \AAA_i \geq 0,
\end{multline*}
and hence
\[
\mu_{\QQ,w}\left(\DDD + \sum\nolimits_{i} t'_{i,n} \BBB_i\right) \leq
\mult_{w}\left(\DDD + \sum\nolimits_{i} t'_{i,n} \BBB_i + \sum\nolimits_{i} (a_i + t_{i,n}) (\phi_i)\right).
\]
Thus, taking the limits as $n \to \infty$ together with Claim~\ref{claim:prop:mu:basic:03}, we have
\[
\mu_{\QQ,w}(\DDD) \leq \mult_{w}\left(\DDD + \sum\nolimits_{i} a_i(\phi_i)\right),
\]
which gives rise to $\mu_{\QQ,w}(\DDD) \leq \mu_{\RR,w}(\DDD)$, so that
(8) follows from (5).
\end{proof}

\ifmonog\section{Sectional decomposition}\fi
\ifpaper\subsection{Sectional decomposition}\fi
Let $\XXX$ be a regular and proper variety over $k^{\circ}$.
Let $\DDD$ be an $\RR$-Cartier divisor on $\XXX$.
We assume that $\Gamma^{\times}(\XXX, \DDD) \not= \emptyset$.
We set 
\[
\begin{cases}
\Bs(\DDD) := \bigcap_{\phi \in \Gamma^{\times}(\XXX, \DDD)} \Supp_{\RR}(\DDD + (\phi)), \\
\FFF(\DDD) := \sum_{\Gamma} \inf \left\{ \mult_{\Gamma}(\DDD + (\phi)) \mid \phi \in \Gamma^{\times}(\XXX, \DDD) \right\} \cdot \Gamma, \\
\PPP(\DDD) := \DDD - \FFF(\DDD),
\end{cases}
\]
where $\Gamma$ runs over all prime divisors on $\XXX$.
\index{\AdelDivSymbol}{0B:Bs(DDD)@$\Bs(\DDD)$}%
\index{\AdelDivSymbol}{0F:FFF(DDD)@$\FFF(\DDD)$}%
\index{\AdelDivSymbol}{0P:PPP(DDD)@$\PPP(\DDD)$}%
Note that the above ``$\inf$'' can be replaced by ``$\min$'' because the set
$\left\{ \mult_{\Gamma}(\DDD + (\phi)) \mid \phi \in \Gamma^{\times}(\XXX, \DDD) \right\}$ is discrete in $\RR$.
The decomposition $\DDD = \PPP(\DDD) + \FFF(\DDD)$ is called the {\em sectional decomposition of $\DDD$}.
\index{\AdelDivSubject}{sectional decomposition@sectional decomposition}%

\begin{Lemma}
\label{lem:properties:sectional:decomp}
\begin{enumerate}
\renewcommand{\labelenumi}{(\arabic{enumi})}
\item
The natural inclusion map 
\[
H^0(\XXX, \PPP(\DDD)) \to H^0(\XXX, \DDD)
\]
is bijective.

\item
$\codim \Bs(\PPP(\DDD)) \geq 2$.
\end{enumerate}
\end{Lemma}

\begin{proof}
By our construction, $\DDD + (\phi) \geq \FFF(\DDD)$ for all $\phi \in \Gamma^{\times}(\XXX, \DDD)$.
Thus (1) follows. Moreover, if $\codim \Bs(\PPP(\DDD)) = 1$,
then there is a prime divisor $\Gamma$ such that 
\[
\mult_{\Gamma}(\PPP(\DDD) + (\phi)) > 0
\]
for all $\phi \in \Gamma^{\times}(\XXX, \DDD)$, that is,
$\mult_{\Gamma}(\DDD + (\phi)) > \mult_{\Gamma}(\FFF(\DDD))$ for all $\phi \in \Gamma^{\times}(\XXX, \DDD)$,
which is a contradiction.
\end{proof}

From now on, we assume that $\Gamma^{\times}_{\QQ}(\XXX, \DDD) \not= \emptyset$.
We set 
\[
N(\DDD) := \{ m \in \ZZ_{\geq 1} \mid \Gamma^{\times}(\XXX, m\DDD) \not= \emptyset \}.
\]%
\index{\AdelDivSymbol}{0N:N(DDD)@$N(\DDD)$}%
Note that $N(\DDD) \not= \emptyset$.
For $m \in N(\DDD)$, we set $\FFF_m := \FFF(m\DDD)$ and $\PPP_m := \PPP(m\DDD)$.

\begin{Lemma}
\label{lem:properties:sectional:decomp:m}
\begin{enumerate}
\renewcommand{\labelenumi}{(\arabic{enumi})}
\item
$\FFF_{m} + \FFF_{m'} \geq \FFF_{m+m'}$ for $m, m' \in N(\DDD)$.
In particular,
\[
\inf_{m \in N(\DDD)} \left\{ \frac{\mult_{w}(\FFF_m)}{m} \right\} = 
\lim_{\substack{m \to \infty,\\  m \in N(\DDD)}} \frac{\mult_{w}(\FFF_m)}{m}
\]
for all $w \in \DVal_k(\XXX)$ \rom{(}cf. \cite[Chapter~3, {\bf 98}]{PS}\rom{)}.

\item ${\displaystyle \mu_{\QQ, \Gamma}(\DDD) = \inf_{m \in N(\DDD)} \left\{ \frac{\mult_{\Gamma}(\FFF_m)}{m} \right\}}$
for all prime divisors $\Gamma$ on $\XXX$.
\end{enumerate}
\end{Lemma}

\begin{proof}
(1) is obvious because $\phi \phi' \in \Gamma^{\times}(\XXX, (m+m')\DDD)$ 
for all $\phi \in \Gamma^{\times}(\XXX, m\DDD)$ and
$\phi' \in \Gamma^{\times}(\XXX, m'\DDD)$.
For (2), note that
\[
\Gamma^{\times}_{\QQ}(\XXX, \DDD) = \bigcup_{m \in N(\DDD)} \left( \Gamma^{\times}(\XXX, m\DDD) \right)^{1/m}.
\]
\end{proof}

\ifmonog\section{Characterization in terms of $\mu_w$}\fi
\ifpaper\subsection{Characterization in terms of $\mu_w$}\fi

The following theorem is a characterization of relatively nef Cartier divisors in terms of the asymptotic multiplicity.

\begin{Theorem}
\label{thm:char:nef:big:alg:var}
Let $\XXX$ be a $(d+1)$-dimensional, proper and normal variety over $k^{\circ}$ and
let $\DDD$ be an $\RR$-Cartier divisor on $\XXX$.
If $\Gamma^{\times}_{\QQ}(\XXX, \DDD) \not= \emptyset$ and $\mu_{\QQ, w}(\DDD) = 0$ for all $w \in \DVal_k(\XXX)$,
then $\DDD$ is relatively nef.
In particular, if $\DDD$ is big, then the following are equivalent:
\begin{enumerate}
\renewcommand{\labelenumi}{(\arabic{enumi})}
\item
$\DDD$ is relatively nef with respect to $\XXX \to \Spec(k^{\circ})$.

\item
$\mu_{\QQ, w}(\DDD) = 0$ for all $w \in \DVal_k(\XXX)$.

\item
$\mu_{\RR, w}(\DDD) = 0$ for all $w \in \DVal_k(\XXX)$.
\end{enumerate}
\end{Theorem}

\begin{proof}
Let us begin with the following claim:

\begin{Claim}
\label{claim:thm:char:nef:alg:var:01}
Let $\YYY$ be a normal and proper variety over $k^{\circ}$ and
let $\nu : \YYY \to \XXX$ be a dominant morphism over $\Spec(k^{\circ})$ such that
$\Rat(\YYY)$ is algebraic over $\Rat(\XXX)$.
If
\[
\Gamma^{\times}_{\QQ}(\XXX, \DDD) \not= \emptyset
\quad\text{and}\quad
\mu_{\QQ, w}(\DDD) = 0
\]
for all $w \in \DVal_k(\XXX)$, then
$\mu_{\QQ, w'}(\nu^*(\DDD)) = 0$ for all $w' \in \DVal_k(\YYY)$.
\end{Claim}

\begin{proof}
Let $w'$ be a divisorial valuation of $\Rat(\YYY)$ over $k$ and let $w$ be the restriction of $w'$ to $\Rat(\XXX)$.
As $\Rat(\YYY)$ is algebraic over $\Rat(\XXX)$, we can see that $\OOO_{w'}/m_{w'}$ is algebraic over $\OOO_w/m_w$, so that
$w$ is a divisorial valuation of $\Rat(\XXX)$ over $k$.
Then,
for an $\RR$-Cartier divisor $\LLL$ on $\XXX$, $\mult_{w'}(\nu^*(\LLL)) = \mult_w(\LLL)$.
Thus, 
\begin{align*}
\mu_{\QQ, w'}(\nu^*(\DDD)) & 
= \inf \left\{ \mult_{w'}(\nu^*(\DDD) + (\psi)) \mid \psi \in \Gamma^{\times}_{\QQ}(\YYY, \nu^*(\DDD)) \right\} \\
& \leq \inf \left\{ \mult_{w'}(\nu^*(\DDD + (\phi))) \mid \phi \in \Gamma^{\times}_{\QQ}(\XXX, \DDD) \right\} \\
& = \inf \left\{ \mult_w(\DDD + (\phi)) \mid \phi \in \Gamma^{\times}_{\QQ}(\XXX, \DDD) \right\} = \mu_{\QQ, w}(\DDD),
\end{align*}
which prove the claim.
\end{proof}

\begin{Claim}
\label{claim:thm:char:nef:alg:var:02}
We may assume that
$\XXX$ is regular and projective over $\Spec(k^{\circ})$.
\end{Claim}

\begin{proof}
We assume that the theorem holds if $\XXX$ is regular and projective.
By de Jong's theorem \cite{deJong},
there is a regular and projective variety $\YYY$ over $k^{\circ}$ together with
a dominant morphism $\mu :  \YYY \to \XXX$ over $\Spec(k^{\circ})$ such that
$\Rat(\YYY)$ is algebraic over $\Rat(\XXX)$.
By the previous claim and our assumption, 
we can see that $\nu^*(\DDD)$ is relatively nef, so that
$\DDD$ is also relatively nef.
\end{proof}

Let $C$ be an irreducible and reduced curve on  $\XXX_{\circ}$.
Let us see $(\DDD \cdot C) \geq 0$.
Clearly we may assume that $\DDD$ is effective  and $C \subseteq \Supp_{\RR}(\DDD)$.
There is a  succession of blowing-ups $\rho : \tilde{\XXX} \to \XXX$ at closed points such that
the strict transform $\tilde{C}$ of $C$ is regular (cf. \cite[Theorem~1.101]{KoSing}).
If $(\rho^*(\DDD) \cdot \tilde{C}) \geq 0$, then 
$(\DDD \cdot C) = (\rho^*(\DDD) \cdot \tilde{C})  \geq 0$,
so that we may assume that $C$ is regular.

Let $\pi : \YYY \to \XXX$ be the blowing-up along $C$ and
let $\EEE$ be the exceptional set of $\pi$. 
Let $\DDD'$ be the strict transform of $\DDD$.
Then $\pi^*(\DDD) = \DDD' + e \EEE$ for some $e \in \ZZ_{>0}$.
Let $H$ be a very ample divisor on $\EEE$.
Choosing general members $H_1, \ldots, H_{d-1}$ of $| H |$,
we set $C' = H_1 \cap \cdots \cap H_{d-1}$ and $\pi_*(C') = aC$ for some $a \in \ZZ_{>0}$.
As $H_1, \ldots, H_{d-1}$ are general, $C' \not\subseteq \Supp_{\RR}(\DDD') \cap \EEE$.
If $(\EEE \cdot C') = (\rest{\OOO_{\XXX}(\EEE)}{\EEE} \cdot H^{d-1}) \geq 0$, then
\[
a(\DDD \cdot C) = (\pi^*(\DDD) \cdot C') = (\DDD' \cdot C') + e(\EEE \cdot C') \geq 0.
\]
Thus we may assume $(\rest{\OOO_{\XXX}(\EEE)}{\EEE} \cdot H^{d-1}) < 0$.

Let $m \pi^*(\DDD) = \PPP_m + \FFF_m$ be the sectional decomposition of $m \pi^*(\DDD)$.
By virtue of (1) in Lemma~\ref{lem:properties:sectional:decomp:m}, 
there are finitely many prime divisors $\Gamma_1, \ldots, \Gamma_r$ of $\YYY$ such that
\[
\FFF_m = a_{m,1} \Gamma_1 + \cdots + a_{m, r} \Gamma_r
\]
for some $a_{m,1},\ldots, a_{m,r} \in \RR_{\geq 0}$. 
First we assume that $\Gamma_i \not= \EEE$ for all $i=1, \ldots, r$.
Choosing general members $H_1, \ldots, H_{d-2} \in |H|$ if necessarily, we have
\[
C' \not\subseteq \left( \Bs(\PPP_m) \cup \Gamma_{1} \cup \cdots \cup \Gamma_r \right) \cap \EEE,
\]
and hence
\[
a m (\DDD \cdot C) =  (\pi^*(m \DDD) \cdot C') =  (\PPP_m \cdot C') + \sum_{i=1}^r a_{m,i} (\Gamma_i \cdot C') \geq 0.
\]
Therefore, we may assume that $\Gamma_1 = \EEE$.
By (2) in Lemma~\ref{lem:properties:sectional:decomp:m} and our assumption,
${\displaystyle \lim_{m \to \infty} a_{m,1}/m = 0}$.
For any $\epsilon > 0$, we choose $m$ such that
$0 \leq a_{m, 1}/m \leq \epsilon$. As before,
choosing general members $H_1, \ldots, H_{d-2} \in |H|$ if necessarily,
\[
C' \not\subseteq \left( \Bs(\PPP_m) \cup \Gamma_{2} \cup \cdots \cup \Gamma_r \right) \cap \EEE
\]
holds, so that
$(\PPP_m \cdot C') \geq 0$ and $(\Gamma_{i} \cdot C') \geq 0$ for $i=2, \ldots, r$.
Thus
\[
a(\DDD \cdot C) =  (\pi^*(\DDD) \cdot C') \geq (a_{m, 1}/m) (\EEE \cdot C') \geq \epsilon 
(\rest{\OOO_{\XXX}(\EEE)}{\EEE} \cdot H^{d-2}).
\]
Therefore $(\DDD \cdot C) \geq 0$ because $\epsilon$ is an arbitrary small number.

\medskip
Finally,
the last assertion of the theorem follows from the first assertion, (7) and (8) in Proposition~\ref{prop:mu:basic}.
\end{proof}

As a corollary, we have the following characterization of relatively nef $\RR$-Cartier divisors.
It is proved in \cite[Theorem~5.11 and Lemma~5.12]{BFJ} in the case where the characteristic of $k^{\circ}/k^{\circ\circ}$ is zero.
In general, it seems to be proved by Thuillier.
Note that our proof is based on de Jong's alteration.

\begin{Corollary}
\label{cor:nef:approx:nef}
Let $X$ be a proper and normal variety over $k$ and let $L$ be an $\RR$-Cartier divisor on $X$.
Let $\XXX$ be a normal model of $X$ over $\Spec(k^{\circ})$
and let $\LLL$ be an $\RR$-Cartier on $\XXX$ with $\LLL \cap X = L$. We assume there is a sequence 
$\left\{ (\XXX_n, \LLL_n) \right\}_{n=1}^{\infty}$
with the following properties:
\begin{enumerate}
\renewcommand{\labelenumi}{(\arabic{enumi})}
\item
$\XXX_n$ is a normal model of $X$ over $\Spec(k^{\circ})$.

\item
$\LLL_n$ is a relatively nef $\RR$-Cartier divisor on $\XXX_n$ such that $\LLL_n \cap X = L$.

\item
$\lim_{n\to\infty} \mult_{w}(\LLL_n) = \mult_{w}(\LLL)$ for
all $w \in \DVal_k(\XXX)$.
\end{enumerate}
Then $\LLL$ is relatively nef.
\end{Corollary}

\begin{proof}
If $v$ is trivial, then the assertion is obvious, so that we assume that $v$ is non-trivial.
Clearly we may assume that there is a birational morphism $\nu_n : \XXX_n \to \XXX$ over $k^{\circ}$.
By using Chow's lemma, we have a birational morphism $\mu : \XXX' \to \XXX$ over $k^{\circ}$ such that
$\XXX'$ is projective over $\Spec(k^{\circ})$.
Let $X'$ be the generic fiber of $\XXX' \to \Spec(k^{\circ})$.
Let $\XXX'_n$ be the normalization of the main component of $\XXX_n \times_{\XXX} \XXX'$, and
let $\mu_n : \XXX'_n \to \XXX_n$ 
be the induced morphism.
We set 
\[
\LLL' := \mu^*(\LLL),\quad
L':= \LLL' \cap X'
\quad\text{and}\quad
\LLL'_n := \mu_n^{*}(\LLL_n).
\]
Then $\XXX'_n$ is a model of $X'$ over $\Spec(k^{\circ})$ and
$\LLL'_n$ is a relatively nef $\RR$-Cartier divisor on $\XXX'_n$ such that $\LLL'_n \cap X' = L'$.
Moreover, 
\[
\mult_{w}(\LLL') = \mult_{w}(\LLL)
\quad\text{and}\quad
\mult_{w}(\LLL'_n) = \mult_{w}(\LLL_n)
\]
for all $w \in \DVal_k(\XXX)$.
Therefore, we may assume that $\XXX$ is projective.

Let $\AAA$ be a relatively nef and big Cartier divisor on $\XXX$.
As $L$ is nef on $X$, $\LLL + \epsilon \AAA$ is nef and big on $X$ for $\epsilon > 0$, so that,
by virtue of Theorem~\ref{thm:char:nef:big:alg:var},
it is sufficient to see that
\[
\mu_{\RR, w} (\LLL + \epsilon \AAA) = 0
\]
for all $w \in \DVal_k(\XXX)$ and $\epsilon > 0$.
Replacing $\XXX$ by a suitable birational model,
we may assume that there is a vertical prime divisor $\Gamma$ on $\XXX$ such that
$w = a \ord_{\Gamma}$ for some positive number $a$.
Let $\XXX_{\circ} = a_1 \Gamma_1 + \cdots + a_r \Gamma_r$ be be the
irreducible decomposition of 
the central fiber $\XXX_{\circ}$
of $\XXX \to \Spec(k^{\circ})$ as a Weil divisor. Renumbering
$\Gamma_1, \ldots, \Gamma_r$, we may set $\Gamma = \Gamma_1$.
Let $w_{\Gamma_i}$ be the additive valuation over $k$ arising from $\Gamma_i$.
Note that $w = w_{\Gamma_1}$.
For a positive number $\delta$,
there is $N$ such that
\[
\left\vert \mult_{w_{\Gamma_i}}(\LLL) - \mult_{w_{\Gamma_i}}(\LLL_n) \right\vert \leq a_i \delta
\]
for all $n \geq N$ and $i = 1, \ldots, r$.
Then $(\nu_n)_*(\LLL_n) - \delta \XXX_{\circ} \leq \LLL$ for $n \geq N$.
Therefore, as $\LLL_n - \delta (\XXX_n)_{\circ} + \epsilon \nu_n^*(\AAA)$ is relatively nef and big,
by (2), (6.2) and (7) in Proposition~\ref{prop:mu:basic},
{\allowdisplaybreaks
\begin{align*}
0 & \leq \mu_{\RR, w} (\LLL + \epsilon \AAA) \\
& \leq \mu_{\RR, w} ((\nu_n)_*(\LLL_n) - \delta \XXX_{\circ} + \epsilon \AAA) \\
& \qquad\qquad\qquad + \mult_{w}(\LLL - (\nu_n)_*(\LLL_n) + \delta \XXX_{\circ}) & (\text{$\because$ (2)}) \\
& = \mu_{\RR, w} ((\nu_n)_*(\LLL_n - \delta (\XXX_n)_{\circ} +\epsilon \nu_n^*(\AAA))) \\
& \qquad\qquad\qquad + \mult_{w}(\LLL - (\nu_n)_*(\LLL_n) + \delta \XXX_{\circ}) \\
& \leq  \mu_{\RR, w} (\LLL_n - \delta (\XXX_n)_{\circ} + \epsilon \nu_n^*(\AAA)) \\
& \qquad\qquad\qquad + \mult_{w}(\LLL - (\nu_n)_*(\LLL_n) + \delta \XXX_{\circ}) & (\text{$\because$ (6.2)})\\
& = \mult_{w}(\LLL - (\nu_n)_*(\LLL_n) + \delta \XXX_{\circ}) & (\text{$\because$ (7)})\\
& \leq \left| \mult_{w}(\LLL)  - \mult_{w} (\LLL_n)\right| + \delta \mult_{w}(\XXX_{\circ}) \leq 2 a_1 \delta
\end{align*}}
for $n \geq N$.
Thus $\mu_{\RR, w} (\LLL + \epsilon \AAA) = 0$ because $\delta$ is an arbitrary positive number.
\end{proof}

\ifmonog
\backmatter
\fi

\ifmonog
\Printindex{AdelDiv-Subject}{Subject Index}
\Printindex{AdelDiv-Symbol}{Symbol Index}
\fi

\ifpaper
\Printindex{AdelDiv}{Index}
\fi

\end{document}